\documentclass[10pt,english,a4paper]{article}
\usepackage[dvipsnames,svgnames]{xcolor}

\usepackage{fullpage}
\usepackage{authblk}

\usepackage[english]{babel}    
\usepackage{amsmath,amsthm}           
\usepackage[utf8]{inputenc}    
\usepackage{xspace}            
\usepackage{tikz}              
\usepackage{tikz-cd}            
\usepackage{amssymb}           
\usepackage{bbm}
\usepackage{mathtools}         
\usepackage{mathrsfs}          
\usepackage{stmaryrd}          
\SetSymbolFont{stmry}{bold}{U}{stmry}{m}{n}
\usepackage[expansion=false    
           ]{microtype}        
\usepackage[backref=page,      
           final=true,         
           pdfpagelabels,      
           colorlinks
           ]{hyperref}         
\usepackage[alphabetic,initials]{amsrefs}
\usepackage[capitalize]{cleveref}
\usepackage[shortcuts]{extdash}
\usepackage{ifthen}
\usepackage{adjustbox}
\usepackage{enumitem}

\usetikzlibrary{decorations.pathreplacing}

\usepackage[bbgreekl]{mathbbol}
\DeclareSymbolFontAlphabet{\mathbbm}{bbold}
\DeclareSymbolFontAlphabet{\mathbb}{AMSb}%

\allowdisplaybreaks

\newcommand{\roverline}[1]{\mathpalette\doroverline{#1}}
\newcommand{\doroverline}[2]{\overline{#1#2}}

\addto\captionsenglish{%
}

\newcommand{\X}{\mathsf X}
\newcommand{\Y}{\mathsf Y}

\newcommand{\NN}{\mathbb N}
\newcommand{\ZZ}{\mathbb Z}

\newcommand{\RR}{\mathbb R}
\newcommand{\Rp}{{\mathbb R}_{>0}}
\newcommand{\Rx}{{\mathbb R}^*}
\newcommand{\CC}{\mathbb C}

\newcommand{\I}{\mathrm i}
\newcommand{\E}{\mathrm e}
\newcommand{\D}{\mathrm d}
\newcommand{\id}{\mathrm{id}}

\newcommand{\abs}[2][]{#1\lvert#2#1\rvert}
\newcommand{\norm}[2][]{#1\lVert#2#1\rVert}

\DeclareFontFamily{U}{mathx}{}
\DeclareFontShape{U}{mathx}{m}{n}{<-> mathx10}{}
\DeclareSymbolFont{mathx}{U}{mathx}{m}{n}
\DeclareMathAccent{\widehat}{0}{mathx}{"70}
\DeclareMathAccent{\widecheck}{0}{mathx}{"71}

\newcommand*{\nb}{\nobreakdash}

\newcommand*{\conj}[1]{\overline{#1}}

\newcommand{\Algebra}[1]{\mathcal{#1}}

\newcommand{\Space}[1]{\underline{#1}}

\newcommand{\tensor}{\otimes}
\newcommand{\completedtensor}{\mathbin{\widehat{\tensor}}}

\newcommand{\argument}{\,\cdot\,}

\newcommand*{\defeq}{\mathrel{\vcentcolon=}}

\DeclareMathOperator{\Ad}{Ad} 
\DeclareMathOperator{\ad}{ad}

\newcommand{\lie}[1]{\mathfrak{#1}} 

\newcommand{\Schwartz}{\mathcal{S}}
\newcommand{\Smooth}{\mathcal{C}^\infty}
\newcommand{\SmoothCompactSupp}{\Smooth_{\mathrm c}}

\newcommand{\Cont}{\mathcal{C}}

\newcommand*{\Cst}{C^*}

\newcommand{\TG}{{\mathcal G}} 
\newcommand{\TGdil}{\mathcal G_D} 
\newcommand{\TGrep}{\mathcal G_R} 

\newcommand{\PPseu}{\mathbbm{\Psi}_\Gamma}
\newcommand{\Pseu}{\Psi_\Gamma}
\newcommand{\Symb}{\Sigma_\Gamma}
\renewcommand{\P}{\mathbb P}
\newcommand{\Q}{\mathbb Q}
\DeclareMathOperator{\Op}{Op}
\DeclareMathOperator{\Bound}{\mathbb{B}}
\DeclareMathOperator{\Comp}{\mathbb{K}}
\DeclareMathOperator{\ind}{ind}
\DeclareMathOperator{\ev}{ev}
\newcommand{\pr}{\mathrm{pr}}
\DeclareMathOperator{\singsupp}{singsupp}

\newcommand{\Sob}{H_\Gamma}
\newcommand{\ess}{\mathrm{Ess}}
\newcommand{\Ess}{\mathrm{Ess}_\Gamma}

\newcommand{\AlgebraicIso}{\Psi}
\newcommand{\IntegrationMap}{{\smallint}}
\newcommand{\OrbitMap}[2]{o^{#1}_{#2}}

\newcommand{\ChainNoArg}[3]{\mathrm{Ch}(#1)^{#2}_{#3}}
\newcommand{\Chain}[4]{\ChainNoArg{#1}{#2}{#3}(#4)}

\newcommand{\set}[3][]{#1\{#2\,#1|\,#3#1\}}
\newcommand{\simpleset}[2][]{#1\{#2#1\}}

\makeatletter
\def\@endtheorem{\endtrivlist}
\makeatother

\makeatletter
\renewenvironment{proof}[1][\proofname]{\par
	\pushQED{\qed}%
	\normalfont \topsep6\p@\@plus6\p@\relax
	\trivlist
	\item[\hskip\labelsep
	\textsc{#1\@addpunct{.}}]\ignorespaces
}{%
	\popQED\endtrivlist
}

\makeatother

\newtheorem{theorem}{Theorem}[section]
\newtheorem{lemma}[theorem]{Lemma}
\newtheorem{proposition}[theorem]{Proposition}
\newtheorem{corollary}[theorem]{Corollary}

\theoremstyle{definition}
\newtheorem{definition}[theorem]{Definition}

\newtheorem{example}[theorem]{Example}
\newtheorem{remark}[theorem]{Remark}

\makeatletter
\newcommand{\myitem}[1]{%
	\item[#1]\protected@edef\@currentlabel{#1}%
}
\makeatother

\newcommand{\refitem}[1]{\textit{\ref{#1}.)}}

\numberwithin{equation}{section}

\newcommand{\MyBoxWidth}{2.25}
\newcommand{\StartX}[1]{#1*\MyBoxWidth-2*\MyBoxWidth/3+\value{mybigboxcounter}*\mybigboxincrease}
\newcommand{\CenterX}[1]{#1*\MyBoxWidth-\MyBoxWidth/3+\value{mybigboxcounter}*\mybigboxincrease}
\newcommand{\EndX}[1]{#1*\MyBoxWidth+\value{mybigboxcounter}*\mybigboxincrease}

\newcommand{\convertType}[1]{\csname IsOfType#1\endcsname}

\newcounter{myboxcounter}
\newcounter{mylinecounter0}
\newcounter{mylinecounter1}
\newcounter{mylinecounter2}
\newcounter{mylinecounter3}
\newcounter{mylinecounter4}
\newcounter{mylinecounter5}
\newcounter{mylinecounter6}
\newcounter{mylinecounter7}
\newcounter{mylinecounter8}
\newcounter{mylineboxcounter0}
\newcounter{mylineboxcounter1}
\newcounter{mylineboxcounter2}
\newcounter{mylineboxcounter3}
\newcounter{mylineboxcounter4}
\newcounter{mylineboxcounter5}
\newcounter{mylineboxcounter6}
\newcounter{mylineboxcounter7}
\newcounter{mylineboxcounter8}
\newcounter{mybigboxcounter}
\newcommand{\resetboxes}{
	\setcounter{myboxcounter}{0}
	\setcounter{mylinecounter0}{0}
	\setcounter{mylinecounter1}{0}
	\setcounter{mylinecounter2}{0}
	\setcounter{mylinecounter3}{0}
	\setcounter{mylinecounter4}{0}
	\setcounter{mylinecounter5}{0}
	\setcounter{mylinecounter6}{0}
	\setcounter{mylinecounter7}{0}
	\setcounter{mylinecounter8}{0}
	\setcounter{mylineboxcounter0}{0}
	\setcounter{mylineboxcounter1}{0}
	\setcounter{mylineboxcounter2}{0}
	\setcounter{mylineboxcounter3}{0}
	\setcounter{mylineboxcounter4}{0}
	\setcounter{mylineboxcounter5}{0}
	\setcounter{mylineboxcounter6}{0}
	\setcounter{mylineboxcounter7}{0}
	\setcounter{mylineboxcounter8}{0}
	\setcounter{mybigboxcounter}{0}
}
\newcommand{\addtolineboxcounters}[1]{
	\addtocounter{mylineboxcounter0}{#1}
	\addtocounter{mylineboxcounter1}{#1}
	\addtocounter{mylineboxcounter2}{#1}
	\addtocounter{mylineboxcounter3}{#1}
	\addtocounter{mylineboxcounter4}{#1}
	\addtocounter{mylineboxcounter5}{#1}
	\addtocounter{mylineboxcounter6}{#1}
	\addtocounter{mylineboxcounter7}{#1}
	\addtocounter{mylineboxcounter8}{#1}
}

\newcommand{\convertLine}[1]{\csname linetemplatecommand#1\endcsname}

\newcommand{\mylineheight}[1]{0.18*#1}
\newcommand{\mybigboxincrease}{0.275}

\newcommand{\draftBoxNumber}[1]{
	\draw[anchor = center,blue] (\CenterX{\value{myboxcounter}},#1-1) node {\arabic{myboxcounter}};
}
\renewcommand{\draftBoxNumber}{}

\newcommand{\DrawBox}[6][0]{
	\stepcounter{myboxcounter}
	\addtolineboxcounters{#1}
	\draw (\StartX{\value{myboxcounter}},\mylineheight{\convertLine{#3}}-0.5) rectangle (\EndX{\value{myboxcounter}}+#1*\mybigboxincrease,\mylineheight{\convertLine{#4}}+0.5);
	\draw[anchor = center] (\CenterX{\value{myboxcounter}}+#1*\mybigboxincrease/2,\mylineheight{\convertLine{#3}}/2+\mylineheight{\convertLine{#4}}/2) node {$#2$};
	\foreach \myline / \mystyle in #5 {
		\draw[\convertType{\mystyle}] (\EndX{\value{mylinecounter\convertLine{\myline}}}-\value{mylineboxcounter\convertLine{\myline}}*\mybigboxincrease+#1*\mybigboxincrease,\mylineheight{\convertLine{\myline}}) -- (\StartX{\value{myboxcounter}},\mylineheight{\convertLine{\myline}});
	}
	\foreach \myline in #6 {
		\setcounter{mylinecounter\convertLine{\myline}}{\value{myboxcounter}}
		\setcounter{mylineboxcounter\convertLine{\myline}}{0}
	}
	\draftBoxNumber{\mylineheight{#3}}
	\addtocounter{mybigboxcounter}{#1}
}

\newcommand{\DrawEmptyBox}[1][0]{\stepcounter{myboxcounter}\addtocounter{mybigboxcounter}{#1}}
\newcommand{\DrawFinalLines}[1]{
	\stepcounter{myboxcounter}
	\foreach \myline / \mystyle in #1 {
		\draw[\convertType{\mystyle}] (\EndX{\value{mylinecounter\convertLine{\myline}}},\mylineheight{\convertLine{\myline}}) -- (\StartX{\value{myboxcounter}},\mylineheight{\convertLine{\myline}});
	} 
}

\newcommand{\DrawOneLineBox}[4][0]{
	\DrawBox[#1]{#2}{#3}{#3}{{#3/#4}}{{#3}}
}
\newcommand{\DrawTwoLineBox}[6][0]{
	\DrawBox[#1]{#2}{#3}{#5}{{#3/#4,#5/#6}}{{#3,#5}}
}
\newcommand{\DrawOneToTwoBox}[5]{
	\DrawBox{#1}{#4}{#5}{{#2/#3}}{{#4,#5}}
}
\newcommand{\DrawTwoToOneBox}[6]{
	\DrawBox{#1}{#2}{#4}{{#2/#3,#4/#5}}{{#6}}
}

\newcommand{\DrawLeftLine}[2]{
	\foreach \start / \end / \linetype in #2 {
		\draw[\convertType{\linetype}] (\EndX{\start},\mylineheight{\convertLine{#1}}) -- (\StartX{\end}-0.2,\mylineheight{\convertLine{#1}});
		\fill (\StartX{\end}-0.2,\mylineheight{\convertLine{#1}}) circle (3pt);
	} 
}
\newcommand{\DrawRightLine}[2]{
	\foreach \start / \end / \linetype in #2 {
		\draw[\convertType{\linetype}] (\EndX{\start}+0.2,\mylineheight{\convertLine{#1}}) -- (\StartX{\end},\mylineheight{\convertLine{#1}});
		\fill (\EndX{\start}+0.2,\mylineheight{\convertLine{#1}}) circle (3pt);
	} 
}

\newcommand{\AlignTo}[1]{(0,\mylineheight{\convertLine{#1}}*\myscale)}

\newcommand{\underbracedistance}{0.1}
\newcommand{\AddUnderbrace}[5][0]{
	\draw [decoration={brace, mirror},	decorate](\StartX{#4}-3*\underbracedistance-#1*\mybigboxincrease,\mylineheight{\convertLine{#3}}-0.5-\underbracedistance) -- (\EndX{#5}+3*\underbracedistance,\mylineheight{\convertLine{#3}}-0.5-\underbracedistance) 
	node [pos=0.5,anchor=north,yshift=-0.05cm] {$#2$};
}
\newcommand{\AddOverbrace}[5][0]{
	\draw [decoration={brace},	decorate](\StartX{#4}-3*\underbracedistance-#1*\mybigboxincrease,\mylineheight{\convertLine{#3}}+0.5+\underbracedistance) -- (\EndX{#5}+3*\underbracedistance,\mylineheight{\convertLine{#3}}+0.5+\underbracedistance) 
	node [pos=0.5,anchor=south,yshift=0.05cm] {$#2$};
}

\newenvironment{boxdiagram}[1]{\begin{tikzpicture}[scale=\myscale, baseline={([yshift=-.6ex]\AlignTo{#1})}]
		\resetboxes
}{\end{tikzpicture}}

\newcommand{\myscale}{.4}

\title{\textbf{Shubin calculi for actions of graded Lie groups}}

\author[1]{Eske Ewert}
\author[1]{Philipp Schmitt}
\affil[1]{\small Institute of Analysis, Leibniz University Hannover, Welfengarten 1, 30167 Hannover, Germany}
\date{}

\begin{document}

\maketitle
\vspace*{-.675cm}

\begin{abstract}
{\footnotesize
	In this article, we develop a calculus of Shubin type pseudodifferential operators on certain non-compact spaces, using a groupoid approach similar to the one of van Erp and Yuncken.
	More concretely, we consider actions of graded Lie groups on graded vector spaces
	and study pseudodifferential operators that generalize fundamental vector fields and multiplication by polynomials. 
Our two main examples of elliptic operators in this calculus are Rockland operators with a potential
and the generalizations of the harmonic oscillator to the Heisenberg group due to Rottensteiner--Ruzhansky.

Deforming the action of the graded group, we define a tangent groupoid which connects pseudodifferential operators to their principal (co)symbols.
We show that our operators form a calculus that is asymptotically complete.
Elliptic elements in the calculus have parametrices, are hypoelliptic, and can be characterized in terms of a Rockland condition.
Moreover, we study the mapping properties as well as the spectra of our operators on Sobolev spaces
and compare our calculus to the Shubin calculus on $\RR^n$ and its anisotropic generalizations.\par
} \end{abstract}

\setcounter{tocdepth}{3}
{\small \tableofcontents}

\section{Introduction} \label{sec:introduction}

An important approach to study elliptic partial differential equations is via pseudodifferential calculi.
They give a conceptual way to invert an elliptic operator up to a smoothing error,
i.e.\ to construct a parametrix.
Typically, in these calculi one starts with a class of symbols,
like the ones introduced by Kohn--Nirenberg and Hörmander \cites{KN65, Hoe85c}, which one then quantizes to operators.
For Hörmander symbols on a compact manifold, the smoothing error in the parametrix construction is a compact operator. 
Therefore elliptic operators are Fredholm,
and their index can be computed in terms of topological properties \cite{AS68}.
However, on a non-compact manifold the calculus is not suitable for index theory since smoothing operators 
are no longer compact.

\paragraph{Shubin calculus}
One way to remedy this on \(\RR^n\) is to also control the behaviour of symbols in space direction.
This leads to so-called \emph{Shubin symbol classes} \cite{Shu87}, 
which were introduced at the same time also by Helffer and Robert \cite{Hel84}.
The space of Shubin symbols $\Gamma^m(\RR^n)$ of order $m \in \RR$ consists of
smooth functions \(p\in \Smooth(\RR^{2n})\) with the property that for all \(a,b\in\NN^n_0\) there exists \(C_{ab}>0\) such that for all \((x,\xi)\in\RR^{2n}\)
\begin{equation}
	\abs{\partial^a_x\partial^b_\xi p(x,\xi)}\leq C_{ab} (1+\norm{x}+\norm{\xi})^{m-\abs{a}-\abs{b}}.
\end{equation}
One easily verifies $\bigcap_{m \in \ZZ} \Gamma^m(\RR^n) = \Schwartz(\RR^n\times\RR^n)$
and that the corresponding operators have Schwartz functions as kernels, so that they are indeed compact.
For this reason, many results and techniques carry over from compact manifolds: 
e.g.\ elliptic operators are Fredholm and one obtains index theorems \cites{Fed70,ENN96} and a Weyl law \cite{HR82}.
Prominent examples of elliptic Shubin operators are the creation and annihilation operator and the harmonic oscillator.

Note however, that Shubin symbols are not invariant under general coordinate changes,
giving limitations to generalize them to other non-compact manifolds than \(\RR^n\).
There are also other ways to control symbols in space directions,
e.g.\ the SG-calculus \cites{Cor95, Par72} which can even be generalized to certain non-compact manifolds \cites{Mel94,Sch87}, 
the global calculus of \cite{NR10},
or, more generally, the Hörmander--Weyl calculus of \cite{Hoe85c}*{Section~18.5}.

\paragraph{Adapted calculi} To gain a better understanding of hypoelliptic operators like Hörmander's sum of squares, which are not elliptic in the standard sense, \cites{FS74,RS76,FS82} studied operators on certain nilpotent Lie groups called graded Lie groups. For these, hypoellipticity can be characterized by the Rockland condition \cites{Roc78,Bea77,HN78} which is based on the representation theory of the group. A graded Lie group $G$ is a  simply connected Lie group whose Lie algebra admits a grading 
\begin{equation} \label{eq:grading}
	\lie{g}=\bigoplus_{j=1}^r\lie{g}_j\qquad\text{ with }[\lie{g}_j,\lie{g}_k]\subseteq \lie{g}_{j+k}.\end{equation}
One important example is the \((2n+1)\)-dimensional Heisenberg group \(H_n\). Its Lie algebra is generated by \(X_1,\ldots, X_{2n+1}\) with \([X_i,X_{n+i}]=X_{2n+1}\) for \(i=1,\ldots,n\). A grading is given by \(\lie{g}_1=\mathrm{span}\{X_1,\ldots,X_{2n}\}\) and \(\lie{g}_2=\RR X_{2n+1}\). 
The grading yields a natural way to redefine the order of differential operators. Namely, identifying \(X\in\lie{g}\) with the corresponding left-invariant differential operator on $G$, the order of \(X\in\lie{g}_j\) is defined to be~\(j\). This notion of order is the cornerstone of various pseudodifferential calculi defined for these groups \cites{Tay84,CGGP92,FR16,FKF20} or also for certain classes of manifolds, like contact manifolds \cites{BG88, Pon08} or filtered manifolds \cites{Mel82,vEY19}.

In \cite{vEY19} an elegant, coordinate-free approach to a calculus on filtered manifolds is introduced. It is built on Connes' tangent groupoid \cite{Con94} and its generalizations for filtered manifolds \cites{vEY17,SH18,CP19c,Moh21}. Moreover, it uses fibred distributions \cites{AS68,LMV17} as well as a zoom action of \(\Rp\) on the tangent groupoid, whose close relation with classical pseudodifferential operators was shown in \cite{DS14}. Ellipticity in this calculus can be characterized using the Rockland condition on certain graded groups associated with the filtration \cites{DH22,CGGP92}.
This new approach lead to recent advances such as the definition of a scale of associated  Sobolev scales \cite{DH22}, a Weyl law and heat asymptotics \cite{DH20}, index theorems \cites{Moh22, Moh22u, GK22, Gof24} and the proof of the Helffer--Nourrigat conjecture \cites{AMY22}.
As van Erp and Yuncken's calculus generalizes operators with classical Hörmander symbols, many of these results are only valid for \emph{compact} filtered manifolds.

\paragraph{Objectives} In this article, we shall therefore investigate the question whether one can also define an adapted Shubin type calculus on graded Lie groups, similarly to passing from classical pseudodifferential operators on compact manifolds to Shubin operators on $\RR^n$.
There are at least two ways to achieve this, and they can be used to study different kinds of operators.
To see this, let us consider the harmonic oscillator \(-\Delta+\norm{x}^2\) on $\RR^n$ as a prototype of an elliptic Shubin operator.

\begin{enumerate}[wide=0pt]
	\myitem{(O1)} \label{item:example:rocklandWithPotential} The first approach is based on the fact that \(-\Delta\) is a constant coefficient, elliptic operator or, equivalently, a left-invariant Rockland operator on the Abelian group \(G=\RR^n\). Therefore, one might be interested in operators obtained by adding potentials to a Rockland operator on a graded Lie group. For example, Sublaplacians with potentials are studied in \cite{BC22} on noncompact Lie groups.

	\myitem{(O2)} \label{item:example:representationGroupoid} The second approach uses the fact that the 
	harmonic oscillator on \(\RR^n\) appears as the Schrödinger representation of the Sublaplacian of the Heisenberg group \(H_n\).
	Based on this observation, we consider images of a Rockland operator on \(\overline{G}\) in a certain representation,
	where \(\overline{G}\) is a bigger group relating to $G$ in the same way as $H_n$ to $\RR^n$ \cites{Dyn75,Fol94, Moh22}.
	This idea was used by Rottensteiner and Ruzhansky to define 
	(an)harmonic oscillators on the Heisenberg group in \cites{RR20,RR22}.
\end{enumerate}

\paragraph{Shubin groupoid}
In this article, we develop a general framework in which both approaches can be realized.
Let $G$ be a graded Lie group with a polynomial (right) action \(\theta^1\) on a graded vector space \(\X\).
Here, polynomial means that the map $\theta^1 \colon \X \times G \to \X$ is polynomial,
where $G$ is identified with the vector space $\lie g$ via the exponential map.
The goal is to construct a calculus of pseudodifferential operators acting on $\X$ which contains 
fundamental vector fields with respect to $\theta^1$ (with order determined by the grading of $G$) and multiplications by polynomials on $\X$ (with order determined by the grading of $\X$).
In our two examples above, \(\X\) is the group itself (viewed as a vector space with a potentially different grading),
but we also give other examples in which this is not the case.

For a graded group \(G\), the grading can also be characterized by the family
of Lie algebra homomorphisms $A_\lambda \colon \lie g \to \lie g$ with $\lambda \in \RR$ defined by \(A_\lambda(X)=\lambda^jX\) for \(X\in\lie{g}_j\). The exponentiated action 
$\alpha_\lambda \colon G \to G$ of the multiplicative group $\Rx \coloneqq \RR \setminus \simpleset{ 0 }$ on $G$ is called the \emph{dilation action}. 
Similarly, the grading on $\X$ is defined by dilations $\beta_\lambda$.

In order to construct a tangent groupoid, we shall deform $\theta^1$ with respect to a parameter \(t\neq 0\) by
\begin{equation*}
	\theta_v^t(x)=\beta_t\circ \theta^1_{\alpha_t(v)}\circ \beta_{t^{-1}}(x) \qquad\text{for \(v\in G\) and \(x\in \X\).}
\end{equation*}
If the family \((\theta^t)_{t \neq 0}\) extends to \(t=0\), we say that $\theta^1$ defines a \emph{Shubin action}.
In this case, we assemble the family of actions $\theta^t$ into an action of $G$ on $\X \times \RR$
and define the \emph{Shubin tangent groupoid} as the action groupoid $\TG = (\X \times \RR) \rtimes G$.
More concretely, its unit space is $\X \times \RR$, the range and source maps are $r(x,t,v) = (x,t)$ and $s(x,t,v) = (\theta^t_v(x),t)$ and the multiplication is determined by $(x,t,v) \cdot (\theta^t_v (x), t, w) = (x,t,vw)$.
It can be equipped with a Shubin zoom action of \(\Rp\) by \(\tau_\lambda(x,t,v)=(\beta_{\lambda^{-1}}(x),\lambda^{-1}t,\alpha_\lambda(v))\) for \(\lambda>0\), \(x\in \X\), \(t\in\RR\) and \(v\in G\). 
In contrast to the tangent groupoid of a general (filtered) manifold, this groupoid is particularly easy as it can be globally described as an action groupoid. 

\paragraph{Fibred distributions}
When trying to apply the groupoid approach of \cite{vEY19} to the Shubin tangent groupoid, 
one difficulty arises due to the non-proper support of Shubin operators. For this reason, one has to use other distribution spaces as \cites{AS68,LMV17,vEY19,AMY21}. Namely, we use the space of slowly increasing functions \(\Algebra{O}_M(\RR^k)\)
(smooth functions on $\RR^k$ with all derivatives bounded by a polynomial) and its dual \(\Algebra{O}_M'(\RR^k)\) to define an algebra of fibred distributions \(\Algebra O'_r(\TG) \cong \Algebra O_M(\X \times \RR) \completedtensor \Algebra O_M'(G)\) for polynomial action groupoids. 
Note that a distribution $u\in\Algebra O'_r(\TG)$ may be viewed as a family of distributions $u_{x,t} \in \Algebra O_M'(G)$ on the range fibres, which varies smoothly along $\X \times \RR$ with controlled growth.
The convolution product of Schwartz functions $\Schwartz(\TG)$ on the groupoid can be extended to \(\Algebra O'_r(\TG)\).

For every $t \in \RR$ one can consider the action groupoid $X \rtimes^t G$ with respect to $\theta^t$.
Making use of $\Algebra O_M(\X \times \RR) \cong \Algebra O_M(\X) \completedtensor \Algebra O_M(\RR)$,
a range fibred distribution \(u \in \mathcal O_r'(\TG)\) can also be viewed as a family of range fibred distributions $u_t \in \Algebra O_r'(\X \rtimes^t G)$. In this case the convolution product becomes the pointwise convolution product, $(u*v)_t = u_t *_t v_t$.

For \(u_t\in\Algebra{O}'_r(\X\rtimes^t G)\), there is a corresponding operator \(\Op_t(u_t)\colon\Algebra O_M(\X)\to\Algebra O_M(\X)\).
The map \(\Op_t \colon \Algebra{O}'_r(\X\rtimes^t G) \to \mathcal L(\Algebra O_M(\X),\Algebra O_M(\X))\) is a representation, i.e. $\Op_t(u_t *_t v_t) = \Op_t(u_t) \circ \Op_t(v_t)$.
However, when the action \(\theta^t\) is not free, $\Op_t$ is in general not injective,
so that the same operator could be described by different range fibred distributions.
In our two main examples $\Op_t$ is injective for $t \neq 0$ and one can identify fibred distributions with the corresponding operators.

The definition of fibred distributions and the convolution product is given for general polynomial actions on vector spaces and thus not limited to the setting of Shubin tangent groupoids. 
Therefore, they could also be employed to define other calculi, like an SG-type calculus. 

\paragraph{Pseudodifferential operators and distributions}

Similar to the approach of van Erp--Yuncken, we say that
an element \(\P_1\in\Algebra O'_r(\X\rtimes^1 G)\) is called a \emph{pseudodifferential distribution} of order \(m\) if it can be extended to a \(\P\in\Algebra{O}_r'(\TG)\) which is essentially homogeneous of order $m$ with respect to the Shubin zoom action, that is
\begin{equation*}
	{\tau_\lambda}_*\P-\lambda^m\P\in\Schwartz(\TG) \quad\text{for all }\lambda>0.
\end{equation*} 
We denote the space of pseudodifferential distributions of order \(m\) by $\Pseu^m$.
The corresponding operator $\Op_1(\P_1) \colon \mathcal O_M(\X) \to \mathcal O_M(\X)$ is called a \emph{pseudodifferential operator} of order~$m$.
 
Following the usual tangent groupoid philosophy, one can interpret this construction as deforming a pseudodifferential distribution $\P_1$ to its principal cosymbol $[\P_0] \in \Algebra O_r'(\X \rtimes^0 G) / \Schwartz(\X \rtimes^0 G)$.
This is achieved by scaling away lower order terms using the zoom action.

Let us describe the tangent groupoid and pseudodifferential operators more concretely for the two main examples from above.
\begin{enumerate}[wide=0pt]
	\myitem{(O1)}\label{item:dd-calculus} Let \(\X=G\) be a graded Lie group with dilations $\alpha$, \(\theta^1(x,v)=x\cdot v\) and \(\beta\) be group dilations, possibly different from $\alpha$.
	For this reason, we call this the \emph{double dilation groupoid}.
	Since \(\theta^t_v(x)=x \cdot \beta_t(\alpha_{t}(v))\),
	the smooth structure of the double dilation groupoid differs from the tangent groupoid of \(G\) used to obtain a Hörmander type calculus. 
	
	For \(G=\RR^n\) with dilations \(\alpha_\lambda(v)=\beta_\lambda(v) = \lambda v\), one sees that under Fourier transform of the distributions in the \(v\)-direction at \(t=0\), the zoom action transforms into \(\widehat\tau_\lambda f(x,\xi)=f(\lambda x,\lambda\xi)\) for \(f\in C^\infty(\RR^{2n})\).
	This is precisely the action used to define homogeneity of symbols in the Shubin calculus,
	and indeed we show that one recovers the classical Shubin calculus in this case. 
	For different dilations \(\alpha,\beta\) on $\RR^n$ one obtains the non-isotropic calculi considered in \cites{BN03,NR10}, allowing one to study, for example, the anharmonic oscillators in \cite{CDR21}.
	
	When \(G\) is the Heisenberg group \(H_1\), denote by \(x_1,x_2,x_3\) exponential coordinates with respect to the basis \(X_1,X_2, X_3\) of \(\lie{g}\) as above.
	Consider $\alpha_\lambda(x_1,x_2,x_{3})=\beta_\lambda(x_1,x_2,x_{3})= (\lambda x_1, \lambda x_{2}, \lambda^2 x_{3})$ for $\lambda >0$.
	Then the polynomials \(x_1\) and \(x_{2}\) have order \(1\) as pseudodifferential operators whereas \(x_{3}\) has order \(2\) and similarly the left-invariant differential operators $X_1$ and $X_2$ have order \(1\) whereas $X_3$ has order \(2\).
		
	\myitem{(O2)} 
	Let $G$ be a graded Lie group with dilations $\alpha$
	and $\X = \lie g$ be the vector space with the opposite grading $\X_j = \lie g_{r-j}$ and dilations $\beta$
	(here $r \in \NN$ is the maximal integer with $\lie g_r \neq \simpleset{0}$ in \eqref{eq:grading}). 
	The definition of $\theta^1$ is more involved, but for $t = 0$ it deforms to the adjoint action of $G$ on $\lie g$.
	
	For $G = \RR^n$, the groupoid $C^*$-algebra of this Shubin tangent groupoid is isomorphic to the group $C^*$-algebra of the Heisenberg group $H_n = \smash{\roverline{\RR^n}}$.
	For general $G$, these two $C^*$-algebras are still related, even though they might not coincide on the nose.
	We call this Shubin tangent groupoid the \emph{representation groupoid} of $G$.
	
	For the Heisenberg group $G = H_1$ the order of polynomials in the corresponding calculus differs from the order of the double dilation groupoid.
	The polynomials \(x_1\) and \(x_2\) have now order \(2\) and \(x_{3}\) has order~\(1\).
	The orders of the fundamental vector fields of $X_1$, $X_2$ and $X_3$ are still $1$, $1$ and $2$.
\end{enumerate}

\paragraph{Properties of the calculus} Let us summarize the most important properties of our calculus:
\begin{enumerate}
	\item For every $m \in \RR$, there is a short exact sequence 
	$0\to\Pseu^{m-1} \to \Pseu^m \to \Symb^m\to 0$
	where the right map is the principal symbol map and $\Symb^m$ is a certain subspace of $\Algebra O_r'(\X \rtimes^0 G) / \Schwartz(\X \rtimes^0 G)$ consisting of essentially $m$-homogeneous elements.
	\item The convolution of $\P_1\in\Pseu^\ell$ and $\Q_1 \in \Pseu^m$ belongs to $\Pseu^{\ell+m}$ and the convolution is compatible with the principal cosymbol map.
	\item Every element of $\Pseu^m$ is a two-sided multiplier of $\Schwartz(\X \rtimes^1 G)$, i.e.\ $\P_1 *_1 f, f *_1 \P_1 \subseteq \Schwartz(\X \rtimes^1 G)$ for all $f \in \Schwartz(\X \rtimes^1 G)$ and $\P_1 \in \Pseu^m$. 
	\item There is an involution $^*$ on $\Pseu^m$ satisfying $(\P_1 *_1 \Q_1)^* = \Q_1^* *_1 \P_1^*$ for all $\P_1\in\Pseu^\ell$ and $ \Q_1 \in \Pseu^m$.
	\item Schwartz functions are the residual class of the calculus,
	meaning that $\bigcap_{m \in \ZZ} \Pseu^m = \Schwartz(\X \rtimes^1 G)$.
	\item The calculus is asymptotically complete.
	\item\label{item:parametrix} An elliptic $\P_1 \in \Pseu^m$ has a parametrix in $\Pseu^{-m}$, i.e.\ an inverse up to $\Schwartz(\X \rtimes^1 G)$.
\end{enumerate}
In \refitem{item:parametrix}, an element $\P_1 \in \Pseu^m$ is called elliptic if its principal cosymbol $[\P_0]$ is invertible (up to Schwartz functions in the convolution algebra $\Algebra O_r'(\X \rtimes^0 G)$).

\paragraph{Property (R)}
The ellipticity of $\P_1 \in \Pseu^m$ (or equivalently invertibility of $[\P_0]$)
can be characterized by a Rockland condition, if the calculus satisfies a certain assumption.
Namely, we say that property (R) is satisfied if the action $\theta^0$ of $G$ on $\X$ is by linear maps.
In this case, $G$ acts also on the dual $\X^*$ of $\X$ and we may consider the group $\X^* \rtimes^0 G$.
Then $\P_0$ may be viewed as a distribution on $\X^* \rtimes^0 G$ via inverse Fourier transform
and this identification is compatible with the respective convolutions.
By results of \cite{CGGP92}, a distribution on $\X^* \rtimes^0 G$ is invertible up to Schwartz functions if and only if the Rockland condition on \(\X^* \rtimes^0 G\) is satisfied,
i.e.\ if the principal symbol becomes invertible in the non-trivial representations of this group.
 
For the double dilation groupoid the action $\theta^0$ is trivial, hence property (R) is fulfilled and $\X^* \rtimes^0 G = \X^* \times G$. Hence, representations of this group are obtained from a point evaluation and a representation of $G$.
For the representation groupoid, $\theta^0$ is the adjoint action, so that property (R) is fulfilled. But now the principal symbol group $\X^* \rtimes^0 G = \lie g^* \rtimes G$ is the semidirect product with respect to the coadjoint action.
 
We show that the respective generalizations of the harmonic oscillators, i.e.
Rockland operators with suitable potentials for the double dilation groupoid or the harmonic oscillators of Rottensteiner--Ruzhansky for the representation groupoid, are indeed elliptic in 
the corresponding calculus.
 
\paragraph{Mapping properties}
We refer to the technical assumption that the shear map $\Theta^1 \colon \X \times G \to \X \times \X$, $\Theta^1(x,v) = (x, \theta^1(x,v))$
has a polynomial inverse as property (P). In particular, this implies that the action $\theta^1$ is free and transitive and that $\Op_1$ is injective.
Under properties (P) and (R) the following holds true:
\begin{enumerate}
	\item For $\P_1 \in \Pseu^m$ with $m < 0$, $\Op_1(\P_1)$ is compact on $L^2(\X)$.
	\item For $\P_1 \in \Pseu^m$ with $m \leq 0$, $\Op_1(\P_1)$ is bounded on $L^2(\X)$.
	\item For every $s \in \RR$, there exists an elliptic operator $P_s \in \Op(\Pseu^s)$,
	which can be used to define \emph{Sobolev spaces} \(\Sob^s(\X)=\set{u\in\Schwartz'(\X) }{ P_su\in L^2(\X)}\).
	\item If $\P_1 \in \Pseu^m$ for $m \in \RR$, then $\Op_1(\P_1) \colon \Sob^s(\X) \to \Sob^{s-m}(\X)$ is bounded for every $s \in \RR$.
	\item If $\P_1 \in \Pseu^m$ is elliptic for $m \in \RR$, then $\Op_1(\P_1) \colon \Sob^s(\X) \to \Sob^{s-m}(\X)$ is Fredholm for every $s \in \RR$.
	\item If $\P_1 \in \Pseu^m$ is elliptic for $m > 0$ and $\Op_1(\P_1)$ is formally self-adjoint on $\Schwartz(\X)$, then $\Op_1(\P_1)$ has discrete spectrum. 
\end{enumerate}
In a sequel to this article, we intend to compute the index of the Fredholm operator $\Op_1(\P_1)$ for elliptic $\P_1$ 
in terms of its principal (co)symbol.

\paragraph{Outline}
This article is organized as follows. In \cref{sec:graded-lie-groups} we summarize known results from the literature concerning graded Lie groups, the Rockland condition and essentially homogeneous distributions.
All other sections of this article contain new results. In \cref{sec:action-groupoids-convolution-algebra} we define the convolution algebra \(\Algebra O'_r(\X\rtimes G)\) of fibred distributions for a polynomial action of a graded group \(G\) on \(\X\). \cref{sec:shubin-tangent-groupoid} introduces abstract Shubin tangent groupoids and our two main examples. In \cref{sec:calculus} a pseudodifferential calculus using the approach of van Erp--Yuncken is defined for Shubin tangent groupoids. In particular, the principal symbol, the \(^*\)-algebra structure, asymptotic completeness and parametrices for elliptic pseudodifferential distributions are discussed.
In \cref{sec:PR} we characterize ellipticity in terms of Rockland conditions, construct a scale of Sobolev spaces and investigate the Fredholm and spectral properties of these Shubin type pseudodifferential operators.
In \cref{sec:comparison} we compare our calculus for \(G=\RR^n\) to the usual Shubin calculus and the anisotropic calculus.
Finally, \cref{sec:homogeneous_distributions} compiles known facts about homogeneous and essentially homogeneous distributions on graded groups. 

\paragraph{Notation} Throughout this article, we write
$\NN = \simpleset{1,2,3,\dots}$ and $\NN_0 = \simpleset{0,1,2,3,\dots}$ for the natural numbers without and with $0$, 
$\RR$ for the real numbers, $\Rp = \set{x \in \RR}{x > 0}$, and $\Rx = \RR \setminus \simpleset 0$.
Since (graded) groups occur mainly in action groupoids, we think of their elements as arrows or vectors and therefore denote them by $v$ and $w$ throughout this article.

\paragraph{Acknowledgements}
The authors would like to thank Ryszard Nest for many inspiring discussions and for suggesting a topic to us that lead to the present article 
and Elmar Schrohe and Robert Yuncken for many insightful conversations.
They are also grateful to Magnus Goffeng, Omar Mohsen, David Rottensteiner, Michael Ruzhansky and Bernhard Helffer for their kind remarks which lead us to generalize our first results to also cover the representation groupoid. They would also like to thank the referee for their helpful suggestions. 

\section{Graded Lie groups}\label{sec:graded-lie-groups}

In this section the definition and some basic properties of graded Lie groups are recalled. 
For a more detailed introduction we refer the reader to \cites{FS82,FR16}. 
All our Lie groups and Lie algebras shall be defined over the field $\mathbb R$. 

\begin{definition}
	A \emph{graded Lie group} is a simply connected Lie group \(G\) whose Lie algebra~\(\lie{g}\) is graded, that is \(\lie{g} = \bigoplus_{j=1}^r {\lie{g}_j},\) with \([X,Y]\in \lie{g}_{j+k}\) for all \(X\in\lie{g}_j\) and \(Y\in\lie{g}_k\). Here, one sets \(\lie{g}_j=\{0\}\) for \(j>r\). 
	The graded Lie group is called \emph{stratified} if \(\lie{g}_1\) generates \(\lie{g}\) as a Lie algebra. 
\end{definition}
Here, and in the following simply connected Lie groups are in particular assumed to be connected. 

\begin{example}\label{ex:heisenberg-algebra}
	The \((2n+1)\)-dimensional \emph{Heisenberg Lie algebra} is generated by \(X_1,\ldots,X_{2n+1}\) such that \([X_j,X_{n+j}]=-[X_{n+j},X_j] = X_{2n+1}\) for all \(j=1,\ldots,n\), while all other commutators of generators vanish. Consider the grading for which $X_1,\ldots,X_{2n}$ have degree $1$ and $X_{2n+1}$ has degree $2$.
	The corresponding simply connected Lie group is called the \emph{Heisenberg Lie group} $H_n$ and is as a space \(\RR^{2n+1}\)
	with the group product 
	\begin{equation*}
		(v_1,\ldots,v_{2n+1})\cdot(w_1,\ldots,w_{2n+1})
		= \biggl(
		v_1+w_1, \ldots, v_{2n}+w_{2n}, v_{2n+1}+w_{2n+1}+\tfrac{1}{2}\sum_{j=1}^n (v_jw_{n+j} -  w_j v_{n+j})
		\biggr).
	\end{equation*}
\end{example}
The condition on the Lie bracket implies that graded Lie groups are nilpotent.  In particular, the exponential map \(\exp\colon\lie{g}\to G\) is a diffeomorphism and the group law on \(G\) is determined by the Lie bracket on~\(\lie{g}\) using the Baker--Campbell--Hausdorff formula \cite{CG90}*{Thm.~1.2.1}.  
Furthermore, graded Lie groups can be always realized as subgroups of upper triangular matrices \cite{CG90}*{Cor.~1.2.3}.

\subsection{Dilations and homogeneity}\label{sec:dilations}

A grading on a (finite dimensional) Lie algebra can equivalently be described by an integer family of dilations:

\begin{definition}\label{def:dilation}
	Let \(\lie{g}\) be a Lie algebra of dimension $n \in \mathbb N$.
	A family \(\{A_\lambda\}_{\lambda >0}\) of Lie algebra homomorphisms \(A_\lambda\colon \lie{g}\to\lie{g}\) is called an \emph{integer family of dilations} if there is a diagonalizable linear map \(D\colon\lie{g}\to\lie{g}\) with eigenvalues \(1\leq q_1\leq q_2\leq\ldots\leq q_n\in\NN\) such that \(A_\lambda=\mathrm{Exp}(\ln(\lambda)D)\) for all \(\lambda>0\). The eigenvalues~\(q_j\in\NN\) are called the \emph{weights} of \(A\).
\end{definition}

\begin{example}
	For a graded Lie group $G$ with Lie algebra $\lie g$, \(A_\lambda(X)=\lambda^jX\) for \(X\in\lie{g}_j\) and \(\lambda>0\) defines an integer family of dilations, called  \emph{standard dilations}.
	Conversely, given an integer family of dilations on a Lie algebra $\lie g$,
	we can fix a basis of eigenvectors \(\{X_1,\ldots,X_n\}\) of \(D\) such that \(A_\lambda(X_j)=\lambda^{q_j}X_j\), which we refer to as a \emph{standard basis}.
	Defining the degree of $X_j$ to be $q_j$ we obtain a grading on $\lie g$,
	which turns the simply connected Lie group integrating $\lie g$ into a graded Lie group.
\end{example}
More generally, one can also consider dilations with positive real weights.
Groups equipped with such dilations are called \emph{homogeneous groups}
in the literature and slightly more general than graded groups \cite{FR16}*{Ex.\ 3.1.11}.
However, we shall only be concerned with integer dilations in this article 
and refer to those simply as dilations.

Note that an integer family of dilations \(\{A_\lambda\}_{\lambda>0}\) defines an \(\Rp\)-action \(A\) on \(\lie{g}\) by Lie algebra automorphisms. The integrated action of \(\Rp\) on the simply connected Lie group $G$ is by Lie group automorphisms. Denote it by \(\alpha\). 

Define the \emph{homogeneous dimension} of \(G\) with respect to \(\alpha\) by \(Q(\alpha)=q_1+\ldots+q_n\). 
Given an integer family of dilations,
fix a standard basis \(\{X_1,\ldots,X_n\}\).
We often identify \(\RR^n\) with \(G\) under the map \((v_1,\ldots,v_n)\mapsto \exp(v_1X_1+\ldots+v_nX_n)\),
referred to as \emph{standard coordinates}.
In particular, \(0\in G\) denotes the unit element and we occasionally write $-v$ for the inverse of $v$.
Note that the coordinates for different choices of bases are related by a linear map,
allowing us to use this identification to define spaces of functions like polynomials or Schwartz functions $\Schwartz(G)$ on \(G\).

\begin{remark} \label{rem:dilationsExtend}
	Integer dilations on a Lie group are given by 
	$\alpha_\lambda(v_1, \dots, v_n) = (\lambda^{q_1} v_1, \dots, \lambda^{q_n} v_n)$ in standard coordinates on $G$.
	This formula still makes sense for $\lambda \in \RR$, and then defines a smooth map
	$\RR \times G \to G$ which restricts to an action of the (multiplicative, non-connected) group $\RR^* \coloneqq \RR \setminus \{0\}$. By abuse of notation, we still denote the extended map by $\alpha_\lambda$. 
\end{remark}

\begin{definition} \label{def:homogeneousLength}
	For a multi-index \(k\in\NN^n_0\) define its \emph{$\alpha$-length} 
	by \[[k]_\alpha=k_1 q_1+\ldots +k_nq_n.\]
	We use the usual multiindex notation, in particular $v^k = v_1^{k_1} \dots v_n^{k_n}$.
\end{definition}

\begin{proposition}[\cite{FS82}*{p.~23}]\label{res:triangular}
	With respect to a standard basis \(\{X_1,\ldots,X_n\}\) there are constants \(c_{j,a,b} \in \mathbb R\)
	such that for all \(v,w\in G\) and \(j=1,\ldots,n\)
	\begin{equation}\label{eq:triangular}
		(v \cdot w)_j = v_j+w_j+\sum_{\substack{a,b\in\NN^n_0\setminus\{0\}\\ [a]_\alpha+[b]_\alpha=q_j}} c_{j,a,b} v^a w^b.
	\end{equation}
\end{proposition}
This triangular group law implies
that the Lebesgue measure on \(\RR^n\) yields a bi-invariant Haar measure on \(G\) 
which we denote by \(\D v\). 

Dilation actions allow one to define a new notion of order for the differential operators and polynomials, which we will use to define a Shubin-type differential calculus.
\begin{definition} \label{def:homogeneousOrder}
	Let \(m\in\RR\). A function \(f\colon G\setminus\{0\}\to \CC\) is \emph{homogeneous of degree \(m\)} (with respect to \(\alpha\)) if \(\alpha_\lambda^*(f)=f\circ\alpha_\lambda=\lambda^m f\) for all~\(\lambda>0\). A (left invariant) differential operator \(D\) on \(G\) is \emph{homogeneous of degree \(m\)} (with respect to \(\alpha\)) if \(D(f\circ\alpha_\lambda)=\lambda^m(Df)\circ\alpha_\lambda\) for all \(\lambda>0\) and \(f\in \Smooth(G)\).
\end{definition}
The Haar measure is \(Q(\alpha)\)-homogeneous with respect to the dilations, in the sense that
\begin{equation} \label{eq:haar_measure_homogeneity}
	\int_G f(\alpha_\lambda(v))\D v = \lambda^{-Q(\alpha)}\int_G f(v)\D v.
\end{equation}
Recall that \(X\in\lie{g}\) defines a left-invariant differential operator on $G$ by
\begin{equation}\label{eq:left-invariant-differential-operator}
	Xf(v)=\frac{\D}{\D s} \Big|_{s=0} f(v\cdot\exp(sX)).
\end{equation}
For \(X\in\lie{g}_j\) the corresponding left-invariant differential operator is homogeneous of degree \(q_j\).
\begin{example}
	Let \(k\in\NN^n_0\). Using the standard coordinates on \(G\) introduced above, the polynomial \(v^k=v_1^{k_1}\cdots v_n^{k_n}\) is homogeneous of degree \([k]_\alpha\). Likewise, the left invariant differential operator \(X^k=X_1^{k_1}\cdots X_n^{k_n}\) is homogeneous of degree  \([k]_\alpha\).
\end{example}
In the following it is handy to introduce a homogeneous quasi-norm on \(G\) as defined in \cite{FS82}*{p.~8}.

\begin{definition}\label{def:quasi-norm}
	Fix a common multiple \(q\) of the weights \(q_1,\ldots,q_n\). We define a \emph{homogeneous quasi-norm} on \(G\) by
	\begin{equation}\label{eq:quasi-norm}\norm{v}_\alpha=\norm{(v_1,\ldots,v_n)}_\alpha
		=\biggl(\sum_{j=1}^n v_j^{\frac{2q}{q_j}}\biggr)^{\!\frac 1 {2q}}\quad\text{for }v\in G.
	\end{equation}
	Moreover, define \(\langle v\rangle_\alpha =(1+\norm{v}_\alpha^{2q})^\frac{1}{2q}\).
\end{definition}
Note that \(\norm{\alpha_\lambda(v)}_\alpha=\lambda\norm{v}_\alpha\) for all \(\lambda>0\) and that \(v\mapsto\norm{v}_\alpha^{2q}\) is a \(2q\)-homogeneous polynomial with respect to \(\alpha\).
The general definition of a homogeneous quasi-norm can be found in \cite{FR16}*{Definition~3.1.33}
and all homogeneous quasi-norms are equivalent \cite{FR16}*{Proposition~3.1.35}.
We may therefore freely change between different homogeneous quasi-norms up to a multiplicative constant, 
and assume that they are of the form \eqref{eq:quasi-norm} if necessary.

\begin{lemma}\label{res:norm_estimates} Let \(G\) be a graded Lie group equipped with the homogeneous quasi-norm \eqref{eq:quasi-norm}. Then the following estimates hold:
	\begin{enumerate}
		\item\label{item:x^alpha} \(\abs{v^k}\leq\norm{v}_\alpha^{[k]_\alpha}\) for all \(k\in\NN^n_{0}\) and $v \in G$,
		\item\label{item:derivatives-mult} there is a constant \(D>0\) such that for all \(k\in\NN^n_0\), \(v,w\in G\) and \(j=1,\ldots,n\) 
		\[\abs{\partial^k_w(v\cdot w)_j}\leq D\langle v\rangle_\alpha^{Q(\alpha)}\langle w\rangle_\alpha^{Q(\alpha)},\]
		\item\label{item:triangle-inequality} there is a constant \(\gamma\geq 1\) such that \(\norm{v \cdot w}_\alpha\leq \gamma(\norm{v}_\alpha+\norm{w}_\alpha)\) for all \(v,w\in G\).
	\end{enumerate}
\end{lemma}

\begin{proof}By definition of the homogeneous quasi-norm  \(\norm{v}_\alpha^{q_j}\geq \abs{v_j}\) holds for \(j=1,\ldots,n\), this implies~\refitem{item:x^alpha}.  The estimate \refitem{item:derivatives-mult} can be deduced from the polynomial group law \eqref{eq:triangular} and \refitem{item:x^alpha}. The analogue of the triangle inequality in \refitem{item:triangle-inequality} is shown in \cite{FS82}*{1.8}.
\end{proof}

\begin{lemma}[\cite{FS82}*{Corollary~1.17}]\label{res:integrable}
	The function \(v\mapsto\langle v\rangle^k_\alpha\) is in \(L^1(G)\) if and only if \(k<-Q(\alpha)\).
\end{lemma}

\subsection{Left-invariant differential operators and the Rockland condition}\label{sec:diff-op-rockland}
One can identify the (complex) universal enveloping algebra~\(\lie{U}(\lie{g})\) of \(\lie{g}\) with the left-invariant differential operators (with complex coefficients) on \(G\), see for example \cite{FR16}*{Section~1.3}.
Fix a standard basis \(X_1,\ldots,X_n\) of \(\lie{g}\) as before.
By the Poincar\'e--Birkhoff--Witt Theorem every left-invariant differential operator can be uniquely written as 
\(\sum_{k \in \mathbb N_0^n} c_k X^k=\sum_{k \in \mathbb N_0^n} c_k X_1^{k_1}\cdots X^{k_n}_n\) 
with \(c_k\in\CC\) non-zero only for finitely many $k$.

The dilations \(A\) induce a grading on \(\lie U(\lie g)\) so that \(\lie{U}^m(\lie g)\) for \(m\in\NN_0\) corresponds precisely to all (with respect to \(\alpha\)) \(m\)-homogeneous left-invariant differential operators. Given a standard basis one has
\begin{align*}
	\lie{U}^m(\lie g)=\set[\Big]{{\sum}_{[k]_\alpha= m} c_k X^k}{c_k\in\CC}.
\end{align*}
A left-invariant homogeneous differential operator \(D\) on the Abelian group \(G=\RR^n\) is a constant coefficients operator. It is hypoelliptic if and only if it is elliptic, that is its Fourier transform, i.e.\ its symbol, \(d(\xi)=\widehat D(\xi)\) is invertible for
 all \(\xi\neq 0\). Rockland \cite{Roc78} generalized this condition to the setting of left-invariant differential operators on graded groups using the representation theory of \(G\).
\begin{definition} \cite{FR16}*{Section 1.7}
	Let \(\pi\) be an irreducible unitary representation of \(G\) on a Hilbert space \(\mathcal H_\pi\). Its \emph{infinitesimal representation} \(\D \pi\) is a Lie algebra representation of \(\lie{g}\) on the space of smooth vectors \(\mathcal H^\infty_\pi\) given by
	\begin{equation}
		\D\pi(X) v= \lim_{s\to 0}\frac{\pi(\exp(sX))v-v}{s} \qquad\text{for }v\in\mathcal H^\infty_\pi.
	\end{equation}	
\end{definition}
By the universal property of the universal enveloping algebra this induces also a representation of \(\lie{U}(\lie{g})\) on \(\mathcal H^\infty_\pi\), also denoted by \(\D\pi\).
\begin{definition}\label{def:rockland-cond-diff}
	A left-invariant homogeneous differential operator \(D\) on \(G\) satisfies the \emph{Rockland condition} if  \(\D \pi(D)\) is injective on \(\mathcal H^\infty_\pi\) for all \(\pi\in\widehat G\setminus\simpleset{\pi_{\mathrm {triv}}}\). Here, \(\widehat G\) denotes the equivalence classes of irreducible unitary representations of \(G\) and $\pi_{\mathrm{triv}}$ denotes the trivial representation on $\CC$.
\end{definition}
A homogeneous left-invariant differential operator on \(G\) is hypoelliptic 
if and only if it satisfies the Rockland condition \cites{Bea77,HN79}.

\begin{example}\label{ex:standard-rockland}
	Let $G$ be any graded Lie group with dilations $\alpha$, weights \(q_1,\ldots,q_n\),
	and standard basis $X_1, \dots, X_n$.
	Let \(q\) be a common multiple of all weights. Then the $2q$-homogeneous operator
	\begin{equation*}
		R = \sum_{j=1}^n(-1)^{\frac{q}{q_j}}X_j^{\frac{2q}{q_j}}
	\end{equation*}
	satisfies the Rockland condition \cite{HN78}*{Lemma~6.2.1}. 	
\end{example}
\begin{example} \label{ex:stratified-rockland}
	The previous example can be adapted under the assumption that \(X_{j_1},\ldots, X_{j_k}\) for a subset \(\{j_1,\ldots,j_k\}\subset \{1,\ldots,n\}\) generate \(\lie{g}\) as a Lie algebra. In this case, for every common multiple \(q\) of \(q_{j_1},\ldots,q_{j_k}\) the following operator satisfies the Rockland condition \cite{FR16}*{Corollary~4.1.10}
	\begin{equation*}
		R = \sum_{i=1}^k(-1)^{\frac{q}{q_{j_i}}}X_{j_i}^{\frac{2q}{q_{j_i}}}.
	\end{equation*}
	In particular, for a stratified group, where \(\lie{g}_1=\bigoplus_{j=1}^k\RR X_j\) generates \(\lie{g}\), the \emph{Sublaplacian} \(\smash{-\sum_{j=1}^k X_j^2}\) satisfies the Rockland condition.
\end{example}
	More examples can be found in \cite{FR16}*{Section 4.1.2}.
\begin{remark}
	Later on, also Rockland conditions for non-differential operators were introduced, see \cites{Glo91,CGGP92}. The definition can be found in \cref{sec:rockland}.
\end{remark}

\subsection{Essentially homogeneous distributions}

A dilation action \(\alpha\) on a graded group \(G\)
also induces \(\Rp\)-actions on the spaces of Schwartz functions and tempered distributions by
\begin{align}\label{eq:dilations_on_schwartz}
	\alpha_\lambda^* f &= f \circ \alpha_\lambda && \text{for \(f\in\Schwartz(G)\) and \(\lambda>0\),} \\
	\langle{\alpha_\lambda}_*u,f\rangle &=\langle u,\alpha_\lambda^* f\rangle &&\text{for \(u\in\Schwartz'(G)\),  \(f\in\Schwartz(G)\) and \(\lambda>0\).}
	\label{eq:dilations_on_tempered}
\end{align}
Note that this definition is set up so that the Fourier transform of a symbol of order $m$ still has order $m$ (as a tempered distribution). In particular, there will be no shift between the order of pseudodifferential operators and their defining distributions later on.

\begin{remark}\label{rem:schwartz-embedding}
	In this article, whenever $\alpha$ is a map between spaces, then $\alpha^*$ shall denote the induced pullback map between function spaces given by \(\alpha^*f=f\circ\alpha\). Moreover, \(\alpha_*\) always means the map between the corresponding distribution spaces induced by duality. 
	
	A point where this might lead to confusions is when one embeds function spaces into distribution spaces.
	Let \(u\colon\Schwartz(G)\hookrightarrow \Schwartz'(G)\) be the map \(f\mapsto f \D v\), where \(\D v\) denotes the Lebesgue measure. Then, using our convention, \({\alpha_\lambda}_*(u(f))=\lambda^{-Q(\alpha)}u(\alpha_{\lambda^{-1}}^*f)\) holds for all \(f\in\Schwartz(G)\). In the following we often drop \(u\) from the notation, so the upper or lower star in \(\alpha^*_\lambda(f)\) or \({\alpha_\lambda}_*(f)\) indicates whether \(f\) should be viewed as a function or a distribution. 
\end{remark}

The following space of essentially homogeneous distributions
generalizes $m$-homogeneous distributions for which $(\alpha_\lambda)_* u = \lambda^m u$ and plays an important role for the calculi on graded groups or filtered manifolds.

\begin{definition}[\cites{Tay84, BG88, vEY19}]\label{def:ess_graded_group}For \(m\in\RR\) a distribution \(u\in\Algebra E'(G)+\Schwartz(G)\) is called \emph{essentially \(m\)-homogeneous} if \({\alpha_\lambda}_* u-\lambda^m u\in\Schwartz(G)\) for all \(\lambda >0\). The space of all essentially \(m\)-homogeneous distributions is denoted by \(\ess^m(G)\). Moreover define the quotient space \(\Sigma^m(G)= \ess^m(G)/\Schwartz(G)\).
\end{definition}
Note that for \(u\in\ess^m(G)\), the singular support must be \(\Rp\)-invariant by the essential homogeneity. As \(u\in\mathcal E'(G)+\Schwartz(G)\) the singular support must also be compact and hence \(\singsupp(u)\subset\{0\}\).

For \(u\in\ess^k(G)\) and \(v\in\ess^\ell(G)\) their convolution $u*v$ as distributions belongs to \(\ess^{k+\ell}(G)\) for all \(k,\ell\in\RR\) \cite{Tay84}*{Proposition~2.3}, which is also a special case of the results in \cref{sec:calculus} for \(\X=\simpleset{0}\). 

Moreover, for \(u\in\ess^m(G)\) let \(u^*\) be defined by \(\langle u, f\rangle =\conj{\langle u, f^*\rangle}\) with \(f^*(v)=\conj{f(v^{-1})}\) for \(v\in G\). Then \(u^*\in\ess^m(G)\) holds. Extending linearly, \(\ess(G) = \bigoplus_{m\in\RR}\ess^m(G)\) is a \(^*\)-algebra and as \(\Schwartz(G)\) is a two-sided \(\Algebra E'(G)\)-module and invariant under the involution, also \(\Sigma(G) =\bigoplus_{m\in\RR}\Sigma^m(G)\) is a \(^*\)-algebra.

\begin{example}\label{ex:diffop-as-hom-distr}
		Suppose \(P\) is an \(m\)-homogeneous left-invariant differential operator on \(G\) or equivalently an element of \(\lie{U}^m(\lie{g})\). Then it defines an \(m\)-homogeneous distribution \(u_P\) by \(\langle u_P, f\rangle =\langle \delta_0, Pf\rangle =Pf(0)\) for \(f\in \Schwartz(G)\), where \(\delta_0\) denotes the Dirac distribution at~\(0\). Consequently,  \(u_P\) can be viewed as an element of \(\ess^m(G)\). 
	\end{example}

\section{Action groupoids and convolution algebras}\label{sec:action-groupoids-convolution-algebra}
In this section, we study convolution algebras of Schwartz functions and certain fibred distributions for a polynomial action of a graded Lie group \(G\) on a vector space~\(\X = \RR^d\).
Afterwards in \cref{sec:shubin-tangent-groupoid}, the notion of a Shubin tangent groupoid is introduced which is an action groupoid of \(G\) on a space \(\X\times\RR\). The corresponding fibred distributions are together with a zoom action of \(\Rp\) crucial to define a Shubin-type pseudodifferential calculus in \cref{sec:calculus}. 

We always consider right actions on manifolds and the induced left actions on function spaces
by pull-backs in order to be compatible with \eqref{eq:dilations_on_schwartz}.
In this article, right actions of a Lie group $G$ on a manifold $N$ will usually be denoted by $\theta \colon N \times G \to N$. We also write $x \cdot v = \theta_v(x) = \theta(x,v)$.

\begin{definition}[Action groupoid]\label{def:action_groupoid}
	Let $N$ be a smooth manifold with a smooth right action by a Lie group \(G\). 
	The \emph{action groupoid} \(N\rtimes G\) is a Lie groupoid with arrow space \(N\times G\) and unit space \(N\). The \emph{unit map} \(u \colon N \to N \times G\), the \emph{range} and \emph{source map} \(r,s \colon N \times G \to N\),
	the \emph{inverse} $I \colon N \times G \to N \times G$ 
	and the \emph{multiplication} $M \colon (N \rtimes G)^{(2)} \to N \times G$ are given by
	\begin{gather*}
		u(x)=(x,e), \qquad r(x,v)=x, \qquad s(x,v)=x\cdot v,\\
		I(x,v) = (x,v)^{-1}=(x\cdot v,v^{-1}) ,\qquad M((x,v), (x \cdot v, w)) = (x,v)\cdot(x\cdot v,w)=(x,vw)
	\end{gather*}
	where $x \in N$ and $v,w \in G$.
	We usually identify $(N \rtimes G)^{(2)} \cong \set{((x,v),(y,w))}{s(x,v) = r(y,w)} = \set{((x,v), (x \cdot v,w))}{x \in N, v,w \in G}$ with $(x,v,w) \in N \times G \times G$.
\end{definition}
We denote the inversion and multiplication of the group $G$ by $i \colon G \to G$ and $m \colon G \times G \to G$.

\subsection{Polynomial actions}

The actions that we shall consider in the rest of this article are all polynomial, in the following sense:

\begin{definition}[Polynomial actions] \label{def:polynomial:action}
	Consider \(\X=\RR^d\). We call a smooth right action \(\theta \colon \X \times G \to \X\) of a graded Lie group \(G\) on \(\X\) \emph{polynomial} if each component \(\theta(x,v)_j\) is a polynomial in \(x\) and \(v\) for \(j=1,\ldots,d\). 
	Here, \(v\) denotes standard coordinates on \(G\) (see \cref{sec:dilations}) 
	and $x$ denotes linear coordinates on $\X$.
\end{definition}

\begin{definition} \label{def:polynomialDiffeo}
	A \emph{polynomial diffeomorphism} $f \colon \RR^d \to \RR^d$ is a bijection 
	such that $f$ and $f^{-1}$ are polynomial.
\end{definition}
Note that if $\theta \colon \X \times G \to \X$ is a polynomial action 
then each $\theta_v \colon \X \to \X$ is a polynomial diffeomorphism with inverse $\theta_{-v}$.
The following lemma can be applied to this ``polynomial family'' of diffeomorphisms:

\begin{lemma} \label{res:determinantIsConstant}
	Let $Y = \RR^k$ and
	$f \colon \RR^d \times Y \to \RR^d$, $g \colon \RR^d \times Y \to \RR^d$ be two polynomials.
	Define $f_y \colon \RR^d \to \RR^d$, $f_y(x) = f(x,y)$ and similarly for $g$ and assume that $f_y$ and $g_y$ are inverse to each other for all $y \in \RR^d$.
	Then the map $(x,y) \mapsto \det(D_x f_y)$ is constant and non-zero.
\end{lemma}
Here $D_x f_y$ is the Jacobi matrix of $f_y \colon \RR^d \to \RR^d$ at $x \in \RR^d$.

\begin{proof}
	Differentiating both sides of $x = g_y \circ f_y(x)$ we obtain
	$\id = D_{f_y(x)} g_y \cdot D_x f_y$.
	Taking determinants,
	\begin{equation*} 
		1 = \det(D_{f_y(x)} g_y) \det( D_x f_y) \,.
	\end{equation*} 
	All entries of the Jacobi matrix $D_x f_y$ are polynomials in $x$ and $y$.
	It follows that $\det(D_{x} f_y)$ is a polynomial.
	Similarly, $\det(D_{x} g_{y})$ is polynomial,
	and so is the expression $\det(D_{f_y(x)} g_{y})$ obtained by substituting $x$ with the polynomial $f_y(x)$.
	Therefore the constant polynomial $1$ is the product of two polynomials,
	which is only possible if these two polynomials are both constant and non-zero.
\end{proof}
Letting $k = 0$, i.e.\ $Y = \simpleset 0$, we obtain in particular that 
for any polynomial diffeomorphism $f \colon \X \to \X$, 
the map $x \mapsto \det(D_x f)$ is constant and non-zero. 

Recall that for any smooth (right) action $\theta \colon N \times G \to N$ of a Lie group $G$ on a manifold $N$ and $X \in \lie g$, the Lie algebra of $G$, we may consider the \emph{fundamental vector field} $\widehat X (x) \coloneqq \frac{\D}{\D s} \big|_{s=0} \theta(x,\exp(sX))$.

\begin{lemma} \label{res:polynomial_action:properties}
	Let $\theta$ be a polynomial right action of a graded Lie group $G$ on $\X = \RR^d$.
	Then:
	\begin{enumerate}
		\item For any $X \in \lie g$, the fundamental vector field $\widehat X$ is polynomial,
		i.e.\ in linear coordinates on $\X$ all coefficients of $\widehat X$ are polynomial. \label{item:polynomial_action:i}
		\item The Jacobian determinant $\det(D_x \theta_v) = 1$ is constant. \label{item:noJ}
	\end{enumerate}
\end{lemma}

\begin{proof}
	The first part is immediate since $\theta(x,\exp(sX))$ is polynomial in $x$ and $s$.
	The second part follows from the previous lemma and the observation that $\det(D_0 \theta_0) = \det(\id) = 1$.
\end{proof}

\begin{lemma}\label{res:polynomial_action}
	Let \(\theta\) be a polynomial right action of a graded Lie group \(G\) on \(\X=\RR^d\). Let \(\norm{\,\cdot\,}\) be a norm on \(\X\) and set \(\langle x\rangle =(1+\norm{x}^2)^\frac12\). Then there are \(B\in \NN\) and \(D_1, D_2 >0\) such that 
	\begin{enumerate}
		\item for all \(k\in\NN^{d+n}_0\), \(j=1,\ldots,d\) and \((x,v)\in \X\times G\) 
		\begin{align}\label{eq:polynomial1}
			\abs{\partial^k\theta(x,v)_j}\leq D_1 \langle x \rangle^B \langle v \rangle_\alpha^B ,\end{align}
		\item for all \((x,v)\in \X\times G\)
		\begin{align}\label{eq:polynomial2}
			\frac{\langle x \rangle}{\langle \theta_v(x) \rangle^B}\leq D_2 \langle v \rangle_\alpha^B.
		\end{align}
	\end{enumerate}
\end{lemma}
\begin{proof}
	As the action is polynomial, only finitely many derivatives have to be considered in \eqref{eq:polynomial1}
	and is it easy to find constants $B$, $D_1$ satisfying \eqref{eq:polynomial1}.
	Then \eqref{eq:polynomial2} follows from \eqref{eq:polynomial1} by setting $k = 0$ 
	and replacing $x$ and $v$ with \(\theta_v(x)\) and \(v^{-1}\). 
\end{proof}
The following Fa\`{a} di Bruno's formula will be useful in several estimates in the next sections.

\begin{lemma} \label{res:chain_rule}
	Let $F \colon \RR^k \to \RR$ and $G \colon \RR^\ell \to \RR^k$ be smooth. For a multiindex
	$L \in \NN_0^\ell$ we can write
	\begin{equation}
		\label{eq:chain_rule_constants}
		\partial_x^L(F \circ G)(x) = \sum_{K \in \NN_0^k, \abs K \leq \abs L} (\partial_x^K F)(G(x)) \cdot \Chain{G}{L}{K}{x} ,
	\end{equation}
	where each $\ChainNoArg G L K$ is a smooth function on $\RR^\ell$.
	Moreover, there exists a constant $C_{L}$ such that the following holds:
	Whenever $\abs{\partial_x^A G_j(x)} \leq H(x)$ holds for some function $H \colon \RR^\ell \to [1,\infty)$, 
	for all $0 \neq A \leq L$ and all $j \in \simpleset{1,\dots, k}$,
	then $\abs{\Chain G L K x} \leq C_{L} H(x)^{\abs L}$ for all $\abs K \leq \abs L$.
\end{lemma}
Here, $A \leq L$ means that $A_i \leq L_i$ for all $i = 1, \dots, \ell$.

\begin{proof}
	Using the chain and the product rule to compute derivatives, it is easy to show that $\partial_x^L(F \circ G)(x)$ is of the given form with
	\begin{equation*}
		\Chain G L K x = \sum_{N} c_N \prod_{\substack{(A,j) \in \NN_0^\ell \times \simpleset{1, \dots, k} \\ 0 \neq A \leq L}} \bigl(\partial_x^A G_j(x)\bigr)^{N(A,j)}
	\end{equation*}
	where $c_N \in \NN_0$ and the sum runs over all maps $N \colon \set{(A,j) \in \NN_0^\ell \times \simpleset{1, \dots, k}}{0 \neq A \leq L} \to \NN_0$
	with the property that $\sum_{A} N(A,j) = K_j$ for each $j = 1, \dots, k$.
	Letting $C_{K,L} = \sum_N c_N$ the estimate $\abs{\Chain G L K x} \leq C_{K,L} H(x)^{\abs K}$ is immediate.
	Then let $C_L = \max_{\abs K \leq \abs L} C_{K,L}$.
\end{proof}
\subsection{Schwartz convolution algebra}
In this section we shall demonstrate that the convolution product of Schwartz functions
on the groupoid $\X \rtimes G$ is well-defined and continuous,
if the action of $G$ on $\X = \RR^d$ is polynomial.

Throughout this section let $\X = \RR^d$ be endowed with a polynomial action $\theta$ of a graded Lie group $G$, see \cref{def:polynomial:action}. By $\Schwartz(\X\rtimes G)$ we denote the Fr\'echet $^*$-algebra obtained
by endowing $\Schwartz(\X \times G)$ with the involution and product defined by
\begin{align}
	f^*(x,v)&=\conj{f(x \cdot v, v^{-1})},\label{eq:involution}\\
	(f*g)(x,v)&=\int_G f(x,w)g(x \cdot w,w^{-1}v)\D w \label{eq:convolution}
\end{align}
for  \(f,g\in \Schwartz(\X \times G)\) and \((x,v)\in \X\times G\).

\begin{lemma} \label{res:algebra:schwartz}
	$\Schwartz(\X \rtimes G)$ is a well-defined Fr\'echet $^*$-algebra.
\end{lemma}

\begin{proof}
	Of course, $\Schwartz(\RR^d \times G)$ can be defined by picking standard coordinates on $G \cong \RR^n$ 
	(which are unique up to linear diffeomorphism), and endowing $\Schwartz(\RR^d \times \RR^n)$ with the usual Schwartz seminorms $\norm\argument_{k,\ell}$ defined by
	\begin{equation} \label{eq:schwartz:seminorm}
		\norm{f}_{k,\ell}
		= \!\!\!
		\sup_{\substack{(x,v)\in \RR^{d + n} \\ (a,b) \in \NN_0^{d+n}\!,\, \abs{(a,b)}\leq k}}\!\!\!(1+\norm{(x,v)})^{\ell}\abs{\partial^{(a,b)} f(x,v)}
	\end{equation}
	for \(k,\ell\in\NN_0\).
	It is well-known that $\Schwartz(\X \rtimes G)$ is a Fr\'echet space when endowed with these seminorms,
	and it only remains to check that the involution and product are continuous 
	and compatible in the usual way $(f * g)^* = g^* * f^*$.
	By our assumptions, the map $\X \times G \ni (x,v) \mapsto (x \cdot v, v^{-1}) \in \X \times G$
	is a polynomial diffeomorphism (recall that $v^{-1} = -v$ in standard coordinates),
	hence $f \mapsto f^*$ is continuous
	(to check this one can use \cref{res:chain_rule}). 
	The convolution can be written as $* = I_2 \circ r_3 \circ \tilde\psi \circ \tensor$,
	with maps 
	\begin{align*}
		\Schwartz(\X \times G) \times \Schwartz(\X \times G) \xrightarrow{\tensor} \Schwartz(\X \times G \times \X \times G)
		\xrightarrow{\tilde\psi} \Schwartz(\X \times G \times \X \times G)
		\xrightarrow{r_3} \Schwartz(\X \times G \times G)
		\xrightarrow{I_2} \Schwartz(\X \times G) 
	\end{align*}
	defined as follows: $\tensor$ is the multiplication, $(f \tensor g)(x,v,x',v') = f(x,v) g(x',v')$;
	$\tilde \psi$ is the pull-back with the polynomial diffeomorphism
	$\psi \colon \X \times G \times \X \times G \to \X \times G \times \X \times G$, 
	$(x,v,x',v') \mapsto (x,v,x'+x \cdot v,v^{-1} v')$;
	$r_3$~restricts to the third variable being $0$, i.e.\ $r_3(f)(x,w,v) = f(x,w,0,v)$;
	and $I_2$ integrates out the second variable, $I_2(f)(x,v) = \int_G f(x,w,v) \D w$.
	It is straightforward to verify that all these four maps are continuous, and so is $*$.
	E.g.\ for $I_2$ this follows from the estimate
	\begin{equation*}
		\norm{I_2(f)}_{k,\ell} = \!\!\!\!
		\sup_{\substack{(x,v)\in \RR^{d+n}\\ (a,b) \in \NN_0^{d+n}, \abs{(a,b)}\leq k}} \!\!\!\!
		\abs[\Big]{\int_G (1+\norm{(x,v)})^{\ell} \partial_{(x,v)}^{(a,b)} f(x,w,v)\D w} 
		\leq 
		\int_G \frac{\norm{f}_{k,\ell+n+1}}{(1+\norm{w})^{n+1}}\D w \leq C \norm f_{k,\ell+n+1}
	\end{equation*} 
	for some constant $C > 0$.
	It is straightforward to verify that $(f*g)^* = g^* * f^*$ for all $f, g \in \Schwartz(\X \times G)$.
\end{proof}

\begin{remark}
	The Fr\'echet $^*$-algebra $\Schwartz(\X \rtimes G)$ can also be obtained as the Fr\'echet algebraic crossed product $\Schwartz(\X) \rtimes G$ \cites{ENN88,Sch93} where the action of $G$ on $\Schwartz(\X)$ is by pull-back.
	This more ``analytic'' interpretation will not be relevant in this article.
\end{remark}

\subsection{Convolution algebra of fibred distributions}\label{subsec:convolution_alg_distr}
In this section, we extend the convolution and involution obtained in the previous section from
$\Schwartz(\X \rtimes G)$ to a bigger algebra of fibred distributions.
The construction is analogous to \cite{LMV17},
but since Shubin-type pseudodifferential operators are in general not properly supported,
we need to adapt the spaces of functions and distributions to our setting.

\paragraph{Slowly increasing functions}

\begin{definition} \label{def:om:oc}
	For \(d\in\NN_0\) let \(\mathcal{O}_M(\RR^d)\) denote the space of \emph{slowly increasing functions}
	\begin{align*}
		\Algebra{O}_M(\RR^d) &= \set{\varphi \in \Smooth(\RR^d)}{\forall f \in\Schwartz(\RR^d) \, \forall b\in\NN^d_0\colon f\cdot\partial^b \varphi \text{ is bounded}}		
	\end{align*}
	with seminorms for \(f \in\Schwartz(\RR^d)\) and \(b\in\NN^d_0\) given by
	\[\norm{\varphi}_{f,b}=\sup_{x\in \RR^d}\abs{f(x)\partial^b \varphi(x)}.\]
	Endow the dual space \(\Algebra O_M'(\RR^d)\) of \(\Algebra O_M(\RR^d)\) 
	with the strong dual topology.
\end{definition}

\begin{proposition}\label{res:properties:om}
	The space of slowly increasing functions and its dual have the following properties:
	\begin{enumerate}
		\item\label{item:om:nuclear} \(\Algebra O_M(\RR^d)\) is complete, nuclear and barrelled and \(\Algebra O_M'(\RR^d)\) is complete and nuclear \cite{Gro55}*{Th\'eor\`eme 16 in Chapitre 2},
		\item for \(\varphi \in\Smooth(\RR^d)\) \cite{Hor66}*{p.~417}
		\[\varphi\in\Algebra O_M(\RR^d) \Longleftrightarrow \forall b\in\NN^d_0\, \exists k\in\NN_0\colon (1+\norm{x}^2)^{-k}\partial^b \varphi(x) \text{ vanishes at infinity},\]
		\item \(\Algebra{O}_M(\RR^d)\) is the multiplier algebra of \((\Schwartz(\RR^d),\,\cdot\,)\) in the sense that for \(\varphi \in\Smooth(\RR^d)\) the map \(f\mapsto \varphi\cdot f\) is a continuous map \(\Schwartz(\RR^d)\to\Schwartz(\RR^d)\) if and only if \(\varphi \in\Algebra{O}_M(\RR^d)\) \cite{Hor66}*{Proposition~4.11.5},
		\item\label{res:continuous_mult_OM} the product \(p\colon \Algebra O_M(\RR^d)\times\Algebra O_M(\RR^d)\to\Algebra O_M(\RR^d)\) is continuous \cite{Schw66}*{p.~248}, see also \cite{Lar13}*{Proposition~3.5}.
	\end{enumerate}
\end{proposition}

\begin{proposition} \label{res:OM:pullback}
	Let $\psi \colon \RR^k \to \RR^\ell$ be a polynomial map. Then the pull-back
	$\psi^* \colon \Algebra O_M(\RR^\ell) \to \Algebra O_M(\RR^k)$ is continuous.
\end{proposition}

\begin{proof}
	For any $f \in \Schwartz(\RR^k)$ and $a \in \NN_0^k$ we estimate 
	\begin{align*}
		\norm{\psi^*(\varphi)}_{f,a} 
		= \sup_{x \in \RR^k}\abs{f(x) \partial^a(\varphi \circ \psi)(x)}
		&= \sup_{x \in \RR^k}\abs[\Big]{f(x) {\sum}_{\abs{b} \leq \abs{a}} \partial^b \varphi(\psi(x)) \cdot \Chain \psi a b x} \\
		&\leq {\sum}_{\abs{b} \leq \abs{a}} \sup_{x \in \RR^k}\abs[\big]{(f \cdot \ChainNoArg \psi a b) (x) \cdot \partial^b \varphi(\psi(x))} ,
	\end{align*}
	where $\ChainNoArg \psi a b$ is polynomial in derivatives of $\psi$,
	and hence a polynomial.
	Therefore $f \cdot \ChainNoArg{\psi}a b \in \Schwartz(\RR^k)$
	and it suffices to estimate $\sup_{x \in \RR^k} \abs{f_1(x) \cdot \partial^b \varphi(\psi(x))}$ by seminorms of $\varphi$ for $f_1 \in \Schwartz(\RR^k)$.
	Consider the function $f_2 \colon \RR^\ell \to \CC$, defined by
	\begin{equation*}
		f_2(y) = \sup_{\substack{x \in \RR^k \\ \psi(x) = y}} \abs{f_1(x)}
	\end{equation*}
	where $\sup \emptyset = 0$.
	We compute for any function $g$ on $\RR^\ell$ that
	\begin{equation*}
		\sup_{y \in \RR^\ell} \abs{f_2(y) g(y)} 
		= \sup_{y \in \RR^\ell} \sup_{\substack{x \in \RR^k \\ \psi(x) = y}} \abs{f_1(x) g(y)}
		= \sup_{y \in \RR^\ell} \sup_{\substack{x \in \RR^k \\ \psi(x) = y}} \abs{f_1(x) g(\psi(x))}
		= \sup_{x \in \RR^k} \abs{f_1(x) g(\psi(x))} .
	\end{equation*}
	In particular, the function $f_2 \cdot p$ is bounded on $\RR^\ell$ for any polynomial $p$ on $\mathbb R^\ell$, i.e.\ $f_2$ is rapidly decreasing. 
	By \cite{Gar04} there exists a Schwartz function $f_3 \in \Schwartz(\RR^\ell)$ such that $f_2(y) \leq f_3(y)$ holds for all $y \in \RR^\ell$.
	Now note that
	\begin{equation*}
		\sup_{x \in \RR^k}\abs[\big]{f_1 (x) \cdot \partial^b \varphi(\psi(x))} 
		=
		\sup_{y \in \RR^\ell} \abs{f_2 (y) \partial^b \varphi(y) }
		\leq
		\sup_{y \in \RR^\ell} \abs{f_3 (y) \partial^b \varphi(y) }
		\leq \norm{\varphi}_{f_3, b} .
	\end{equation*}
	Therefore $\norm{\psi^*(\varphi)}_{f,a}$ can be estimated by a finite linear combination of seminorms of $\varphi$, showing that $\psi^*$ is indeed continuous.
\end{proof}
In particular, this proposition applies to the polynomial maps $i$, $m$, $r$, $s$, $I$ and $M$.
Moreover, we may define $\Algebra{O}_M(G)$ for a graded Lie group $G$ by using standard coordinates on $G$ 
(which are unique up to a linear diffeomorphism).

It is straightforward to verify that the inclusions $\Schwartz(\RR^d) \to \Algebra{O}_M(\RR^d) \to \Algebra E(\RR^d)$ are continuous with dense image, where $\Algebra E(\RR^d) = \Smooth(\RR^d)$. Therefore, there are inclusions $\Algebra E'(\RR^d) \to \Algebra O'_M(\mathbb R^d) \to \Schwartz'(\RR^d)$,
and one can take the Fourier transform of an element of $\Algebra O'_M(\mathbb R^d)$.
The image is characterized as follows:

\begin{proposition}[\cite{Schw66}*{Th\'eor\`eme XV on p.~268}]\label{res:properties:oc}
	The Fourier transform defines an isomorphism
	\(\Algebra{O}_M'(\RR^d)\to\Algebra{O}_C(\RR^d)\)
	where \begin{equation}\Algebra{O}_C(\RR^d)=\set{\varphi \in 
		\Smooth(\RR^d) }{ \exists k\in\NN_0\,\forall b\in\NN^d_0 \colon (1+\norm{x}^2)^{-k}\partial^b \varphi(x) \text{ vanishes at infinity}}.\end{equation}
	Defining the topology of $\Algebra O_C(\RR^d)$ through this isomorphism
	yields a complete nuclear and barrelled space $\Algebra O_C(\RR^d)$,
	for which the inclusion $\Algebra{O}_C(\RR^d) \to \Algebra{O}_M(\RR^d)$ is continuous with dense image.
	
	Moreover, the inclusion $\Schwartz(\RR^d) \to \mathcal O_C(\RR^d)$ is continuous with dense image.
	Therefore $\mathcal O'_C(\RR^d) \to \Schwartz'(\RR^d)$ is a continuous injection when the dual $\Algebra{O}_C'(\RR^d)$ of $\Algebra O_C(\RR^d)$ is endowed with the strong dual topology.
	The restriction of the Fourier transform gives a topological isomorphism \(\Algebra{O}_C'(\RR^d)\to\Algebra{O}_M(\RR^d)\).
\end{proposition}
Note that in particular the Fourier transform of a distribution in $\Algebra O'_M(\RR^d)$ is in $\Algebra O_M(\RR^d)$.
Using standard coordinates one can also introduce $\Algebra O_C(G)$ for a graded Lie group $G$.

In the following, $E \completedtensor F$ denotes a completed tensor product of locally convex topological vector spaces. In this article, at least one of the spaces is nuclear, so that all reasonable tensor products agree.

\begin{lemma}[\cite{Schw54}*{p.~115}] \label{res:OM_continuity}For all \(k,\ell\in\NN\)
	the map \(\Algebra O_M(\RR^k)\times\Algebra O_M(\RR^\ell)\to \Algebra O_M(\RR^{k}\times\RR^\ell)\) defined by \((\varphi_1,\varphi_2)\mapsto \varphi_1\tensor \varphi_2\) with $\varphi_1 \tensor \varphi_2(x,y) = \varphi_1(x)\varphi_2(y)$ is continuous and induces an isomorphism
	\(\Algebra O_M(\RR^k)\completedtensor \Algebra O_M(\RR^\ell)\cong \Algebra O_M(\RR^k\times\RR^\ell)\).
	Moreover, \(\Algebra O'_C(\RR^k)\completedtensor \Algebra O'_C(\RR^\ell)\cong \Algebra O'_C(\RR^k\times\RR^\ell)\) holds.
\end{lemma}
\begin{proof}
	 We start by showing that the map \(\Algebra O_M(\RR^k)\times \Algebra O_M(\RR^\ell)\to\Algebra O_M(\RR^{k+\ell})\) given by \((\varphi_1,\varphi_2)\mapsto \varphi_1\tensor \varphi_2\) is continuous.
	 Note that the map $\mathcal O_M(\RR^k) \to \mathcal O_M(\RR^{k+\ell})$, $\varphi_1 \mapsto \varphi_1 \tensor 1$ is just the pull-back with the projection $\RR^{k+\ell} \to \RR^k$ onto the first $k$ coordinates
	 and therefore continuous by \cref{res:OM:pullback}. Similarly the map $\varphi_2 \mapsto 1 \tensor \varphi_2$ is continuous.
	Since the multiplication \(\Algebra O_M(\RR^{k+\ell})\times\Algebra O_M(\RR^{k+\ell})\to\Algebra O_M(\RR^{k+\ell})\) is continuous by \cref{res:properties:om} \refitem{res:continuous_mult_OM} the claim follows.
	To show that the induced map \(\Algebra O_M(\RR^k)\completedtensor \Algebra O_M(\RR^\ell)\to\Algebra O_M(\RR^{k+\ell})\) is an isomorphism one can argue as in the proof of \cite{Tre67}*{Theorem~51.6}. This uses that \(\Algebra{O}_M(\RR^k)\) is nuclear, see \cref{res:properties:om} \refitem{item:om:nuclear}.
	The statement for $\mathcal O_C'$ follows by applying Fourier transform, see \cref{res:properties:oc}.
\end{proof}

\paragraph{Fibred distributions}

\begin{definition}
	For \(m\geq n\), let \(Y=\RR^m\), \(B=\RR^n\) and  \(F=\RR^{m-n}\). A \emph{polynomial fibre projection} is a map \(\pi\colon Y\to B\) together with a polynomial diffeomorphism \(\psi_\pi\colon Y\to B\times F\) such that \(\pi={\pr_1}\circ \psi_\pi\).
\end{definition}
\begin{example} \label{ex:polynomialFibreProjection}
	In the following let $\mathcal G = \X \rtimes G$ be the action groupoid of a polynomial right action of a graded Lie group $G$ on $\X = \mathbb R^d$.
 	The range and source map \(r,s\colon \mathcal G\to \X\) are polynomial fibre projections. We fix as polynomial diffeomorphisms \(\psi_r=\id\) and \(\psi_s(x,v)=(x\cdot v, v)\).
\end{example}
Note that \(\pi^*\colon \Algebra O_M(B)\to \Algebra O_M(Y)\) is continuous by \cref{res:OM:pullback}.

\begin{definition}
	Let \(\pi\colon Y\to B\) be a polynomial fibre projection. 
	Consider  $\mathcal O_M(Y)$ as an $\mathcal O_M(B)$-module via
	\begin{equation*}
		(f \cdot_\pi \varphi)(y)=\pi^* f(y)\cdot \varphi(y)= f(\pi(y))\cdot \varphi(y)
	\end{equation*}
	for $f\in\Algebra O_M(B)$ and $\varphi\in\Algebra O_M(Y)$,
	and $\mathcal O_M(B)$ as an $\mathcal O_M(B)$-module with respect to the multiplication.
	We define the \emph{$\pi$-fibred distribution} $\mathcal O_\pi'(Y) \coloneqq \mathcal L_{\Algebra O_M(B)}^\pi(\Algebra O_M(Y),\Algebra O_M(B))$
	as the space of continuous $\Algebra O_M(B)$-linear maps $\mathcal O_M(Y) \to \mathcal O_M(B)$.
\end{definition}
\begin{example} \label{ex:schwartz-as-fibred}
	Given a Schwartz function $f \in \Schwartz(Y)$, it is straightforward to verify that
	\begin{align*}
		v_f(\varphi)(x)=\int_F (\psi_s^{-1})^*f(x,v)(\psi_s^{-1})^*\varphi(x,v)\D v
	\end{align*}
	defines a \(\pi\)-fibred distribution \(v_f\in\Algebra O_\pi'(Y)\). In particular,
	\begin{align}
		u_f(\varphi)(x) &= \int_G f(x,v) \varphi(x,v) \D v , \label{eq:Schwartz-as-r-fibred}
		\\
		\tilde{u}_f(\varphi)(x) &= \int_G f(x\cdot v^{-1}, v)\varphi(x\cdot v^{-1},v)\D v  \label{eq:Schwartz-as-s-fibred}
	\end{align}
	with $\varphi \in \mathcal O_M(\X \times G)$ and $x \in \X$,
	define an $r$-fibred distribution $u_f \in \mathcal O_r'(\mathcal G)$ and an $s$-fibred distribution $\tilde u_f \in \mathcal O_s'(\mathcal G)$.
\end{example}
A fibred distribution $u \in \Algebra O'_\pi(Y)$ shall be thought of as a special distribution on $Y$,
which can be obtained by plugging a test function $\varphi \in \SmoothCompactSupp(Y) \eqqcolon \Algebra D(Y)$ into $u$
and integrating the result over $B$.
In other words such a distribution is really a ``smooth family'' over \(B\) of distributions 
on the fibre \(F\).
To make this interpretation more concrete, define the integration map 
\begin{equation}
	\IntegrationMap_\pi \colon \Algebra O'_\pi(Y) \to \Algebra D'(Y) \,,
	\quad
	\langle \IntegrationMap_\pi(u), \varphi\rangle = \int_B u(\varphi)(x) \D x
\end{equation}
for $\varphi \in \Algebra D(Y)$. As \(\varphi\) is smooth with compact support, it belongs to $\mathcal O_M(Y)$, so that we can apply $u$ to it. Suppose that the support of \(\varphi\) is contained in a compact set \(K\subset Y\). Let \(\chi\in\SmoothCompactSupp(B)\) be constant~\(1\) on the compact set \(\pi(K)\subset B\). Then by \(\Algebra O_M(B)\)-linearity of \(u\) one has \begin{equation*}\chi\cdot u(\varphi)=u(\pi^*\chi\cdot \varphi)=u(\varphi).\end{equation*}
This shows that \(u(\varphi)\) is compactly supported and, hence, that the integration is well-defined. 
Note that \(\IntegrationMap_r(u_f)=\IntegrationMap_s(\tilde u_f)\) holds for the $r$ and $s$-fibred distributions associated to a Schwartz function $f \in \Schwartz(Y)$ (see \cref{ex:schwartz-as-fibred}). 

\begin{lemma} \label{res:PhiR:injective}
	The map $\IntegrationMap_\pi \colon \mathcal O_\pi'(Y) \to \mathcal D'(Y)$ is injective.
\end{lemma}

\begin{proof} Let $u_1, u_2 \in \mathcal O'_\pi(Y)$ and assume that $u_1 \neq u_2$. 
	Since $\Algebra D(Y)$ is dense in $\Algebra O_M(Y)$,
	there exists a test function $\varphi \in \Algebra D(Y)$ such that $u_1(\varphi) \neq u_2(\varphi)$.
	There is a point $y \in B$ such that $u_1(\varphi)(y) \neq u_2(\varphi)(y)$,
	and by continuity there exist an open neighbourhood $U \subset B$ of $y$
	such that $\abs{u_i(\varphi)(x) - u_i(\varphi)(y)} \leq \frac 1 3 \abs{u_1(\varphi)(y) - u_2(\varphi)(y)}$ holds for all $x \in U$ and $i=1,2$.
	Choose a smooth cut-off function $\chi \colon B \to [0,1]$ with support in $U$ satisfying $\chi(y) = 1$.
	Then 
	\begin{align*}
		&\abs[\big]{\langle \IntegrationMap_\pi(u_1) - \IntegrationMap_\pi(u_2), \pi^* \chi \cdot \varphi \rangle}
		= \abs[\Big]{\int_B (u_1 - u_2)(\pi^* \chi \cdot \varphi)(x) \D x}
		= \abs[\Big]{\int_B (u_1 - u_2)(\varphi)(x) \chi(x) \D x}
		\\&\!\geq \abs[\Big]{ \!\int_B (u_1(\varphi)(y) - u_2(\varphi)(y)) \chi(x) \D x } - \! \int_B \abs{u_1(\varphi)(x) - u_1(\varphi)(y)} \chi(x) \D x - \!\int_B \abs{u_2(\varphi)(x) - u_2(\varphi)(y)} \chi(x) \D x 
		\\&\!\geq \frac 1 3 \abs{u_1(\varphi)(y) - u_2(\varphi)(y)} \int_B \chi(x) \D x > 0 .
	\end{align*}
	Hence $\IntegrationMap_\pi(u_1) \neq \IntegrationMap_\pi(u_2)$.
\end{proof}
For the fibred distributions \(\mathcal E'_\pi(Y)\) of \cite{LMV17} for a submersion \(\pi\colon Y\to B\), one can locally trivialize \(Y=B\times F\) and obtain a description of fibred distributions as a family of \(u_x\in\mathcal E'(F)\) for \(x\in B\), see \cite{LMV17}*{Proposition~2.7}. In our given setting, $Y$ can be globally identified with the product $B \times F$.
This allows one to describe fibred distributions in the following way.
By extending functions constantly in $B$, one obtains a map
\begin{equation}
	\AlgebraicIso_\pi \colon \mathcal L^\pi_{\mathcal O_M(B)}(\mathcal O_M(Y), \mathcal O_M(B)) \to \mathcal L(\mathcal O_M(F), \mathcal O_M(B)) ,
	\quad
	\AlgebraicIso_\pi(u)(\varphi) = u (\psi_\pi^*(1 \tensor \varphi)) .
\end{equation}
\begin{lemma} \label{res:identifications:fibredDistributions}
	The map $\AlgebraicIso_\pi \colon \mathcal O'_\pi( Y) \to \mathcal L(\Algebra O_M(F),\Algebra O_M(B))$
	is an isomorphism of vector spaces.
\end{lemma}

\begin{proof}
	We define a map $\AlgebraicIso_\pi^{-1} \colon \mathcal L(\Algebra O_M(F),\Algebra O_M(B)) \to \mathcal O'_\pi(Y)$
	and check that it is the inverse of $\AlgebraicIso_\pi$. Given $U \in  \mathcal L(\Algebra O_M(F),\Algebra O_M(B))$ define
	$\AlgebraicIso_\pi^{-1}(U) = p \circ (\mathrm{id} \tensor U)\circ(\psi_\pi^{-1})^*$,
	where $\mathrm{id} \tensor U \colon \Algebra O_M(B) \completedtensor \Algebra O_M(F) \to \Algebra O_M(B) \completedtensor \Algebra O_M(B)$ is continuous
	and $p \colon \Algebra O_M(B) \completedtensor \Algebra O_M(B) \to \Algebra O_M(B)$ is the map induced by the continuous multiplication, see \cref{res:properties:om}~\refitem{res:continuous_mult_OM}. Here we identify $\Algebra O_M(B\times F)$ with $\Algebra O_M(B) \completedtensor \Algebra O_M(F)$, see \cref{res:OM_continuity}. Observe that
	\begin{equation}\label{eq:pi-pullback}
		\pi^*f=\psi_\pi^*(\pr_1^*f) =\psi_\pi^*(f\otimes 1) \quad\text{ for all }f\in\Algebra O_M(B).
	\end{equation}
	To show that \(\AlgebraicIso_\pi^{-1}(U)\) is indeed \(\Algebra O_M(B)\)-linear, it suffices to consider by continuity all \(f\in\Algebra O_M(B)\) and functions of the form $\varphi = \psi_\pi^*(\varphi_1 \tensor \varphi_2)$ for \(\varphi_1\in \Algebra O_M(B)\) and \(\varphi_2\in \Algebra O_M(F)\). For these one has by \eqref{eq:pi-pullback}
	\begin{equation*}
		\Psi_\pi^{-1}(U)(\pi^*f\cdot \varphi)=p\circ(\id\otimes U)\circ(\psi_\pi^{-1})^*(\psi^*_\pi(f\otimes 1)\cdot\psi_\pi^*(\varphi_1\otimes\varphi_2))=f\cdot \Psi_\pi^{-1}(U)(\varphi).
	\end{equation*} 
	It is clear that $\Psi_\pi(\Psi_\pi^{-1} U)(\varphi) = \AlgebraicIso_\pi^{-1}(U)(\psi_\pi^*(1 \tensor \varphi)) = U(\varphi)$ for all \(\varphi\in \Algebra O_M(F)\).
	It remains to check $\Psi_\pi^{-1}(\Psi_\pi(u)) = u$ for \(\varphi\in\Algebra O_M(Y)\). As both maps are continuous and \(\Algebra O_M(B)\)-linear it suffices to show they coincide on functions of the form \(\psi^*_\pi(1\otimes\varphi)\) for \(\varphi\in\Algebra O_M(F)\) for which one computes
	\begin{equation*}
		\Psi_\pi^{-1}(\Psi_\pi(u))(\psi_\pi^*(1\tensor \varphi)) = \Psi_\pi(u)(\varphi) 
		 =u(\psi_\pi^* (1\tensor \varphi)).\qedhere
	\end{equation*} 
\end{proof}

\begin{example}
	Recall that Schwartz functions give rise to \(r\)- and \(s\)-fibred distributions on \(\mathcal G\), 
	see \cref{ex:schwartz-as-fibred}. Under the isomorphisms from the previous lemma one obtains
	\begin{align*}
		U_f(\varphi)(x)&=\Psi_r(u_f)(\varphi)(x)=\int_G f(x,v)\varphi(v)\D v,\\
		\tilde U_f(\varphi)(x)&=\Psi_s(\tilde u_f)(\varphi)(x)=\int_G f(x\cdot v^{-1} ,v)\varphi(v)\D v.
	\end{align*}
\end{example}
It follows from \cite{Tre67}*{Proposition~50.5} that the there are isomorphisms of vector spaces
\begin{equation} \label{eq:fibredDistribution:Isos}
	\mathcal O_M(F)' \completedtensor \mathcal O_M(B) 
	\cong \mathcal L(\Algebra O_M(F), \Algebra O_M(B))
	\cong \mathcal L_{\Algebra O_M(B)}^\pi (\Algebra O_M(Y), \Algebra O_M(B))
	= \mathcal O'_\pi(Y) .
\end{equation}

\begin{remark}
	In \cite{Tre67}*{Proposition~50.5} it is even shown that the isomorphism
	$\mathcal O_M(F)' \completedtensor \mathcal O_M(B) \cong \mathcal L(\Algebra O_M(F), \Algebra O_M(B))$
	is an isomorphism of topological vector spaces 
	if the latter space is endowed with the topology of uniform convergence on bounded sets.
	One can show that under $\Psi_\pi^{-1}$ from \cref{res:identifications:fibredDistributions},
	the topology of uniform convergence on bounded subsets of \(\mathcal L(\Algebra O_M(F),\Algebra O_M(B))\)
	induces the topology of uniform convergence on sets of the form $\overline{\mathrm{conv}}(\psi_\pi^*(B_1 \tensor B_2))$ 
	on \(\mathcal L^\pi_{\mathcal O_M(B)}(\Algebra O_M(Y),\Algebra O_M(B))\), 
	where $B_1 \subset \Algebra O_M(F)$ and $B_2 \subset \Algebra O_M(B)$ are both bounded and $\overline{\mathrm{conv}}$ denotes the closed convex hull.
	Characterizing spaces for which this topology coincides with the topology of uniform convergence on bounded sets
	is Grothendieck's problem of topologies \cite{Gro55}*{Ch.~I, p.~33-34}, but the authors do not know whether it has a positive answer for $\Algebra O_M(\RR^k)$.
	Since these topologies do not play a role in the following, we omit the details.
\end{remark}

\paragraph{Convolution}

That \(\Algebra O'_r(\mathcal G)\) is a convolution algebra could be seen using the same arguments as in \cite{LMV17}*{Theorem~3.2}.
Here, we give a description of the convolution under the isomorphisms of \cref{res:identifications:fibredDistributions}.

\begin{proposition}\label{res:convolution}
	Let \(\mathcal G=\X\rtimes G\) be the action groupoid of a polynomial action of the graded Lie group \(G\) on a space \(\X\cong\RR^d\). Then \(\Algebra{O}'_r(\mathcal G)\) admits a convolution extending \eqref{eq:convolution} that gives it the structure of an associative unital algebra.
\end{proposition}
\begin{proof}
	We define a convolution on $\mathcal L(\Algebra O_M(G), \Algebra O_M(\X))$ 
	and use the isomorphism $\Psi_r$ to obtain a convolution on $\Algebra O'_r(\mathcal G)$.
	Denote by $\Delta \colon \X \to \X \times \X$, $\Delta(x) = (x,x)$ the diagonal embedding
	and by \(F\colon G \times G \to G \times G\), \(F(v,w)=(w,v)\) the flip map. 
	In the following, we use the identification $\Algebra O_M(\RR^k \times \RR^\ell) \cong \Algebra O_M(\RR^k) \completedtensor \Algebra O_M(\RR^\ell)$ from \cref{res:OM_continuity}
	and the map $\psi_s$ from \cref{ex:polynomialFibreProjection}.
	For $U_1, U_2 \in \mathcal L(\Algebra O_M(G), \Algebra O_M(\X))$, let $U_1 * U_2$ be the composition
\begin{multline*}
	\Algebra O_M(G) 
	\xrightarrow{m^*} \Algebra O_M(G \times G)\xrightarrow{F^*} \Algebra O_M(G \times G)
	\xrightarrow{U_2 \tensor \mathrm{id}} \Algebra O_M(\X \times G)
	\\
	\xrightarrow{\psi_s^*} \Algebra O_M(\X \times G)
	\xrightarrow{\mathrm{id} \tensor U_1} \Algebra O_M(\X \times \X)
	\xrightarrow{\Delta^*} \Algebra O_M(\X) \,,
\end{multline*}
which we also write graphically as
\begin{equation*}
	U_1 * U_2 = 
	\begin{boxdiagram}{M}
		\DrawOneToTwoBox{m^*}{M}{G}{L}{U}
		\DrawTwoLineBox{F^*}{L}{G}{U}{G}
		\DrawOneLineBox{U_2}{U}{G}
		\DrawTwoLineBox{\psi_s^*}{L}{G}{U}{S}
		\DrawOneLineBox{U_1}{L}{G}
		\DrawTwoToOneBox{\Delta^*}{L}{S}{U}{S}{M}
		\DrawFinalLines{{M/S}}
	\end{boxdiagram} \,.
\end{equation*}
	The meaning of the graphical expression should be clear: it needs to be read from left to right,
	a solid line stands for a tensor factor $\Algebra O_M(G)$ and a dashed line for $\Algebra O_M(\X)$,
	maps are represented by boxes and applied to the tensor factors corresponding to the lines that enter the boxes.
	The maps $m^*$, \(F^*\), $\psi_s^*$ and $\Delta^*$ are continuous by \cref{res:OM:pullback}.
	Therefore $U_1 * U_2$ is indeed a continuous linear map from $\Algebra O_M(G)$ to $\Algebra O_M(G)$. 
	
	To see that the convolution defined above generalizes the convolution in \eqref{eq:convolution},
	we compute for \(f_1,f_2\in\Schwartz(\mathcal G)\) and \(\varphi\in\Algebra O_M(G)\) that 
	\begin{align} \label{eq:UonSchwartzFunction}
		\psi_s^* \circ ( U_{f_2}\tensor\id) \circ F^*\circ m^* \varphi & \colon & (y,w) & \mapsto \int_G f_2(y \cdot w,v) \varphi(w v) \D v 
		\\
		\Delta^* \circ (\id\tensor U_{f_1}) \circ \psi_s^* \circ ( U_{f_2}\tensor\id) \circ F^*\circ m^* \varphi & \colon & x & \mapsto \int_{G \times G} f_1(x,w) f_2(x \cdot w,v) \varphi(w v) \D w \D v  \notag
		\\
		&&&\phantom{XXX} =
		\int_{G} \int_{G} f_1(x,w) f_2(x \cdot w,w^{-1}v) \D w \varphi(v) \D v. \notag
	\end{align}
	This shows $U_{f_1} * U_{f_2} = U_{f_1 * f_2}$, hence $*$ extends the previously defined convolution on $\Schwartz(\X \rtimes G)$.
	
	To show associativity of $*$, let \(\overline m= m\circ F\) to shorten notation. 
	Associativity of $m$ implies that $\overline m \circ (\overline m \times \id)(u,v,w) = \overline m(vu,w) = wvu = \overline m(u,wv) = \overline m \circ (\id \times \overline m)(u,v,w)$.
	Similarly, $(\Delta \times \id) \circ \Delta = (\id \times \Delta) \circ \Delta$ agree as maps $\X \to \X \times \X \times \X$.
	Dualizing we obtain
	\begin{equation*}
		\begin{boxdiagram}{B}
		\DrawBox{\overline m^*}{C}{U}{{HL/G}}{{U,C}}
		\DrawBox{\overline m^*}{B}{A}{{U/G}}{{B,A}}
		\DrawFinalLines{{C/G,B/G,A/G}}
		\end{boxdiagram}
	=
		\begin{boxdiagram}{B}
			\DrawBox{\overline m^*}{L}{A}{{HU/G}}{{A,L}}
			\DrawBox{\overline m^*}{C}{B}{{L/G}}{{C,B}}
			\DrawFinalLines{{C/G,B/G,A/G}}
		\end{boxdiagram}
	\quad
	\text{and}
	\quad
		\begin{boxdiagram}{B}
			\DrawBox{\Delta^*}{B}{A}{{B/S, A/S}}{{U}}
			\DrawBox{\Delta^*}{C}{U}{{U/S,C/S}}{{HL}}
			\DrawFinalLines{{HL/S}}
		\end{boxdiagram}
	=
		\begin{boxdiagram}{B}
			\DrawBox{\Delta^*}{C}{B}{{C/S,B/S}}{{L}}
			\DrawBox{\Delta^*}{L}{A}{{A/S,L/S}}{{HU}}
			\DrawFinalLines{{HU/S}}
		\end{boxdiagram} \,\,\,.
	\end{equation*}
	A similar computation yields
	\begin{equation*}
		\begin{boxdiagram}{B}
			\DrawBox{\psi_s^*}{L}{A}{{L/G,A/S}}{{L,A}}
			\DrawBox{\overline m^*}{C}{B}{{L/G}}{{C,B}}
			\DrawFinalLines{{C/G,B/G,A/S}}
		\end{boxdiagram} 
	= 
		\begin{boxdiagram}{B}
			\DrawBox{\overline m^*}{C}{B}{{L/G}}{{C,B}}
			\DrawBox{\psi_s^*}{B}{A}{{B/G,A/S}}{{B,A}}
			\DrawBox{\psi_s^*}{C}{A}{{C/G,A/S}}{{C,A}}
			\DrawFinalLines{{C/G,A/S}}
			\DrawLeftLine{B}{{2/3/G}}
			\DrawRightLine{B}{{3/4/G}}
		\end{boxdiagram} 
\quad\text{and}\quad
		\begin{boxdiagram}{B}
			\DrawBox{\psi_s^*}{C}{A}{{C/G,A/S}}{{C,A}}
			\DrawBox{\psi_s^*}{C}{B}{{C/S}}{{C,B}}
			\DrawBox{\Delta^*}{B}{A}{{B/S,A/S}}{{U}}
			\DrawFinalLines{{U/S}}
			\DrawLeftLine{B}{{0/1/S}}
			\DrawRightLine{B}{{1/2/S}}
		\end{boxdiagram} 
		= 
			\begin{boxdiagram}{B}
				\DrawBox{\Delta^*}{B}{A}{{B/S,A/S}}{{U}}
				\DrawBox{\psi_s^*}{C}{U}{{C/G,U/S}}{{C,U}}
				\DrawFinalLines{{C/G,U/S}}
			\end{boxdiagram} 
	\end{equation*}
	where the dots on the second line next to $\psi_s^*$ shall indicate that this map does not act on the second argument, i.e.\ acts on the first and third arguments. Using these equations and that we may change the order of maps acting on different tensor factors we compute
\begin{align*}
	(U_1 * U_2) * U_3 &= 
	\begin{boxdiagram}{HU}
		\DrawBox{\overline m^*}{L}{A}{{HU/G}}{{L,A}}
		\DrawBox{U_3}{A}{A}{{A/G}}{{A}}
		\DrawBox{\psi_s^*}{L}{A}{{L/G,A/S}}{{L,A}}
		\DrawEmptyBox
		\DrawBox{\overline m^*}{C}{B}{{L/G}}{{B,C}}
		\DrawBox{U_2}{B}{B}{{B/G}}{{B}}
		\DrawBox{\psi_s^*}{C}{B}{{C/G,B/S}}{{C,B}}
		\DrawBox{U_1}{C}{C}{{C/G}}{{C}}
		\DrawBox{\Delta^*}{C}{B}{{C/S,B/S}}{{L}}
		\DrawBox{\Delta^*}{L}{A}{{L/S,A/S}}{{HU}}
		\DrawFinalLines{{HU/S}}
	\end{boxdiagram}
	\\
	&=
	\begin{boxdiagram}{HU}
		\DrawBox{\overline m^*}{L}{A}{{HU/G}}{{L,A}}
		\DrawBox{U_3}{A}{A}{{A/G}}{{A}}
		\DrawBox{\overline m^*}{C}{B}{{L/G}}{{B,C}}
		\DrawBox{\psi_s^*}{B}{A}{{B/G,A/S}}{{B,A}}
		\DrawBox{\psi_s^*}{C}{A}{{C/G,A/S}}{{C,A}}
		\DrawBox{U_2}{B}{B}{}{{B}}
		\DrawBox{\psi_s^*}{C}{B}{{C/G,B/S}}{{C,B}}
		\DrawBox{U_1}{C}{C}{{C/G}}{{C}}
		\DrawBox{\Delta^*}{C}{B}{{C/S,B/S}}{{L}}
		\DrawBox{\Delta^*}{L}{A}{{L/S,A/S}}{{HU}}
		\DrawFinalLines{{HU/S}}
		\DrawLeftLine{B}{{4/5/G}}
		\DrawRightLine{B}{{5/6/G}}
	\end{boxdiagram} 
	\\
	&=
	\begin{boxdiagram}{HL}
		\DrawBox{\overline m^*}{C}{U}{{HL/G}}{{C,U}}
		\DrawBox{\overline m^*}{B}{A}{{U/G}}{{B,A}}
		\DrawBox{U_3}{A}{A}{{A/G}}{{A}}	
		\DrawBox{\psi_s^*}{B}{A}{{B/G,A/S}}{{B,A}}
		\DrawBox{U_2}{B}{B}{{B/S}}{{B}}
		\DrawBox{\psi_s^*}{C}{A}{{C/G,A/S}}{{C,A}}
		\DrawBox{\psi_s^*}{C}{B}{{C/G}}{{C,B}}
		\DrawBox{\Delta^*}{B}{A}{{B/S,A/S}}{{U}}	
		\DrawBox{U_1}{C}{C}{{C/G}}{{C}}
		\DrawBox{\Delta^*}{C}{U}{{C/S,U/S}}{{HL}}
		\DrawFinalLines{{HL/S}}
		\DrawLeftLine{B}{{5/6/G}}
		\DrawRightLine{B}{{6/7/G}}
	\end{boxdiagram} 
	\\
	&=
	\begin{boxdiagram}{HL}
		\DrawBox{\overline m^*}{C}{U}{{HL/G}}{{C,U}}
		\DrawBox{\overline m^*}{B}{A}{{U/G}}{{B,A}}
		\DrawBox{U_3}{A}{A}{{A/G}}{{A}}	
		\DrawBox{\psi_s^*}{B}{A}{{B/G,A/S}}{{B,A}}
		\DrawBox{U_2}{B}{B}{{B/S}}{{B}}
		\DrawBox{\Delta^*}{B}{A}{{B/S,A/S}}{{U}}
		\DrawEmptyBox
		\DrawBox{\psi_s^*}{C}{U}{{C/G,U/S}}{{C,U}}
			\DrawBox{U_1}{C}{C}{{C/G}}{{C}}
		\DrawBox{\Delta^*}{C}{U}{{C/S,U/S}}{{HL}}
		\DrawFinalLines{{HL/S}}
	\end{boxdiagram} 
	= U_1 * (U_2 * U_3) .
\end{align*}
	
	Finally, we check that the map $\delta_0 \in \mathcal L(\mathcal O_M(G), \mathcal O_M(\X))$, $\delta_0(\varphi)(x) = \varphi(0)$ is a unit for $*$.  
	Indeed, for $\varphi \in \mathcal O_M(G)$, we obtain that
	$\psi_s^* \circ (\delta_0 \tensor \mathrm{id}) \tensor \overline m^*(\varphi)$ maps $(x,w)$ to $\varphi(w)$,
	hence $(\delta_0 * U)(\varphi) = U(\varphi)$.
	Moreover, $(U \tensor \mathrm{id}) \tensor \overline m^*(\varphi)$ maps $(x,w)$ to $U(v \mapsto \varphi(wv))(x)$,
	hence $\psi_s^* \circ (U \tensor \mathrm{id}) \tensor \overline m^*(\varphi)$ maps 
	$(x,w)$ to $U(v \mapsto \varphi(wv))(x w)$.
	Applying $\mathrm{id} \tensor \delta_0$ evaluates this expression for $w = 0$,
	hence $(U * \delta_0)(\varphi) = U(\varphi)$.
\end{proof}
Recall that $\Psi_r^{-1}(U) = \Delta^* \circ (\mathrm{id}\tensor U)$. Therefore $\Psi_r^{-1}(U_1 * U_2)$ is given by
\begin{multline*}
	\begin{boxdiagram}{M}
		\DrawOneToTwoBox{\overline m^*}{L}{G}{C}{B}
		\DrawOneLineBox{U_2}{B}{G}
		\DrawTwoLineBox{\psi_s^*}{C}{G}{B}{S}
		\DrawOneLineBox{U_1}{C}{G}
		\DrawTwoToOneBox{\Delta^*}{C}{S}{B}{S}{L}
		\DrawTwoToOneBox{\Delta^*}{L}{S}{A}{S}{HU}
		\DrawFinalLines{{HU/S}}
	\end{boxdiagram}
	=
	\begin{boxdiagram}{M}
		\DrawOneToTwoBox{\overline m^*}{L}{G}{C}{B}
		\DrawOneLineBox{U_2}{B}{G}
		\DrawTwoLineBox{\psi_s^*}{C}{G}{B}{S}
		\DrawOneLineBox{U_1}{C}{G}
		\DrawTwoToOneBox{\Delta^*}{B}{S}{A}{S}{U}
		\DrawTwoToOneBox{\Delta^*}{C}{S}{U}{S}{HL}
		\DrawFinalLines{{HL/S}}
	\end{boxdiagram}
	\\
	=
	\begin{boxdiagram}{M}
		\DrawOneToTwoBox{ m^*}{L}{G}{C}{B}
		\DrawTwoLineBox{F^*}{C}{G}{B}{S}
		\DrawOneLineBox{U_2}{B}{G}
		\DrawTwoLineBox{\psi_s^*}{C}{G}{B}{S}
		\DrawTwoToOneBox{\Delta^*}{B}{S}{A}{S}{U}
		\DrawTwoToOneBox{u_1}{C}{G}{U}{S}{HL}
		\DrawFinalLines{{HL/S}}
	\end{boxdiagram}
	= 
	\begin{boxdiagram}{M}
		\DrawOneToTwoBox{ m^*}{L}{G}{C}{B}
		\DrawBox[3]{s^* u_2}{C}{A}{{C/G,B/G,A/S}}{{L,U}}
		\DrawTwoToOneBox{u_1}{L}{G}{U}{S}{M}
		\DrawFinalLines{{M/S}}
	\end{boxdiagram}
\end{multline*}
if we define
\begin{equation}
	s^* u_2 \colon \mathcal O_M(\X \times G \times G) \to \mathcal O_M(\X \times G) ,
	\quad
	s^* u_2 = (\Delta^* \tensor \id)
	\circ (\id \tensor \psi_s^*)
	\circ (\id \tensor \AlgebraicIso_r(u_2) \tensor \id)
	\circ (\id \tensor F^*) .
\end{equation}
Using that $\id \tensor m^* = M^*$, the following description of the convolution on $\mathcal O_r'(\mathcal G)$, similar to \cite{LMV17}*{Theorem~3.2},
is immediate.

\begin{corollary}
	The convolution of two distributions $u_1, u_2 \in \mathcal O_r'(\mathcal G)$ is given by
	\begin{equation}
		u_1 * u_2 = u_1 \circ s^* u_2 \circ M^* .
	\end{equation}
\end{corollary}

\paragraph{Involution}
It will not be possible to define an involution on all of $\Algebra O'_{r}(\mathcal G)$.
Similar to \cite{LMV17}, we therefore introduce:

\begin{definition}	
	Define \(\Algebra O'_{r,s}(\mathcal G)=\Algebra O'_r(\mathcal G)\cap \IntegrationMap_r^{-1}\IntegrationMap_s(\Algebra O'_s(\mathcal G))\),
	i.e.\ the space of range fibred distributions that can also be fibred along the source fibres.
\end{definition}
\begin{example} 
	For \(f\in\Schwartz(\mathcal G)\) the \(r\)-fibred distribution \(u_f\in\Algebra O'_r(\mathcal G)\) is contained in \(\Algebra O'_{r,s}(\mathcal G)\). Namely using \cref{res:polynomial_action:properties} \refitem{item:noJ} one computes \(\int_r(u_f)=\int_s(\tilde u_f)\). 
\end{example}
The groupoid involution yields a map \(I_*\colon \Algebra O'_r(\mathcal G)\to\Algebra O'_s(\mathcal G)\) defined by \(I_*(u)=u\circ I^*\) and similarly \(I_*\colon \Algebra O'_s(\mathcal G)\to\Algebra O'_r(\mathcal G)\). 
The following diagram commutes:
\begin{equation} \label{diag:inversions}
			\begin{tikzcd}
					\Schwartz(G \times \X)  \arrow[d,"I^*"]\arrow[r,"u"] &  \mathcal O'_r(\mathcal G) \arrow[d,"I_*"] \arrow[r,"\IntegrationMap_r"]&\mathcal D'(\X \times G) \phantom{.}\arrow[d,"I_*"]\\
					\Schwartz(G \times \X)  \arrow[r,  "\tilde u"] &\mathcal O'_s(\mathcal G)  \arrow[r,"\IntegrationMap_s"]&\mathcal D'(\X \times G) .
				\end{tikzcd}
		\end{equation}

Similar to \cite{AMY21}*{Lemma~7.3}, one can show that the convolution can be restricted to \(\Algebra O'_{r,s}(\mathcal G)\).
\begin{proposition}\label{res:convolution-on-rs-fibred}
	The convolution on \(\Algebra O'_r(\mathcal G)\) restricts to \(\Algebra O'_{r,s}(\mathcal G)\).
\end{proposition}
\begin{proof}
	Define a convolution of \(s\)-fibred distributions by \(v_1\mathbin{\tilde{*}}v_2=I_*(I_*(v_2)*I_*(v_1))\) for \(v_1,v_2\in\Algebra O'_s(\mathcal G)\). 
		Let \(u_1,u_2\in\Algebra O'_{r,s}(\mathcal G)\) with \(\IntegrationMap_r(u_i)=\IntegrationMap_s(\tilde{u}_i)\) for \(i=1,2\). We show that \(\IntegrationMap_r(u_1*u_2)=\IntegrationMap_s(\tilde{u}_1\mathbin{\tilde{*}}\tilde{u}_2)\) to see that \(u_1*u_2\in\Algebra O'_{r,s}(\mathcal G)\).
	One has on the one hand
	\begin{equation*}
		\IntegrationMap_r(u_1*u_2)=\IntegrationMap\circ u_1\circ s^*u_2\circ M^* = \IntegrationMap \circ \tilde{u}_1\circ s^*u_2\circ M^*
	\end{equation*}
	and on the other hand
	\begin{equation*}
		\IntegrationMap_s(\tilde u_1\mathbin{\tilde *}\tilde u_2)
		=\IntegrationMap \circ I_*(\tilde u_2) \circ s^*(I_* \tilde u_1)\circ M^*\circ I^* 
		= \IntegrationMap \circ u_2 \circ I^* \circ  s^*(I_*\tilde u_1)\circ (\id\tensor F^*) \circ I^*_{13} \circ (I^* \tensor \id) \circ M^*
	\end{equation*}
	where the notation $I^*_{13}$ shall mean that the map $I^*$ is applied to the first and third tensor factor.
	We used that $I \circ M = M \circ (I \times \id) \circ I_{13} \circ (\id \times F)$,
	which the reader can easily verify by evaluating both sides on $(x,v,w) \in \X \times G \times G$. 
	Similarly, one checks that
	\begin{gather*}
		\begin{boxdiagram}{M}
			\DrawTwoLineBox{\psi_s^*}{C}{G}{B}{S}
			\DrawTwoToOneBox{\Delta^*}{B}{S}{A}{S}{U}
			\DrawTwoLineBox{I^*}{C}{G}{U}{S}
			\DrawFinalLines{{U/S,C/G}}
		\end{boxdiagram}
		=
		\begin{boxdiagram}{M}
			\DrawTwoLineBox{I^*}{C}{G}{A}{S}
			\DrawBox{\Delta^*}{B}{A}{{A/S}}{{U}}
			\DrawLeftLine{B}{{0/1/S}}
			\DrawRightLine{B}{{1/2/S}}
			\DrawFinalLines{{U/S,C/G}}
		\end{boxdiagram} \,\,\,\raisebox{-.2cm}{,}
		\\
		\begin{boxdiagram}{M}
			\DrawTwoLineBox{I^*}{L}{G}{U}{S}
			\DrawOneLineBox{i^*}{L}{G}
			\DrawFinalLines{{U/S,L/G}}
		\end{boxdiagram}
		=
		\begin{boxdiagram}{M}
			\DrawTwoLineBox[5]{(\psi_s^{-1})^*}{L}{G}{U}{S}
			\DrawFinalLines{{U/S,L/G}}
		\end{boxdiagram}
		\qquad\text{and}\qquad
		\begin{boxdiagram}{M}
			\DrawLeftLine{B}{{0/1/S}}
			\DrawTwoLineBox[5]{(\psi_s^{-1})^*}{C}{G}{A}{S}
			\DrawRightLine{B}{{1/2/S}}
			\DrawBox{\Delta^*}{B}{A}{{A/S}}{{U}}
			\DrawFinalLines{{U/S}}
		\end{boxdiagram}
		=
		\begin{boxdiagram}{M}
			\DrawTwoLineBox{\psi_s^*}{C}{G}{B}{S}
			\DrawTwoToOneBox{\Delta^*}{B}{S}{A}{S}{U}
			\DrawTwoLineBox[5]{(\psi_s^{-1})^*}{C}{G}{U}{S}
			\DrawFinalLines{{C/G,U/S}}
		\end{boxdiagram} \,\,\,\raisebox{-.2cm}{.}
	\end{gather*}
	Using also that $\AlgebraicIso_r(I_* \tilde u_1)(\varphi) = \tilde u_1(I^*(1 \tensor \varphi)) = \tilde u_1(1 \tensor i^* \varphi) = \AlgebraicIso_s(\tilde u_1)(i^* \varphi) = \tilde U_1 \circ i^*(\varphi)$
	holds for all $\varphi \in \Algebra O_M(G)$ and that 
	$(\id \times F) \circ (\Delta \times \id) \circ \Delta = (\Delta \times \id) \circ \Delta$
	holds on $\X$,
	we compute
	\begin{align*}
		u_2\circ I^*\circ s^*(I_*\tilde u_1)\circ F^*_{23} \circ I^*_{13} \circ I^*_{12}
		&=
		\begin{boxdiagram}{M}
			\DrawTwoLineBox{I^*}{B}{G}{A}{S}
			\DrawTwoLineBox{I^*}{C}{G}{A}{S}
			\DrawOneToTwoBox{F^*}{C}{G}{C}{B}
			\DrawTwoLineBox{F^*}{C}{G}{B}{G}
			\DrawOneLineBox{i^*}{B}{G}
			\DrawOneLineBox{\text{\small$\tilde U_1$}}{B}{G}
			\DrawTwoLineBox{\psi_s^*}{C}{G}{B}{S}
			\DrawTwoToOneBox{\Delta^*}{B}{S}{A}{S}{U}
			\DrawTwoLineBox{I^*}{C}{G}{U}{S}
			\DrawOneLineBox{U_2}{C}{G}
			\DrawTwoToOneBox{\Delta^*}{C}{S}{U}{S}{HL}
			\DrawFinalLines{{HL/S}}
			\DrawLeftLine{B}{{1/2/G}}
			\DrawRightLine{B}{{2/3/G}}
			\AddOverbrace{s^*(I_*\tilde u_1)}{A}{4}{8}
			\AddOverbrace{u_2}{U}{10}{11}
		\end{boxdiagram}
		\\
		&= \begin{boxdiagram}{M}
			\DrawTwoLineBox{I^*}{B}{G}{A}{S}
			\DrawTwoLineBox{I^*}{C}{G}{A}{S}
			\DrawBox{i^*}{B}{B}{}{{B}}
			\DrawOneLineBox{\text{\small$\tilde U_1$}}{B}{G}
			\DrawTwoLineBox{I^*}{C}{G}{A}{S}
			\DrawBox{\Delta^*}{B}{A}{{A/S}}{{U}}
			\DrawOneLineBox{U_2}{C}{G}
			\DrawTwoToOneBox{\Delta^*}{C}{S}{U}{S}{HL}
			\DrawFinalLines{{HL/S}}
			\DrawLeftLine{B}{{1/2/G,4/5/S}}
			\DrawRightLine{B}{{2/3/G,5/6/S}}
		\end{boxdiagram}
		\\
		&= \begin{boxdiagram}{M}
			\DrawOneLineBox{U_2}{C}{G}
			\DrawBox[5]{(\psi_s^{-1})^*}{B}{A}{{B/G,A/S}}{{B,A}}
			\DrawOneLineBox{\text{\small$\tilde U_1$}}{B}{G}
			\DrawTwoLineBox{F^*}{C}{S}{B}{S}
			\DrawTwoToOneBox{\Delta^*}{B}{S}{A}{S}{U}
			\DrawEmptyBox
			\DrawTwoToOneBox{\Delta^*}{C}{S}{U}{S}{HL}
			\DrawFinalLines{{HL/S}}
		\end{boxdiagram}
		\\
		&= \begin{boxdiagram}{M}
			\DrawTwoLineBox{F^*}{C}{G}{B}{G}
			\DrawOneLineBox{U_2}{B}{G}
			\DrawLeftLine{B}{{2/3/S}}
			\DrawBox[5]{(\psi_s^{-1})^*}{C}{A}{{C/G,A/S}}{{C,A}}
			\DrawEmptyBox
			\DrawRightLine{B}{{3/5/S}}
			\DrawBox{\Delta^*}{B}{A}{{A/S}}{{U}}
			\DrawOneLineBox{\text{\small$\tilde U_1$}}{C}{G}
			\DrawTwoToOneBox{\Delta^*}{C}{S}{U}{S}{HL}
			\DrawFinalLines{{HL/S}}
		\end{boxdiagram}
		\\
		&= \begin{boxdiagram}{M}
			\DrawTwoLineBox{F^*}{C}{G}{B}{G}
			\DrawOneLineBox{U_2}{B}{G}
			\DrawTwoLineBox{\psi_s^*}{C}{G}{B}{S}
			\DrawTwoToOneBox{\Delta^*}{B}{S}{A}{S}{U}
			\AddUnderbrace{s^* u_2}{C}{1}{4}
			\DrawBox[5]{(\psi_s^{-1})^*}{C}{U}{{C/G,U/S}}{{C,U}}
			\DrawOneLineBox{\text{\small$\tilde U_1$}}{C}{G}
			\DrawTwoToOneBox{\Delta^*}{C}{S}{U}{S}{HL}
			\DrawFinalLines{{HL/S}}
			\AddUnderbrace[4]{\tilde u_1}{C}{5}{7}
		\end{boxdiagram}
		= \tilde u_1\circ s^*u_2 \,. \tag*{\qedhere}
	\end{align*}
\end{proof}
Analogously to \cite{LMV17}*{Theorem~3.2} \(\Algebra O_{r,s}'(\mathcal G)\) can be equipped with an involution.
\begin{lemma}\label{res:equiv_rs_fibred}
	Let \(u\in\Algebra{O}'_r(\mathcal G)\). Then the following are equivalent:
	\begin{enumerate}
		\item \(u\in\Algebra{O}'_{r,s}(\mathcal G)\),\label{item:range_source}
		\item there is \(u^*\in\Algebra O_r'(\mathcal G)\) such that 
		\(\IntegrationMap_r(u^*) = \conj{I_*(\IntegrationMap_r(u))} \),\label{item:adjoint}
		\item there is \(u^t\in\Algebra O_r'(\mathcal G)\) such that \(\IntegrationMap_r(u^t)=I_*(\IntegrationMap_r(u))\).\label{item:transpose}
	\end{enumerate}  
\end{lemma}

\begin{proof}
	To show \refitem{item:range_source} implies \refitem{item:adjoint} let \(\tilde{u}\in \Algebra O'_s(\mathcal G)\) be such that \(\IntegrationMap_r(u)=\IntegrationMap_s(\tilde{u})\). 
	Define \(u^* 
	\in\Algebra{O}'_r(\mathcal G)\) by \(u^*(\varphi)=\conj{\tilde{u}(I^*\conj \varphi)}\) for $\varphi \in \mathcal O_M(\mathcal G)$.
	It is immediate from \eqref{diag:inversions} that 
	$\IntegrationMap_r(u^*) = \overline{I_*(\IntegrationMap_s(\tilde u))} = \overline{I_*(\IntegrationMap_r(u))}$.
	For \refitem{item:adjoint} implies \refitem{item:transpose} define \(u^t(\varphi)=\conj{u^*(\conj{\varphi})}\) for \(\varphi\in\Algebra O_M(\mathcal G)\). Lastly, suppose that the transpose \(u^t\in\Algebra O_r'(\mathcal G)\) in \refitem{item:transpose} exists. Then define \(\tilde{u} = I_* u^t \in\Algebra O_s'(\mathcal G)\).
	By \eqref{diag:inversions} we obtain 
	$\IntegrationMap_s \tilde u = I_* \IntegrationMap_r(u^t) = \IntegrationMap_r u$. 
\end{proof}
By \cref{res:PhiR:injective} the adjoint in \refitem{item:adjoint} and transpose in \refitem{item:transpose} are unique if they exist. 
\begin{proposition}\label{res:involution}
	The map \(^*\colon\Algebra{O}'_{r,s}(\mathcal G)\to\Algebra{O}'_{r,s}(\mathcal G)\) defines an involution on the algebra \(\Algebra O_{r,s}'(\mathcal G)\) extending the involution \eqref{eq:involution} on \(\Schwartz(\mathcal G)\).
\end{proposition}
\begin{proof}
	Let $u \in \mathcal O_{r,s}'(\mathcal G)$ and $u^* \in \mathcal O_r'(\TG)$ be the unique element satisfying \(\IntegrationMap_r(u^*) = \smash{\conj{I_*(\IntegrationMap_r(u))}} \).
	This implies \( \smash{\conj{I_*(\IntegrationMap_r(u^*))}} = \IntegrationMap_r(u) \),
	so that $(u^*)^* = u$ and $u^* \in \Algebra O_{r,s}'(\TG)$ by \cref{res:equiv_rs_fibred} \refitem{item:adjoint} $\Rightarrow$ \refitem{item:range_source}. So the involution $^*$ is well-defined on $\Algebra O'_{r,s}(\mathcal G)$.
	
	Moreover, it is easy to check that \(u\mapsto u^*\) is antilinear. To show that \((u_1*u_2)^*=u_2^**u_1^*\), note that this is equivalent to \((u_1*u_2)^t=u_2^t*u_1^t\). 
	By the proof of \cref{res:convolution-on-rs-fibred} one has $\widetilde{u_1 * u_2} = \tilde u_1 \mathbin{\tilde*} \tilde u_2$, so
	\begin{align*}
		(u_1*u_2)^t=\widetilde{u_1*u_2}\circ I^* =(\tilde{u}_1 \mathbin{\tilde *}\tilde{u}_2)\circ I^* = (I_*\tilde u_2)*(I_*\tilde u_1)= u_2^t*u_1^t.
	\end{align*}	
To see that the involution extends the one on \(\Schwartz(\mathcal G)\), one needs to show \(u_{f}^*=u_{f^*}\) for \(f\in\Schwartz(\mathcal G)\) which holds as
\begin{equation*}
	u_f^*(\varphi)(x)=\conj{\tilde u_f(I^*\conj{\varphi})}(x)=\conj{\int_G f(x\cdot v^{-1},v)I^*\conj\varphi(x\cdot v^{-1},v)\D v}
	=\int_G \conj{f(x\cdot v,v^{-1})}\varphi(x,v)\D v = u_{f^*}(\varphi)(x)
\end{equation*}
for all \(\varphi\in\Algebra O_M(\X\times G)\) and \(x\in \X\).
\end{proof}

\begin{remark}
	An important ingredient in the calculus of \cite{vEY19} is that properly supported smooth densities form a two-sided ideal (see \cite{vEY19}*{Proposition~9} and \cite{LMV17}*{Proposition~3.3}). For the Shubin calculus in \cref{sec:calculus} we will use Schwartz functions as the residual class. However, \(\Schwartz(\mathcal G)\) does not define an ideal in \(\mathcal O'_{r,s}(\mathcal G)\).  For example, consider the action of \(G\) on \(\X=G\) by multiplication. Let \(\varphi\in\Schwartz(G\times G)\) with \(\varphi(0,0)\neq 0\) and define \(u\in\Algebra{O}_{r,s}'(\mathcal G)\) 
	by \(u(f)(x,w)=f(x,x)\). Then 
	\begin{equation}
		u*\varphi(x,v)=\varphi(0,x^{-1}v)
	\end{equation}
	does not decay rapidly on the diagonal in \(G\times G\).	
	Only when we restrict later on to essentially homogeneous distributions on the tangent groupoid, Schwartz functions become an ideal. We remark that in the general situation above one can show that \(\Algebra O_M(\X)\completedtensor\Schwartz(G)\) defines a two-sided ideal in \(\mathcal O_{r,s}'(\mathcal G)\). 
\end{remark}

\paragraph{Operator representation}
For \(u\in\Algebra O'_r(\mathcal G)\) there is an operator 
\begin{equation} \label{eq:Op}
	\Op(u)\colon\Algebra O_M(\X)\to\Algebra O_M(\X) ,
	\quad 
	\Op(u) = u \circ s^* ,
\end{equation}
i.e.\ for $\varphi\in\Algebra O_M(\X)$ this means that $\mathrm{Op}(u) \varphi = u(s^* \varphi)$. 
Under the isomorphisms of \cref{res:identifications:fibredDistributions} we obtain
$\mathrm{Op}(U) = \Delta^* \circ (\mathrm{id} \tensor U) \circ s^*$
if $U = \AlgebraicIso_r(u)$.
As in \cite{AMY21}*{Prop.~5.1} one can show the following result.
\begin{lemma}\label{res:fibred-distr-as-operators}
	The map \(\Op\colon\Algebra O_r'(\mathcal G)\to\mathcal L(\Algebra O_M(\X),\Algebra O_M(\X))\) is a homomorphism. 
\end{lemma}
\begin{proof}
	One easily checks that 
	\begin{equation*}
		\begin{boxdiagram}{M}
			\DrawOneToTwoBox{s^*}{M}{S}{L}{A}
			\DrawOneToTwoBox{\overline m^*}{L}{G}{C}{B}
			\DrawFinalLines{{A/S,B/G,C/G}}
		\end{boxdiagram}
		=
		\begin{boxdiagram}{M}
			\DrawOneToTwoBox{s^*}{U}{S}{B}{A}
			\DrawOneToTwoBox{s^*}{A}{S}{C}{A}
			\DrawFinalLines{{A/S,C/G}}
			\DrawLeftLine{B}{{1/2/G}}
			\DrawRightLine{B}{{2/3/G}}
		\end{boxdiagram}
		\quad\text{and}\quad
		\begin{boxdiagram}{M}
			\DrawOneToTwoBox{s^*}{A}{S}{C}{A}
			\DrawOneToTwoBox{\psi_s^*}{C}{G}{C}{B}
			\DrawTwoToOneBox{\Delta^*}{B}{S}{A}{S}{U}
			\DrawFinalLines{{U/S,C/G}}
			\DrawLeftLine{B}{{0/1/S}}
			\DrawRightLine{B}{{1/2/S}}
		\end{boxdiagram}
		=
		\begin{boxdiagram}{M}
			\DrawTwoToOneBox{\Delta^*}{L}{S}{U}{S}{M}
			\DrawOneToTwoBox{s^*}{M}{S}{L}{U}
			\DrawFinalLines{{U/S,L/G}}
		\end{boxdiagram} \,\,\,\raisebox{-0.2cm}{.}
	\end{equation*}
	Now let \(U_1, U_2 \in \mathcal L(\mathcal O_M(G), \mathcal O_M(\X))\). Then we compute
	\begin{align*}
		\Delta^* \circ (\mathrm{id} \tensor (U_1 * U_2)) \circ s^*  &= 
		\begin{boxdiagram}{HU}
			\DrawOneToTwoBox{s^*}{HU}{S}{L}{A}
			\DrawOneToTwoBox{\overline m^*}{L}{G}{C}{B}
			\DrawOneLineBox{U_2}{B}{G}
			\DrawTwoLineBox{\psi_s^*}{C}{G}{B}{S}
			\DrawOneLineBox{U_1}{C}{G}
			\DrawTwoToOneBox{\Delta^*}{C}{S}{B}{S}{L}
			\DrawTwoToOneBox{\Delta^*}{L}{S}{A}{S}{HU}
			\DrawFinalLines{{HU/S}}
		\end{boxdiagram} \\
	&= 
	\begin{boxdiagram}{U}
		\DrawOneToTwoBox{s^*}{U}{S}{B}{A}
		\DrawOneLineBox{U_2}{B}{G}
		\DrawOneToTwoBox{s^*}{A}{S}{C}{A}
		\DrawOneToTwoBox{\psi_s^*}{C}{G}{C}{B}
		\DrawTwoToOneBox{\Delta^*}{B}{S}{A}{S}{U}
		\DrawOneLineBox{U_1}{C}{G}
		\DrawTwoToOneBox{\Delta^*}{C}{S}{U}{S}{HL}
		\DrawLeftLine{B}{{2/3/S}}
		\DrawRightLine{B}{{3/4/S}}
		\DrawFinalLines{{HL/S}}
	\end{boxdiagram} \\
	&= 
	\begin{boxdiagram}{M}
		\DrawOneToTwoBox{s^*}{M}{S}{L}{U}
		\DrawOneLineBox{U_2}{L}{G}
		\DrawTwoToOneBox{\Delta^*}{L}{S}{U}{S}{M}
		\DrawEmptyBox
		\DrawOneToTwoBox{s^*}{M}{S}{L}{U}
		\DrawOneLineBox{U_1}{L}{G}
		\DrawTwoToOneBox{\Delta^*}{L}{S}{U}{S}{M}
		\DrawFinalLines{{M/S}}
	\end{boxdiagram} \\
		&= \Delta^* \circ (\mathrm{id} \tensor U_1) \circ s^* \circ \Delta^* \circ (\mathrm{id} \tensor U_2)\circ s^*.
		\tag*{\qedhere}
	\end{align*}
\end{proof}
For a right action of a graded Lie group \(G\) on \(\X = \RR^d\), define the \emph{shear map} 
\begin{equation} \label{eq:shearmap}
	\Theta\colon \X\times G\to \X\times \X , \qquad
	(x,v)\mapsto(x,\theta_v(x)) .
\end{equation} In fact, it defines a groupoid homomorphism from the action groupoid \(\X\rtimes G\) to the pair groupoid \(\X\times \X\), restricting to the identity on the unit space $\X$.
The action is free if and only if \(\Theta\) is injective. It is transitive if and only if \(\Theta\) is a surjective, and in this case $\Theta$ is even a surjective submersion \cite{Hel79}*{Prop.~II.4.3}. 
The following strengthening of the two properties will be useful.
\begin{definition}\label{def:polynomially-free}
	The polynomial right action \(\theta\) of \(G\) on \(\X\) is called \emph{polynomially free} if the shear map \(\Theta\colon \X\times G\to \X\times \X\) admits a polynomial left inverse \(\Omega\colon \X\times \X\to \X\times G\). It is called \emph{polynomially transitive} if \(\Theta\) is a polynomial fibre projection.
\end{definition}

\begin{lemma}\label{res:Phi}
	Let \(\theta\) denote a polynomial right action of \(G\) on \(\X\). Then there is a map \(\Theta_*\colon \Algebra O'_r(\mathcal G)\to \Algebra O'_r(\X\times \X)\) defined by \(\Theta_*u=u\circ\Theta^*\). It has the following properties
		\begin{enumerate}
			\item\label{item:free} if \(\theta\) is polynomially free, \(\Theta_*\) is injective,
			\item\label{item:transitive} if \(\theta\) is polynomially transitive, \(\Theta_*\) restricts to a surjective map \(\Theta_*\colon\Schwartz(\mathcal G)\to\Schwartz(\X\times \X)\).
		\end{enumerate}	
\end{lemma}
\begin{proof}
	It is easy to check that \(u\circ\Theta^*\) is \(\Algebra O_M(\X)\)-linear with respect to the range map \(r\colon \X\times \X\to \X\) given by \((x,y)\mapsto x\) since \(\Theta\)
	intertwines the range maps. Therefore, \(\Theta_*\) is well-defined.
	
	To show \refitem{item:free} let \(\Omega\) be a polynomial map satisfying \(\Omega\circ\Theta=\id\), then \(\Theta^*\circ\Omega^*=\id_{\Algebra O_M(\X\times G)}\). Consequently, every \(\varphi\in\Algebra O_M(\X\times G)\) can be written as \(\Theta^*\tilde\varphi\) for some \(\tilde\varphi\in\Algebra O_M(\X\times \X)\). Thus, \(\Theta_*u=0\) for \(u\in\Algebra O'_r(\mathcal G)\) implies \(u=0\).
	
	When \(\theta\) is polynomially transitive, there is a polynomial diffeomorphism \(\psi\colon \X\times G\to \X\times \X\times\RR^{n-d}\) such that \(\Theta=(\pr_1,\pr_2)\circ\psi\). 
	Using this identification we show below that the map $\Theta_*$ becomes, up to multiplication by a non-zero constant, the map
	$\Schwartz(\X\times \X\times\RR^{n-d})	\to \Schwartz(\X \times \X)$
	which integrates out the last variables. From this, \refitem{item:transitive} follows easily.
		
	Note that \(\psi\) is of the form \(\psi(x,v)=(x,\psi_x(v))\), where \(\psi_x\colon G\to \X\times\RR^{n-d}\) is a polynomial diffeomorphism for every \(x\in \X\).
	The inverse is of the form $\psi^{-1}(x,y,z) = (x, \psi_x^{-1}(y,z))$.
	Hence \cref{res:determinantIsConstant} shows that \((x,v)\mapsto \abs{\det D_v(\psi_x)}\) 
	is a non-zero constant, denoted by \(c\) in the following.
	Thus, for \(f\in\Schwartz(\X \times \X \times \RR^{n-d})\), \(\varphi\in\Algebra O_r(\X \times \X)\) and \(x\in \X\) we compute
	\begin{multline*}
		\Theta_*(u_{\psi^* f})(\varphi)(x) 
		= u_{\psi^* f}(\Theta^* \varphi)(x)
		= \int_{G} \psi^* f(x,v) \Theta^* \varphi(x,v) \D v
		\\ 
		= \int_{G} f(x, \psi_x(v)) \varphi((\pr_1,\pr_2)(x,\psi_x(v))) \D v
		= \frac 1 c \int_{\X \times \RR^{n-d}} f(x, y,z) \varphi(x,y) \D y \D z
		=  u_{\tilde f}(\varphi)(x),
	\end{multline*}
	where $\tilde f \in \Schwartz(\X \times \X)$ is obtained from $f$ by integrating out the $z$-variable and multiplying with $\frac 1 c$.
\end{proof}

\begin{corollary}\label{res:Op-injective}
		Let \(\theta\) denote a polynomial right action of \(G\) on \(\X\). Then the following holds
		\begin{enumerate}
			\item\label{item:Op-free} when \(\theta\) is polynomially free, then \(\Op\colon\Algebra O'_r(\mathcal G)\to\mathcal L(\Algebra O_M(\X),\Algebra O_M(\X))\) is injective.
			\item\label{item:Op-transitive} when \(\theta\) is polynomially transitive, then \(\Op(\Schwartz(\mathcal G))\) consists precisely of operators that extend to continuous maps \(\Schwartz'(\X)\to\Schwartz(\X)\).
		\end{enumerate}	
\end{corollary}

\begin{proof}
	Note that \(\Op(u)=\Theta_*(u)\circ s^*\), where \(s\colon \X\times \X\to \X\) is the source map \(s(x,y)=y\). The Schwartz kernel of \(\Op(u)\) is, therefore, given by \(\IntegrationMap_r(\Theta_*(u))\).
	
	To see \refitem{item:Op-free} suppose that \(\Op(u)=\Op(v)\) for \(u,v\in\Algebra O'_r(\mathcal G)\), then their Schwartz kernels coincide. By the injectivity of \(\IntegrationMap_r\) (\cref{res:PhiR:injective}) this implies \(\Theta_*(u)=\Theta_*(v)\). Hence, the first claim follows as \(\Theta_*\) is injective by \cref{res:Phi}~\refitem{item:free}.
	Claim \refitem{item:Op-transitive} follows from \cref{res:Phi}~\refitem{item:transitive} and \(\Schwartz(\X\times \X)\cong\mathcal L(\Schwartz'(\X),\Schwartz(\X))\).
\end{proof}

\begin{lemma} \label{res:OpDual}
	For all $f_1, f_2 \in \SmoothCompactSupp(\X)$ and $u \in \Algebra O_{r,s}'(\mathcal G)$ the following equality holds
	\begin{equation} \label{eq:OpDual}
		\int_\X (\Op(u) f_1)(x) f_2(x) \D x = \int_\X f_1(x) (\Op(u^t) f_2)(x) \D x .
	\end{equation}
\end{lemma}

\begin{proof}
	Using \(\IntegrationMap_ru =I_*\left(\IntegrationMap_r u^t\right) \) we compute 
	\begin{multline*}
		\int_\X (\Op(u) f_1)(x) f_2(x) \D x
		=
		\int_\X u(s^* f_1)(x) f_2(x) \D x
		=
		\int_\X u(s^* f_1 \cdot r^* f_2 )(x) \D x
		\\
		\overset{\mathclap{(*)}}{=}
		\int_\X u^t(I^*(s^* f_1 \cdot r^* f_2) )(x) \D x
		= 
		\int_\X u^t(r^* f_1 \cdot s^* f_2) (x) \D x
		=
		\int_\X f_1(x) (\Op(u^t)f_2) (x) \D x .
	\end{multline*}
	We need to justify ($*$) since we only know that 
	$\langle\IntegrationMap_r u, \varphi\rangle = \langle \IntegrationMap_s u^t \circ I^*, \varphi\rangle$
	holds for all $\varphi \in \mathcal D'(\TG)$,
	but $\varphi = s^* f_1 \cdot r^* f_2$ might have non-compact support.
	Choose a smooth cut-off function $\chi \in \SmoothCompactSupp(\TG)$ which is $1$ in a neighbourhood of $(0,0)$ and let $\chi_j(x,v) = \chi(\frac 1 j x, \frac 1 j v)$
	where $G$ is identified with the vector space $\lie g$ using standard coordinates.
	Then $\chi_j \varphi \to \varphi$ in $\Algebra O_M(\TG)$,
	hence $u(\chi_j \varphi) \to u(\varphi)$ in $\Algebra O_M(\X)$.
	But the support of $u(\chi_j \varphi)$ is contained in the support of $u(\varphi)$,
	hence contained in the compact support of $f_2$.
	Therefore $u(\chi_j\varphi) \to u(\varphi)$ in $\SmoothCompactSupp(\X)$
	and $\langle\IntegrationMap_r u, \chi_j\varphi\rangle =\int_\X u(\chi_j\varphi)(x) \D x \to \int_\X u(\varphi)(x) \D x = \langle\IntegrationMap_r u, \varphi\rangle$.
	Similarly, $\langle \IntegrationMap_s u^t \circ I^*, \chi_j \varphi \rangle \to \langle \IntegrationMap_s u^t \circ I^*, \varphi \rangle$
	and the equality above remains true for $\varphi = s^* f_1 \cdot r^* f_2$.
\end{proof}

\paragraph{Groupoid homomorphisms}
Let us examine functoriality of action groupoids and their algebras of fibred distributions.
Later, we will apply these results to the zoom action on the Shubin tangent groupoid. 

Throughout this section, assume that $G$, $H$ are graded Lie groups, that $\X$, $\Y$ are vector spaces
and that \(\theta\colon \X\times G\to \X\) and \(\vartheta\colon \Y\times H\to \Y\) are polynomial right actions.
The following lemma is easy to prove.
\newcommand{\XX}{\,\X} 
\begin{lemma}\label{res:groupoidHom}
	Let \(\tau_G\colon G\to H\) be a Lie group homomorphism and \(\tau_{\XX}\colon \X\to \Y\) be a smooth map. Suppose \(\tau_{\XX}\) is compatible with the actions in the sense that \(\tau_{\XX}(\theta_v(x))=\vartheta_{\tau_G(v)}(\tau_{\XX}(x))\) for all \(x\in \X\) and \(v\in G\). Then \(\tau\colon \X\rtimes G \to\Y \rtimes H\) with \(\tau(x,v)=(\tau_{\XX}(x),\tau_G(v))\) defines a Lie groupoid homomorphism.
\end{lemma}
\begin{lemma}\label{res:homomorphism} 
	Let \(\tau_G\colon G\to H\) be a Lie group homomorphism and \(\tau_{\XX}\colon \X\to \Y\) a polynomial map.  Suppose \(\tau_{\XX}\) has a polynomial right inverse \(\sigma_{\Y}\colon \Y\to \X\), i.e.\ \(\tau_{\XX}\circ\sigma_\Y=\id_\Y\), which satisfies \(\sigma_\Y(\vartheta_{\tau_G(v)}(y))=\theta_v(\sigma_\Y(y))\) for all \(y\in\Y\) and \(v\in G\). Then the following holds:
	\begin{enumerate}
		\item \label{item:conv-auto} There is a homomorphism \(\tau_*\colon\Algebra O'_r(\X\rtimes G)\to \Algebra O'_r(\Y\rtimes H)\) with respect to convolution defined by	
		\begin{equation*}
			\tau_*u(\varphi)=\sigma_\Y^*(u(\tau^*\varphi)) \quad \text{for }u\in\Algebra O'_r(\X\rtimes G)\text{ and } \varphi\in\Algebra O_M(\Y\rtimes H).
		\end{equation*}				
		\item \label{item:schwartz-auto} If \(\tau_G\colon G\to G\) is a group automorphism, \(\tau_*\) restricts to a homomorphism \(\Schwartz(\X\rtimes G)\to\Schwartz(\Y\rtimes G)\) with \(\tau_*(u_f)=u_{\abs{\det(D_v\tau_G^{-1})}(\sigma_\Y\times\tau_G^{-1})^*f}\) for \(f\in\Schwartz(\X\rtimes G)\),
		where \(v\mapsto \det(D_v\tau_G^{-1})\) is constant by \cref{res:determinantIsConstant}.
		\end{enumerate}
\end{lemma}
\begin{proof}				
		To see \refitem{item:conv-auto}, note first that any Lie group homomorphism is linear in standard coordinates, hence polynomial.
		The polynomial maps \(\tau\) and \(\sigma_{\Y}\) induce continuous maps \(\tau^*\colon \Algebra O_M(\Y\rtimes H)\to\Algebra O_M(\X\rtimes G)\) and \(\sigma_{\Y}^*\colon\Algebra O_M(\X)\to\Algebra O_M(\Y)\), so that \(\tau_*u\colon \Algebra O_M(\Y\rtimes H)\to\Algebra O_M(\Y)\) is well-defined and continuous. It is \(\Algebra O_M(\Y)\)-linear as \(\tau^*(r^*(f))=r^*(\tau_{\XX}^*(f))\) for \(f\in\Algebra O_M(\Y)\) so that for all \(\varphi\in\Algebra O_M(\Y\rtimes H)\)
		\begin{equation*}
			\tau_*u(r^*f\cdot \varphi)=\sigma_\Y^*(u(r^*(\tau_{\XX}^*f)\cdot \tau^*\varphi))=(\tau_{\XX}\circ\sigma_\Y)^*(f)\cdot \tau_*u(\varphi)=f\cdot \tau_*u(\varphi).
		\end{equation*}
		To show that \(\tau_*\) is a homomorphism, we use the isomorphism from \cref{res:identifications:fibredDistributions}. One computes that \(\Psi_r(\tau_*u)=\sigma_\Y^*\circ \AlgebraicIso_r(u) \circ\tau_G^*\) for \(u\in\Algebra O'_r(\X\rtimes G)\). Let \(U_1,U_2\in\mathcal L(\Algebra O_M(G),\Algebra O_M(\X))\) and note that \(\sigma_\Y^*\circ \Delta^*=\Delta^*\circ (\sigma_\Y^*\tensor\sigma_\Y^*)\) and \(F^*\circ m^*\circ\tau_G^*=(\tau_G^*\tensor\tau_G^*)\circ F^*\circ m^*\). Moreover, using that \(\sigma_{\Y}\) and \(\tau_G\) are compatible, one gets \((\id\times\tau_G)\circ\psi_s\circ(\sigma_\Y\times\id)=(\sigma_\Y\times\id)\circ\psi_s\circ(\id\times\tau_G)\) and thus
		\begin{align*}
			\sigma_\Y^*\circ (U_1*U_2)\circ\tau_G^* &= \Delta^*\circ (\sigma_\Y^*\tensor \sigma_\Y^*)\circ (\id\otimes U_1)\circ\psi^*_s\circ(U_2\otimes \id)\circ(\tau_G^*\tensor\tau_G^*)\circ F^*\circ m^*\\
			&= \Delta^*\circ (\id\tensor \sigma_\Y^*\circ U_1 )\circ(\sigma_\Y\times\id)^*\circ\psi^*_s\circ(\id\times\tau_G)^*\circ(U_2\circ\tau_G^*\tensor\id)\circ F^*\circ m^*\\
			&=(\sigma_\Y^*\circ U_1\circ\tau_G^*)*(\sigma_\Y^*\circ U_2\circ\tau_G^*).
		\end{align*}		
		For \refitem{item:schwartz-auto} suppose that \(\varphi_G\) is an automorphism of \(G\), let \(f\in\Schwartz(\X\rtimes G)\), \(\varphi\in\Algebra O_M(\Y\rtimes G)\) and compute
		\begin{align*}
			\tau_*u_f(\varphi)(y)&= u_f(\tau^*\varphi)(\sigma_\Y(y))=\int_G f(\sigma_\Y(y),v)\varphi(\tau(\sigma_\Y(y),v))\D v\\
			& = \int_G \abs{\det D_v\tau_G^{-1}} f(\sigma_\Y(y),\tau_G^{-1}(v))\varphi(y,v)\D v =u_{\abs{\det D_v\tau_G^{-1}}(\sigma_\Y\times\tau_G^{-1})^*f}(\varphi)(y). \qedhere
		\end{align*}
\end{proof}
When \(\X=\Y\times\RR^m\) is equipped with a family of \(G\)-actions in the fibre \(\Y\) and \(\tau_{\XX}\colon\X\to\Y\) is a fibre projection, \(\tau_*\) restricts to a \(^*\)-homomorphism \(\Algebra O'_{r,s}(\X\rtimes G)\to \Algebra O'_{r,s}(\Y\rtimes H)\).
\begin{lemma}\label{res:*-homomorphism}
	Let \(\X=\Y\times\RR^m\) for some \(m\in\NN_0\), \(\tau_\X(y,z)=y\) and \(\sigma_\Y(y)=(y,z_0)\) for a fixed \(z_0\in\RR^m\). Suppose that  \(\theta\colon \X\times G\to \X\) and \(\vartheta^z\colon \Y\times H\to \Y\) for \(z\in \RR^m\) are polynomial actions satisfying \(\theta_v(y,z)=(\vartheta^z_{\tau_G(v)}(y),z)\) for all \((y,z)\in\X\) and \(v\in G\). Then 
	\begin{equation*}
		\tau_*\colon \Algebra O'_{r,s}(\X\rtimes G)\to \Algebra O'_{r,s}(\Y\rtimes H)
	\end{equation*}
	is a well-defined \(^*\)-homomorphism. 
\end{lemma}
\begin{proof}
		The assumptions on the actions imply that \(\sigma_\Y(\vartheta_{\tau_G(v)}(y))=\theta_v(\sigma_\Y(y))\) for all \(y\in\Y\) and \(v\in G\), therefore \(\tau_*\colon  \Algebra O'_{r}(\X\rtimes G)\to \Algebra O'_{r}(\Y\rtimes H)\) is a homomorphism by \cref{res:homomorphism}. To see that \(\tau_*\) restricts to \(\Algebra O_{r,s}'(\X\rtimes G)\), define \(\tau_*\tilde u=\sigma_\Y^*\circ \tilde u \circ \tau^* =\sigma_\Y^*\circ \tilde u\circ I^*\circ \tau^*\circ I^*\) for \(\tilde u\in\Algebra O'_s(\X\rtimes G)\). The last equality holds as \(\tau|_{\Y\times\{z_0\}\times G}=I\circ\tau\circ I|_{\Y\times\{z_0\}\times G}\). The second description shows that it defines indeed a map \(\tau_*\colon\Algebra O'_s(\X\rtimes G)\to\Algebra O'_s(\Y\rtimes H)\) as \(I^*\circ\tau^*\circ I^*\circ s^*=I^*\circ \tau^*\circ r^*=I^*\circ r^*\circ\tau_X^*=s^*\circ\tau_{\XX}^*\).
		
		Suppose that \(u\in\Algebra O_{r,s}'(\X\rtimes G)\) and that \(\tilde u\in\Algebra O_s'(\X\rtimes G)\) satisfies \(\IntegrationMap_ru =\IntegrationMap_s\tilde u\). We show that \(\IntegrationMap_r(\tau_*u)=\IntegrationMap_s(\tau_*\tilde u)\) to see that \(\tau_*u\in\Algebra O'_{r,s}(\Y\rtimes H)\).
	Assume on the contrary that \(\langle \IntegrationMap_r(\tau_*u),\varphi\rangle\neq\langle\IntegrationMap_s(\tau_*\tilde u),\varphi\rangle\) for some \(\varphi\in\mathcal D(\Y\rtimes H)\) and let
	\begin{equation*}
		\varepsilon= \left|\int_\Y u(\tau^*\varphi)(\sigma_\Y(y))\D y- \int_\Y \tilde u(\tau^*\varphi)(\sigma_\Y(y))\D y\right|>0.
	\end{equation*}
	Consider the continuous maps
	\begin{equation*}
		p(u)\colon z\mapsto \int_\Y u(\tau^*\varphi)(y,z)\D y \quad\text{ and }\quad p(\tilde u)\colon z\mapsto \int_\Y \tilde u(\tau^*\varphi)(y,z)\D y. 
	\end{equation*}
	By continuity one can find a neighbourhood of \(z_0\) in which \(\abs{p(u)(z)-p(\tilde u)(z)}\geq \frac\varepsilon 3\). Take a non-zero \(\chi\in\SmoothCompactSupp(\RR^m)\) which is supported in this neighbourhood and satisfies \(\chi(z)\geq 0\) for all \(z\in\RR^m\). Note that the range and source map of \(\X\rtimes G\) satisfy \(\pr_2\circ r= \pr_2\circ s\) so that \((\pr_2\circ r)^*\chi=(\pr_2\circ s)^*\chi\). Estimating as in the proof of \cref{res:PhiR:injective} gives
	\begin{align*}
		&\abs{\langle\IntegrationMap_r(u)-\IntegrationMap_s(\tilde u),(\pr_2\circ r)^*\chi\cdot\tau^*\varphi\rangle} = \left|\int_\X u(r^*(\pr_2^*\chi)\cdot\tau^*\varphi)- \tilde u(s^*(\pr_2^*\chi)\cdot\tau^*\varphi)\D x\right|\\
		= &\left|\int_{\RR^m} \chi(z)(p(u)(z)-p(\tilde u)(z) )\D z \right|\geq  \frac \varepsilon 3\int_{\RR^m} \chi(z)\D z>0.
	\end{align*}
	This contradicts \(\IntegrationMap_ru =\IntegrationMap_s\tilde u\). It remains to show that \(\tau_*\) is compatible with the involutions 
	\begin{equation*}
		(\tau_* u)^*=\widetilde{\tau_*u}\circ I^* = (\sigma_\Y^{*}\circ \tilde u \circ I^*\circ\tau^*\circ I^*)\circ I^*= \sigma_\Y^{*}\circ \tilde u \circ I^*\circ \tau^* =\tau_*(u^*).\qedhere
	\end{equation*}
\end{proof}
\begin{corollary}\label{res:automorphism} 
	Let \(\tau_G\colon G\to G\) be a Lie group automorphism and \(\tau_{\XX}\colon \X\to \X\) a polynomial diffeomorphism with \(\tau_{\XX}(\theta_v(x))=\theta_{\tau_G(v)}(\tau_{\XX}(x))\) for all \(x\in \X\) and \(v\in G\). Then the following holds:	 
	\begin{enumerate}
		\item The map \(\tau_*\colon\Algebra O'_r(\mathcal G)\to \Algebra O'_r(\mathcal G)\) with \(\tau_*u=(\tau_{\XX}^{-1})^*\circ u\circ\tau^*\) is an automorphism with inverse \((\tau^{-1})_*\).
		\item It restricts to a \(^*\)-automorphism of \(\Algebra O'_{r,s}(\mathcal G)\).
		\item  It restricts to a \(^*\)-automorphism of \(\Schwartz(\mathcal G)\) with \(\tau_*(u_f)=u_{\abs{\det(D_v\tau_G^{-1})}(\tau^{-1})^*f}\) for \(f\in\Schwartz(\mathcal G)\),
		where \(v\mapsto \det(D_v\tau_G^{-1})\) is constant by \cref{res:determinantIsConstant}.
	\end{enumerate}	 
\end{corollary}
\begin{proof}
	Under the assumptions \(\tau\colon \mathcal G\to\mathcal G\) is a groupoid automorphism with inverse \(\tau^{-1}=\tau_{\XX}^{-1}\times\tau_G^{-1}\). Then \((\tau^{-1})_*\) is a homomorphism by \cref{res:homomorphism} and is inverse to \(\tau_*\). The remaining claims follow from \cref{res:homomorphism} and \cref{res:*-homomorphism}.
\end{proof}

\section{Shubin tangent groupoids}\label{sec:shubin-tangent-groupoid}

In this section we define an abstract Shubin tangent groupoid.
Our two main examples are the double dilation groupoid and the representation groupoid, introduced in \cref{sec:shubin:DDG} and \cref{sec:shubin:Rep}.
In both cases, the following assumptions will be satisfied:

\begin{definition} \label{assumption}
	Let $G$ be a graded Lie group with dilations $\alpha$, $\X = \RR^d$ a graded vector space with dilations \(\beta\) and \(\theta^1\colon \X\times G\to \X\) a right action of \(G\) on $\X$. We say that this data defines a \emph{Shubin action} if the following assumptions are satisfied:
	\begin{enumerate}
		\item\label{assumption:polynomial} The action $\theta^1$ is a polynomial action of $G$ on $\X$ 
		in the sense of \cref{def:polynomial:action}.
		\item \label{assumption:compatible} For \(t\neq 0\) set \(\theta^t_v=\beta_t\circ \theta^1_{\alpha_t(v)}\circ\beta_{t^{-1}}\) (see \cref{rem:dilationsExtend} for $t < 0$). Then the map given by
		\begin{equation*}
			\theta(x,t,v)=(\theta^t_v(x),t)
		\end{equation*}
		extends to a smooth map \(\theta\colon \X\times\RR\times G\to \X\times\RR\).
	\end{enumerate}
	We say it has property 
	\begin{enumerate}
	\myitem{(P)}\label{assumption:free} 
	if the shear map \(\Theta^1\colon \X\times G\to \X\times \X\), \((x,v)\mapsto(x,\theta^1_v(x))\) from \eqref{eq:shearmap} is a polynomial diffeomorphism (see \cref{def:polynomialDiffeo}),
	\myitem{(R)}\label{assumption:homomorphisms}
	if $\theta^0_v \colon \X \to \X$ is linear for all $v \in G$.
	\end{enumerate}	
\end{definition}
Suppose that \(\theta^1\) defines a Shubin action, then \(\theta\) defines a polynomial action of \(G\) on \(\X\times\RR\). Namely as \(\alpha\), \(\beta\) and \(\theta\) are polynomial one can write
\[\theta^t_v(x)=\sum_{j=-k}^k p_j(x,v)t^j\]
for some polynomials \(p_j\) and \(k\in\NN_0\). As this extends smoothly to \(t=0\), we must have \(p_j=0\) for \(j<0\). Hence, \(\theta\) is polynomial in \(x,t,v\).
It is easy to check that all $\theta^t$ are actions of $G$ on $\X$. 
Then $\theta$ is an action of $G$ on $\X \times \RR$.

Furthermore, one has for all \(\lambda,t\neq 0\) and \(v\in G\)
\begin{equation}\label{eq:compatibility}
	\beta_\lambda\circ \theta^t_{\alpha_\lambda(v)}\circ\beta_{\lambda^{-1}}=\beta_{\lambda t}\circ\theta^1_{\alpha_{\lambda t}(v)}\circ\beta_{(\lambda t)^{-1}}=\theta^{\lambda t}_v.
\end{equation}
By continuity the action $\theta^0$ satisfies \(\beta_\lambda\circ\theta^0_{\alpha_\lambda(v)}\circ\beta_{\lambda^{-1}}=\theta^0_v\) for all \(v\in G\) and \(\lambda\neq0\).
Conversely, we show in the following lemma that if a field of actions over \(\RR\) satisfies a similar compatibility condition to \eqref{eq:compatibility} where \(\lambda\) one the right hand side is replaced by \(\lambda^k\) for some \(k\in\NN\), then it can be rescaled to define a Shubin action.
\begin{lemma}\label{res:rescaling}
	Let \(G\) be a graded group with dilations \(\alpha\) and \(\X=\RR^d\) a graded vector space with dilations  \(\beta\). Suppose $\vartheta \colon \X \times \RR \times G \to \X \times \RR$ is a right polynomial action of $G$, which is a field of actions \((\vartheta^t)_{t\in\RR}\) with \(\vartheta((x,t),v)=(\vartheta^t(x,v),t)\) for all \(t\in\RR\). Suppose there is \(k\in\NN\) such that
	\begin{equation}\label{eq:compatibility:rescaled}\beta_{\lambda}\circ\vartheta^{ t}_{\alpha_\lambda(v)}\circ\beta_{\lambda^{-1}} = \vartheta^{\lambda^kt}_{v}\quad \text{for all }\lambda\neq 0\text{, }t\in\RR\text{ and }v\in G\end{equation}
	holds, then \(\theta^1=\vartheta^1\) defines a Shubin action with \(\theta^t=\vartheta^{t^k}\) for all \(t\in\RR\).
\end{lemma}
\begin{proof}
	The action \(\theta^1\) is polynomial and $\theta^t_v=\smash{\beta_t\circ\vartheta^1_{\alpha_t(v)}\circ\beta_{t^{-1}}=\vartheta^{t^k}_v}$ holds for \(t\neq 0\) by \eqref{eq:compatibility}.
	So $\theta^t$ extends smoothly to \(t=0\) by \(\vartheta^0\).	
\end{proof}

\begin{definition}[Shubin tangent groupoid and zoom action] \label{def:shubin_groupoid}
	Given the data from \cref{assumption}, 
	the \emph{Shubin tangent groupoid} \(\TG\) is defined as the smooth action groupoid,
	see \cref{def:action_groupoid}, of the right \(G\)-action $\theta$ on \(\X \times\RR\).
	In particular its object space is $\X \times \RR \times G$ and its unit space is $\X \times \RR$.
	The \emph{Shubin zoom action} of \(\Rp\) on \(\TG\) is given by
	\begin{equation}
		\tau_\lambda(x,t,v)=(\beta_{\lambda^{-1}}(x),\lambda^{-1}t,\alpha_{\lambda}(v))
	\end{equation}
	for $(x,t,v) \in \TG$, with underlying map 
	\(\tau_\lambda^0(x,t)=(\beta_{\lambda^{-1}}(x),\lambda^{-1}t)\) on the unit space.
\end{definition}

The fact that \(\theta\) is a field of actions \(\theta^t\) yields that the Shubin tangent groupoid is a smooth field of Lie groupoids over \(\RR\)  in the sense of \cite{LR01}*{Def.~5.2} with respect to the projection \(p\colon\TG\to\RR\) given by \(p(x,t,v)=t\).
In particular, for each $t \in \RR$, the preimage $\X\rtimes^tG\defeq p^{-1}(t)$ is the action Lie groupoid with respect to \(\theta^t\). 
Since the actions $\theta$ and $\theta^t$ are polynomial,
we can use the convolution algebras from \cref{sec:action-groupoids-convolution-algebra}.

The convolution of $\Algebra O_r'(\TG)$ is `pointwise in $t$' in the following sense.

\begin{lemma} \label{res:evaluationOfDistributions}
	The evaluation map $\ev_t \colon \mathcal O'_r(\mathcal G) \to \mathcal O'_r(\X \rtimes^t G)$, $u \mapsto \ev_t (u) \eqqcolon u_t$ is a homomorphism for the respective convolutions
	defined in \cref{res:convolution}.
	Its restriction $\mathcal O'_{r,s}(\mathcal G) \to \mathcal O'_{r,s}(\X \rtimes^t G)$ is a $^*$-homomorphism for the respective convolutions and involutions from \cref{res:involution}.
\end{lemma}

\begin{proof}
	Consider the maps \(\pi\colon \X\times\RR\to \X\) defined by \((x,t)\mapsto x\). Then \(\iota_t\colon \X\to \X\times\RR\) defined by \(\iota_t(x)=(x,t)\) is a polynomial right inverse of \(\pi\). Moreover, \(\iota_t(\sigma^t_v(x))=\sigma_v(\iota_t(x))\) holds for all \(x\in\X\), \(t\in\RR\) and \(v\in G\). Then by \cref{res:homomorphism} the induced map \(\ev_t\colon \Algebra O'_r(\TG)\to\Algebra O'_r(X\rtimes^t G)\) given by \(u\mapsto \iota_t^*\circ u\circ(\pi\times\id_G)^*\) is a homomorphism with respect to convolution. By \cref{res:*-homomorphism}, \(\ev_t\) restricts to a \(^*\)-homomorphism $\mathcal O'_{r,s}(\mathcal G) \to \mathcal O'_{r,s}(\X \rtimes^t G)$.
\end{proof}
This structure is essential to have an interpretation of the tangent groupoid
as a deformation of an operator at $t=1$ to its symbol at $t=0$,
which is explained in detail in \cref{sec:calculus}.

Property \refitem{assumption:compatible} in \cref{assumption} is needed to show the following result.

\begin{lemma}\label{res:zoom_action}
	The zoom action $\tau_\lambda$ on $\TG$ is by Lie groupoid automorphisms. 
\end{lemma}

\begin{proof}
	One computes for all \(\lambda>0\) and \((x,t,v)\in\TG\) with \(t\neq 0\) that
	\begin{align*}
		\tau_\lambda^0(\theta_v(x,t))&=\tau_\lambda^0(\beta_t(\theta^1_{\alpha_t(v)}(\beta_{t^{-1}}(x)),t) = (\beta_{\lambda^{-1}t}(\theta^1_{\alpha_t(v)}(\beta_{t^{-1}}(x)),\lambda^{-1}t) 
		\\& = (\theta_{\alpha_\lambda(v)}^{\lambda^{-1}t }(\beta_{\lambda^{-1}}(x)), \lambda^{-1}t)
		= \theta_{\alpha_\lambda(v)}(\beta_{\lambda^{-1}}(x),\lambda^{-1}t)=\theta_{\alpha_\lambda(v)}(\tau^0_\lambda(x,t)).
	\end{align*}
	By continuity, this also holds for \(t=0\). Then the claim follows from \cref{res:groupoidHom}.
\end{proof}
Furthermore, by \cref{res:automorphism} there is an induced \(\Rp\)-action on \(\Algebra O'_r(\TG)\) whose properties are summarized in the following result.
\begin{corollary} \label{res:zoomAutomorphisms}
	There is an \(\Rp\)-action \(\tau_*\) on $\Algebra{O}'_r(\TG)$ by homomorphisms defined by 
	\begin{equation*}
		({\tau_\lambda}_* u)f = (\tau^0_{\lambda^{-1}})^*(u(\tau_{\lambda}^* f)) \quad\text{for \(\lambda>0\), \(u\in\Algebra O'_r(\TG)\), \(f\in\Algebra O_M(\TG).\)}
	\end{equation*}
	It restricts to an action on $\Algebra O'_{r,s}(\TG)$
	by $^*$-homomorphisms.
	Moreover, the action can be restricted to \(\Schwartz(\mathcal G)\subseteq \Algebra O'_{r,s}(\TG)\) with \({\tau_\lambda}_*u_f=u_{{\tau_\lambda}_*f}\) where \begin{equation*}{\tau_\lambda}_*f(x,t,v)=\lambda^{-Q(\alpha)}f(\tau_{\lambda^{-1}}(x,v,t))\qquad\text{for } f\in\Schwartz(\TG).\end{equation*}
\end{corollary}
In the following, it will be useful to know that properties of the action \(\theta^1\) transfer to \(\theta^t\) for \(t\neq 0\).
\begin{lemma} \label{res:pair_groupoid} 
	Let $\theta^1$ be a Shubin action of \(G\) on \(\X\). When \(\theta^1\) is polynomially free, polynomially transitive or satisfies \ref{assumption:free}, then \(\theta^t\) has the same property for all \(t\neq 0\).
\end{lemma} 

\begin{proof} 
	Let \(t\neq 0\). As \(\theta^t_v = \beta_t\circ \theta^1_{\alpha_t(v)}\circ\beta_{t^{-1}}\), one deduces that the shear map \(\Theta^t\) of \(\theta^t\) is given by
	\begin{equation*}
		\Theta^t=(\beta_{t}\times\beta_t)\circ \Theta^1\circ(\beta_{t^{-1}}\times\alpha_{t}).
	\end{equation*}
	Since $\beta_t \times \beta_t$ and $\beta_{t^{-1}} \times \alpha_t$ are polynomial diffeomorphisms, it is straightforward to verify that the properties of $\Theta^1$ transfer to $\Theta^t$.
\end{proof}
Recall that the shear map $\Theta^t$ is a groupoid homomorphism.
In particular, if property~\ref{assumption:free} holds, then $\Theta^t$ is an isomorphism between the action groupoid $\X \rtimes^t G$ and the pair groupoid $\X \times \X$ for every $t \neq 0$.
In this case, $\Algebra O_r'(\X \rtimes^t G)$ for $t \neq 0$ can be identified with Schwartz kernels of certain operators on $\X$, see \cref{res:Op-injective}.

In general, the maps \(\theta^t_v\colon \X\to \X\) for \(v\in G\) define a polynomial family of diffeomorphisms in the sense of \cref{res:determinantIsConstant}. The same is true for the orbit maps under assumption \ref{assumption:free}.
\begin{lemma}\label{res:orbitmaps}
	Consider for \(t\neq 0\) and \(x\in \X\) the orbit map \(\OrbitMap t x\colon G\to \X\) defined by \(v\mapsto\theta^t(x,v)\). Suppose that property \ref{assumption:free} holds, then \(\OrbitMap t x\) for \(x\in \X\) is a family of polynomial diffeomorphisms in the sense of \cref{res:determinantIsConstant}.  In particular, \((x,v)\mapsto \det D_v(\OrbitMap 1 x) \) is a non-zero constant.
\end{lemma}
\begin{proof}
	Apply \cref{res:determinantIsConstant} with $Y = \X$ to the maps
	\(\theta^t \colon \X\times G\to \X \) and \(\pr_2\circ(\Theta^t)^{-1}\colon \X\times \X\to G\).
\end{proof}
Property~\ref{assumption:homomorphisms} is needed to obtain a Rockland condition for elliptic elements in the corresponding calculus.
Properties \ref{assumption:free} and \ref{assumption:homomorphisms} are satisfied in our two main examples, 
but can be given up in the study of more general calculi, see \cref{sec:shubin:weird-examples}.

\subsection{Shubin double dilation groupoid} \label{sec:shubin:DDG}
In this section, we specify the Shubin action for our first goal \ref{item:example:rocklandWithPotential}, which is to study Rockland operators with potentials.

For a graded Lie group $G$ we write $\Space G$ for the graded vector space $\RR^n$ obtained by using standard coordinates on $G$ and forgetting the group structure.

\begin{lemma}
	Let $G$ be a Lie group with two gradings defined by dilations $\alpha$ and $\beta$ and let $\X = \Space G$.
	Then  $\theta^1(x,v) \coloneqq x v$ defines a Shubin action satisfying properties \ref{assumption:free} and \ref{assumption:homomorphisms}, see \cref{assumption}.
\end{lemma}

\begin{proof}
	That the action \(\theta^1\) is polynomial follows from the polynomial group law \eqref{eq:triangular}.
	We compute 
	\begin{equation*}
		\theta^t_v(x)=\beta_t\circ\theta^1_{\alpha_t(v)}\circ\beta_{t^{-1}}(x)=\beta_t(\beta_{t^{-1}}(x)\alpha_t(v))=x \beta_t(\alpha_t(v))
	\end{equation*}
	for \(t\neq 0\).
	This extends smoothly to \(t=0\) with \(\theta^0_v(x)=x\). In particular, property \ref{assumption:homomorphisms} holds. Moreover, the action of \(G\) on itself by right multiplication satisfies property \ref{assumption:free} as \(\Theta^1(x,v)=(x,x v)\) has the polynomial inverse \((x,y)\mapsto (x,x^{-1} y)\).	
\end{proof}

\begin{definition}[Dilation groupoid] \label{def:doubleDilationGroupoid}
	Let $G$ be a Lie group with two gradings defined by dilations $\alpha$ and $\beta$.
	The \emph{double dilation groupoid} $\TGdil$ of $G$ is the Shubin tangent groupoid, see \cref{def:shubin_groupoid},
	for $\theta(x,t,v) = x \beta_t \alpha_t (v)$.
	If $G$ is a graded Lie group with dilations $\alpha$, the Shubin tangent groupoid obtained for $\beta = \alpha$ 
	is also referred to as the \emph{standard dilation groupoid} of $G$. 
\end{definition}
More explicitly, the groupoid \(\TGdil = \Space G \times \RR \times G\) has unit space \(\TGdil^{(0)} = \Space G \times\RR\)
and structure maps
\begin{gather*}
	u(x,t)=(x,t,0), \qquad r(x,t,v)=(x,t), \qquad s(x,t,v)=(x\beta_{t}(\alpha_t(v)),t),\\
	(x,v,t)^{-1}=(x\beta_{t}(\alpha_t(v)),t,v^{-1}) ,\qquad (x,t,v)\cdot(x\beta_{t}(\alpha_t(v)),t,w)=(x,t,vw),
\end{gather*}
for $x \in \Space  G$, $v,w\in G$ and $t\in \RR$ (see \cref{def:action_groupoid}).

Recall that $\theta^0$ is just the identity, so that the range and source map coincide at $t = 0$. 
Therefore the groupoid \(\Space{G} \rtimes^0 G = \Space{G} \times G\) can be understood as the tangent bundle of $G$.
The groupoid structure is such that the fibre \(\simpleset x \times G \cong r^{-1}(x,0)=s^{-1}(x,0)\) is equipped with the multiplication on \(G\).

\begin{remark}\label{rem:isomorphism_standard_form} 
	By assembling the isomorphisms \(\Phi^t = \Theta^t \colon \Space G \rtimes^t G \to \Space G\times \Space G\) for \(t\neq 0\) from \cref{res:pair_groupoid} and \(\Phi^0 = \id\) for \(t=0\) one obtains an isomorphism \(\Phi\colon\TGdil\to\mathbb T_HG\), which allows one to write the tangent groupoid in the more familiar form 

	\[ \mathbb T_HG\coloneqq\Space G \times G \times\{0\}\cup \Space G\times \Space G\times\RR^*\rightrightarrows \Space G\times\RR.\]
	The isomorphism is equivariant for the \(\Rp\)-action $\tau_\lambda$ on $\TGdil$
	and the $\Rp$-action on \(\mathbb T_H G\) given by
	\begin{align*}
		(x,y,t) \cdot \lambda&=
		(\beta_{\lambda^{-1}}(x),\beta_{\lambda^{-1}}(y),\lambda^{-1}t)\quad\text{for }t\neq 0,\\
		(x,v,0) \cdot \lambda&=(\beta_{\lambda^{-1}}(x),\alpha_\lambda(v),0).		
	\end{align*}
	for \(\lambda>0\), \(x,y\in\Space G\) and \(v\in G\). 
	Note that the smooth structure and the $\Rp$-action above differ from the smooth structure 
	and zoom action of the tangent groupoid used for a Hörmander type calculus (see \cite{Ewe23a}). 
	There, one uses the action \((x,t)\cdot v= (x\gamma_{t}(v),t)\) for a dilation \(\gamma\), 
	which is not necessarily a composition of two dilations, to describe the tangent groupoid as an action groupoid. 	
\end{remark}

\subsection{Shubin representation groupoid} \label{sec:shubin:Rep}
As sketched in the introduction,
we would also like to study operators on $G$, like the Harmonic oscillator on \(\RR^n\),
that are representations of Rockland operators on a higher step group $\overline G$,
see \ref{item:example:representationGroupoid}.

For a graded Lie group \(G\) of dimension \(n\) and highest weight \(r \coloneqq q_n\), 
Mohsen constructs in \cite{Moh22} a \((2n+1)\)-dimensional graded group \(\overline{G}\) of highest weight \(r+1\). 
The construction yields $\overline{\RR^n} = H_n$, 
and \(\overline{H_n}\) is the Dynin--Folland group used in \cites{Dyn75,Fol94,RR20,RR22}.
In the following, we recall the construction of \(\overline{G}\) as a semidirect product.
Using Fourier transform there is a corresponding action groupoid \((\lie{g}\times\RR)\rtimes_\vartheta G\),
from which we obtain a Shubin tangent groupoid. 

Let \(D\colon\lie{g}\to\lie{g}\) be the linear map from \cref{def:dilation} such that \(A_\lambda=\exp(\ln(\lambda)D)\) are the dilations on~$\lie g$ for \(\lambda>0\). In the following, denote the coadjoint action of $G$ on $\lie g^*$ by $\Ad^*$,
and its infinitesimal action of $\lie g$ on $\lie g^*$ by $\ad^*$.
Define
\begin{align*}
	\widehat\vartheta^\lie g &\colon \lie g \times (\lie g^* \times \RR) \to \lie g^* \times \RR,
	&
	\widehat\vartheta^{\lie g} (X,\xi,\tau) &= \widehat\vartheta^{\lie g}_X (\xi,\tau) = (\ad^*_X \xi, \xi(D X)),
	\\
	\widehat\vartheta &\colon G \times (\lie g^* \times \RR) \to \lie g^* \times \RR , &
	\widehat{\vartheta}(\exp X, \xi,\tau) &= \widehat{\vartheta}_{\exp X}(\xi,\tau)
	= \Bigl(\Ad^*_{\exp X}\xi,\tau + \sum_{k=1}^\infty \frac{1}{k!}(\ad^*_X)^{k-1}\xi(DX)\Bigr).
\end{align*}
Here, we use that the exponential map \(\exp\colon\lie{g}\to G\) is a diffeomorphism to define \(\widehat\vartheta\).
The reason for the notation $\widehat\vartheta$ will become clear in \cref{res:representation-groupoid-fundamentals}.

\begin{lemma}
	The map $\widehat\vartheta$ defines a left action of $G$ on $\lie g^* \times \RR$
	with infinitesimal action $\widehat\vartheta^{\lie g}$ of $\lie g$ on $\lie g^* \times \RR$.
\end{lemma}

\begin{proof}
	Note that \(D\colon \lie{g}\to\lie{g}\) is a derivation, 
	that is \(D([X,Y])=[DX,Y]+[X,DY]\) for \(X,Y\in\lie{g}\).
	Indeed, this equality is clearly true for \(X\in\lie{g}_i\) and \(Y\in\lie{g}_j\), where \(1\leq i,j\leq r\).
	Using this, it is straightforward to verify that 
	$[\widehat\vartheta^{\lie g}_X, \widehat\vartheta^{\lie g}_Y] = \widehat\vartheta^{\lie g}_{[X,Y]}$.
	Since
	\begin{multline*}
		\sum_{k=0}^\infty\frac{1}{k!} (\widehat\vartheta^{\lie g}_X)^k (\xi,\tau)
		=\left(\sum_{k=0}^\infty \frac{1}{k!}(\ad^*_X)^k\xi,\tau+ \sum_{k=1}^\infty \frac{1}{k!}(\ad^*_X)^{k-1}\xi(DX)\right)
		\\
		= \Bigl(\Ad^*_{\exp X}\xi,\tau + \sum_{k=1}^\infty \frac{1}{k!}(\ad^*_X)^{k-1}\xi(DX)\Bigr)
		= \widehat{\vartheta}_{\exp X}(\xi,\tau)
	\end{multline*}
	the map $\widehat \vartheta$ is obtained by exponentiating $\widehat\vartheta^{\lie g}$,
	hence $\widehat\vartheta_{vw} = \widehat\vartheta_v \widehat\vartheta_w$ for all \(v,w\in G\).
\end{proof}

\begin{definition} Let $G$ be a graded Lie group. We define the Lie group $\overline G$ as the semidirect product
	$(\lie g^* \times \RR)\rtimes_{\widehat\vartheta} G$ and denote its Lie algebra by $\overline{\lie g}$.
	More concretely, $\overline G = \lie g^* \times \RR \times G$ and $\overline{\lie g} = \lie g^* \times \RR \times \lie g$ as sets and the product and Lie bracket are determined by
	\begin{align} \label{eq:Dynin-Folland:group-product}
		(\xi, \tau, v) (\eta, \kappa, w) &= ((\xi,\tau) + \widehat\vartheta_v(\eta, \kappa), v w),
		\\
		\label{eq:Dynin-Folland:Lie-bracket}
		[(\xi, \tau, X), (\eta,\kappa, Y)] 
		&= (\widehat\vartheta^{\lie g}_X(\eta, \kappa) - \widehat\vartheta^{\lie g}_Y(\xi,\tau), [X,Y])
	\end{align}
	where $\xi,\eta \in \lie g^*$, $\tau, \kappa \in \RR$, $v,w \in G$ and $X,Y \in \lie g$.
\end{definition}
Assume that the dilations $\alpha$ of $G$ have weights $1 \leq q_1 \leq \dots \leq q_n$
and write $r \coloneqq q_n$ for the highest weight. Choose a standard basis $X_1, \dots, X_n$.
Let $\alpha^\vee$ be the dilations on $\lie g$ determined by $\alpha^\vee_\lambda X_j = \lambda^{r+1-q_j} X_j$. 
We also write $\alpha^\vee$ for the dual dilations on $\lie g^*$
determined by $\langle \alpha^\vee_\lambda \xi, X \rangle = \langle \xi, \alpha^\vee_\lambda X \rangle$ for $\xi \in \lie g^*$, $X \in \lie g$. 
Note that \(Q(\alpha^\vee) = n(r+1)-Q(\alpha)\).
Let~$\rho_\lambda$ denote the dilation $\rho_\lambda(\tau) = \lambda^{r+1} \tau$ on $\RR$.
Mohsen proved the following result in \cite{Moh22}.

\begin{lemma}
	The group $\overline G$ becomes a graded Lie group when its Lie algebra is equipped with the dilations $\gamma_\lambda(\xi,\tau,v) = (\alpha^\vee_\lambda \xi, \rho_\lambda \tau ,\alpha_\lambda v)$.
\end{lemma}
One proves this by showing that $\gamma_\lambda$ defines Lie algebra automorphisms of $\lie g$.
It follows that $\gamma_\lambda$ acts by Lie group automorphisms.
Writing this out using \eqref{eq:Dynin-Folland:group-product}, one obtains 
\begin{equation} \label{eq:compatibility:Dynin-Folland}
	\widehat{\vartheta}_{\alpha_\lambda(v)} \circ (\alpha_\lambda^\vee \times \rho_\lambda) = (\alpha_\lambda^\vee \times \rho_\lambda) \circ \widehat\vartheta_v .
\end{equation}

\begin{remark}
	The group $\overline G$ is a central extension of the semidirect product $\lie g^* \rtimes G$.
	Indeed, the Lie bracket in $\lie g^* \rtimes \lie g$ is $[(\xi,X),(\eta,Y)] = (\ad^*_X \eta - \ad^*_Y \xi, [X,Y])$. Identifying $\overline{\lie g}$ with $\lie g^* \times \lie g \times \RR$
	by swapping $\lie g$ and $\RR$, the Lie bracket of $\overline{\lie g}$ is of the form
	\begin{equation}
		[(\xi,X,\tau), (\eta,Y,\kappa)]
		=
		([(\xi,X), (\eta,Y)], \varepsilon((\xi,X), (\eta,Y)))
	\end{equation}
	for the \(2\)-cocycle $\varepsilon((\xi,X), (\eta,Y)) = \eta(D X) - \xi(D Y)$,
	so that $\overline{\lie g}$ is indeed a central extension as claimed.		
\end{remark}
For $f \in \Schwartz(\lie g \times \RR)$, define its Euclidean Fourier transform $\mathcal F(f) \coloneqq \widehat{f} \in\Schwartz(\lie g^* \times \RR)$ by
\begin{equation*}
	\widehat{f}(\xi,\tau)
	= \int_{\RR^{n+1}} \E^{-\I\langle (\xi, \tau), (x,t) \rangle}f(x,t)\D x \D t 
\end{equation*}
where $\langle(\xi,\tau), (x,t) \rangle = \xi(x) + \tau t$.
The following proposition shows that $\overline G$ can be obtained as the Fourier transform of
(a rescaled) Shubin tangent groupoid.

\begin{proposition} \label{res:representation-groupoid-fundamentals}
	Let $G$ be a graded Lie group with dilations $\alpha$.
	\begin{enumerate}
		\item \label{item:repGroupoid:i}
		The formula
		\begin{equation}\label{eq:action-for-representation-groupoid}
			\vartheta_{\exp X}(x,t)= \Bigl(\Ad(\exp (-X))x+t\sum_{k=1}^\infty\frac{1}{k!}\ad(-X)^{k-1}(DX), t\Bigr)
		\end{equation}
		defines a polynomial right action of $G$ on $\lie g\times\RR$, which satisfies
		$\widehat{\vartheta_v^* f} = \widehat\vartheta_{-v}^* \widehat f$ for $f \in \Schwartz(\lie g \times \RR)$ and $v \in G$. 
		
		\item \label{item:repGroupoid:ii} The equality
		$\mathcal F((\alpha_\lambda^\vee \times \rho_\lambda)^* f) = \lambda^{-Q(\alpha^\vee)-r-1} (\alpha_{\lambda^{-1}}^\vee \times \rho_{\lambda^{-1}})^* \mathcal F(f)$
		holds for all $f \in \Schwartz(\lie g \times \RR)$ and $\lambda \neq 0$. 
		
		\item \label{item:repGroupoid:iii} The actions $\alpha$, $\alpha^\vee$ and $\vartheta$ are compatible in the sense that
		$\alpha^\vee_{\lambda}\circ\vartheta^{ t}_{\alpha_\lambda(v)}\circ\alpha^\vee_{\lambda^{-1}} = \vartheta^{\lambda^{r+1} t}_{v}$
		holds for all $\lambda\neq 0$, $t\in\RR$ and $v \in G$.
	\end{enumerate}
\end{proposition}

\begin{proof}
	We use again that $G$ is nilpotent, so that $\exp \colon \lie g \to G$ is a diffeomorphism and
	all sums in \eqref{eq:action-for-representation-groupoid} are finite.
	It follows that the action $\vartheta^t$ of $G$ on $\lie g$ is polynomial.
	One computes for \(X\in\lie{g}\), \((\xi,\tau)\in\lie{g}^*\times\RR\) and \((x,t)\in \lie{g}\times\RR\) 
	\begin{multline*}
		\langle (\xi,\tau),\vartheta_{\exp X}(x,t)\rangle 
		= \xi\Bigl(\Ad(\exp(-X))x + t\sum_{k=1}^\infty \frac{1}{k!}\ad(-X)^{k-1}(DX)\Bigr)+\tau t\\
		= \Ad^*_{\exp(X)}\xi(x)+\left(\tau +\sum_{k=1}^\infty \frac{1}{k!} (\ad^*_X)^{k-1}\xi(DX)\right) t 
		= \langle \widehat\vartheta_{\exp X}(\xi,\tau),(x,t)\rangle .
	\end{multline*}
	Since $\widehat \vartheta$ is a left action, it follows that $\vartheta$ is a right action.
	Therefore $\vartheta_v$ is a polynomial diffeomorphism whose Jacobian determinant is constantly $1$ by \cref{res:polynomial_action:properties}~\refitem{item:noJ}.
	Hence
	\begin{align*}
		\widehat{\vartheta^*_v f}(\xi,\tau)
		&= \int_{\RR^{n+1}} \E^{-\I\langle(\xi,\tau),(x,t)\rangle}f(\vartheta_v(x,t))\D x \D t
		= \int_{\RR^{n+1}} \E^{-\I\langle(\xi,\tau),\vartheta_{-v}(x,t)\rangle}f(x,t)\D x \D t
		\\
		&= \int_{\RR^{n+1}} \E^{-\I\langle\widehat\vartheta_{-v}(\xi,\tau),(x,t)\rangle}f(x,t)\D x \D t 
		= (\widehat{\vartheta}_{-v}^*\widehat{f})(\xi,\tau) .
	\end{align*}
	This proves \refitem{item:repGroupoid:i}. Similarly, we compute
	\begin{align*}
		\mathcal F((\alpha_\lambda^\vee \times \rho_\lambda)^* f)(\xi,\tau)
		&= \int_{\RR^{n+1}} \E^{-\I\langle(\xi,\tau),(x,t)\rangle}f(\alpha_\lambda^\vee x,\rho_\lambda t)\D x \D t\\
		&= \lambda^{-Q(\alpha^\vee)-r-1} \int_{\RR^{n+1}} \E^{-\I\langle(\xi,\tau),(\alpha_{\lambda^{-1}}^\vee x,\rho_{\lambda^{-1}} t)\rangle}f(x,t)\D x \D t	\\
		&= \lambda^{-Q(\alpha^\vee)-r-1} (\alpha_{\lambda^{-1}}^\vee \times \rho_{\lambda^{-1}})^* \mathcal F(f).
	\end{align*}
	The statement in \refitem{item:repGroupoid:iii} is equivalent to $(\alpha^\vee_\lambda \times \rho_\lambda) \circ \vartheta_{\alpha_\lambda(v)} \circ (\alpha^\vee_{\lambda^{-1}} \times \rho_{\lambda^{-1}}) = \vartheta_v$.
	It suffices to check that both sides induce the same pull-back on Schwartz functions. Under Fourier transform one has
	\begin{equation*}
		\mathcal F(((\alpha^\vee_\lambda \times \rho_\lambda) \circ \vartheta_{\alpha_\lambda(v)} \circ (\alpha^\vee_{\lambda^{-1}} \times \rho_{\lambda^{-1}}))^* f)
		=
		((\alpha^\vee_{\lambda^{-1}} \times \rho_{\lambda^{-1}}) \circ \widehat\vartheta_{\alpha_\lambda(-v)} \circ (\alpha^\vee_\lambda \times \rho_\lambda))^*  \mathcal F f
		=
		\widehat\vartheta_{-v} ^*  \mathcal F f
		=
		\mathcal F(\widehat\vartheta_{v} ^*  f)
	\end{equation*}
	where we used the results of part \refitem{item:repGroupoid:i} and \refitem{item:repGroupoid:ii} and \eqref{eq:compatibility:Dynin-Folland}.
\end{proof}

\begin{example}\label{ex:representation_groupoid_heisenberg}
	For \(G=\RR^n\) with trivial grading (all elements of $\lie g$ are of degree $1$) one computes \(\vartheta_v(x,t)=(x+tv,t)\). For the \((2n+1)\)-dimensional Heisenberg group \(G=H_n\) as in \cref{ex:heisenberg-algebra} one gets 
	\begin{equation*}
		\vartheta_{v}(x,t)=\left(x_1+tv_1,\ldots, x_{2n}+tv_{2n},x_{2n+1}+\sum_{j=1}^n (v_{n+j}x_j- v_j x_{n+j}) +2tv_{2n+1},t\right).
	\end{equation*}	
\end{example}
By \cref{res:rescaling}, the rescaled action $\theta^t_{v}(x) = (\vartheta_{v}^{t^{r+1}}(x), t)$ of $G$ on $\lie g$
together with the dilations $\alpha$ and $\beta = \alpha^\vee$ defines a Shubin groupoid.

\begin{definition}[Representation groupoid]
	Let $G$ be a graded Lie group with dilations $\alpha$ and \(\X=\underline{\lie{g}}\) the underlying vector space of $\lie g$.
	The Shubin groupoid of the rescaled action $\theta$ and the dilations $\alpha$, $\beta=\alpha^\vee$
	is called \emph{representation groupoid} \(\TGrep\) of $G$.
\end{definition}

\begin{proposition}
	Let $G$ be a graded Lie group with dilations $\alpha$. Then the representation groupoid $\TGrep$ has properties \ref{assumption:free} and \ref{assumption:homomorphisms}.
\end{proposition}

\begin{proof}
	Note that the rescaling from \cref{res:rescaling} does not change the actions at \(t=0\) and \(t=1\), so we can consider \(\vartheta^0\) and \(\vartheta^1\).
	Since \(\vartheta^0\) is up to an inverse just the adjoint representation
	it is clear that \ref{assumption:homomorphisms} holds and it remains to show that
	\(\vartheta^1\) satisfies Property \ref{assumption:free}.
	Let \(x=(x_1,\ldots,x_n)\in \lie g\) and \(y=(y_1,\ldots,y_n)\in \lie g\). We show that there is a unique \(v=(v_1,\ldots,v_n)\in G\) such that \(\vartheta^1_v(x)=y\), which depends polynomially on \(x,y\). This is done iteratively for every component, making use of the triangular group law. We show for \(j=1,\ldots,n\) that there are polynomials \(p_j\) and only depending on \(x, v_1,\ldots, v_{j-1}\) such that
	\begin{equation}\label{eq:diagonal_action}
		(\vartheta_v(x))_j=p_j(x,v_1,\ldots,v_{j-1})+ q_j v_j.
	\end{equation}
	This implies that \(v\) is uniquely determined by setting iteratively \(v_j=\frac1{q_j}(y_j-p_j(x,v_1,\ldots,v_{j-1}))\) for \(j=1,\ldots,n\).
	To show \eqref{eq:diagonal_action}, note first that \((\Ad(-v)x)_j\) only depends on \(x,v_1,\ldots,v_{j-1}\). This holds as \(\Ad_{-v}=v^{-1}xv\) and by the polynomial group law \cref{res:triangular} there is a polynomial \(\widetilde p_j\) such that
	\begin{equation*}(v^{-1}xv)_j=x_j+\widetilde p_j(v_1,\ldots,v_{j-1},x_1,\ldots,x_{j-1}).\end{equation*} 	
	Next consider the terms \((\frac{1}{k!}\ad(-X)^{k-1}(DX))_j\) for \(X=v_1X_1+\ldots+v_nX_n\) and \(k\in\NN\). For \(k=1\) this is \(q_jv_j\). For \(k>1\) we note that \(\ad(-X)^{k-1}(DX)_j\) can only depend on \(v_1,\ldots,v_{j-1}\) as \(\lie{g}\) is graded. This shows that \eqref{eq:diagonal_action} holds. 
\end{proof}
For the representation groupoid, the action \(\theta^0=\vartheta^0\) is not trivial. Hence, the range and source map of the groupoid \(\underline{\lie{g}}\rtimes^0 G\) at \(t=0\) do not coincide, so contrary to the double dilation groupoid it cannot be viewed as a bundle of groups. 

By construction, the \(C^*\)-algebra of the groupoid \((\underline{\lie{g}}\times\RR)\rtimes_\vartheta G\) is isomorphic to the group \(C^*\)-algebra \(C^*(\overline G)\). However, the smooth structures of the groupoid \((\X\times\RR)\rtimes_\vartheta G\) and the groupoid of the rescaled action \((\X\times\RR)\rtimes_\theta G\) differ. When \(r\) is even, they still have isomorphic groupoid \(C^*\)-algebras, as the map \(t\mapsto t^{r+1}\) is a homeomorphism in this case. When \(r\) is odd, \((\X\times\RR)\rtimes_\theta G\) only captures the behaviour of \((\X\times\RR)\rtimes_\vartheta G\) for \(t\geq 0\).
\subsection{Examples without assumption (P) or (R)} \label{sec:shubin:weird-examples}
Finally, we give examples for Shubin actions which do not satisfy assumption \ref{assumption:free} or \ref{assumption:homomorphisms} but could still be interesting to study.
\begin{example}[Group bundle]\label{ex:bundle-action}
	Let $G$ be a Lie group with two gradings defined by dilations \(\alpha,\beta_2\) 
	and $\RR^d$ be a graded vector space with dilations \(\beta_1\).
	Let \(\X=\RR^d\times G\) be equipped with the \(\Rp\)-action \(\beta=(\beta_1,\beta_2)\) and \(\theta^1_v(x,w)=(x,w v)\) for \(x\in\RR^d\) and \(v,w\in G\). This defines a Shubin action with property~\ref{assumption:homomorphisms}. 
	The action \(\theta^1\) is polynomially free, but not transitive.
\end{example}
A closed, simply connected subgroup \(H\subseteq G\) with Lie algebra \(\lie{h}\) is called a \emph{graded subgroup} if  \(\lie{h}_i=\lie{g}_i\cap \lie{h}\) for \(i \in \NN \) defines a grading on \(\lie{h}\).
In particular, one can choose a standard basis \(Y_{1},\ldots,Y_n\) of \(\lie{h}\). It can be extended to a standard basis of \(\lie{g}\), denote the additional basis elements by \(X_1,\ldots,X_d\). We identify in the following \(H\) with \(\RR^n\) and \(G\) with \(\RR^{n+d}\) using these standard bases and exponential coordinates as before. Denote these coordinates by \(\phi_H\) and \(\phi_G\), respectively. Moreover, \(\RR^{d}\) is identified with the homogeneous space \(H\backslash G\) via \(\phi_{H\backslash G}\colon (x_1,\ldots,x_d)\mapsto [\exp(x_dX_d)\cdots\exp(x_1X_1)]\), see \cite{CG90}*{Thm.~1.2.12}. 
\begin{example}[Normal subgroup]
	Let \(H\subseteq G\) be a normal graded subgroup. Then set \(\X=H\), where we use the basis \(Y_1, \dots, Y_n\) above to identify \(H\) with \(\RR^n\), and set \(\beta_\lambda =\alpha_\lambda|_H\) for \(\lambda > 0\). Now \(\theta^1_v(x)=v^{-1} x v\) for $v \in G$
	defines a Shubin action of $G$ on $\X$ with property \ref{assumption:homomorphisms} which is in general neither free nor transitive.
\end{example}

\begin{example}[Homogeneous space]\label{ex:homogeneous-space}
	Let \(H\subseteq G\) be a graded subgroup, let \(\X=H\backslash G\) and identify it with \(\RR^d\) using the basis \(X_1,\ldots,X_d\) as above.
	Set \(\beta_\lambda[x]=[\alpha_\lambda(x)]\) for \(x\in G\) and \(\lambda > 0\). We claim that \(\theta^1_v([x])=[x\cdot v]\) is a polynomial action of $G$ on $H \backslash G$	which is polynomially free. To see this, we define new coordinates $\phi \colon \RR^n \times \RR^d \to G$ on $G$ by
	\begin{equation*}
		\phi(v'_1, \dots v'_n, v_1, \dots, v_d) = \exp(v'_nY_n)\cdots\exp(v'_1Y_1)\exp(v_dX_d)\cdots\exp(v_1X_1) .
	\end{equation*}
	As \(Y_n,\ldots, Y_1, X_d,\ldots, X_1\) is a weak Malcev basis, there is a polynomial diffeomorphism \(p\colon \RR^{n}\times\RR^d\to\RR^{n}\times\RR^d\) such that \(\phi_G=\phi\circ p\). Hence, the multiplication of $G$ is polynomial also in the new coordinates~$\phi$.
	In these coordinates, the projection $G \to H \backslash G$ is simply the projection $\pi_2 \colon \RR^n \times \RR^d \to \RR^d$ to the second factor, hence $\theta^1$ is polynomial.
	Moreover, the map $\X \times G \to \X \times G$, $(x,(v',v)) \mapsto (x, (0,x)(v',v))$ is a polynomial diffeomorphism
	and $\Theta^1(x, (v',v)) = (x, \pi_2 ((0,x) (v',v)))$, showing that $\theta^1$ is polynomially transitive.
	 However, \(\theta^1\) is not free unless \(H\) is trivial.
\end{example}
The following example is a special case of \cref{ex:homogeneous-space}, using the standard basis \(X_1,X_2,X_3\) of the \(3\)-dimensional Heisenberg Lie algebra and the graded subalgebra generated by \(X_2\). 
\begin{example}\label{ex:grushin-action}
	Consider the Heisenberg group \(H_1=\{(v_1,v_2,v_3)\colon v_i\in\RR\}\) from \cref{ex:heisenberg-algebra} with dilations \(\alpha_\lambda(v_1,v_2,v_3)=(\lambda^kv_1,\lambda^lv_2,\lambda^{k+l}v_3)\) with \(k,l\in\NN\). Let \(\X=\RR^{2}=\{(x,y)\colon x,y\in\RR\}\) be equipped with dilations \(\beta_\lambda(x,y)=(\lambda^p x,\lambda^q y)\).  Consider the following right polynomial action
	\begin{equation*}
		\theta^1_{(v_1,v_2,v_3)}(x,y)=\left(x+v_1,y+v_3+\tfrac12 xv_2\right).
	\end{equation*}
	One computes
	\begin{align*}
		\beta_t\circ\theta^1_{\alpha_t(v_1,v_2,v_3)}\circ\beta_{t^{-1}}(x,y)=\left(x+t^{k+p}v_1,y+t^{k+l+q}v_3+\tfrac{1}{2}t^{l+q-p}xv_2\right),
	\end{align*} 
	Hence, \(\theta^1\) is a Shubin action if and only if \(l+q\geq p\).
	In this case, it satisfies property \ref{assumption:homomorphisms}. Furthermore, one has \(\theta^0_{v}=\id\) for all \(v\in H_1\) whenever \(l+q>p\). The action \(\theta^1\) is polynomially transitive, but not free.
	In \cref{ex:grushin-fundamental-vectorfields} the relation of this action with Grushin and Kolmogorov operators is discussed.
\end{example}
The following example shows that property \ref{assumption:homomorphisms} is not automatically satisfied.
\begin{example}
	Let \(K_3\) be the step \(3\) graded Lie group defined in \cite{CG90}*{Example~1.3.10}, also called the \emph{Engel group}.
	As a space \(K_3\) is \(\RR^4\) with the following group law containing quadratic terms
	\begin{multline*}
		(v_1,v_2,v_3,v_4)\cdot(w_1,w_2,w_3,w_4)=(v_1+w_1,v_2+w_2,v_3+w_3+\tfrac{1}{2}(v_1w_2-w_1v_2),\\ v_4+w_4+\tfrac{1}{2}(v_1w_3-w_1v_3)+\tfrac{1}{12}(v_1^2w_2-v_1w_1(v_2+w_2)+w_1^2v_2)) .
	\end{multline*}
	Its standard dilations are given by \(\alpha_\lambda(v_1,v_2,v_3,v_4)=(\lambda v_1, \lambda v_2, \lambda^2v_3,\lambda^3v_4)\) for \(\lambda>0\).
	Let \(\X=\Space{K_3}\) and let \(G=K_3\) act on \(\X\) by right multiplication \(\theta_v(x)=x\cdot v\). Contrary to the double dilation groupoid, we let \(\beta_\lambda(x_1,x_2,x_3,x_4)=(\lambda x_1,\lambda x_2,\lambda x_3,\lambda x_4)\) for \(\lambda >0\) which is not a group dilation on \(K_3\). Then one computes
	\begin{align*}
		\theta^0_v(x)=\lim_{t\to 0}\beta_t(\beta_{t^{-1}}(x)\cdot\alpha_t(v))=\left(x_1,x_2,x_3,x_4+\tfrac{1}{12}x_1(x_1v_2-x_2v_1)\right),
	\end{align*} 
	so the data defines a Shubin action. However, \(\theta^0_v\) is not linear for all \(v\in G\) so that property \ref{assumption:homomorphisms} is not satisfied.
\end{example}

\section{Shubin-type pseudodifferential calculus}\label{sec:calculus}

In the following, we always assume that \(G\) is a graded Lie group with dilations \(\alpha\), \(\X\cong \RR^d\) a graded vector space with dilations \(\beta\) and \(\theta^1\) a Shubin action of $G$ on $\X$ as introduced in \cref{assumption}. We write \(\TG\) for the corresponding Shubin tangent groupoid from \cref{def:shubin_groupoid}.
We denote the weights of $\alpha$ by $q_1, \dots, q_n$ and the weights of $\beta$ by $r_1, \dots, r_d$.

\subsection{Shubin-type differential calculus}
To motivate our construction of the pseudodifferential calculus, we define an algebra of certain differential operators with polynomial coefficients on \(\X\) which will be contained in our pseudodifferential calculus. Namely, we set for \(m\in\NN_0\) 
\[\Algebra{A}_m=\set[\bigg]{\sum\nolimits_{[a]_\alpha+[b]_\beta\leq m}c_{a,b}x^b (\widehat X^1)^a }{ c_{a,b}\in\CC}\subseteq \Algebra{A}_{m+1}\subseteq \ldots\]
where $\widehat X^t_j$ denotes the fundamental vector field of $X_j$ with respect to $\theta^t$.
Here, $a \in \NN_0^n$, $b \in \NN_0^d$ and we use the multiindex notation from 
\cref{def:homogeneousOrder} and similarly $\smash{(\widehat X^1)^a} = (\widehat X_1^1)^{a_1} \dots (\widehat X_n^1)^{a_n}$. 
By \cref{res:polynomial_action:properties} \refitem{item:polynomial_action:i}
$\Algebra A_m$ consists of differential operators with polynomial coefficients.
Moreover, \(\Algebra A =\bigcup_{m\in\NN_0}\Algebra A_m\) is a filtered algebra, which is an easy consequence of the following result.

Recall from \cref{def:homogeneousOrder} that we say that a polynomial $p$ on $\X$ is  \emph{homogeneous of degree} $k$ with respect to $\beta$ if $p \circ \beta_\lambda = \lambda^k p$.
We shall say that $p$ has $\beta$-\emph{degree} $\leq k$ if $p(x) = \sum_{[b]_\beta \leq k} c_{b} x^b$.

\begin{lemma} \label{res:fundamentalVFsOnPolynomials}
	Using standard coordinates on $G$ and $\X$ (so that $X_i$ has degree $q_i$ and $x_j$ has degree $r_j$),
	$\widehat X_i^0 x_j$ is a homogeneous polynomial of degree $q_i + r_j$ with respect to $\beta$,
	and $\widehat X_i^1 x_j$ is a polynomial of $\beta$-degree $\leq q_i + r_j$ of the form
	$\widehat X_i^1 x_j = \widehat X_i^0 x_j + (\text{terms of $\beta$-degree $\leq q_i + r_j -1$})$. 
\end{lemma}

\begin{proof}
	As a consequence of \cref{res:polynomial_action} \refitem{item:polynomial_action:i},
	$\widehat X_i^t x_j$ is a polynomial on $\X$ for all $t \in \RR$. 
	Using the compatibility of the actions \eqref{eq:compatibility}, we compute that for all $t \in \RR$ and $\lambda > 0$
	\begin{equation*}
		\beta_\lambda^* (\widehat X^t_i x_j)(x)
		= \frac{\D}{\D s} \Big|_{s=0} x_j\bigl(\theta^t_{\exp(s X_i)}(\beta_\lambda x)\bigr)
		= \frac{\D}{\D s} \Big|_{s=0} x_j\bigl(\beta_\lambda \theta_{\exp(\alpha_\lambda(s X_i))}^{t/\lambda}(x)\bigr)
		= \lambda^{q_i + r_j} \widehat X_i^{t/\lambda} x_j(x) \,.
	\end{equation*}
	Hence $\beta_\lambda^* (\widehat X^t_i x_j) = \lambda^{q_i + r_j} \widehat X_i^{t/\lambda} x_j$.
	Choosing $t = 0$ shows that $\widehat X^0_i x_j$ is of homogeneous order $q_i + r_j$.
	Taking $t = 1$ we obtain 
	\begin{equation*}
		\lambda^{-q_i - r_j} \beta_\lambda^* (\widehat X^1_i x_j) = \widehat X_i^{1/\lambda} x_j \,.
	\end{equation*}
	Since the action $\theta$ is polynomial in all arguments we have 
	$\lim_{\lambda \to \infty} \widehat X_i^{1/\lambda} = \widehat X_i^0$, 
	so that the limit $\lambda \to \infty$ exists on both sides of the previous equation. 
	It follows that the homogeneous degree of $\widehat X^1_i x_j$ is at most $q_i + r_j$,
	and that this polynomial is of the form 
	$\widehat X^0_i x_j + (\text{terms of $\beta$-degree $\leq q_i + r_j -1$})$. 
\end{proof}

\begin{example}
	For the double dilation groupoid, the action \(\theta^1_v(x)=x\cdot v\) is given by right multiplication, so that \(\widehat{X}^1\) for \(X\in\lie{g}\) is the corresponding left-invariant differential operator \eqref{eq:left-invariant-differential-operator}.
\end{example}
\begin{example}\label{ex:fundamental-vf-RG}
	For the representation groupoid of the Heisenberg group \(H_n\) we compute using  \cref{ex:representation_groupoid_heisenberg} for \(j=1,\ldots,n\)
	\begin{equation*}
		\widehat X_j^1 = \frac\partial{\partial x_j}-x_{n+j}\frac{\partial}{\partial x_{2n+1}},  \qquad
		\widehat X_{n+j}^1 = \frac\partial{\partial x_{n+j}}+x_j\frac{\partial}{\partial x_{2n+1}}, \qquad
		\widehat X_{2n+1}^1 = 2\frac{\partial}{\partial x_{2n+1}}.
	\end{equation*}
\end{example}

\begin{example}
	For the group bundle from \cref{ex:bundle-action} one computes for \(X\in\lie{g}\)
	\begin{equation*}
		\widehat X^1 f(x,w) = X(f(x))(w).
	\end{equation*}
	Hence, the operators we get only act as differential operators in the fibres. In particular, this shows that \(\Algebra A\) is in general only a subalgebra of differential operators with polynomial coefficients.
\end{example}

\begin{example}\label{ex:grushin-fundamental-vectorfields}
	For the action of \(H_1\) on \(\X=\RR^2\) from \cref{ex:grushin-action} one computes 
	\begin{equation*}
		\widehat X_1^1 = \frac{\partial}{\partial x} \qquad
		\widehat X_{2}^1 = \frac 12x\frac\partial{\partial y}, \qquad
		\widehat X_{3}^1 = \frac{\partial}{\partial y}.
	\end{equation*}
	Consider the Sublaplacian \(X_1^2+4X_2^2\) on the Heisenberg group $H_1$. Then the corresponding operator \((\widehat X^1_1)^2+4\cdot(\widehat X^1_{2})^2 = \partial_x^2+x^2 \partial_{y}^2\) on \(\X\) is a Grushin operator.	
	Another Rockland operator on \(H_1\) is \(-X_1^2+2X_2\). Under the action \(\theta^1\) it corresponds to the Kolmogorov operator \(-(\widehat X^1_1)^2+2\widehat X^1_{2} = -\partial^2_x+x\partial_{y}\).
	
	We observe in this example that elements of \(\Algebra A\) do not necessarily have a unique writing in the multi-index notation: \(\widehat X_{2}^1 = \tfrac12 x\partial_{y}=\tfrac12 x\widehat X_{3}^1\).
\end{example}

\begin{lemma} \label{res:propertyP:uniqueRep}
	Suppose that property \ref{assumption:free} is satisfied. Then \(\Algebra A\) is the algebra of all differential operators on \(\X\) with polynomial coefficients. Moreover, every \(P\in \Algebra A\) can be uniquely written as \(P=\sum_{a,b} c_{a,b}x^b(\widehat X^1)^a\).
\end{lemma}

\begin{proof}
	Recall that $\widehat X^1_j(x) = \frac{\partial}{\partial v_j} \theta^1(x,v)_k |_{v = 0} \frac{\partial}{\partial x_k} = \frac{\partial}{\partial v_j} \OrbitMap 1 x(v)_k |_{v = 0} \frac{\partial}{\partial x_k}$.
	Therefore $\widehat X^1(x) = D_v \OrbitMap 1 {x} |_{v = 0} \frac{\partial}{\partial x}$,
	using the obvious matrix notation and the orbit map from \cref{res:orbitmaps}.
	By this lemma, $\det(D_v \OrbitMap 1 x)$ is constant and non-zero as a function of $v$ and $x$.
	Therefore $D_v \OrbitMap 1 x |_{v = 0}$ is invertible and the inverse matrix has again polynomial entries.
	Hence $\smash{\frac{\partial}{\partial x}} = (D_v \OrbitMap 1 x |_{v = 0})^{-1} \smash{\widehat X^1} \in \Algebra A$.
	Since every differential operator $P$ with polynomial coefficients can be uniquely written in the form 
	\(P=\sum_{a,b}d_{a,b}x^b\big(\smash{\frac \partial {\partial x}}\big)^a\), it follows that it can also be uniquely written as above with $\smash{\widehat X^1}$.
\end{proof}
Assume that the Shubin groupoid has property \ref{assumption:homomorphisms}. 
Then the polynomial action $\theta^0$ of $G$ on $\X = \RR^d$ is by linear maps,
hence the fundamental vector fields have linear coefficients,
and $\widehat X_i^0 x_j$ is a linear combination of $x_k$ and thus in $\X^*$.
Therefore $X \cdot \xi \coloneqq \widehat X^0 \xi$ for $X \in \lie g$ and $\xi \in \lie g^*$ 
defines an action of $\lie g$ on $\X^*$.
\cref{res:fundamentalVFsOnPolynomials} even gives that this action is compatible with the grading of $\X^*$,
induced by $\beta$.
Therefore one obtains the following additional structure at \(t=0\):
\begin{lemma}\label{res:r-implies-group-at-0}
Under assumption \ref{assumption:homomorphisms}, the semidirect product $\X^* \rtimes^0 \lie g$ is a graded Lie algebra,
in which $[X_j, x_i] = \widehat X_j^0 x_i$.
\end{lemma}

\begin{example} \label{ex:representationLA:heisenbergGroup}
	Consider the representation groupoid of the Heisenberg group. Then $\lie h^* \rtimes^0 \lie h$ has Lie brackets
	\begin{equation*}
		[X_j,X_{n+j}] = X_{2n+1} , \quad [X_j, x_{2n+1}] = -x_{n+j}, \quad [X_{n+j}, x_{2n+1}] = x_j,
	\end{equation*}
	for $j = 1, \dots, n$ and the Lie brackets of other generators are $0$.
	Since $(\lie h^* \rtimes^0 \lie h)_1$ is spanned by $X_1, \dots, X_{2n}$ and $x_{2n+1}$ and generates $\lie h^* \rtimes^0 \lie h$,
	it follows that $\lie h^* \rtimes^0 \lie h$ is stratified.
\end{example}
\begin{example}\label{ex:symbol-group-grushin}
	For the action of the Heisenberg group \(H_1\) on \(\X=\RR^2\) from \cref{ex:grushin-action} one has to distinguish the two cases \(l+p>q\) and \(l+p=q\). When \(l+p>q\), the action at \(t=0\) is trivial and, hence, \(\X^*\rtimes^0\lie{h}_1=\X^*\times \lie{h}_1\). When \(l+p=q\), one has \(\theta^0_{(v_1,v_2,v_3)}(x,y)=(x,y+\tfrac{1}{2}xv_2)\). Then one computes that the non-trivial brackets in \(\X^*\rtimes^0 \lie{h}_1\) are given by \([X_1,X_2]=X_3\) and \([X_2,y]=x\). This means that the Lie algebra \(\X^*\rtimes^0 \lie{h}_1\) is generated by \(X_1,X_2,y\).
\end{example}
Now assume that the Shubin groupoid also has property \ref{assumption:free},
so that any element of $\Algebra A$ can be uniquely written in the form of \cref{res:propertyP:uniqueRep}.
Then define the \emph{principal cocosymbol} of an operator of order \(m\) by
\begin{equation}\label{eq:principalcoco}\begin{aligned}
		\widecheck\sigma_m\colon \mathcal{A}_m&\to \lie{U}^m(\X^* \rtimes^0 \lie{g})\\
		\sum_{[a]_\alpha+[b]_\beta\leq m}c_{a,b}x^b (\widehat X^1)^a&\mapsto \sum_{[a]_\alpha+[b]_\beta= m}c_{a,b} ( -\I x )^b  X^a. 
\end{aligned}\end{equation}
Recall that the (complex) universal enveloping algebra \(\lie{U}(\X^* \rtimes^0 \lie{g})\) of the graded Lie algebra $\X^* \rtimes^0 \lie g$ is graded
and \(\lie{U}^m(\X^* \rtimes^0 \lie{g})\) denotes the elements of degree $m$, see \cref{sec:diff-op-rockland}.
The name principal cocosymbol is justified since it behaves like the usual principal symbol: 
\begin{proposition}
	Under assumptions \ref{assumption:free} and \ref{assumption:homomorphisms}
	the maps $\widecheck \sigma_m$ assemble into an isomorphism of algebras
	\begin{equation*}
		\widecheck\sigma\colon\bigoplus_{m\in\NN_0}\mathcal{A}_m/\mathcal{A}_{m-1}\to \lie{U}(\X^* \rtimes^0 \lie{g}) \,.
	\end{equation*}
\end{proposition}

\begin{proof}
	Clearly \(\Algebra{A}_{m-1} \subseteq \ker(\widecheck\sigma_m)\),
	so that $\widecheck \sigma_m$ descends to a map $\Algebra A_m / \Algebra A_{m-1} \to \lie U^m(\X^* \rtimes^0 \lie g)$. Extending linearly, one obtains the  map $\widecheck \sigma$.
	
	Note that $\Algebra A$ is the quotient of the tensor algebra of $\X^* \times \lie g$
	by the ideal generated by elements $x_i \tensor x_j - x_j \tensor x_i$,
	$X_i \tensor x_j - x_j \tensor X_i - \smash{\widehat X^1_i(x_j)}$ and $X_i \tensor X_j - X_j \tensor X_i - [X_i, X_j]$,
	whereas $\mathfrak U(\X^* \rtimes^0 \lie g)$ is the quotient of the tensor algebra of $\X^* \times \lie g$ by the ideal generated by $(-\I x_i) \tensor (-\I x_j) - (-\I x_j) \tensor (-\I x_i)$,
	$X_i \tensor (-\I x_j) - (-\I x_j) \tensor X_i - (-\I \smash{\widehat X^0_i(x_j)})$ 
	and $X_i \tensor X_j - X_j \tensor X_i - [X_i, X_j]$. 
	When taking the quotient $\Algebra A_m / \Algebra A_{m-1}$ and direct sum over $m$,
	all terms of lower order drop out of the above generators.
	The first and third sets of generators are homogeneous and thus remain unchanged, 
	while it follows from \cref{res:fundamentalVFsOnPolynomials} that the second one becomes
	$X_i \tensor x_j - x_j \tensor X_i - \smash{\widehat X^0_i(x_j)}$.
	Consequently the map $\widecheck \sigma$, 
	induced by mapping $x_i$ to $-\I x_i$ and $X_j$ to $X_j$, is an isomorphism of algebras.
\end{proof}
Note that the factors of $-\I$ in the definition of $\widecheck \sigma$ and the previous proof
are not necessary to obtain an isomorphism, but are included to make the formulas compatible with the results
of the next section, where $\widecheck \sigma$ is obtained via an inverse Fourier transform.
\begin{remark}
	For the double dilation groupoid, one can also verify this more concretely:
	$X^a (x^b)$ is of homogeneous order $[b]_\beta - [a]_\alpha$,
	which is (unless $a = 0$) always strictly smaller than $[b]_\beta + [a]_\alpha$.
	The non-commutativity in the principal cocosymbols for a double dilation groupoid comes from the non-commutativity of $G$, but there is no additional contribution from the action since $\theta^0$ is trivial.
	In other words, the principal cocosymbol lives in $\lie U(\X^* \times \lie g)$. 
\end{remark}

\subsection{Definition of pseudodifferential operators and their principal cosymbol}

In order to define a pseudodifferential calculus as in \cite{vEY19} based on essentially homogeneous distributions, we employ the space \(\Algebra{O}'_r(\TG)\) introduced in \cref{subsec:convolution_alg_distr} on a Shubin tangent groupoid \(\TG\). 
The space of proper distributions used by \cite{vEY19} is not suitable here as Shubin operators are not properly supported in general. 

Recall from \cref{res:evaluationOfDistributions} that a distribution \(\P\in \Algebra O'_r(\TG)\) determines a family of fibred distributions \(\P_t = \ev_t (\P)\) on \(\X\rtimes^t G\) for \(t\in\RR\). There are representations \(\Op_t\colon\Algebra O'_r(\TG)\to\mathcal L(\Algebra{O}_M(\X),\Algebra{O}_M(\X))\) defined by
\(\Op_t(\P)\varphi=\Op (\ev_t(\P)) \varphi = \P_t((\theta^t)^*\varphi)\), 
where $\Op$ is the map from \eqref{eq:Op} for the groupoid $\X \rtimes^t G$
and $\theta^t$ is the source map in the groupoid $\X\rtimes^t G$.

\begin{definition}
	For \(m\in\RR\) let \(\PPseu^m\) denote the space of all \(\mathbb P\in \Algebra O'_r(\TG) \) which are \emph{essentially homogeneous of order \(m\)}, that is,  
	\begin{equation}
		\lambda^m\mathbb P-{\tau_\lambda}_* \mathbb P\in\Schwartz(\TG) 
		\qquad\text{for all $\lambda > 0$.}
	\end{equation}
	Here, we implicitly use the inclusion \(u\colon \Schwartz(\TG)\hookrightarrow \Algebra O'_r(\TG)\) from \cref{ex:schwartz-as-fibred}.
	
	An element \(\P_1\in\Algebra O_r'(\X\rtimes^1 G)\) is called a \emph{pseudodifferential fibred distribution of order \(m\)} if there is an extension \(\mathbb P\in\PPseu^m \).
	The corresponding continuous map \(\Op(\P_1)\colon \Algebra O_M(\X)\to\Algebra O_M(\X)\) is a called \emph{Shubin pseudodifferential operator of order \(m\)} on \(\X\).	We denote by \(\Pseu^m\) the space of all pseudodifferential fibred distributions of order $m$
	and by \(\Op(\Pseu^m)\) the space of Shubin pseudodifferential operators of order \(m\).
\end{definition}
We shall also write \(\PPseu\defeq\bigcup_{m\in\RR}\PPseu^m\). 
We first show that the pseudodifferential calculus contains the algebra \(\mathcal A\) of certain differential operators described in the previous section.
\begin{example}\label{ex:diff_ops_pol_coeff} 
Consider \(P=\sum_{[a]_\alpha+[b]_\beta\leq m}c_{a,b}x^b (\widehat X^1)^a\in\Algebra A_m\) with \(m\in\NN_0\) and \(c_{a,b}\in\CC\). 
	Define \(\P\in\Algebra O_r'(\TG)\) by setting for \(\varphi\in \Algebra O_M(\TG)\)
	\[\P(\varphi)(x,t)=\sum_{[a]_\alpha+[b]_\beta\leq m}t^{m-[a]_\alpha-[b]_\beta}c_{a,b}x^b X^a_v \varphi(x,t,0).\] 
	Here $X^a_v$ means that the left invariant differential operator $X^a$ is applied to the last variable $v$.
	We compute for \(\lambda>0\), \(\varphi \in\Algebra O_M(\TG)\) and \((x,t)\in \X\times\RR\) \begin{align*}
		{\tau_\lambda}_*(\P)(\varphi)(x,t)&=\P(\varphi\circ\tau_\lambda)(\beta_\lambda(x),\lambda t)= \sum_{[a]_\alpha+[b]_\beta\leq m}(\lambda t)^{m-[a]_\alpha-[b]_\beta}c_{a,b}(\beta_\lambda(x))^b X^a_v(\varphi\circ\tau_\lambda )(\beta_\lambda(x),\lambda t,0)\\
		&= \lambda^m\sum_{[a]_\alpha+[b]_\beta\leq m}t^{m-[a]_\alpha-[b]_\beta}c_{a,b}x^b X^a_v(\varphi)(x,t,0)=\lambda^m\P(\varphi)(x,t),
	\end{align*} so that \(\P\) is \(m\)-homogeneous with respect to the zoom action.
	In particular, it is essentially homogeneous of order $m$.
	To compute the corresponding operator at \(t=1\) note that for \(X\in\lie{g}\) and \(\varphi\in\Algebra O_M(\X)\)
	\begin{align*}
		X_v((\theta^1)^*\varphi)(x,v)&=\frac{\D}{\D s}\Big|_{s=0} (\theta^1)^*\varphi(x,v\cdot\exp(sX)) =\frac{\D}{\D s}\Big|_{s=0} \varphi(\theta^1_{\exp(sX)}(\theta^1_v(x)) \\&
		=\widehat X^1 \varphi(\theta^1_v(x))=(\theta^1)^*(\widehat X^1\varphi)(x,v).
	\end{align*}
	Therefore, one has for \(x\in X\)
	\begin{align*}
		\Op_1(\P)(\varphi)(x)&=\P_1((\theta^1)^*\varphi)(x)=\sum_{[a]_\alpha+[b]_\beta\leq m}c_{a,b}x^b X^a_v((\theta^1)^*\varphi)(x,0) \\&
		=\sum_{[a]_\alpha+[b]_\beta\leq m}c_{a,b}x^b (\theta^1)^*((\widehat X^1)^a\varphi)(x,0)=P\varphi(x).
	\end{align*}
 Hence, \(P=\Op_1(\P)\) belongs to \(\Op(\Pseu^m)\). 
 Also note that \(\P_0\) consists of the highest order part in the sense that for \(\varphi\in\Algebra{O}_M(\X\rtimes^0 G)\)
 \begin{equation*}
 	\P_0(\varphi)(x)=\sum_{[a]_\alpha+[b]_\beta= m}c_{a,b}x^bX^a_v\varphi(x,0).
 \end{equation*}
 In fact, we see that \(\Algebra A_m \subseteq \Op(\Pseu^m)\cap \Algebra A\). 
\end{example}
\begin{example}\label{ex:identity_operator}
	As a particular case of the previous example, the identity operator \(\id\colon\Algebra O_M(\X)\to \Algebra O_M(\X)\) is contained in \(\Op(\Pseu^0)\). Namely, for \(\mathbb I\in \PPseu^0\) defined by
	\(\mathbb I \varphi(x,t)=\varphi(x,t,0)\) for \(\varphi\in\Algebra O_M(\TG)\), we get that \(\Op_1(\mathbb I)=\id\).
\end{example}
\begin{remark}
	As \cref{ex:grushin-fundamental-vectorfields} shows, the map \(\Op\colon\Pseu^m\to\Op(\Pseu^m)\) is not necessarily injective. However, it is injective by \cref{res:Op-injective} when \(\theta^1\) is polynomially free in the sense of \cref{def:polynomially-free}. This is, for example, the case under property \ref{assumption:free}.
\end{remark}
Recall that $\Algebra O_r'(\TG) \cong \Algebra{O}_M(\X\times\RR)\completedtensor\Algebra O_M'(G)$ holds by \eqref{eq:fibredDistribution:Isos}.
In the following, we consider the Euclidean Fourier transform in the \(v\)-direction (understanding \(v\in G\) as an element of the underlying vector space \(\underline{\lie g}\) of the Lie algebra $\lie g$)
\[\widehat{\argument}=\mathcal F_{v\to\xi}\colon \Algebra{O}_M(\X\times\RR)\completedtensor\Algebra O_M'(G)\to \Algebra O_M(\X\times\RR\times \lie g^*) .\]
More precisely, $\widehat{\argument}$ is obtained from the Euclidean Fourier transform $\Algebra O_M'(\underline{\lie g}) \to \Algebra O_C(\lie g^*)$, the inclusion $\Algebra O_C(\lie g^*) \subseteq \Algebra O_M(\lie g^*)$, see \cref{res:properties:oc}, and the isomorphism $\Algebra{O}_M(\X\times\RR)\completedtensor\Algebra O_M(\lie g^*) \cong \Algebra{O}_M(\X\times\RR\times\lie g^*)$, see \cref{res:OM_continuity}.
We denote also by $\alpha$ the induced dilations on the dual $\lie g^*$, defined by
\(\langle \alpha_\lambda(\xi), X\rangle =\langle \xi,\alpha_\lambda(X)\rangle\) for $X \in \lie g$, $\xi \in \lie g^*$.
The Shubin zoom action transforms as follows
$\widehat{{\tau_\lambda}_* \P}(x,t,\xi)
=\widehat{\P}(\beta_\lambda(x),\lambda t,\alpha_{\lambda}(\xi))$.
We therefore equip \(\X \times\RR\times \lie g^*\) with the dilations 
\begin{equation}
	\widehat\tau_\lambda(x,t,\xi)=(\beta_\lambda(x),\lambda t,\alpha_\lambda(\xi)) ,
\end{equation}
so that $\widehat{{\tau_\lambda}_* \P} = \widehat\tau_\lambda^* \widehat \P$.
Additionally, the inverse Fourier transform in the \((x,t)\)-direction will be useful
\[\widecheck{\argument}=\mathcal F^{-1}_{(x,t)\to(\eta,\tau)}\colon \Algebra{O}_M(\X\times\RR)\completedtensor\Algebra O_M'(G)\to \Algebra O_C'(\X^*\times\RR\times G), \]
where we used that the inverse Fourier transform maps $\Algebra{O}_M(\X \times \RR)$ to $\Algebra{O}_C'(\X^* \times \RR)$, 
that $\Algebra{O}'_M(\X^*)$ includes into $\Algebra{O}'_C(\X^*)$ by \cref{res:properties:oc}
and that $\Algebra O_C'(\X^*\times \RR) \completedtensor\Algebra O_C'(G) \cong \Algebra O_C'(\X^*\times\RR\times G)$ by \cref{res:OM_continuity}.
It satisfies \(\widecheck{{\tau_\lambda}_*\P}={(\widecheck\tau_\lambda)}_*\widecheck{\P}\), where \begin{equation}
	(\widecheck\tau_\lambda(\eta,\tau,v)=(\beta_\lambda(\eta),\lambda\tau,\alpha_\lambda(v))
\end{equation}
and \((\widecheck\tau_\lambda)_*\) is defined as in \eqref{eq:dilations_on_tempered}.

\begin{proposition}[\cites{Tay84,BG88}]\label{res:equivalence-pseudo-symbol}
	Let \(\P\in\Schwartz'(\TG)\). Then the following are equivalent:
	\begin{enumerate}
		\item\label{item:pseudo} \(\P\in\PPseu^m\),
		\item\label{item:widehat1} \(\widehat \P\in\Smooth(\X\times\RR\times\lie g^{*})\) and \(\widehat{\P}(\beta_\lambda(x),\lambda t, \alpha_\lambda(\xi))-\lambda^{m} \widehat{\P}(x,t,\xi)\in\Schwartz(\X\times\RR\times\lie{g}^*)\) for all \(\lambda >0\),
		\item\label{item:widehat2} there is a \(m\)-homogeneous function \(P\in \Smooth(\X\times\RR\times\lie{g}^*\setminus\{(0,0,0)\})\), i.e. \(P(\beta_\lambda(x),\lambda t, \alpha_\lambda(\xi))=\lambda^{m} P(x,t,\xi)\) for all \(\lambda>0\),  such that for all smooth \(\chi\) vanishing in a neighbourhood of \((0,0,0)\) and constant \(1\) outside a compact set (equivalently: for one such $\chi$) there is a \(f\in \Schwartz(\X\times\RR\times\lie{g}^*)\) such that \(\widehat{\P}=\chi\cdot P+f\), 
		\item\label{item:widecheck} \(\widecheck\P\in \Algebra E'(\X^*\times\RR\times G)+\Schwartz(\X^*\times\RR\times G)\), has singular support in \(\{(0,0,0)\}\) and satisfies \begin{equation*}{(\widecheck\tau_\lambda)}_*\widecheck \P-\lambda^m\widecheck{\P}\in\Schwartz(\X^*\times\RR\times G)\quad \text{for all }\lambda>0.\end{equation*}
	\end{enumerate}
\end{proposition}

\begin{remark}
	We often start with a function $\varphi$ (``$= \widehat \P$'') satisfying \refitem{item:widehat1}. Then $\varphi$ defines a tempered distribution,
	hence is of the form $\varphi = \widehat\P$ for some $\P\in \Schwartz'(\TG)$. 
	The previous proposition then implies $\P \in \PPseu^m$.
	Similarly, we may start with a distribution $u$ (``$= \widecheck \P$'') satisfying \refitem{item:widecheck},
	which is automatically a tempered distribution, hence of the form $\widehat\P$ for some $\P \in \PPseu^m$.
\end{remark}
\begin{proof}
	The implication \refitem{item:pseudo}\(\Rightarrow\)\refitem{item:widehat1} is clear. The equivalence of \refitem{item:widehat1} and \refitem{item:widehat2} is \cite{Tay84}*{Lemma~2.2} or \cite{BG88}*{Proposition~12.72}.
	 Furthermore, \refitem{item:widehat1} implies \refitem{item:widecheck} is shown in \cite{Tay84}*{Proposition~2.1}. The last implication \refitem{item:widecheck}\(\Rightarrow\)\refitem{item:pseudo} is clear once we note that \(\mathcal F_{(\eta,\tau)\to(x,t)}\widecheck\P\) belongs to \(\Algebra O'_r(\mathcal G)\). By \cite{Tre67}*{Theorem~51.6, Corollary~51.7}
	\begin{equation*}\Algebra E'(\X^*\times\RR\times G)+\Schwartz(\X^*\times\RR\times G)\cong \Algebra E'(\X^*\times\RR)\completedtensor\Algebra E'(G)+\Schwartz(\X^*\times\RR)\completedtensor\Schwartz(G)
	\end{equation*}
	holds, $\Algebra E'(G)$ and $\Algebra S(G)$ are contained in $\Algebra O_M'(G)$ 
	and the Fourier transforms of $\Algebra E'(\X^* \times \RR)$ and $\Schwartz(\X^* \times \RR)$ are in $\Algebra O_M(\X \times \RR)$.
\end{proof}
In particular, \refitem{item:widehat2} implies the following symbol estimates.
Fix homogeneous quasi-norms $\norm\argument_\alpha$ on $G$ for the dilations $\alpha$ 
and $\norm\argument_\beta$ on $\X$ for $\beta$ as in \eqref{eq:quasi-norm} with $q$ being a common multiple of the weights $q_1, \dots, q_n$ of $\alpha$ and the weights $r_1, \dots, r_d$ of $\beta$.
\begin{corollary}\label{res:symbol_estimates}
	Let \(\mathbb P\in\PPseu^m\), then for all \((a,b,c)\in\NN^{n+d+1}_0\) there is a constant \(C_{a,b,c}>0\) such that
		\begin{align}\label{eq:symbol_estimates}
			\abs{\partial^a_\xi\partial^b_x \partial^c_t\widehat{\P}(x,t,\xi)}\leq C_{a,b,c} (1+\norm{\xi}_\alpha+\norm{x}_\beta+\abs{t})^{m-[a]_\alpha-[b]_\beta-c}	.	\end{align}
\end{corollary}
In this estimate, we may replace \(1+\norm{\xi}_\alpha+\norm{x}_\beta+\abs{t}\)  by 
\begin{equation} \label{eq:JapaneseBracket:homogeneous}
	\langle (x,t,\xi)\rangle_H = \left(1+\norm{\xi}_\alpha^{2q}+\norm{x}_\beta^{2q}+ t^{2q}\right)^{\frac{1}{2q}}
\end{equation}
using the equivalence of all homogeneous quasi-norms.

\begin{remark}
	If $\P \in \PPseu^m$ for some $m \in \RR$, then one can study the $\lambda$-dependence of the Schwartz function \(\lambda^m{\tau_{\lambda^{-1}}}_* \P-\P\).
	Choose a smooth cut-off $\omega \colon \RR \to [0,1]$ such that
	\(\omega(\lambda)=0\) when \(\lambda \leq 1\) and \(\omega(\lambda)=1\) when \(\lambda \geq 2\).
	Then the function \((\lambda,x,t,v)\mapsto\omega(\lambda) (\lambda^m{\tau_{\lambda^{-1}}}_* \P- \P)(x,t,v)\) belongs to \(\Algebra{O}_M(\RR)\completedtensor \Schwartz(\TG)\).
\end{remark}

\begin{proposition}\label{res:singular_support}
	The singular support of \(\P\in\PPseu^m\) is contained in \(\X\times\RR\times\{0\}\).
\end{proposition}

\begin{proof}
	Let \(v_0 \in G \setminus \simpleset 0\).
	Choose a cut-off \(\chi\in\SmoothCompactSupp(G)\) which is constant~\(1\) near \(v_0\) and vanishes in a neighbourhood of~\(0\).
	By \cref{res:equivalence-pseudo-symbol}, \(\widecheck{\P}=\mathcal F^{-1}_{(x,t)\to(\eta,\tau)} \P\) belongs to \(\mathcal E'(\X^*\times\RR\times G)+\Schwartz(\X^*\times\RR\times G)\) and its singular support is contained in~\(\{(0,0,0)\}\). 
	Then $\mathcal F^{-1}_{(x,t)\to(\eta,\tau)}((1_{\X \times \RR} \tensor \chi)\cdot \P)=\chi\cdot \widecheck{\P}$ is in $\Schwartz(\X^*\times\RR\times G)$, hence
	\((1_{\X \times \RR} \tensor \chi)\cdot \P\) is smooth.
\end{proof}

\paragraph{Principal cosymbol} To define a principal cosymbol as in \cite{vEY19} the following lemma is crucial.
\begin{lemma}\label{res:smoothingat1implies0}
	Let \(\mathbb P\in\PPseu^m\) such that \(\P_1\in\Schwartz(\X\rtimes^1 G)\). Then \(\P_0\in\Schwartz(\X\rtimes^0 G)\).
\end{lemma}
\begin{proof}
	It suffices to show that \(\widehat{\P}_0\) is a Schwartz function. As \(\widehat{\P}_1\) is Schwartz and \(\mathbb P\) is essentially homogeneous, also \(\widehat{\P}_t\) is Schwartz for all \(t>0\). Let \(P\) be the homogeneous function as in \cref{res:equivalence-pseudo-symbol}, it must vanish for \(t>0\) and therefore by continuity also at \(t=0\). Hence, \(\widehat{\P}_0\) is a Schwartz function.
\end{proof}	
\begin{definition}
	For \(m\in\RR\), define the \emph{space of essentially homogeneous distributions at \(t=0\)} by
	\begin{equation*}
		\Ess^m=\set{ u\in \Algebra O'_r(\X\rtimes^0 G) }{ {\tau_\lambda}_*u-\lambda^mu\in\Schwartz(\X\rtimes^0 G) \text{ for all $\lambda > 0$}},
	\end{equation*} Here, \(\tau\) is understood as the restriction of the zoom action to \(t=0\). 
\end{definition}
Note that an analogous result to \cref{res:equivalence-pseudo-symbol} holds for \(\Ess^m\).
\begin{definition}
Define the
	\emph{space of principal cosymbols of order \(m\)} as the quotient space \[\Symb^m=\Ess^m/\Schwartz(\X\rtimes^0 G).\]  The \emph{principal cosymbol map}
	\(\sigma_m\colon\Pseu^m\to\Symb^m\) is defined as follows. For \(\P_1\in\Pseu^m\) take any extension \(\P\in\PPseu^m\) and let \(\sigma_m(\P_1)=[\P_0]\). The principal cosymbol map is well-defined by \cref{res:smoothingat1implies0}.
\end{definition}

\begin{remark} \label{remark:P:principalSymbolOfOperators}
	Under the assumption that \(\theta^1\) is polynomially free also \(\sigma_m(P)\) for \(P\in\Op(\Pseu^m)\) is well-defined by \cref{res:Op-injective}. This fails if \(\Op_1\) is not injective, as \(P\in\Op(\Pseu^m)\) may lift to different \(\P_1\in\Pseu^m\) having different principal cosymbols. 
	
	This is similar to the situation in \cite{AMY22}. They circumvent the problem by defining a principal symbol which does not depend on all representations of the group at \(t=0\) but only on a subset -- the Helffer--Nourrigat cone. It is conceivable that a similar approach works if property \ref{assumption:homomorphisms} holds.
\end{remark}

\begin{proposition}\label{res:short_exact_sequence}
	For \(m\in\RR\) the map \(\PPseu^{m-1}\hookrightarrow\PPseu^m\) with \(\P\mapsto t\P\) induces a short exact sequence
	\begin{equation*}
		\begin{tikzcd}
		0\arrow[r]&	\PPseu^{m-1} \arrow[r] & \PPseu^m\arrow[r,"\ev_0"] & \Ess^m\arrow[r] &0.
		\end{tikzcd}
	\end{equation*}
	It admits a linear split \(r_m\colon \Ess^m\to\PPseu^m\) with the property that \(r_m(u)_1=u\) under the canonical identification \(\X\rtimes^1 G= \X\times G=\X\rtimes^0 G\) as range fibred spaces over \(\X\). Moreover, this induces a short exact sequence
	\begin{equation*}
		\begin{tikzcd}
			0\arrow[r]&	\Pseu^{m-1} \arrow[r] & \Pseu^m\arrow[r,"\sigma_m"] & \Symb^m\arrow[r] &0.
		\end{tikzcd}
	\end{equation*}
\end{proposition}
\begin{proof}
	The inclusion \(\PPseu^{m-1}\subseteq \ker(\ev_0)\) is evident. For the converse inclusion suppose that \(\P \in \PPseu^m\) is in the kernel of \(\ev_0\). To show that \(\P=t\Q\) for some \(\Q\in\PPseu^{m-1}\), consider \(\widehat \P\). By \cref{res:equivalence-pseudo-symbol} it is a smooth function and for \(\lambda>0\) there is \(f_\lambda\in\Schwartz(\X\times\RR\times\lie{g}^*)\) such that 
	\begin{equation*}
		\widehat{\P}(\beta_\lambda(x),\lambda t, \alpha_\lambda(\xi))-\lambda^{m} \widehat{\P}(x,t,\xi) =f_\lambda(x,t,\xi).
	\end{equation*}
	As \(\widehat \P\) vanishes at \(t=0\), the function \(\widehat\Q\defeq t^{-1}\widehat \P\) is well-defined and smooth. Moreover, one computes
	\begin{equation*}
		\widehat{\Q}(\beta_\lambda(x),\lambda t, \alpha_\lambda(\xi))-\lambda^{m-1} \widehat{\Q}(x,t,\xi) 
		=\frac{f_\lambda(x,t,\xi)}{\lambda t}.
	\end{equation*}	
	Note that \(f_\lambda\) also vanishes at \(t=0\), so that \((x,t,\xi)\mapsto(\lambda t)^{-1}f_\lambda(x,t,\xi)\) is a well-defined Schwartz function. Then \(\Q\in\PPseu^{m-1}\) holds by \cref{res:equivalence-pseudo-symbol}.

	Next, we show surjectivity of $\ev_0$.
	Fix smooth cut-offs $\chi_1 \colon X \times \lie g^* \to [0,1]$ and $\chi_2 \colon X \times \RR \times \lie g^* \to [0,1]$
	which vanish in a neighbourhood of ${(0,0)}$ and ${(0,0,0)}$ and are constant $1$ outside compact sets and
	$\omega \colon \RR \to [0,1]$ such that $\omega(t) = 1$ for $\abs t \geq 2$ and $\omega(t) = 0$ for $\abs t \leq 1$. 
	Let \(u\in\Ess^m\).
	By the analogous result to \cref{res:equivalence-pseudo-symbol},
	there is a $m$-homogeneous function $U \in \Smooth(X \times \lie g^* \setminus \simpleset{(0,0)})$
	and a Schwartz function $f \in \Schwartz(\X \times \lie g^*)$ such that
	$\widehat u = \chi_1 \cdot U + f$.
	Then $P(x,t,\xi) = \omega(\frac 1 t \norm{(x,\xi)}_{\beta,\alpha}) U(x,\xi)$ defines an $m$-homogeneous function on $\X \times \RR \times \lie g^* \setminus \simpleset{(0,0,0)}$.
	Here $\norm{\argument}_{\beta,\alpha}$ denotes a homogeneous quasi-norm with respect to the dilations $(\beta,\alpha)$ on $\X \times \lie g^*$
	and for $t = 0$, $(x,\xi) \neq (0,0)$ we suppose that $\omega(\frac 1 t \norm{(x,\xi)}_{\beta,\alpha}) = 1$.
	Let $\widehat \P \in \Smooth(\X \times \RR \times \lie g^*)$ be defined by
	\begin{equation*}
		\widehat \P(x,t,\xi) = \chi_2(x,t,\xi) P(x,t,\xi) + (1-\omega(t))( \widehat u(x,\xi) - \chi_2(x,t,\xi) P(x, t, \xi)).
	\end{equation*}
	Note that $(1-\omega(t)) (\chi_1(x,\xi) - \chi_2(x,t,\xi) \omega(\frac 1 t \norm{(x,\xi)}_{\beta,\alpha}))$ has compact support,
	therefore the second summand of $\widehat \P$ is Schwartz. 
	By \cref{res:equivalence-pseudo-symbol}, $\widehat \P$ is the Fourier transform of an element $\P \in \PPseu^m$.
	Let $r_m(u) = \P$. Clearly, $r_m$ is linear and $r_m(u)_t = u$ holds for all $\abs t \leq 1$.
	
	Exactness of the second sequence is easily checked using exactness of the first.
\end{proof}

\paragraph{Automorphisms} We examine now under which changes of coordinates the calculus is invariant. 
\begin{lemma}\label{res:diffeo_groupoid}
	Let \(F\colon \X\to \X\) be a polynomial diffeomorphism satisfying \(F\circ \theta^1_v=\theta^1_v\circ F\) for all \(v\in G\).
	Suppose that the map \((x,t)\mapsto (\beta_t\circ F \circ\beta_{t^{-1}}(x),t)\) extends smoothly to a diffeomorphism \(\mathbb F^{(0)}\) of \(\TG^{(0)}=\X\times\RR\). Then the map \(\mathbb F\colon \TG\to\TG\) defined by 
	\begin{equation*}
		\mathbb F(x,t,v)=\left(\mathbb F^{(0)}(x,t),v\right)\quad\text{for }(x,t,v)\in\TG,
	\end{equation*}
	defines a zoom-equivariant automorphism of $\TG$. 
\end{lemma}
\begin{proof}
	First note that $\mathbb F^{(0)} \circ \theta_v = \theta_v \circ \mathbb F^{(0)}$ for all $v \in G$.
	Namely for all \(t\neq 0\), we have by the compatibility \eqref{eq:compatibility} and using \(F\circ \theta^1_v=\theta^1_v\circ F\)
	\begin{equation*}
		(\beta_t\circ F\circ\beta_{t^{-1}})\circ \theta^t_v= \beta_t\circ F\circ\beta_{t^{-1}}\circ (\beta_t\circ\theta^1_{\alpha_t(v)}\circ\beta_{t^{-1}})=\beta_t\circ\theta^1_{\alpha_t(v)}\circ F\circ\beta_{t^{-1}}=\theta^t_v\circ(\beta_t\circ F\circ\beta_{t^{-1}}).
	\end{equation*}
	For \(t=0\) this follows by continuity. The map \(\mathbb F^{(0)}\) is polynomial by similar arguments as after \cref{assumption}. By \cref{res:automorphism} the \(G\)-equivariant diffeomorphism \(\mathbb F^{(0)}\) of the unit space \(\TG^{(0)}=\X\times\RR\) extends to an automorphism \(\mathbb F\colon \TG\to\TG\) defined by \((x,t,v)\mapsto(\mathbb F^{(0)}(x,t),v)\). To see that it is zoom-equivariant, we compute for \(t\neq 0\)
	\begin{align*}
		\mathbb F(\tau_\lambda(x,t,v))=\mathbb F(\beta_{\lambda^{-1}}(x),\tfrac{t}\lambda,\alpha_\lambda(x))= (\beta_{t\lambda^{-1}}(F(\beta_{t^{-1}}(x))),\tfrac{t}{\lambda},\alpha_\lambda(x))=\tau_\lambda(\mathbb F(x,t,v))
	\end{align*}
and argue for \(t=0\) using continuity. 
\end{proof}
\begin{corollary}\label{res:invariance}
	In the situation of \cref{res:diffeo_groupoid}, the automorphism \(\mathbb F\) induces an automorphism \(\mathbb F_*\) of \(\PPseu\) defined by \(\mathbb F_*(\P)=((\mathbb F^{(0)})^{-1})^*\circ\P\circ\mathbb F^*\).  It satisfies for \(\P\in\PPseu\)
	\begin{equation*}
		\Op_1({\mathbb F}_*\P)=(F^{-1})^*\circ\Op_1(\P)\circ F^*.
	\end{equation*}
\end{corollary}
\begin{proof}
	As \(\mathbb F^{(0)}\) and therefore also \(\mathbb F\) are polynomial, the map \(\mathbb F_*\colon\Algebra O_r'(\TG)\to\Algebra{O}_r'(\TG)\) is an automorphism by \cref{res:automorphism}. As \(\mathbb F\) is zoom-equivariant and \(\mathbb F_*\colon\Algebra \Schwartz(\TG)\to\Schwartz(\TG)\), it restricts to maps \(\mathbb F_*\colon \PPseu^m\to\PPseu^m\) for every \(m\in\RR\). One computes for \(\P\in\PPseu\) and \(\varphi\in\Algebra O_M(\X)\) 
	\begin{equation*}
		\Op_1(\mathbb F_*\P)(\varphi)=(\mathbb F_*\P_1)((\theta^1)^*\varphi)=(F^{-1})^*\P_1((\theta^1)^*(F^*\varphi))= (F^{-1})^*(\Op_1(\P)(F^*\varphi)).\qedhere
	\end{equation*}
\end{proof}
Note that the principal cosymbol of \(\mathbb F_*\P\) changes according to \(\mathbb F^{t=0}\).
\begin{example}
	For the double dilation groupoid the unit element \(0\in \underline G\) seems to play a special role for the calculus as it is fixed by the dilations \(\beta\). This is not the case in the sense that the calculus is invariant under shifts by \(x_0\in \underline G\) and that the principal cosymbol is invariant under such shifts. This can be seen by applying the previous results to the diffeomorphism \(F(x)=x_0 x\). In this case \(\mathbb F^{(0)}(x,t)=(\beta_t(x_0) x,t)\) extends smoothly to the identity at \(t=0\). Therefore, the principal cosymbols of \(\P_1\in\Pseu^m\) and the shifted operator \(\mathbb F^1_*\P_1\in\Pseu^m\) coincide. 
\end{example}

\subsection[\texorpdfstring{Filtered \(^*\)-algebra structure}{Filtered *-algebra structure}]{\texorpdfstring{Filtered \boldmath\(^*\)-algebra structure}{Filtered *-algebra structure}}

To see that the pseudodifferential calculus of a Shubin tangent groupoid is closed under composition, 
one needs that \(\Schwartz(\TG)\) forms a two-sided ideal in~$\PPseu$.
We start by showing that it is a left ideal. 

\begin{proposition}\label{res:leftschwartzmultiplier}
	Let \(\P\in\PPseu^m\) and $f \in \Schwartz(\TG)$. Then $\P * f \in \Schwartz(\TG)$ is given by
	\begin{equation} \label{eq:Pstarf}
		\P * f(x,t,v) = (2\pi)^{-n}\int \E^{\I\langle w, \xi\rangle}\widehat{\P}(x,t,\xi)f(\theta(x,t,w),w^{-1}v)\D w\D \xi .
	\end{equation}
	Moreover, $\P$ is a left multiplier of $\Schwartz(\TG)$, i.e.\ \(f\mapsto \P * f\) defines a continuous map \(\Schwartz(\TG)\to\Schwartz(\TG)\).
\end{proposition}
\begin{proof}
	Let $f \in \Schwartz(\TG)$. Then $\int \E^{\I\langle w, \xi\rangle} f(\theta(x,t,w),w^{-1}v)\D w$ is Schwartz in $\xi$,
	hence the right hand side of \eqref{eq:Pstarf} is well-defined as an iterated integral.
	In \textbf{(1)} we show that every Schwartz seminorm $\norm{\argument}_{k,\ell}$, see \eqref{eq:schwartz:seminorm}, 
	with $k, \ell \in \NN_0$ of it can be bounded by a seminorm of \(f\).
	Afterwards, we show in \textbf{(2)} that $\P * f$ is indeed given by the claimed formula.
	
	\textbf{(1)} Recall that $\langle (x,t,v) \rangle_H$ was defined in \eqref{eq:JapaneseBracket:homogeneous}.
	Since all homogeneous quasi-norms are equivalent, it suffices to estimate the absolute value of 
	\(\smash{\langle (x,t,v)\rangle_H^\ell\partial^\gamma_{(x,t,v)}(\P*f)}\)
	for all \(\gamma\in\NN_0^{m+1+n}\) with $\abs \gamma \leq k$.
	It is a finite linear combination of terms of the form 
	\begin{equation*}\label{eq:osc_integral} \tag{$*$}
		\langle(x,t,v)\rangle^\ell_H \int \E^{\I \langle w, \xi\rangle}\partial^a_{(x,t)} \widehat{\P}(x,t,\xi)\partial^{b}_{(x,t,v)}(f(\theta(x,t,w),w^{-1}v))\D w\D \xi
	\end{equation*}
	with $a \in \NN_0^{d+1}$, $b \in \NN_0^{d+1+n}$, $\abs a + \abs b \leq \abs{\gamma}$.	
	Recall from \cref{def:quasi-norm} the notation
	\(\langle w\rangle_\alpha =(1+\norm{w}_\alpha^{2q})^\frac{1}{2q}\),
	where $q$ is a common multiple of the weights $q_1, \dots, q_n$ of $\alpha$. 
	Define for \(K\in 2q\NN_0\)
	\[ 
	\langle D_w\rangle^K_\alpha
	=
	\biggl(
	1+\sum_{j=1}^n(-1)^{\frac{q}{q_j}}\partial_{w_j}^{\frac{2q}{q_j}}
	\biggr)^\frac{K}{2q} .
	\]
	One computes that for all \(M,K\in 2q\NN_0\)
	\begin{equation} \label{eq:oscillatoryDecay}
		\frac{\langle D_{\xi} \rangle_\alpha^M}{\langle w  \rangle^M_\alpha}\E^{\I\langle w,\xi\rangle}=\E^{\I\langle w,\xi\rangle}
		\quad\text{ and }\quad 
		\frac{\langle D_{w}   \rangle_\alpha^K}{\langle \xi\rangle^K_\alpha}\E^{\I\langle w,\xi\rangle}=\E^{\I\langle w,\xi\rangle}.
	\end{equation}
	Using partial integration with respect to $w$ (which is justified since the integrand is Schwartz in that variable),
	one can rewrite \eqref{eq:osc_integral} for \(K \in 2q\NN_0\), $K > Q(\alpha) + m$ as
	\begin{equation*}
		\langle(x,t,v)\rangle_H^\ell 
		\int \E^{\I\langle w, \xi\rangle} 
		\langle \xi \rangle_{\alpha}^{-K} 
		\partial^a_{(x,t)} \widehat{\P}(x,t,\xi)
		\langle D_w \rangle^K_\alpha \partial^{b}_{(x,v,t)}(f(\theta(x,t,w),w^{-1}v))
		\D w\D \xi .
	\end{equation*}
	Since $\widehat \P$ is a symbol of order $m$, the integrand decays faster than $\langle\xi\rangle_\alpha^{-Q(\alpha)}$
	and is therefore integrable in $\xi$ by \cref{res:integrable}.
	We may therefore change the order of integration and perform another partial integration with respect to $\xi$, obtaining for any $M \in 2q\NN_0$
	\begin{equation*}\label{eq:stuffToEstimate}
		\int \langle(x,t,v)\rangle_H^\ell\langle w\rangle_\alpha^{-M} \E^{\I \langle w, \xi\rangle}
		\underbrace{\langle D_\xi\rangle_\alpha^M(\langle \xi\rangle_{\alpha}^{-K}\partial^a_{(x,t)}\widehat{\P}(x,t,\xi))}_{(\text{i})}
		\underbrace{\langle D_w\rangle^K_\alpha\partial^{b}_{(x,v,t)}(f(\theta(x,t,w),w^{-1}v))}_{(\text{ii})}\D w\D \xi .
		\tag{$**$}
	\end{equation*}
	To estimate the terms (i) and (ii) we frequently use Peetre's inequality 
	$\langle(x,t,\xi)\rangle_H^j\leq \langle(x,t)\rangle_\beta^{\abs j}\langle \xi\rangle^j_\alpha$ for $j \in \RR$, where by abuse of notation $\beta$ denotes the extension of the dilations $\beta$ on $\X$ to $\X \times \RR$ giving $t$ the weight $1$.
	By the symbol estimates for \(\P\) shown in \cref{res:symbol_estimates} and the estimate 
	$\abs{\partial_\xi^e(\langle \xi\rangle_{\alpha}^{-K}) }\lesssim \langle \xi\rangle_{\alpha}^{-K}$
	we get
	\[
	\abs{
		\langle D_\xi\rangle_\alpha^M (\langle \xi\rangle_{\alpha}^{-K} \partial^{a}_{(x,t)} \widehat{\P}(x,t,\xi))
	} 
	\lesssim 
	\langle \xi\rangle_{\alpha}^{-K} \langle (x,t,\xi)\rangle^m_H
	\lesssim
	\langle \xi\rangle_{\alpha}^{m-K} \langle (x,t)\rangle_\beta^{\abs{m}} ,
	\]
	where ``$\lesssim$'' indicates that we omitted a multiplicative constant. 
	Note that (ii) is a linear combination 
	of terms of the form $\partial^b_{(x,t,v)} \partial^c_w (f(\theta(x,t,w),w^{-1}v))$ with $\abs c \leq \abs K$,
	which are estimated by
	\begin{align*}
		\abs{\partial^b_{(x,t,v)} &\partial^c_w (f(\theta(x,t,w),w^{-1}v))} \\
		&\leq 
		\sum_d \abs[\big]{(\partial^d_{(x,t,v)} f)(\theta(x,t,w),w^{-1}v)} 
		\cdot \abs[\big]{\Chain{\theta(x,t,w), w^{-1} v}{(b,c)}{d}{x,t,v,w}} \\
		&\lesssim \norm{f}_{N,\abs b+\abs c} \langle(\theta(x,t,w), w^{-1}v)\rangle_H^{-N}
		(\langle(x,t)\rangle_\beta \langle v\rangle_\alpha \langle w\rangle_\alpha)^{(\abs b + \abs c)(B + Q(\alpha))} 
	\end{align*}
	for any $N \in \NN$	with \(B\in\NN\) being the constant from \cref{res:polynomial_action} 
	for the action \(\theta\) of \(G\) on \(\X\times\RR\).
	Here the first inequality follows from \cref{res:chain_rule} and the triangle inequality.
	The second inequality follows again from \cref{res:chain_rule} since a common bound for the partial derivatives of the argument of $\mathrm{Ch}$ is (up to a multiplicative constant) 
	$( \langle (x,t) \rangle_\beta
	   \langle v \rangle_\alpha 
	   \langle w \rangle_\alpha)^{B + Q(\alpha)}$:
	Indeed, by \eqref{eq:polynomial1} partial derivatives of $\theta$ are bounded by 
	$\langle (x,t) \rangle_\beta^B \langle w \rangle_\alpha^B$
	and by \cref{res:norm_estimates} \refitem{item:derivatives-mult} partial derivatives of 
	$w^{-1} v$ are bounded by $\langle v \rangle_\alpha^{\rule[-1pt]{0pt}{0pt}\smash{Q(\alpha)}} \langle w \rangle_\alpha^{\rule[-1pt]{0pt}{0pt} \smash{Q(\alpha)}}$.
	
	Define $A \coloneqq (k + K)(B + Q(\alpha)) + \abs m + \ell$
	and assume that $N \geq 2AB$.
	Note $\abs b + \abs c \leq k + K$, so that
	the absolute value of the integrand in \eqref{eq:stuffToEstimate} is bounded by 
	\begin{align*}
		\langle w\rangle_\alpha^{-M} 
		&\langle \xi\rangle_{\alpha}^{m-K}
		\norm{f}_{N,k + K} \langle(\theta(x,t,w), w^{-1}v)\rangle_H^{-N}
		(\langle(x,t)\rangle_\beta \langle v\rangle_\alpha \langle w\rangle_\alpha)^{A}
		\\
		&\lesssim
		\langle \xi\rangle_{\alpha}^{m-K}
		\norm{f}_{N,k + K} 
		\frac{
			\langle(x,t)\rangle_\beta^{A}
		}{
			\langle \theta(x,t,w)\rangle_\beta^{N/2}
		}
		\frac{
			\langle v\rangle^{A}_\alpha
		}{
			\langle w^{-1}v \rangle_\alpha^{N/2}
		}
		\langle w\rangle^{A-M}_\alpha
		\\
		&\lesssim
		\langle w\rangle_\alpha^{2 A+AB-M} 
		\langle \xi\rangle_{\alpha}^{m-K}
		\norm{f}_{N,k + K} .
	\end{align*}
	To obtain the second estimate, we used the inequalities 
	\[
	\frac{\langle(x,t)\rangle_\beta}{\langle \theta(x,t,w)\rangle^B_\beta }\lesssim \langle w\rangle_\alpha^B
	\quad\text{ and }\quad
	\frac{\langle v\rangle_\alpha}{\langle w^{-1}v\rangle_\alpha}\lesssim \langle w\rangle_\alpha
	\]
	from \cref{res:polynomial_action} and \cref{res:norm_estimates}.
	
	Choosing $M \geq 2 A + A B + Q(\alpha)$, the right hand side of the above estimate is integrable by \cref{res:integrable},
	we obtain $\norm{\P*f}_{k,\ell} \lesssim \sum\sup_{(x,t,v)} \abs{\eqref{eq:osc_integral}} \lesssim \norm f_{k + K, N}$. Here \eqref{eq:osc_integral} shall be replaced by the expression in that equation,
	and $\sum$ is a reminder that we need to take linear combinations of such terms,
	which was absorbed in $\lesssim$.
	
	\textbf{(2)} Let $\varphi \in \SmoothCompactSupp(\TG)$. 
	Then
	$s^* u_f( M^*(\varphi) )(x,t,w) = \int_G f(\theta(x,t,w), v) \varphi(w v) \D v$
	holds by a computation similar to \eqref{eq:UonSchwartzFunction}.
	We compute
	\begin{multline*}
		(\P * f)(\varphi)(x,t) = \P \circ s^* u_f \circ M^* (\varphi)(x,t)
		= \langle \P_{x,t}, (s^* u_f \circ M^* (\varphi))_{x,t} \rangle
		\\ = \langle \widehat \P_{x,t}, \mathcal F^{-1}_{v \to \xi} (s^* u_f \circ M^* (\varphi))_{x,t} \rangle
		= (2\pi)^{-n} \int \E^{\I\langle w,\xi\rangle} \widehat \P(x,t,\xi) f(\theta(x,t,w), w^{-1} v) \varphi(v) \D w \D \xi \D v 
	\end{multline*}
	where the last equality follows by substituting $v$ with $w^{-1}v$ and applying Fubini's theorem to change the order of integration.
    Then \eqref{eq:Pstarf} holds since $ \SmoothCompactSupp(\TG) $ is dense in $\Algebra O_M(\TG)$.
\end{proof}
Under assumption \ref{assumption:free}, we have seen in \cref{res:orbitmaps} that the orbit maps \(G\to \X\) with \(\OrbitMap 1 x \colon  v\mapsto\theta^1_v(x)\) are a polynomial family of diffeomorphisms and that $c = \abs{\det D_v(\OrbitMap 1 x)}$ is independent of $x$ and $v$.
In this case, we can describe the operator \(\Op_1(\P)\) more explicitly as
\begin{multline} \label{eq:rep_as_osc_int}
	\Op_1(\P)f(x)=\langle \P_{x,1}, (\OrbitMap{1}{x})^*f\rangle 
	= \frac{1}{(2\pi)^{n}}\int \E^{\I \langle v,\xi\rangle }\widehat{\P}_1(x,\xi)f(\OrbitMap 1 x(v))\D v\D\xi 
	\\ =\frac{c}{(2\pi)^{n}}\int \E^{\I \langle (\OrbitMap{1}{x})^{-1}(y),\xi\rangle }\widehat{\P}_1(x,\xi)f(y) \D y\D\xi.
\end{multline}
Hence, the operator can be viewed as a Fourier integral operator with symbol \(\widehat\P_1\) and phase function \((x,y,\xi)\mapsto \langle (\OrbitMap{1}{x})^{-1}(y),\xi\rangle\). For example, for the double dilation groupoid one has $(\OrbitMap{1}{x})^{-1}(y)=x^{-1}y$.

\begin{corollary}\label{res:op_on_schwartz}
	Suppose \(\theta^1\) is polynomially transitive and let \(\P\in\PPseu^m\). Then for every \(t\neq 0\), \(\Op_t(\P)\) restricts to a continuous operator \(\Schwartz(\X)\to\Schwartz(\X)\). 
\end{corollary}
\begin{proof}
	By \cref{res:leftschwartzmultiplier}, \(\P_t\) is a left multiplier of \(\Schwartz(\X\rtimes^{t} G)\). Using \cref{res:Op-injective}~\refitem{item:Op-transitive}, \((\Theta^t)_*\P_t\) is a left multiplier of \(\Schwartz(\X\times \X)\). 
	Observe that the induced convolution $*$ on $\Algebra O_r'(\X \times \X)$ is just the usual composition of Schwartz kernels.
	We claim that any distribution \(u\in\Algebra O_r'(\X\times \X)\) which is a left multiplier of \(\Schwartz(\X\times \X)\) defines a continuous map \(\Op(u) \colon \Schwartz(\X)\to\Schwartz(\X)\). 
	Fix \(\varphi\in\SmoothCompactSupp(\X)\) with \(\varphi(0)=1\).
	Then $\Op(u)$ is continuous since it can be written as the composition
	\begin{equation*}
		\Schwartz(\X) \xrightarrow{f \mapsto f \tensor \varphi} 
		\Schwartz(\X \times \X) \xrightarrow{g \mapsto u * g} 
		\Schwartz(\X \times \X) \xrightarrow{g \mapsto (x \mapsto g(x,0))} \Schwartz(\X) . \qedhere
	\end{equation*}
\end{proof}
To show the existence of an adjoint on $\PPseu^m$, let us verify the condition
\refitem{item:adjoint} in \cref{res:equiv_rs_fibred}.

\begin{proposition}\label{res:adjoint}
	Let \(\P\in\PPseu^m\). Then there is a unique \(\P^*\in\PPseu^m\) satisfying \(\IntegrationMap_r(\P^*)=\conj{I_*\IntegrationMap_r(\P)}\),
	where \(I\) denotes the inverse of \(\TG\).
\end{proposition}

\begin{proof}
	The uniqueness of $\P^*$ is immediate from the injectivity of $\IntegrationMap_r$ (see \cref{res:PhiR:injective}) and it remains to show existence.
	Using \cref{res:chain_rule}, the symbol estimates~\eqref{eq:symbol_estimates}, and 
	$\abs{\theta(x,t,w)} \lesssim (\langle x \rangle \langle t \rangle \langle w \rangle_\alpha)^{B}$
	from \cref{res:polynomial_action}
	we obtain
	\begin{align*}
		\abs[\big]{\partial_{(x,t,w,\xi,\eta)}^{(a,b,c,d,e)} \conj{\widehat{\P}(\theta(x,t,w),\xi+\eta)}}
		&= \abs[\Big]{\sum_f \partial^f \widehat{\P}(\theta(x,t,w), \xi+\eta) \cdot \Chain{\theta(x,t,w), \xi+\eta }{(a,b,c,d,e)}{f}{x,t,w,\xi,\eta} }\\
		&\lesssim \langle (\theta(x,t,w), \xi+\eta) \rangle_H^m (\langle x \rangle \langle t \rangle \langle w \rangle_\alpha)^{B(\abs a + \abs b + \abs c)} \\
		&\lesssim \langle x \rangle^{B(\abs m+\abs a)} \langle t \rangle^{B(\abs m+\abs b)} \langle w \rangle^{B(\abs m+\abs c)} \langle \xi \rangle^{\abs m} \langle \eta \rangle^{\abs m}. \tag{$\#$} \label{eq:symbolComposedWithPolynomial}
	\end{align*}
	Choose a cut-off function $\chi \in \Smooth(G)$ which is constant $1$ inside a ball of radius $1$ and
	$0$ outside a ball of radius $2$, and set $\chi_j(w) = \chi(w/j)$. We define
	\[
	\widehat{\P^*}(x,t,\xi)
	=
	(2 \pi)^{-n} \lim_{j \to \infty} \int_{G \times \lie g^*} 
		\E^{\I \langle w,\eta\rangle} \chi_j(w) \conj{\widehat{\P}(\theta_w^t(x),t,\xi+\eta)} \D w \D \eta.
	\]
	Let us show that the limit exists:
	Using \eqref{eq:oscillatoryDecay} and integrating by parts we obtain
	\begin{align*}
	\int_{G \times \lie g^*} 
	\E^{\I \langle w,\eta\rangle} \chi_j(w) \conj{\widehat{\P}(\theta_w^t(x),\xi+\eta)} \D w \D \eta
	&=
	\int_{G \times \lie g^*} 
	\E^{\I \langle w,\eta\rangle} \frac{ \langle D_w \rangle_\alpha^K }{\langle \eta \rangle_\alpha^K} \chi_j(w) \conj{\widehat{\P}(\theta_w^t(x),t,\xi+\eta)} \D w \D \eta
	\\
	&=
	\int_{G \times \lie g^*} 
	\E^{\I \langle w,\eta\rangle}  \frac{ \langle D_\eta \rangle_\alpha^M }{\langle w \rangle_\alpha^M}  \frac{ \langle D_w \rangle_\alpha^K }{\langle \eta \rangle_\alpha^K} \chi_j(w) \conj{\widehat{\P}(\theta_w^t(x),t,\xi+\eta)} \D \eta \D w.
	\end{align*}
	Here, we choose $K \in 2q\NN$, $K > \abs m + Q(\alpha)$, 
	so that \cref{res:integrable} and the above estimate \eqref{eq:symbolComposedWithPolynomial} show
	that integrand on the right of the first line is integrable with respect to $\eta$
	also without integrating with respect to $w$ first.
	We can then swap the order of integration by Fubini's theorem and integrate by parts again.
	Note that
	$\langle D_\eta \rangle_\alpha^M  \langle \eta\rangle^{-K}_\alpha
	=
	\langle \eta\rangle^{-K}_\alpha \sum_{e} c_e(\eta) \partial_\eta^e $
	where all coefficients $c_e$ are bounded in $\eta$.
	So another application of \cref{res:integrable} and the above estimate \eqref{eq:symbolComposedWithPolynomial}
	show that for big enough $M \in 2 q \NN_0$ the integral at the end of the equation above exists also without the oscillatory factor.
	By the dominated convergence theorem, the limit $j \to \infty$ exists and is given by
	\begin{equation*} \label{eq:Pstarhat}
	\widehat{\P^*}(x,t,\xi)
	= (2\pi)^{-n}
	\int_{G \times \lie g^*} 
	\E^{\I \langle w,\eta\rangle}  \frac{ \langle D_\eta \rangle_\alpha^M }{\langle w \rangle_\alpha^M}  \frac{ \langle D_w \rangle_\alpha^K }{\langle \eta \rangle_\alpha^K} \conj{\widehat{\P}(\theta^t_w(x),t,\xi+\eta)} \D \eta \D w. \tag{$\#\#$}
	\end{equation*}
	In particular, $\widehat{\P^*}$ does not depend on the cut-off $\chi_j$ that is chosen in order to define it,
	nor on the choice of $K$ or $M$ in the previous equation (as long as $K$ and $M$ are big enough).
	
	Next, we demonstrate that $\widehat{\P^*}$ is smooth and essentially homogeneous of order $m$.
	Then \(\mathbb P^*\in\PPseu^m\) by \cref{res:equivalence-pseudo-symbol}.
	Choosing $M$ big enough and using the estimate \eqref{eq:symbolComposedWithPolynomial}
	the integral in \eqref{eq:Pstarhat} still exists when $\widehat\P$ is replaced by partial derivatives
	$\partial^{(a,b,d)}_{(x,t,\xi)} \widehat \P$; 
	exchanging the partial differentiations and integration we obtain that $\smash{\widehat{\P^*}}$ is smooth. 
	Moreover, by the essential homogeneity of $\smash{\widehat\P}$, there exists for every $\lambda > 0$ a Schwartz function $f_\lambda$ such that
	$\lambda^m \widehat\P(x,t,\xi) = \widehat\P(\beta_\lambda(x), \lambda t, \alpha_\lambda(\xi)) + f_\lambda(x,t,\xi)$. Hence
	\begin{align*}
		\widehat{\P^*}&(\beta_\lambda(x),\lambda t,\alpha_\lambda(\xi))
		= (2 \pi)^{-n} \lim_{j \to \infty} \int_{G \times \lie g^*} 
		\E^{\I \langle w,\eta\rangle} \chi_j(w) \conj{\widehat{\P}(\theta_w^{\lambda t}(\beta_\lambda(x)),\lambda t,\alpha_\lambda(\xi)+\eta)} \D \eta \D w
		\\
		&= (2 \pi)^{-n} \lim_{j \to \infty} \int_{G \times \lie g^*} 
		\E^{\I \langle \alpha_{\lambda^{-1}}(w),\alpha_\lambda(\eta)\rangle} 
		\chi_j(\alpha_{\lambda^{-1}}(w)) \conj{\widehat{\P}(\theta_{\alpha_{\lambda^{-1}}(w)}^{\lambda t}(\beta_\lambda(x)),\lambda t,\alpha_\lambda(\xi)+\alpha_\lambda(\eta))} \D \eta \D w
		\\
		&= (2 \pi)^{-n} \lim_{j \to \infty} \int_{G \times \lie g^*}
		\E^{\I \langle w,\eta\rangle} \chi_j(\alpha_{\lambda^{-1}}(w)) \big(\lambda^m \conj{\widehat\P(\theta_w^t(x),t, \xi+\eta)}
		- \conj{ f_\lambda(\theta_w^t(x),t, \xi + \eta)}\big)\D \eta\D w 
		\\
		&= \lambda^m \widehat{\P^*}(x, t, \xi) - (2 \pi)^{-n} \lim_{j \to \infty} \int_{G \times \lie g^*}
		\E^{\I \langle w,\eta\rangle} \chi_j(\alpha_{\lambda^{-1}}(w)) \conj{ f_\lambda(\theta_w^t(x),t, \xi + \eta)}\D \eta\D w 
	\end{align*}
	where we have used the compatibility $\theta_{\alpha_{\lambda^{-1}}(w)}^{\lambda t}(\beta_\lambda(x)) = \beta_\lambda(\theta_w^t(x))$ 
	and that $\chi_j \circ \alpha_{\lambda^{-1}}$ is a cut-off, so that it can be replaced by $\chi_j$
	since we have already seen that $\smash{\widehat{\P^*}}$ does not depend on the choice of cut-off.
	The second integral can be modified just as in our derivation of \eqref{eq:Pstarhat} 
	(yielding that the second integral equals \eqref{eq:Pstarhat} with $f_\lambda$ instead of $\widehat \P$).
	That this defines a Schwartz function can be shown exactly as in the proof of \cref{res:leftschwartzmultiplier}.
	
	Finally, we verify that \(\IntegrationMap_r(\P^*) = \conj{I_*\IntegrationMap_r(\P)}\). 
	Write $\widecheck \varphi$ for the inverse Fourier transform of \(\varphi\in\mathcal D(\TG)\) with respect to the last variable
	and recall that it decays rapidly with respect to this variable.
	We compute
	\newcommand{\stepFirstSubstitution}{1}
	\newcommand{\stepRemoveDecayingFactor}{2}
	\newcommand{\stepFourierInversion}{3}
	\newcommand{\stepDominatedConvergence}{4}
	\newcommand{\stepSecondSubstituion}{5}
		\begin{align*}
		\langle \IntegrationMap_r(\P^*),\varphi\rangle 
		&= \int \P^*(\varphi)(x,t) \D x \D t 
		= \int \widehat{\P^*}(x,t,\xi) \widecheck \varphi(x,t,\xi) \D \xi \D x \D t 
		\\
		&= (2 \pi)^{-n}
		\lim_{j \to \infty}
		\int \E^{\I \langle w,\eta\rangle} 
		\frac{ \langle D_{w} \rangle_\alpha^K }{ \langle \eta\rangle^{K}_\alpha } \chi_j(w)
		\conj{\widehat{\P}(\theta_w^t(x),t,\xi+\eta)} \widecheck\varphi(x,t,\xi)
		\D w \D \eta\D \xi \D x \D t 
		\\
		&\overset{\mathclap{(\stepFirstSubstitution)}}{=} (2 \pi)^{-n}
		\lim_{j \to \infty}
		\int \E^{\I \langle w,\eta\rangle} 
		\frac{ \langle D_{w} \rangle_\alpha^K }{ \langle \eta\rangle^{K}_\alpha } \chi_j(w)
		\conj{\widehat{\P}(\theta^t_w(x),t,\xi)} \widecheck\varphi(x,t,\xi-\eta)
		\D w \D \eta \D \xi \D x \D t 
		\\
		&\overset{\mathclap{(\stepRemoveDecayingFactor)}}{=}
		(2\pi)^{-2n}
		\lim_{j \to \infty}
		\int \E^{\I (\langle w,\eta\rangle + \langle v, \xi - \eta \rangle)} 
		\chi_j(w)
		\conj{\widehat{\P}(\theta^t_w(x),t,\xi)} \varphi(x,t,v)
		\D w \D \eta \D v \D \xi \D x \D t 
		\\
		&\overset{\mathclap{(\stepFourierInversion)}}{=} (2 \pi)^{-n}
		\lim_{j \to \infty}
		\int \E^{\I \langle v, \xi \rangle} 
		\chi_j(v)
		\conj{\widehat{\P}(\theta_v^t(x),t,\xi)} \varphi(x,t,v)
	    \D v \D \xi \D x \D t 
		\\
		&\overset{\mathclap{(\stepDominatedConvergence)}}{=} (2 \pi)^{-n}
		\int \E^{\I \langle v, \xi \rangle} 
		\conj{\widehat{\P}(\theta_v^t(x),t,\xi)} \varphi(x,t,v)
		\D v \D \xi \D x \D t 
		\\
		&\overset{\mathclap{(\stepSecondSubstituion)}}{=} (2\pi)^{-n} \int \E^{\I \langle v, \xi \rangle} \conj{\widehat\P(x,t,\xi)}  \varphi (\theta_{-v}^t(x),t,v) \D v \D \xi \D x \D t 
		\\
		&= (2\pi)^{-n} \conj{ \int  \E^{\I \langle v, \xi \rangle} \widehat\P(x,t,\xi) \conj{\varphi (\theta^t_v(x),t,-v)} \D v \D \xi \D x \D t }
		= \conj{ \int \widehat\P( \widecheck{ \overline{\varphi} \circ I })(x,t) \D x \D t }
		= \langle \conj{I_*\IntegrationMap_r(\P)},\varphi\rangle .
	\end{align*}
	In (\stepFirstSubstitution) we have substituted $\xi$ by $\xi-\eta$.
	Note that in order to do so, we first need to change the order of integration to first integrate with respect to $\xi$,
	then $w$, then $\eta$.
	This is justified by Fubini's theorem since the integral still exists when the integrand is replaced by its absolute value.
	Then we can substitute and revert to the original order of integration.
	In (\stepRemoveDecayingFactor) we substituted the definition of $\widecheck\varphi$,
	used Fubini's theorem again to swap the integration over $v$ and $\eta$,
	and removed $ \langle D_{\eta} \rangle_\alpha^M / \langle w\rangle^{M}_\alpha$ by partial integration.
	Step (\stepFourierInversion) is just the Fourier inversion formula, 
	which applies because the integrand is smooth and compactly supported in $w$, 
	and in (\stepDominatedConvergence) we used the dominated convergence theorem (note that $\varphi$ has compact support).
	In step (\stepSecondSubstituion) we substitute $(x,t)$ by $\theta_v(x,t)$, where the Jacobian determinant is $1$ by \cref{res:polynomial_action:properties}.
	Note that we can only do this after changing the order of integration,
	which we can justify by Fubini's theorem once we know that the integrand decays in $\xi$.
	This can be shown as above: Use partial integration to introduce $ \langle D_{v} \rangle_\alpha^L / \langle \xi\rangle^{L}_\alpha $ into the integral, then use the dominated convergence theorem to introduce a cut-off $\chi_j(\xi)$
	and integrate by parts again to get rid of $ \langle D_{v} \rangle_\alpha^L / \langle \xi\rangle^{L}_\alpha $.
\end{proof}

\begin{corollary}\label{res:Schwartz_multiplier}
	Let \(\P\in\PPseu^m\). Then the following holds:
	\begin{enumerate}
		\item\label{item:r,s} \(\P\in\mathcal O'_{r,s}(\TG)\), i.e.\ there exist \(\widetilde{\P}\in\Algebra O'_s(\TG)\) such that \(\IntegrationMap_s(\widetilde{\P})=\IntegrationMap_r(\P)\),		
		\item\label{item:multiplier} \(\P\) is a two-sided multiplier of \(\Schwartz(\TG)\), that is, 
		the maps \(\P * \argument \colon \Schwartz(\TG) \to \Schwartz(\TG)\) and 
		\(\argument * \P \colon \Schwartz(\TG) \to \Schwartz(\TG)\),
		which convolve a Schwartz function with $\P$ from the left and right,
		are both well-defined and continuous,
		\item\label{item:transpose-p} there is a unique \(\P^t\in\PPseu^m\) satisfying \(\IntegrationMap_r(\P^t)=I_*\IntegrationMap_r(\P)\).
		\item\label{item:op-transpose} suppose \(\theta^1\) is polynomially transitive, then \(\Op_t(\P)\) extends to a continuous operator \(\Schwartz'(\X)\to\Schwartz'(\X)\) for every \(t\neq 0\).	
	\end{enumerate}
\end{corollary}
\begin{proof}
		Using that \(\P^*\in\Algebra O'_r(\TG)\) exists by \cref{res:adjoint}, \refitem{item:r,s} and \refitem{item:transpose-p} are immediate from \cref{res:equiv_rs_fibred}.	
		
		To see \refitem{item:multiplier} note that for \(f\in\Schwartz(\TG)\) one has \(f*\mathbb P=(\mathbb P^**f^*)^*\) by \cref{res:involution}. Hence, the claim follows from \cref{res:leftschwartzmultiplier} and the continuity of the involution on \(\Schwartz(\TG)\), see \cref{res:algebra:schwartz}.
		
		For \refitem{item:op-transpose} suppose now that \(\theta^1\) is polynomially transitive, then $\Op_t(\P) \colon \Schwartz(\X) \to \Schwartz(\X)$ is continuous by \cref{res:op_on_schwartz}. Therefore,
		the equation \eqref{eq:OpDual} in \cref{res:OpDual} remains valid for all $f_1, f_2 \in \Schwartz(\X)$,
		meaning precisely that \(\Op_t(\P^t)=\Op_t(\P)^t\). The claim follows by duality as \(\Op_t(\P^t)\colon\Schwartz(\X)\to\Schwartz(\X)\) is continuous by \cref{res:op_on_schwartz}. \qedhere
	
\end{proof}
Now, we can show that our operators form indeed a calculus in the sense
that their compositions and formal adjoints also belong to it. Note that by \cref{res:Schwartz_multiplier}, \(u\in\Ess^m\) is a two-sided multiplier of \(\Schwartz(\X\rtimes^0 G)\), so that the convolution \([u]*[v]=[u*v]\) of \([u]\in\Symb^\ell\) and \([v]\in\Symb^m\) is well-defined. Furthermore, \([u]^*=[u^*]\) and \([u]^t=[u^t]\) are well-defined.
\begin{theorem}
	The calculus has the following properties:
	\begin{enumerate}
		\item For \(\P_1\in\Pseu^\ell\) and \(\Q_1\in\Pseu^m\) their composition satisfies \(\P_1*\Q_1\in \Pseu^{\ell+m}\) and \(\sigma_{\ell+m}(\P_1*\Q_1)=\sigma_{\ell}(\P_1)*\sigma_{m}(\Q_1)\),
		\item For \(\P_1\in\Pseu^m\) also \(\P_1^*,\P_1^t\in \Pseu^{m}\) with \(\sigma_{m}(\P_1^*)=\sigma_{m}(\P_1)^*\) and \(\sigma_{m}(\P_1^t)=\sigma_{m}(\P_1)^t\).		
		\end{enumerate}
\end{theorem}
\begin{proof}
	\begin{enumerate}
		\item Let \(\P\in\PPseu^\ell\) and \(\Q\in\PPseu^m\) extend \(\P_1\) and \(\Q_1\). Then \(\P*\Q\in\Algebra{O}'_r(\TG)\) and moreover, since both are essentially homogeneous
		\begin{align*}\lambda^{\ell+m}\mathbb (\P*\Q)-{\tau_\lambda}_*(\P*\Q)
			&=(\lambda^\ell\P-{\tau_\lambda}_* \P)*\lambda^m\Q+{\tau_\lambda}_*\P*(\lambda^m\Q-{\tau_\lambda}_*\Q)\\
			&\subset \Schwartz(\TG)*\lambda^m\Q+{\tau_\lambda}_*\P*\Schwartz(\TG),\end{align*}
		where we used \cref{res:zoomAutomorphisms}.
		As \({\tau_\lambda}_*\P\) and \(\lambda^m \Q\) are Schwartz multipliers by \cref{res:Schwartz_multiplier},
		we obtain \(\P*\Q\in\PPseu^{\ell+m}\), so that \(\P_1*\Q_1\in\Pseu^{\ell+m}\). Moreover \([\P_0*\Q_0]=[\P_0]*[\Q_0]\) shows the claim on the principal cosymbol. 
		\item This follows directly from \cref{res:adjoint} and \([\P_0]^*=[\P_0^*]\), respectively, \([\P_0]^t=[\P_0^t]\).\qedhere
	\end{enumerate}		
\end{proof}

\subsection{Asymptotic completeness and ellipticity}
First, we determine the residual class of the calculus. We define $\Pseu^{-\infty} = \Schwartz(X \rtimes^1 G)$, which is justified by the following result.
\begin{lemma}
	Let $(m_j)_{j \in \NN_0}$ be a sequence of real numbers such that $m_j \to -\infty$ as $j \to \infty$.
	A fibred distribution \(\P_1 \in \Algebra O_r'(\X \rtimes^1 G)\) belongs to \(\bigcap_{j\in\NN_0}\Pseu^{m_j}\) if and only if \(\P_1\in\Schwartz(\X\rtimes^1G)\). When \(\theta^1\) is polynomially transitive, this implies that \(\Op(\P_1)\) extends to a continuous map \(\Op(\P_1)\colon\Schwartz'(\X)\to\Schwartz(\X)\).
\end{lemma}
\begin{proof}
If \(\P_1\in\Schwartz(\X\rtimes^{1} G)\), it is clear that it can be extended to \(\P\in\Schwartz(\TG)\subset\PPseu^m\) for all \(m\in\RR\).
Conversely, suppose that \(\P_1\in \bigcap_{j\in\NN_0}\Pseu^{m_j}\). By \cref{res:symbol_estimates} its full symbol \(\widehat{\P}_1\) is a Schwartz function. Hence, \(\P_1\) belongs to \(\Schwartz(\X\rtimes^1 G)\). When \(\theta^1\) is polynomially transitive, \(\Op\) maps \(\Schwartz(\X\rtimes^1 G)\) to \(\mathcal L(\Schwartz'(\X),\Schwartz(\X))\) by \cref{res:Op-injective} \refitem{item:Op-transitive}. 
\end{proof}
We show that the calculus is asymptotically complete.
\begin{proposition}\label{res:asymptotic_completeness} 
	Let $m \in \RR$ and \((\P^k_{1})_{k\in\NN_0}\) be a sequence such that \(\P^k_{1}\in\Pseu^{m-k}\).
	Then there is a \(\P_1\in\Pseu^{m}\) such that for all \(j\in\NN_0\)
	\begin{equation*}
		\P_1 -\sum_{k=0}^{j-1}\P^k_{1}\in\Pseu^{m-j}.
	\end{equation*}
\end{proposition}
If $\P_1 \in \Pseu^{m}$ has these properties, we write \(\P_1 \sim \sum_{k=0}^{\infty} \P^k_{1}\).
It is not surprising that the estimates in the following proof are very similar 
to the usual estimates when showing asymptotic completeness of certain classes of pseudodifferential operators,
see e.g.\ \cite{Kum81}*{Lemma~3.2} or \cite{Shu87}*{Proposition~3.5}, as they are just the adaption to the tangent groupoid formalism.

\begin{proof}
	Throughout the proof, fix a smooth cut-off function $\chi \colon \RR \to [0,\infty)$ such that 
	$\chi(t) = 1$ for $\abs t \geq 2$ and $\chi(t) = 0$ for $\abs t \leq 1$.
	
	Extend $\P^k_{1} \in \Pseu^{m-k}$ to $\P^k \in \PPseu^{m-k}$. 
	By \cref{res:equivalence-pseudo-symbol} \refitem{item:widehat2} there exist $(m-k)$-homogeneous functions $p_k \in \Smooth(\X \times \RR \times \lie g^* \setminus \simpleset{(0,0,0)})$
	and Schwartz functions $f_k \in \Schwartz(\X \times \RR \times \lie g^*)$ 
	satisfying that $\widehat{\P^k} = \chi(\norm{(x,t,\xi)}) p_k(x,t,\xi) + f_k(x,t,\xi)$.
	We show in \textbf{(1)} that there are $\varepsilon_0, \varepsilon_1, \ldots \in (0,1]$ 
	such that  
	\begin{equation*}
		p(x,t,\xi) = \sum_{k=0}^\infty \underbrace{t^{k} p_k(x,t,\xi) \chi(\varepsilon_k t^{-1} \norm{(x,\xi)}_{\beta,\alpha})}_{\eqqcolon q_k(x,t,\xi)}
	\end{equation*}
	is a well-defined $m$-homogeneous function $p \in \Smooth(\X \times \RR \times \lie g^* \setminus \simpleset{(0,0,0)})$.
	By \cref{res:equivalence-pseudo-symbol} there exists $\P \in \PPseu^{m}$
	with $\smash{\widehat \P = \chi(\norm{(x,t,\xi)}) p(x,t,\xi)}$,
	and we prove in \textbf{(2)} that the corresponding operator $\P_1 \in \Pseu^{m}$ has the required properties.

	\textbf{(1)} If $\varepsilon_k \to 0$ for $k \to \infty$
	and $(x,t,\xi) \in \X \times \RR \times \lie g^*$ is fixed, 
	then the expression $t^{k} \chi(\varepsilon_k t^{-1} \norm{(x,\xi)}_{\beta,\alpha})$ becomes $0$ for $k$ sufficiently large.
	In this case, the sum is finite at every point $(x,t,\xi)$ and $p$ is well-defined.
	Since each of the summands $q_k$ is $m$-homogeneous, see the proof of \cref{res:short_exact_sequence}, so is $p = \sum_{k=0}^\infty q_k$.
	It remains to verify that $p \in \Smooth(\X \times \RR \times \lie g^* \setminus \simpleset{(0,0,0)})$.
	
	For every point $(x_0,t_0,\xi_0)$ with $t_0 \neq 0$
	choose a small enough neighbourhood of $(x_0, t_0, \xi_0)$, 
	on which $t^{-1} \norm{(x,\xi)}_{\beta,\alpha}$ remains bounded.
	On this neighbourhood, only finitely many smooth summands contribute to $p$,
	so $p$ is indeed smooth there.
	
	However, on any open neighbourhood of $(x_0, 0, \xi_0) \neq (0,0,0)$ infinitely many summands may contribute to~$p$.
	We claim that if the $\varepsilon_k$ are chosen small enough, 
	then partial derivatives of arbitrary order exist nevertheless and are of the form
	$\partial_{(x,t,\xi)}^{(a,b,c)} p(x,t,\xi) = \sum_{k=0}^\infty q^{(a,b,c)}_k(x,t,\xi)$ with
	\begin{multline*}
		q_k^{(a,b,c)}(x,t,\xi) =
		\sum_{a' \leq a, b'+b'' \leq b, c' \leq c} C_{a',b',b'',c'} t^{k - b''} \partial_{(x,t,\xi)}^{(a',b',c')} p_k(x,t,\xi) \\
		\sum_{r=0}^{\abs{a} + b + \abs{c}} (\partial^r\chi)(\varepsilon_k t^{-1} \norm{(x,\xi)}_{\beta,\alpha}) \Chain {\varepsilon_k t^{-1} \norm{(x,\xi)}_{\beta,\alpha}}{(a-a',b-b'-b'',c-c')}{r}{x,t,\xi}
	\end{multline*}
	where $\mathrm{Ch}$ was introduced in \cref{res:chain_rule}.
	Indeed this follows easily if differentiation under the sum is allowed, 
	for which it suffices to check that $\sum_{k=0}^\infty q_k^{(a,b,c)}(x,t,\xi)$ 
	converges uniformly in a small enough neighbourhood of $(x_0,0,\xi_0)$.
	In fact, it suffices to show that this sum converges uniformly on 
	$K_R \coloneqq \set{(x,t,\xi)}{\norm{(x,\xi)} \in [R^{-1}, R], \abs t \leq 1}$ for every $R \in \NN$.
To this end, we verify that for $k \in \NN$ we have 
\begin{equation*}
	\max_{\substack{(a,b,c) \in \NN_0^{d+1+n}\\(\abs a + b + \abs c+1)^2 \leq k-1}}\sup_{(x,t,\xi) \in K_k} \abs{q_k^{(a,b,c)}(x,t,\xi)} \lesssim \varepsilon_k
\end{equation*}
	where the omitted multiplicative constant depends only on $p_k$ and its derivatives.
	Choosing $\varepsilon_k$ so small that the left hand side is bounded by $2^{-k}$ the claimed uniform convergence follows.
	
	Fix $k \in \NN$. The following estimates hold for $(x,t,\xi) \in K_k$. Since $K_k$ is compact,
	continuous functions like $\norm{(x,\xi)}_{\beta,\alpha}$ or its derivatives are bounded on $K_k$,
	hence can be absorbed into the omitted constants.
The cut-off $(\partial^r\chi)(\varepsilon_k t^{-1} \norm{(x,\xi)}_{\beta,\alpha})$ is bounded
and non-zero only if $\abs t \leq \varepsilon_k \norm{(x,\xi)}_{\beta,\alpha} \lesssim \varepsilon_k$.
Let $L = \abs a + b + \abs c$.
Using \cref{res:chain_rule} we estimate
$\abs{\Chain{\varepsilon_k t^{-1} \norm{(x,\xi)}_{\beta,\alpha}}{(a-a',b-b'-b'',c-c')}{r}{x,t,\xi}} \lesssim \abs t^{-L(L+1)}$
since partial derivatives of $\varepsilon_k t^{-1} \norm{(x,\xi)}_{\beta,\alpha}$ of order $\ell$ can be estimated by $\abs t^{-\ell-1}$.
Absorbing $p_k$ and its derivatives into the constant, we obtain $\abs{q_k^{(a,b,c)}(x,t,\xi)} \lesssim \abs t^{k - L} \abs t^{- L(L+1)} \abs{\partial \chi({\dots})} \lesssim \varepsilon_k^{k - (L+1)^2}$.
	
	\textbf{(2)} We check that $\P_1 - \sum_{k=0}^{j-1}\P^k_{1}\in\Pseu^{m-j}$,
	i.e.\ that $\P - \sum_{k=0}^{j-1} t^{k}\P^k \in t^{j}\PPseu^{m-j} + \Schwartz(\TG)$.
	Note that 
	\begin{align*}
		\widehat \P - \sum_{k=0}^{j-1} t^{k}\widehat {\P^k} 
		&= \chi(\norm{(x,t,\xi)})\Bigl(p(x,t,\xi) - \sum_{k=0}^{j-1} t^{k} p_k(x,t,\xi)\Bigr) - \sum_{k=0}^{j-1} t^{k} f_k(x,t,\xi) \\
		&= \chi(\norm{(x,t,\xi)}) \sum_{k=0}^{j-1} t^{k} p_k(x,t,\xi) (\chi-1)(\varepsilon_k t^{-1} \norm{(x,\xi)}_{\beta,\alpha}) \\
		&\phantom{XXX} + t^{j} \chi(\norm{(x,t,\xi)}) \sum_{k=j}^\infty  t^{k-j} p_k(x,t,\xi) \chi(\varepsilon_k t^{-1} \norm{(x,\xi)}_{\beta,\alpha}) - \sum_{k=0}^{j-1} t^{k} f_k(x,t,\xi) . \tag{$*$} \label{eq:errorDecomposition}
	\end{align*}
	Note that the support of the function $h(x,t,\xi) = (\chi-1)(\frac{\varepsilon_k}{t} \norm{(x,\xi)}_{\beta,\alpha}) \chi(\norm{(x,t,\xi)})$
		is contained in
		$\set{(x,t,\xi)}{\abs{t} \geq \frac {\varepsilon_k} 2 \norm{(x,\xi)}_{\beta,\alpha}, \norm{(x,t,\xi)} \geq 1}$,
		which is contained in $\set{(x,t,\xi)}{\abs t \geq  c}$ for a small enough $c > 0$.
		Therefore $t^{-j} h(x,t,\xi)$ is still smooth. Moreover, it is essentially homogeneous of order $-j$.
		Using \cref{res:equivalence-pseudo-symbol}, the inverse Fourier transform of the first summand in 
		 \eqref{eq:errorDecomposition} is indeed in $t^{j} \PPseu^{m-j}$. It is clear that this also holds for the second summand
		 and that the third summand is Schwartz.
\end{proof}
In analogy to the classical theory, elliptic shall mean that the principal (co)symbol is invertible.

\begin{definition}
	Let \(\P_1\in\Pseu^m\) be a pseudodifferential fibred distribution with extension \(\P\in\PPseu^m\).
	Then $\P_1$ is called \emph{elliptic} if $\P_0$ has an inverse $\Q_0 \in \Algebra O_r'(\X \rtimes^0 G)$ up to Schwartz functions,
	meaning that $\P_0 * \Q_0 - \mathbb{I}_0 \in \Schwartz(\X \rtimes^0 G)$
	and $\Q_0 * \P_0 - \mathbb{I}_0 \in \Schwartz(\X \rtimes^0 G)$.
\end{definition}
Here, \(\mathbb I\) was defined in \cref{ex:identity_operator}.
To see that ellipticity of $\P_1$ does not depend on the chosen extension $\P$, we need the following lemma,
which can be proven in the same way as \cite{vEY19}*{Lemma~55}.
\begin{lemma}
	Let \(\P_0 \in \Ess^m\). If there exists $\Q_0 \in \Algebra O_r'(\X \rtimes^0 G)$ such that $\P_0 * \Q_0 - \mathbb{I}_0 \in \Schwartz(\X \rtimes^0 G)$
	and $\Q_0 * \P_0 - \mathbb{I}_0 \in \Schwartz(\X \rtimes^0 G)$ then $\Q_0 \in \Ess^{-m}$.
\end{lemma}
Hence if $\P,\tilde \P \in \PPseu^m$ are two extensions of $\P_1 \in \Pseu^m$, then $\P_0 - \tilde \P_0 \in \Schwartz(\X \rtimes^0 G)$ by \cref{res:smoothingat1implies0}.
If $\Q_0$ is an inverse of $\P_0$ up to Schwartz functions, then it is essentially homogeneous by the previous lemma, 
hence a multiplier of $\Schwartz(\X \rtimes^0 G)$ by \cref{res:Schwartz_multiplier}.
Therefore $\tilde \P_0 * \Q_0 - \mathbb I_0 = \P_0 * \Q_0 - \mathbb I_0 - (\P_0 - \tilde \P_0) * \Q_0 \in \Schwartz(\X \rtimes^0 G)$
and a similar computation with the order of $\tilde \P_0$ and $\Q_0$ reversed
show that also $\tilde \P_0$ is invertible up to Schwartz functions.
Quotienting out Schwartz functions, we obtain that for $\P_1$ elliptic $\sigma_m(\P_1)$ has an inverse in $\Sigma_\Gamma^{-m}$.

Using the properties of the calculus, like the short exact sequence induced by the principal cosymbol and the asymptotic completeness, one can construct a parametrix using the standard argument (see \cite{vEY19}*{Theorem~60}).
\begin{theorem}\label{res:parametrix}
	Let \(\P_1\in\Pseu^m\). Then \(\P_1\) is elliptic if and only if it has a parametrix \(\Q_1\in\Pseu^{-m}\), that is \(\P_1*\Q_1-\mathbb I_1\) and \(\Q_1*\P_1-\mathbb I_1\) belong to \(\Pseu^{-\infty}\).
	
	When \(\theta^1\) is polynomially transitive, this implies that \(\Op(\P_1)\Op(\Q_1)-\id\) and  \(\Op(\Q_1)\Op(\P_1)-\id\) are continuous maps \(\Schwartz'(\X)\to\Schwartz(\X)\).
\end{theorem}
In particular, \(\Op (\P_1)\) is hypoelliptic in the following sense when \(\theta^1\) is polynomially transitive.
\begin{corollary}\label{res:hypoelliptic}
	Suppose \(\theta^1\) is polynomially transitive and let \(\P_1 \in \Pseu^m\) be elliptic. Then \(\Op(\P_1)u\in\Schwartz(\X)\) for \(u\in\Schwartz'(\X)\) implies that \(u\in\Schwartz(\X)\). 
\end{corollary}
\begin{proof}
	Let \(\Q_1\) be a parametrix for \(\P_1\) and \(f=\Q_1*\P_1-\mathbb I_1\in\Schwartz(\X\rtimes^1 G)\). Then \(u=\Op(\Q_1)\Op(\P_1)u-\Op(f)u\) is contained in \(\Schwartz(\X)\) by \cref{res:op_on_schwartz} and \cref{res:Op-injective}.
\end{proof}
\section{Shubin-type calculus under properties (P) and (R)}\label{sec:PR}

In the following, we briefly rephrase ellipticity in terms of a Rockland condition when property \ref{assumption:homomorphisms} is satisfied. We show that elements of \(\Pseu^m\) for \(m<0\) belong to the groupoid \(C^*\)-algebra \(C^*(\X\rtimes^1G)\). Moreover, under assumption \ref{assumption:homomorphisms}, elements of \(\Pseu^0\) belong to the multiplier algebra \(M(C^*(\X\rtimes^1 G))\).
Afterwards, we restrict to the case where properties \ref{assumption:homomorphisms} and \ref{assumption:free} hold. We study further mapping properties of the pseudodifferential operators, in particular a scale of Sobolev spaces is introduced. Moreover, spectral properties are shown.

\subsection{Property (R) and the Rockland condition}
In the following we assume that property \ref{assumption:homomorphisms} is satisfied and characterize ellipticity in terms of a Rockland condition as described in \cref{sec:homogeneous_distributions}.
Recall that under assumption \ref{assumption:homomorphisms} \(\X^*\rtimes^0 G\) is a graded group with dilations \(\beta_\lambda\times\alpha_\lambda\) as seen in \cref{res:r-implies-group-at-0}. 

Recall from  \cref{def:ess_graded_group} that \(\Sigma^m(\X^*\rtimes^0 G)\) denotes the quotient space of all essentially \(m\)-homogeneous distributions with respect to \(\beta_\lambda\times\alpha_\lambda\) by Schwartz functions.
By the same arguments as in the proof of \cref{res:equivalence-pseudo-symbol}, the partial inverse Fourier transform in the \(x\)-direction \(u\mapsto\widecheck{u}\) induces a bijection
\begin{equation*}
	\widecheck{\,\cdot\,}\colon\Symb^m\to \Sigma^m(\X^*\rtimes^0 G).
\end{equation*}
Note that \(\widecheck{u*v}=\widecheck u*\widecheck v\) for \(u\in\Symb^\ell\) and \(v\in\Symb^m\), 
whereas \(\langle \widecheck{u^*}, f \rangle = \langle \widecheck{u}^*, f_-\rangle\) where $f_-(\eta,v) = f(-\eta,v)$.	
Consequently, \(u\in\Symb^m\) has an inverse in $\Symb^{-m}$ if and only if \(\widecheck u\in \Sigma^m(\X^*\rtimes^0 G)\) has an inverse in $\Sigma^{-m}(\X^* \rtimes^0 G)$. The latter can be characterized by the Rockland Theorem of \cite{CGGP92}*{Theorem~6.2}, see \cref{res:rockland:ess-hom} for a reformulation in terms of essentially homogeneous distributions. 
\begin{proposition}
	Suppose property \ref{assumption:homomorphisms} is satisfied and let \(\P_1\in\Pseu^m\). Then \(\P_1\) is elliptic if and only if \(\widecheck\sigma_m(\P_1)\) and \(\widecheck\sigma_m(\P_1)^*\) satisfy the Rockland condition on the graded group \(\X^*\rtimes^0 G\).
\end{proposition}
Consider a pseudodifferential fibred distribution giving rise to a differential operator with polynomial coefficients as in \cref{ex:diff_ops_pol_coeff} \begin{align*}
	\P_1&=\sum_{[a]_\alpha+[b]_\beta\leq m}c_{a,b}x^b X_v^a \delta_{v=0} &\text{i.e. }&\P_1(\varphi)(x)=\sum_{[a]_\alpha+[b]_\beta\leq m}c_{a,b}x^b X^a \varphi(x,0).
\intertext{As differential operator $\P_1$ has a unique $m$-homogeneous extension $\P \in \PPseu^m$, see \cref{ex:diff_ops_pol_coeff}.
	Then} 
	\widecheck \P_0&=	\sum_{[a]_\alpha+[b]_\beta= m}c_{a,b}(-\I \partial_\eta)^b X_v^a \delta_{(\eta,v) = (0,0)}
	&\text{i.e. }&\langle \widecheck \P_0, \varphi\rangle = \sum_{[a]_\alpha+[b]_\beta= m}c_{a,b}(-\I \partial_\eta)^b X^a \varphi(0,0)
\end{align*}
is a representative of the principal cosymbol $\widecheck\sigma_m(\P_1)$.
The principal cocosymbol from \eqref{eq:principalcoco} coincides with this expression under the inclusion \(\mathfrak U^m(\X^*\rtimes^0G)\hookrightarrow \ess^m(\X^*\rtimes^0G)\), which maps $x^b$ to $\partial_\eta^b$.

In particular, \(\P_1\) is elliptic if and only if the left invariant differential operators corresponding to \(\widecheck\P_0\) and \((\widecheck\P_0)^*\)  satisfy the Rockland condition on the graded group \(\X^*\rtimes^0 G\) in the sense of \cref{def:rockland-cond-diff}.

\begin{example}
	Let $G$ be a graded group with dilations $\alpha$, $\X$ be a graded vector space with dilations $\beta$ and $\theta^1$ a Shubin action of \(G\) on \(\X\) with property \ref{assumption:homomorphisms}. Denote the weights of \(\alpha\) and \(\beta\) by \(q_1,\ldots,q_n\) and \(r_1,\ldots,r_d\) respectively.
	Let \(q\) be a common multiple of all weights and
	\begin{equation*}
		\P_1 = \sum_{j=1}^n(-1)^{\frac{q}{q_j}}X_j^{\frac{2q}{q_j}} \delta_{v=0}+ \sum_{j=1}^d x_j^{\frac{2q}{r_j}} \delta_{v=0}\in \Pseu^{2q}.
	\end{equation*}
	The corresponding operator is $\Op(\P_1) = \sum_{j=1}^n(-1)^{q/q_j} \widehat X_j^{2q/q_j} + \norm{x}_\beta^{2q}$. Then \(\P_1\) is elliptic as
	\[\widecheck{\P}_0=(\widecheck{\P}_0)^*=\sum_{j=1}^n(-1)^{\frac{q}{q_j}}X_j^{\frac{2q}{q_j}} \delta_{(\eta,v) = (0,0)} +\sum_{j=1}^{d}(-1)^{\frac{q}{r_j}}\partial_{\eta_j}^{\frac{2q}{r_j}}\delta_{(\eta,v) = (0,0)}\]
	defines a Rockland operator on the group \(\X^*\rtimes^0 G\) with dilations \(\beta_\lambda\times\alpha_\lambda\) by \cref{ex:standard-rockland}.
\end{example}
When property \ref{assumption:free} holds, principal cosymbols and ellipticity can also be defined for pseudodifferential operators $\Op(\P_1)$, see \cref{remark:P:principalSymbolOfOperators}.
For the double dilation groupoid one can check ellipticity of differential operators with polynomial coefficients as follows. 
\begin{proposition}
	Consider the double dilation groupoid, \cref{def:doubleDilationGroupoid}, of a Lie group $G$ with dilations $\alpha$ and $\beta$.
	Let \(P=\smash\sum_{[a]_\alpha+[b]_\beta\leq m}c_{a,b}x^b X^a \in\Op(\Pseu^m)\). Then \(P\) is elliptic if and only if
	\begin{equation*}
		\sum_{[a]_\alpha+[b]_\beta= m}c_{a,b}x_0^b \pi(X)^a\quad\text{and}\quad \sum_{[a]_\alpha+[b]_\beta= m}\conj{c_{a,b}}x_0^b (-\pi(X))^a
	\end{equation*}
are injective on \(\mathcal H^\infty_\pi\) for all \((x_0,\pi)\in \Space{G}\times\widehat G\setminus\{(0,\pi_\mathrm{triv})\}\) . 
\end{proposition}
\begin{proof}
	This follows from the fact that for the double dilation groupoid \(\X^*\rtimes^0 G=\X^*\times G\), so that \(\widehat{\X^*\rtimes^0 G}\cong \X\times\widehat G=\Space G\times\widehat{G}\). Hence, it suffices to consider for \(x_0\in\Space G\) and \(\pi\) a unitary, irreducible representation of \(G\) the representation of \(\X^*\times G\) given by \(\rho_{(x_0,\pi)}(\eta,v)=\E^{\I\langle \eta ,x_0\rangle}\pi(v)\in\mathcal U(\mathcal H_\pi)\). For these one has
	\[\rho_{(x_0,\pi)}(\widecheck \P_0)=\sum_{[a]_\alpha+[b]_\beta= m}c_{a,b}x_0^b \pi(X)^a\]
	and similarly for \((\widecheck\P_0)^*\). Note that the trivial representation of \(\X^*\times G\) is \(\rho_{(0,\pi_\mathrm{triv})}\). 
\end{proof}
For the representation groupoid, the relevant group is \(\X^*\rtimes^0 G=\lie{g}^*\rtimes G\), where \(G\) acts on \(\lie{g}^*\) by the coadjoint representation.
For the Heisenberg \(H_1\), the irreducible representations of \(\lie{h}_1^*\rtimes H_1\) are computed in \cite{Nie83}*{group \(G_{6,15}\)}.
\begin{example}
	For the Shubin representation calculus of the Heisenberg group \(H_n\), consider the operator
	\begin{equation*}P=-\sum_{j=1}^{2n}\widehat X_j^2+x_{2n+1}^2 \in \Op(\Pseu^2) \end{equation*}
	where the fundamental vector fields were computed in \cref{ex:fundamental-vf-RG}.
	This is the harmonic oscillator on the Heisenberg group defined in \cite{RR20}.
	Its principal cocosymbol is \(\widecheck \P_0 = (\widecheck \P_0)^* = (-\sum_{j=1}^{2n}X_j^2 - \partial_{\eta_{2n+1}}^2) \delta_{(\eta,v) = (0,0)}\). 
	Recall from \cref{ex:representationLA:heisenbergGroup} that \(\lie{h}^*\rtimes H_n\) is a stratified group with Lie algebra generated by \(X_j\) for \(j=1,\ldots,2n\) and \(\partial_{\eta_{2n+1}}\).
	As the corresponding Sublaplacian is a Rockland operator by \cref{ex:stratified-rockland}, we see that \(P\) is elliptic.  
\end{example}\textit{}
\begin{example}
	For the action of \(H_1\) on \(\X=\RR^2\) from \cref{ex:grushin-action}, we have seen in \cref{ex:grushin-fundamental-vectorfields} that certain Rockland operators on \(H_1\) get mapped by \(\Op\) to the Grushin, respectively, Kolmogorov operator on \(\RR^2\). We study now which polynomial potentials can be added to obtain elliptic pseudodifferential distributions in the corresponding Shubin calculus. As the action is polynomially transitive, \cref{res:parametrix} and \cref{res:hypoelliptic} yield then that the corresponding operators have parametrices and are hypoelliptic. 
	
	For the Grushin operator, set \(k=l=p=q=1\). Then by \cref{ex:symbol-group-grushin}, \(\X^*\rtimes^0 \lie{h}_1=\X^*\times\lie{h}_1\), so we get that \(\P_1=X_1^2+4X_2^2-x^2-y^2\in\Pseu^2\) is elliptic. The corresponding operator is \(\Op(\P_1)=\partial_x^2+x^2\partial_y^2-x^2-y^2\).
	
	It is also possible to set \(k=l=q=1\) and \(p=2\), in which case \(\theta^0\) is non-trivial. Now, \(\P_1=X_1^2+4X_2^2-y^2\in\Pseu^2\) is elliptic as its principal cocosymbol is a Sublaplacian on the stratified group \(\X^*\rtimes^0 H_1\). Recall here from \cref{ex:symbol-group-grushin} that its Lie algebra is generated by $X_1, X_2, y$. In this case, \(\Op(\P_1)=\partial_x^2+x^2\partial_y^2-y^2\) holds. 
	
	For the Kolmogorov operator, set \(k=1\) and \(l=2\) so that \(-X^2+2Y\) is homogeneous of order \(2\). Again one can choose \(p=q=1\) and obtain that \(\P_1=-X^2+2Y+x^2+y^2\in\Pseu^2\) is elliptic with \(\Op(\P_1)=-\partial_x^2+x\partial_y+x^2+y^2\). Another possibility is \(p=2\) and \(q=1\), in which case \(\P_1=-X^2+2Y+y^2\in\Pseu^2\) is elliptic with \(\Op(\P_1)=-\partial_x^2+x\partial_y+y^2\).
\end{example}
\begin{lemma}\label{res:existence-elliptic}
	Suppose property \ref{assumption:homomorphisms} holds. Then for every \(m\in\RR\) there is an elliptic \(\P_{m,1}\in\Pseu^m\).
\end{lemma}
\begin{proof}
	By \cite{CGGP92}*{Theorem~6.1} and the isomorphism from \cref{res:equivalence_elliptic_cggp} there is a family \((u_m)_{m\in\RR}\) of \(u_m \in \Sigma^m(\X^*\rtimes^0 G)\) satisfying \(u_{-m}*u_m=u_m*u_{-m}=[\delta_{(\eta,v) = (0,0)}]\) for all \(m\in\RR\). 
	By \cref{res:short_exact_sequence}, we can choose \(\P_{m,1}\in\Pseu^m\) with \(\widecheck\sigma_m(\P_{m,1})=u_m\). 
\end{proof}
\begin{remark}
	Our results can also be adapted to define Shubin type calculi of operators acting on vector bundles over \(\X\), or equivalently matrices of operators as all vector bundles over \(\X\cong\RR^d\) are trivial. In particular, under assumption \ref{assumption:homomorphisms}, one can use the matrix Rockland condition described in \cref{rem:rockland-matrix}. 
\end{remark}
\subsection{Pseudodifferential distributions of negative order and order zero}
The groupoid \(C^*\)-algebra \(C^*(\X\rtimes^1 G)\) (see \cite{Ren80} for the definition) is a \(C^*\)-completion of the convolution algebra \(\Schwartz(\X\rtimes^1 G)=\Pseu^{-\infty}\). In this section, we show that also \(\P_1\in\Pseu^m\) for \(m<0\) is an element of this groupoid \(C^*\)-algebra. Moreover, recall that \(\P_1\in\Pseu^m\)  defines a two-sided multiplier of \(\Schwartz(\X\rtimes^1 G)\) by \cref{res:adjoint}.  We shall show that \(\P_1\in\Pseu^0\) extends to a two-sided multiplier of \(C^*(\X\rtimes^1 G)\). Note that \(G\) is amenable as a nilpotent Lie group and consequently, the full and reduced \(C^*\)-algebra of the transformation groupoid \(\X\rtimes^1 G\) coincide. 
\begin{lemma}\label{res:small-order-groupoid-c-*}
	Suppose \(\P_1\in\Pseu^m\) for \(m<-Q(\alpha)\). Then \(\P_1\in C^*(\X\rtimes^1 G)\). 
\end{lemma}
\begin{proof}As \(\X\rtimes^1 G\) is a transformation groupoid it suffices to show that \(\P_1\) lies in the completion of \(\Cont_c(G,\Cont_0(\X))\) with respect to the norm
	\begin{align*}
		\norm{f}_1=\int_G \sup_{x\in\X}{\abs{f(v)(x)}}\,\D v.
	\end{align*}
	In the following it will be useful to observe that by \cref{res:integrable} there is \(D>0\) such that for all \(x\in \X\)
	\begin{multline}\label{eq:estimate-shubin-norm}
		\int_{\lie{g}^*}(1+\norm{x}_\beta+\norm{\xi}_\alpha)^m\D\xi	 = (1+\norm{x}_\beta)^m\int_{\lie{g}^*}\left(1+\frac{\norm{\xi}_\alpha}{1+\norm{x}_\beta}\right)^m\D\xi\\ =(1+\norm{x}_\beta)^m\int_{\lie{g}^*}\left(1+\norm{\alpha_{(1+\norm{x}_\beta)^{-1}}(\xi)}_\alpha\right)^m\D\xi
		 = D(1+\norm{x}_\beta)^{m+Q(\alpha)}.
	\end{multline}
	Consequently, using the symbol estimates from \cref{res:symbol_estimates} there is \(C_1>0\) such that
	\begin{equation}\label{eq:-q(alpha)-symbol}
		\int_{\lie{g}^*} \abs{\widehat \P_1(x,\xi)}\D \xi\leq C_1\int_{\lie{g}^*}(1+\norm{x}_\beta+\norm{\xi}_\alpha)^{m}\D\xi \\
		= C_1D(1+\norm{x}_\beta)^{m+Q(\alpha)}.
	\end{equation}
	We show first that \(\P_1\in \Cont(\X\times G)\). Let \((x,v)\in \X\times G\) and \(\varepsilon>0\). By the estimates in \eqref{eq:-q(alpha)-symbol} \(\widehat \P_{1}(x)\) is in \(L^1(\lie{g}^*)\) so that \(\P_{1}(x)\) belongs to \(\Cont_0(G)\). Let \(U\) be a neighbourhood of \(v\) in \(G\) such that \(\abs{\P_1(x,v)-\P_1(x,w)}<\frac\varepsilon 2\) for all \(w\in U\). Using once more the symbol estimates and the mean value theorem, there is a \(C_2>0\) such that for all \(y\in \X\) and \(\xi\in\lie g^*\)
	\begin{equation*}
		\abs{\widehat \P_1(x,\xi)-\widehat \P_1(y,\xi)}\leq C_2\norm{x-y}_\beta(1+\norm{\xi}_\alpha)^m
	\end{equation*}
	so that \(\norm{\P_1(x)-\P_1(y)}_{\Cont_0(G)}\leq C_2D\norm{x-y}_\beta\). Therefore, for all \((y,w)\in\X\times G\) such that \(\norm{x-y}_\beta<\frac{\varepsilon}{2C_2D}\) and \(w\in U\) one has
	\begin{equation*}
		\abs{\P_1(x,v)-\P_1(y,w)}\leq \abs{\P_1(x,v)-\P_1(x,w)}+\abs{\P_1(x,w)-\P_1(y,w)}<\varepsilon.
	\end{equation*}
	Moreover, as \(\abs{\P_1(x,v)}=(2\pi)^{-n}\abs{\int_{\lie{g}^*}\E^{\I\langle v,\xi\rangle}\widehat\P_1(x,\xi)\D\xi}\) equation \eqref{eq:-q(alpha)-symbol} implies that for each \(\varepsilon>0\) one can find \(K\subset \X\) compact so that \(\abs{\P_1(x,v)}\leq \varepsilon\) for all \(x\notin K\) and \(v\in G\).
	Choose a sequence of smooth cutoff functions \(\chi_k\in\SmoothCompactSupp(G)\) converging pointwise to \(1\). Then \((f_k)_{k\in\NN}\) defined by \(f_k(x,v)=\chi_k(v)\P(x,v)\) is a sequence of functions in \(\Cont_c(G,\Cont_0(\X))\). We claim that \(f_k\) converges to \(\P\) with respect to \(\norm{\,\cdot\,}_1\). As it converges pointwise to \(\P\), it suffices to show by the dominated convergence theorem that \(v\mapsto\norm{\P_1(v)}_{\Cont_0(\X)}\) belongs to \(L^1(G)\). To see this let $M \in 2q\NN_0$ be such that \(M>Q(\alpha)\). Then there is a constant \(C_3>0\) such that 
	\begin{align*}
		\abs{\P_1(x,v)} &= (2\pi)^{-n}\abs[\Big]{\int_{\lie{g}^*} \E^{\I\langle v,\xi\rangle}\widehat \P_1(x,\xi)\D\xi}
		={(2\pi)^{-n}}\abs[\Big]{\int_{\lie{g}^*} \E^{\I\langle v,\xi\rangle}\langle v\rangle_\alpha^{-M}\langle D_\xi\rangle_\alpha^M\widehat \P_1(x,\xi)\D\xi}\\&\leq C_3 \langle v\rangle_\alpha^{-M}\int_{\lie{g}^*}(1+\norm{x}_\beta+\norm{\xi}_\alpha)^m\D\xi=C_3 D \langle v\rangle_\alpha^{-M}
	\end{align*}
	for all \((x,v)\in \X\times G\). Therefore, one obtains	\(
		\int_G {\sup_{x\in\X}\abs{\P_1(x,v)}}\D v\leq C_3D \int_G \langle v\rangle^{-M}_\alpha\D v<\infty
	\).
\end{proof}
\begin{theorem}\label{res:order0_bounded}
	Suppose property \ref{assumption:homomorphisms} holds. Then \(\P_1\in\Pseu^m\) with \(m\leq 0\) belongs to the multiplier algebra \(M(C^*(\X\rtimes^1 G))\). If also property~\ref{assumption:free} holds, \(\Op(\P_1)\) extends to a bounded operator on \(L^2(\X)\).
\end{theorem}
\begin{proof}
	Note that for \(\P_1\in\Pseu^m\) one has for all \(f\in\Schwartz(\X\rtimes^1 G)\) by the \(C^*\)-identity
	\begin{equation*}
		\norm{\P_1*f}^2=\norm{(\P_1*f)^* * (\P_1*f)} \leq \norm{\P_1^**\P_1 *f}\norm{f}.
	\end{equation*}
	Hence, if \(\P_1^**\P_1\) extends to a bounded left multiplier of \(C^*(\X\rtimes^1 G)\), so does \(\P_1\). Consider the case \(m<0\) first, then \((\P_1^**\P_1)^{2^k}\) belongs to \(C^*(\X\rtimes^1 G)\) when \(k\) is an integer with \(2^k>-\frac{Q(\alpha)}{2m}\) by \cref{res:small-order-groupoid-c-*}. Hence, applying the observation above to \((\P_1^**\P_1)^{2^k}=((\P_1^**\P_1)^{2^{k-1}})^**(\P_1^**\P_1)^{2^{k-1}}\) iteratively yields that \(\P_1\) is a bounded left multiplier. Similarly, \(\P_1\) is a bounded right multiplier. 
	
	Let now \(\P_1\in\Pseu^0\). Under assumption \ref{assumption:homomorphisms}, the principal cosymbol algebra \(\Symb^0\) has a \(C^*\)-closure, see \cref{res:c-star-and-psi-star-algebra}. Consider \(\sigma_0(\P_1)\in\Cst(\Symb^0)\) and let \(C>\norm{\sigma_0(\P_1)^**\sigma_0(\P_1)}\). Then \(C\cdot 1-\sigma_0(\P_1)^**\sigma_0(\P_1)\) is a positive element of the principal symbol \(C^*\)-algebra and consequently \(q\defeq \sqrt{C\cdot 1-\sigma_0(\P_1)^**\sigma_0(\P_1)}\) exists in \(C^*(\Symb^0)\) and is self-adjoint. As \(\Symb^0\) is closed under holomorphic functional calculus by \cref{res:c-star-and-psi-star-algebra}, one has in fact \(q\in\Symb^0\). The principal symbol map is surjective by \cref{res:short_exact_sequence}, so one can choose a \(\Q_1\in\Pseu^0\) with \(\sigma_0(\Q_1)=q\). Consider \(\mathbb R_1=\mathbb Q_1^**\Q_1-C\cdot \mathbb I_1+\P_1^**\P_1\) and note that \(\mathbb R_1^*=\mathbb R_1\).
	Then one computes \(\sigma_0(\mathbb R_1)=0\), so that \(\mathbb R_1\in\Pseu^{-1}\). Therefore, \(\mathbb R_1\) defines a bounded, two-sided multiplier of \(C^*(\X\rtimes^1 G)\) by the argument above. One has for all \(f\in\Schwartz(\X\rtimes^1 G)\)
	\begin{equation*}
		0\leq (\P_1*f)^**(\P_1*f) = -(\Q_1*f)^**(\Q_1*f)+C\cdot f^**f+f^**\mathbb R_1 *f\leq C\cdot f^**f+f^**\mathbb R_1 *f
	\end{equation*}
	which implies
	\begin{equation*}
			\norm{\P_1*f}^2=\norm{(\P_1*f)^**(\P_1*f)}\leq C\norm{f}^2+\norm{\mathbb R_1}\norm{f}^2.
	\end{equation*}
	Therefore, \(\P_1\colon\Schwartz(\X\rtimes^1 G)\to\Schwartz(\X\rtimes^1 G)\) extends to a left multiplier of \(C^*(\X\rtimes^1 G)\) and by similar arguments also to a right multiplier.
	
If \ref{assumption:free} is satisfied, the shear map \(\Theta^1\) defines an isomorphism between \(\X\rtimes^1 G\)  and the pair groupoid \(\X\times \X\). Then \(C^*(\X\rtimes^1G)\cong C^*(\X\times \X)\cong\Comp(L^2\X)\) and \(M(C^*(\X\rtimes^1G))\cong\ M(\Comp(L^2 \X))=\Bound(L^2\X)\) hold.
\end{proof} 
\begin{corollary}\label{res:neg-order-compact}
	Suppose property \ref{assumption:homomorphisms} holds. Then \(\P_1\in\Pseu^{m}\) with \(m<0\) belongs to \(C^*(\X\rtimes G)\). If also property \ref{assumption:free} holds, \(P\) defines a compact operator on \(L^2(\X)\).
\end{corollary}
\begin{proof}
	As \(\P_1\in\Pseu^m\) is a multiplier of \(C^*(\X\rtimes^1 G)\) by the previous result, it suffices to show \(\pi(\P_1)=0\), where \(\pi\colon M(C^*(\X\rtimes^1 G))\to M(C^*(\X\rtimes^1 G))/C^*(\X\rtimes^1 G)\) is the projection onto the Calkin algebra. By the \(C^*\)-identity \(\norm{\pi(\P_1)}^2=\norm{\pi(\P_1^**\P_1)}\), it suffices to show that \(\P_1^**\P_1\in\Pseu^{2m}\) belongs to \(C^*(\X\rtimes^1 G)\). Applying this finitely many times one can use that \((\P_1^**\P_1)^K\in C^*(\X\rtimes^1G)\) by \cref{res:small-order-groupoid-c-*} for \(K\in\NN\) large enough. 
\end{proof}
\subsection{Sobolev spaces and mapping properties}
By definition, the pseudodifferential operators act on \(\Algebra O_M(\X)\). So far, we know that under property~\ref{assumption:free} the operators are also continuous as maps \(\Schwartz(\X)\to\Schwartz(\X)\) by \cref{res:op_on_schwartz} and \(\Schwartz'(\X)\to\Schwartz'(\X)\) by \cref{res:Schwartz_multiplier}. In this section, further mapping properties are studied, in particular, a corresponding scale of Sobolev spaces is defined. 
\begin{remark}
	We will always assume property \ref{assumption:free} now. In \cite{AMY22} a scale of Sobolev Hilbert modules over the \(C^*\)-algebra of the foliation defining the pseudodifferential calculus is introduced. A similar approach should also  be possible in the setting of Shubin tangent groupoids. 
\end{remark}
Recall that by \cref{res:orbitmaps} the orbit maps \(\OrbitMap 1 x\colon G\to \X\) for \(x\in \X\) are a family of polynomial diffeomorphisms and that \(c= \abs{\det D_v(\OrbitMap 1 x)}\) is independent of \(x\in \X\) and \(v\in G\).
\begin{lemma}\label{res:Hilbert-Schmidt}
	Suppose property \ref{assumption:free} holds and let \(P\in\Op(\Pseu^m)\) with \(m<-\frac {Q(\alpha)+Q(\beta)}2\). Then \(P\) is a Hilbert--Schmidt operator on \(L^2(\X)\).
\end{lemma}
\begin{proof}
	The operator \(P=\Op(\P_1)\) for \(\P_1\in\Pseu^m\) is Hilbert--Schmidt if and only if its kernel, which is formally given by \((x,y)\mapsto \P_1((\Theta^1)^{-1}(x,y))\), belongs to \(L^2(\X\times \X)\). As \((\Theta^1)^{-1}(x,y)=(x,(\OrbitMap 1 x)^{-1}(y))\) holds, the \(L^2\)-norm of its kernel is \(\sqrt{c}\norm{\P_1}_{L^2(\X\times G)}\). Using Fourier transform, this norm is equal to \(\tfrac {\sqrt c }{(2\pi)^{n/2}}\norm{\widehat{\P}_1}_{L^2(\X\times\lie g^*)}\), which is finite by the symbol estimates in \cref{res:symbol_estimates} and \cref{res:integrable}.
\end{proof}

\begin{lemma}\label{res:l2_to_continuous}
	Suppose that property \ref{assumption:free} is satisfied and let \(P\in\Op(\Pseu^m)\) with \(m<-\frac {Q(\alpha)}2\). Then \(P\) defines a continuous operator \(L^2(\X)\to \Cont_0(\X)\).
\end{lemma}
\begin{proof}
	We first show that \(\P_1\in\Cont_0(\X,L^2(G))\), this is equivalent to \(\widehat{\P}_1\in\Cont_0(\X,L^2(\lie{g}^*))\). By the symbol estimates \cref{res:symbol_estimates} there is a constant \(C>0\) such that for all \(x\in G\)
	\begin{equation*}
		\abs{\widehat \P_1(x,\xi)}\leq C(1+\norm{\xi}_\alpha+\norm{x}_\beta)^m.
	\end{equation*}
	Using a similar computation as in \eqref{eq:estimate-shubin-norm} one obtains that for every \(\varepsilon>0\) there is a compact set \(K\subset \X\) such that \(\norm{\widehat{\P}_1(x)}_{L^2(\lie{g}^*)}<\varepsilon\) for all \(x\notin K\).
	Furthermore, one finds \(D>0\) such that for all \(x,y\in \X\) by the mean value theorem 
	\begin{equation}\label{eq:cont-in-x}
		\abs{\widehat \P_1(x,\xi)-\widehat \P_1(y,\xi)}\leq D\norm{x-y}_\beta(1+\norm{\xi}_\alpha)^m.
	\end{equation}
	Then for \(f\in\Schwartz(\X)\), we have 
	\begin{equation*}
		\abs{Pf(x)}=\abs{\langle \P_1(x),(\OrbitMap 1 x)^*f\rangle} \leq \tfrac{1}{\sqrt c} \norm{\P_1(x)}_{L^2(G)}\norm{f}_{L^2(\X)}\leq \tfrac{1}{\sqrt c} \sup_{x\in G}\norm{\P_1(x)}_{L^2(G)}\,\norm{f}_{L^2(\X)}.
	\end{equation*}
	Moreover, one has for all \(x,y\in \X\)
	\begin{align*}
		\abs{Pf(x)-Pf(y)}\leq \,&\abs{\langle \P_1(x)-\P_1(y),(\OrbitMap 1 x)^*f\rangle}+\abs{\langle \P_1(y),(\OrbitMap 1 x)^*f-(\OrbitMap 1 y)^*f\rangle}\\
		\leq \,& \tfrac{1}{\sqrt c}\norm{\P_1(x)-\P_1(y)}_{L^2(G)}\norm{f}_{L^2(\X)}+\norm{(\OrbitMap 1 x)^*f-(\OrbitMap 1 y)^*f}_{L^2(G)}\sup_{x\in G}\norm{\P_1(x)}_{L^2(G)}.
	\end{align*}
	For \(x\to y\) this converges to zero by \eqref{eq:cont-in-x} and as \(x\mapsto \sqrt{c}\,(\OrbitMap 1 x)^*\) defines a strongly continuous representation \(G\to\mathcal U (L^2(\X)\to L^2(G))\).
\end{proof}
In the following, we fix an elliptic operator \(P_s\in\Op(\Pseu^s)\) for every \(s\in\RR\), this is possible by \cref{res:existence-elliptic}. Then by \cref{res:parametrix} there are \(Q_s\in\Op(\Pseu^{-s})\) and \(R_s\in\Op(\Pseu^{-\infty})\) with \(Q_sP_s=\id+R_s\).
\begin{lemma}
	Suppose property \ref{assumption:free} and \ref{assumption:homomorphisms} are satisfied. Let \(P\in\Op(\Pseu^m)\) with \(m<-Q(\alpha)-Q(\beta)\), then \(P\) is a trace class operator on \(L^2(\X)\).
\end{lemma}
\begin{proof}
	Let \(s=-\frac{m}{2}\) and write \(P=Q_s^2P_s^2P+R\) for some \(R\in\Op(\Pseu^{-\infty})\). Then \(R\) is of trace class. Note that \(Q_s\in\Op(\Psi^{\frac m2})\) is Hilbert--Schmidt by \cref{res:Hilbert-Schmidt}, so that \(Q_s^2\) is of trace class. Moreover, \(P_s^2P\in\Op(\Pseu^{0})\) and is, therefore, bounded by \cref{res:order0_bounded}.
\end{proof}
\begin{definition}
	For \(s\in\RR\) let the \emph{Sobolev space of order \(s\)} be defined by 
	\(\Sob^s(\X)=\set{u\in\Schwartz'(\X) }{ P_su\in L^2(\X)}\) and define for \(u,v\in \Sob^s(\X)\)
	\begin{equation}\label{eq:sobolev_inner_product}
		\langle u,v\rangle_{s}=\langle P_su,P_sv\rangle_{L^2(\X)}+\langle R_s u, R_s v\rangle_{L^2(\X)}.
	\end{equation}
	 We shall also write \(\Sob^s\), when it is clear which \(G\)-space \(\X\) we consider.
\end{definition}
\begin{lemma}\label{res:sobolev_independent_hilbert}
	Suppose that \ref{assumption:free} and \ref{assumption:homomorphisms} are satisfied.
	The space \(\Sob^s(\X)\) has the following properties:
	\begin{enumerate}
		\item it does not depend on the choice of the elliptic pseudodifferential operator \(P_s\) and its parametrix~\(Q_s\),
		\item it is a Hilbert space with respect to the inner product \eqref{eq:sobolev_inner_product},
		\item \(P_s,R_s\colon \Sob^s(\X)\to L^2(\X)\) and \(Q_s\colon L^2(\X)\to \Sob^s(\X)\) are continuous.
	\end{enumerate}
\end{lemma}
\begin{proof}
	Using that order zero operators on \(L^2(\X)\) are bounded by \cref{res:order0_bounded}, one can show the first two claims as in \cite{NR10}*{Proposition~1.5.3}. The last claim follows from the definition of the norm on \(\Sob^s(\X)\).
\end{proof}
The proof of the following proposition is adapting standard arguments to our situation, see for example \cite{Shu87}*{Section~I.7} and the proof of \cite{DH22}*{Proposition~3.17}.  
\begin{proposition}\label{res:sobolev_properties}
Suppose \ref{assumption:free} and \ref{assumption:homomorphisms} hold. Then the Sobolev spaces have the following properties:
	\begin{enumerate}
		\item \(\Sob^0(\X)=L^2(\X)\),
		\item\label{item:inclusion_schwartz_tempered_distr} for every \(s\in\RR\) one has continuous inclusions
		\(
			\Schwartz(\X)\hookrightarrow \Sob^{s}(\X)\hookrightarrow \Schwartz'(\X)
		\) and \(\Schwartz(\X)\) is dense in \(\Sob^s(\X)\),
		\item\label{item:continuity_pseudo_sob}\(P\in\Op(\Pseu^m)\) defines for every \(s\in\RR\) a continuous operator \(\Sob^s(\X)\to\Sob^{s-m}(\X)\),
		\item\label{item:sobolev_regularity}let \(P\in\Op(\Pseu^m)\) be elliptic and \(u\in\Schwartz'(\X)\), then \(Pu\in\Sob^s(\X)\) for some \(s\in\RR\) implies \(u\in\Sob^{s+m}(\X)\),
		\item for \(s_1< s_2\) there is a compact embedding
		\( \Sob^{s_2}(\X)\hookrightarrow \Sob^{s_1}(\X)\),
		\item\label{item:duality} the pairing \(\Schwartz(\X)\times\Schwartz(\X)\to\CC\) given by \(\langle u,v\rangle=\int u(x)v(x)\D x\) extends for every \(s\in\RR\) to a continuous bilinear map \(\Sob^s(\X)\times\Sob^{-s}(\X)\to\CC\) and \((\Sob^s(\X))'=\Sob^{-s}(\X)\),
		\item\label{item:sobolev_embedding} one has a continuous embedding \(\Sob^s(\X)\hookrightarrow \Cont_0(\X)\) for \(s>\frac{Q(\alpha)}{2}\),
		\item \(\Schwartz(\X)=\bigcap_{s\in\RR}\Sob^s(\X)\) and \(\Schwartz'(\X)=\bigcup_{s\in\RR}\Sob^s(\X)\).
		\end{enumerate} 
\end{proposition}
\begin{proof}
	\begin{enumerate}[wide=0pt]
		\item This is clear from \cref{res:sobolev_independent_hilbert} and choosing \(P_0=\id\).
		\item The continuity of \(\Schwartz\hookrightarrow\Sob^{s}\) follows from the continuity of \(P_{s},R_{s}\colon\Schwartz \to\Schwartz \) and \(\Schwartz \hookrightarrow L^2\). To see that the image is dense, let \(u\in\Sob^s\) and let \((f_n)_{n\in\NN}\) be a sequence in \(\Schwartz(\X)\) with \(\lim_{n\to\infty} f_n = P_su\in L^2\). Then \((Q_sf_n-R_su)_{n\in\NN}\) is a sequence of Schwartz functions converging to \(u\) in \(\Sob^s\).
		
		For the second embedding, write \(\id =Q_{s}P_{s}-R_{s}\). Then \(R_{s}, P_{s}\colon \Sob^s\to L^2\), \(L^2\hookrightarrow \Schwartz'\) and \(Q_{s}\colon\Schwartz'\to\Schwartz'\) are continuous, so that \(\Sob^{s}\hookrightarrow \Schwartz'\) is continuous. 		
		\item We first show \(P\colon \Sob^{m}\to L^2\) is continuous by writing \(P=PQ_mP_m-PR_m\). Then \(P_m\colon \Sob^m\to L^2\) is continuous and \(PQ_m\in\Op(\Pseu^0)\) defines a continuous map \(L^2\to L^2\) by \cref{res:order0_bounded}. Furthermore, \(R_m\colon\Schwartz'\to\Schwartz\) and \(P\colon \Schwartz\to\Schwartz\hookrightarrow L^2\) are continuous.
		
		Let \(s\in\RR\) be now arbitrary and write \(P=Q_{s-m}P_{s-m}P-R_{s-m}P\). Then \(P_{s-m}P\in\Op(\Pseu^s)\) defines a continuous map \(\Sob^s\to L^2\) by the previous argument and \(Q_{s-m}\colon L^2\to \Sob^{s-m}\) is continuous. Moreover, \(\Sob^{s}\hookrightarrow\Schwartz'\), \(P\colon \Schwartz'\to\Schwartz'\), \(R_{s-m}\colon\Schwartz'\to\Schwartz\) and \(\Schwartz\hookrightarrow \Sob^{s-m}\) are continuous. 
		\item Let \(Q\in\Op(\Pseu^{-m})\) be a parametrix for \(P\). Then writing \(u=QP-R\) for some \(R\in\Op(\Pseu^{-\infty})\) implies the claim using \refitem{item:continuity_pseudo_sob}.
		\item One can write \(\id=Q_{s_1}P_{s_1}Q_{s_2}P_{s_2}+R\) for some \(R\in\Op(\Pseu^{-\infty})\). Then \(P_{s_2}\colon \Sob^{s_2}\to L^2\) is continuous, \(P_{s_1}Q_{s_2}\in\Op(\Pseu^{s_1-s_2})\) is compact as a map \(L^2\to L^2\)  by \cref{res:neg-order-compact} and \(Q_{s_1}\colon L^2\to \Sob^{s_1}\) is continuous. Moreover, \(R\) is compact as every bounded sequence in \(\Schwartz\) has a convergent subsequence \cite{Hor66}*{p.~240}. 
		\item For \(s=0\) this is clear. For other \(s\in\RR\) use the pairing on \(L^2\) to define for \(u\in\Sob^s\) and \(v\in\Sob^{-s}\) 
		\begin{equation*}
			\langle u, v\rangle =\langle P_s u, Q_s^t v\rangle- \langle R_su,v\rangle. 
		\end{equation*}
		From the shown mapping properties, one sees that this is separately continuous. As both spaces are Hilbert spaces, the map is also jointly continuous. Furthermore, it extends the pairing on~\(\Schwartz\).
		
		Consider the continuous map \(l\colon \Sob^{-s}\to(\Sob^s)'\) defined by \(v\mapsto l_v\) with \(l_v(u)=\langle u,v\rangle\) for \(u\in\Sob^s\). The map is injective as the embedding \(\Sob^s\hookrightarrow \Schwartz'\) can be written as the composition of \(l\) with the embedding \((\Sob^s)'\hookrightarrow\Schwartz'\). For surjectivity let \(L\in(\Sob^s)'\). Then there must be a \(v\in\Schwartz'\) such that \(L(u)=\langle u,v\rangle\) for all \(u\in\Schwartz\). To show that \(P_sv\in L^2\), note that one can define \(k\colon u\mapsto L(P_s^tu)=\langle P_s^tu,v\rangle =\langle u,P_sv\rangle\) for \(u\in\Schwartz\). As \(L\) is defined on \(\Sob^s\), one can extend \(k\) to a continuous functional on \(L^2\). As \((L^2)'\cong L^2\) under the pairing, we get \(P_sv\in L^2\).
		By the bounded inverse theorem \(l\) is an isomorphism.
		\item Write \(\id =Q_sP_s-R_s\), then \(P_s\colon \Sob^s\to L^2\) and \(\Sob^s\hookrightarrow L^2\) are continuous and \(Q_s, R_s\colon L^2\to\Cont_0\) by \cref{res:l2_to_continuous}. 
		\item Suppose \(u\in \Sob^s\) for all \(s\in \RR\). To see that \(u\in\Schwartz\), we show that \(Pu\in\Cont_0\) for all differential operators with polynomial coefficients \(P\). By \cref{res:propertyP:uniqueRep}, we know that \(P\in\Op(\Pseu^m)\) for some \(m\in\NN_0\). Consider \(s>\frac{Q(\alpha)}{2}\). 
		As \(u\in \Sob^{s+m}\) one has \(Pu\in \Sob^s \subseteq \Cont_0\) by \refitem{item:sobolev_embedding}. 
		
		Let now \(u\in\Schwartz'\). By the characterization of tempered distributions \cite{Hor66}*{p.~410} it can be written as a finite sum of \(P\phi\), where \(P\) is a differential operator with polynomial coefficients and \(\phi\in(\Cont_0)'\). By \cref{res:propertyP:uniqueRep} we may assume \(P\in\Op(\Pseu^m)\) for some \(m\in\NN_0\). For \(s>\frac{Q(\alpha)}{2}\) it defines a continuous functional on \(\Sob^{s+m}\) as \(\langle P\phi, f\rangle =\langle \phi, P^tf\rangle \) and \(P^t\colon \Sob^{s+m}\to\Cont_0\) is continuous by the previous observations. Hence,  \(P\phi\in \Sob^{-s-m}\) by duality \refitem{item:duality}.\qedhere
		\end{enumerate}
\end{proof}

\subsection{Fredholm operators and spectral properties}
The following properties are shown as for \(G=\RR^n\), see for example \cite{Shu87}.
\begin{proposition}
	Suppose that property \ref{assumption:free} and \ref{assumption:homomorphisms} hold and that \(P\in\Op(\Pseu^m)\) is elliptic. Then it has the following properties:
	\begin{enumerate}
		\item \(P\) defines a Fredholm operator \(P\colon \Sob^s(\X)\to \Sob^{s-m}(\X)\) for all \(s\in\RR\),
		\item its Fredholm index does not depend on \(s\in\RR\) and is given by
		\(\ind P =\dim\ker P-\dim\ker P^*,\)
	\item the Fredholm index just depends on the principal cosymbol of \(P\).
	\end{enumerate}	
\end{proposition}
\begin{proof}
	\begin{enumerate}
		\item Let \(s\in\RR\) and \(Q\) be a parametrix for \(P\) such that \(QP-\id=R\) and \(PQ-\id=S\) with \(R,S\in\Op(\Pseu^{-\infty})\). In particular, \(R\) defines a compact operator \(\Sob^{s}\to\Sob^s\) and \(S\) defines a compact operator \(\Sob^{s-m}\to \Sob^{s-m}\). Hence, \(P\) is Fredholm by Atkinson's Theorem.
		\item Note that the formal adjoint \(P^*\colon\Schwartz\to\Schwartz\) extends to a continuous map \(P^*\colon \Sob^{m-s}\to\Sob^{-s}\). Using the identifications \(\Sob^{m-s}\cong(\Sob^{s-m})'\) and \(\Sob^{-s}\cong(\Sob^{s})'\), one checks that \(P^*\) is the adjoint of \(P\colon \Sob^s\to\Sob^{s-m}\). Consequently, \(\ind P= \dim\ker P-\dim\ker P^*\).
		
		As \(P\) is elliptic \(Pu=0\) for \(u\in\Schwartz'\) implies by \cref{res:hypoelliptic} that \(u\in\Schwartz\). Hence, \(\ker(P)\subseteq \Schwartz\) holds and also \(\ker(P^*)\subseteq \Schwartz\) as \(P^*\) is elliptic, too, so that the index does not depend on \(s\).
		\item If \(\widetilde{P}\) has the same principal cosymbol as \(P\), one gets \(P-\widetilde{P}\in\Op(\Pseu^{m-1})\). As operators in \(\Op(\Pseu^{m-1})\) map continuously from \(\Sob^s\) to \(\Sob^{s-m+1}\) and the embedding \(\Sob^{s-m+1}\hookrightarrow\Sob^{s-m}\) is compact, one sees that~\(\widetilde{P}\) is a compact perturbation of \(P\) and therefore has the same Fredholm index. \qedhere
	\end{enumerate}
\end{proof}
In the following, we consider \(P\in\Op(\Pseu^m)\) elliptic with \(m>0\) as an unbounded operator on \(L^2(\X)\) with domain \(\Sob^m(\X)\). To see that it is closed, let \((u_n)_{n\in\NN}\) be a sequence in \(\Sob^m(\X)\) such that \(u_n\to u\) and \(Pu_n\to f\) in \(L^2(\X)\). As \(L^2(\X)\hookrightarrow\Schwartz'(\X)\) and \(P\colon\Schwartz'(\X)\to\Schwartz'(\X)\) are continuous, we see that \(Pu=f\). By \cref{res:sobolev_properties} \refitem{item:sobolev_regularity} this implies \(u\in \Sob^m(\X)\).

In particular, one can consider its spectrum \(\sigma(P)\) and we get the following inverse operator theorem, see also \cite{Shu87}*{Theorem~8.2, Theorem~25.4}
\begin{lemma}\label{res:inverse_operator}
	Suppose \ref{assumption:free} and \ref{assumption:homomorphisms} hold. If \(P\in\Op(\Pseu^m)\) for \(m>0\) is elliptic and \(\lambda \notin\sigma(P)\), then \((P-\lambda\cdot \id)^{-1}\in\Op(\Pseu^{-m})\). 
\end{lemma} 
\begin{proof}
	As \(P-\lambda\cdot \id\in\Op(\Pseu^m)\) is also elliptic, it suffices to show the claim for \(\lambda =0\).
	Let \(Q\in\Op(\Pseu^{-m})\) be a parametrix for \(P\) with \(PQ-\id=R\) and \(R\in\Op(\Pseu^{-\infty})\). Then for all \(u\in\Schwartz\) one has \(P^{-1}u=Qu-P^{-1}Ru\). Therefore, it suffices to show that \(P^{-1}R\colon\Schwartz'\to\Schwartz\). This is the case if \(P^{-1}\) restricts to a continuous map \(P^{-1}\colon\Schwartz\to\Schwartz\). To see this, we show that for every differential operator with polynomial coefficients \(D\), the map \(DP^{-1}\colon\Schwartz\to\Cont_0\) is continuous. By \cref{res:propertyP:uniqueRep} we may suppose \(D\in\Op(\Pseu^k)\). Choose \(s\in\RR\) with \(s>\frac{Q(\alpha)}{2}\) and \(s+k\geq m\). Then \(P^{-1}\colon L^2\to \Sob^m\) restricts by the elliptic regularity shown in \cref{res:sobolev_properties} \refitem{item:sobolev_regularity} to a map \(\Sob^{s+k-m}\to\Sob^{s+k}\) and it is continuous by the inverse operator theorem. Using the continuity of \(\Schwartz\hookrightarrow \Sob^{s+k-m}\), \(D\colon \Sob^{s+k}\to\Sob^s\) and \(\Sob^{s}\hookrightarrow \Cont_0\) from \cref{res:sobolev_properties}, the claim follows. 
\end{proof}
Let \(P=P^*\in\Op(\Pseu^m)\) be formally self-adjoint with \(m>0\). Then denote by \(P_0\) be the unbounded operator on \(L^2(\X)\) with domain \(D(P_0)=\Schwartz(\X)\) given by \(P_0=P|_{\Schwartz(\X)}\). This operator is symmetric as \(P\) is formally self-adjoint, hence \(P_0\) is closable. Analogously to \cite{Shu87}*{Theorem~26.2, Theorem~26.3} one obtains the following result. 
\begin{lemma}
	Suppose \ref{assumption:free} and \ref{assumption:homomorphisms} hold. Let \(P\in\Op(\Pseu^m)\) for \(m>0\) be formally self-adjoint and elliptic. Then the following holds for the unbounded operator \(P_0\):
	\begin{enumerate}
		\item \(P_0\) is essentially self-adjoint,
		\item its closure \(\overline{P_0}\) has compact resolvent,
		\item the spectrum of \(\overline{P_0}\) is discrete and consists of eigenvalues \((\lambda_j)_{j\in\NN}\) of finite multiplicities with \(\abs{\lambda_j}\to\infty\) as \(j\to\infty\). Moreover, there is an orthonormal basis \((\varphi_j)_{j\in\NN}\) of \(L^2(\X)\) such that \(\overline{P_0}\varphi_j=\lambda_j\varphi_j\).
	\end{enumerate}  
\end{lemma}
\begin{proof}
	As \(P_0\) is symmetric it suffices by \cite{Shu87}*{Theorem~26.1} to show that \(\ker(P_0^*\pm i\cdot \id)\subset D(\overline{P}_0)\) to deduce that \(P_0\) is essentially self-adjoint. Using \(P=P^*\), one deduces that the domain of \(P_0^*\) consists of all \(u\in L^2\) such that \(Pu\in L^2\) and that for these \((P_0)^*u=Pu\). Thus, the kernel of \(P_0^*\pm i\cdot\id\) is the space of all \(u\in L^2\) such that \((P\pm i\cdot\id)u=0\). As \(P\pm i\cdot\id\) is elliptic, \cref{res:hypoelliptic} implies \(\ker(P_0^*\pm i\cdot \id)\subset \Schwartz=D(P_0)\).
	
	Note that \(\overline{P_0}\) is essentially self-adjoint, hence, its spectrum is real. If \(\sigma(\overline {P_0})=\RR\) was true, for every~\(\lambda\in\RR\) a function \(\varphi_\lambda\in\Schwartz\) would exist with \(P_0\varphi_\lambda =\lambda\varphi_\lambda\). As \(P_0\) is symmetric, they would all be orthogonal, contradicting the separability of \(L^2\). Hence, one can choose \(\lambda_0\in\RR\setminus\sigma(\overline{P_0})\). Then \((P-\lambda_0\cdot \id)^{-1}\in \Pseu^{-m}\) by \cref{res:inverse_operator}, and extends therefore to a compact operator on \(L^2\). It is well-known that compact resolvent implies the claim on the spectrum, see for example \cite{NR10}*{Theorem~4.1.6}
\end{proof}
\section{\texorpdfstring{Comparison to other calculi on }{Comparison to other calculi on R\^{}n}}\label{sec:comparison}
In this section, we compare the calculus of the double dilation groupoid and representation groupoid to known calculi when \(G\) is the Abelian group \(\RR^n\). 
\subsection{Comparison to the calculus of Shubin and Helffer--Robert}
Recall that the calculus of Shubin and Helffer--Robert is based on the following symbol classes.
\begin{definition}[\cite{Shu87}*{Def.~23.1}, \cite{Hel84}*{Def.~1.1.1}]
	Let \(\Gamma^m\) denote the space of all \(p\in C^\infty(\RR^{2n})\) such that for all \(a,b\in\NN^n_0\) there is \(C_{ab}>0\) such that for all \((x,\xi)\in\RR^{2n}\)\begin{equation*}\abs{\partial^a_x\partial^b_\xi p(x,\xi)}\leq C_{ab} (1+\norm{x}+\norm{\xi})^{m-\abs{a}-\abs{b}}.\end{equation*}
\end{definition}
For \(p\in \Gamma^m\) the corresponding operator \(\Op(p)\colon\Schwartz(\RR^n)\to\Schwartz(\RR^n)\) is defined by
\begin{equation*}
	\Op(p)f(x)=(2\pi)^{-n}\int \E^{\I\langle y-x,\xi\rangle}p(x,\xi)f(y)\D\xi\D y.
\end{equation*}
Denote by \(\Psi^m(\RR^n)\) the class of all such pseudodifferential operators with symbol in \(\Gamma^m\). 
Note that this gives the same class of operators as the standard Kohn-Nirenberg quantization formula using \(\langle x-y,\xi\rangle=\langle y-x,-\xi\rangle\) as \(\Gamma^m\) is invariant under \((x,\xi)\mapsto(x,-\xi)\).

As a subclass, we denote by \(\Psi^m_{cl}(\RR^n)\) all pseudodifferential operators of order \(m\) which are classical in the sense that their symbol \(p\) admits an homogeneous expansion with respect to the scalings \(\lambda\cdot (x,\xi)=(\lambda x,\lambda \xi)\). This means that there are \(p_i\in\Gamma^{m-i}\) homogeneous of degree \(m-i\) for \(\norm{(x,\xi)}\geq 1\) for all \(i\in\NN_0\) such that for every \(k\in\NN_0\) one has
\begin{equation*}p-\sum_{i=0}^k p_i\in\Gamma^{m-k-1}.\end{equation*}
In this case, one writes \(p\sim\sum_{i=0}^\infty p_i\).

We shall compare \(\Psi^m_{cl}(\RR^n)\) to our operators \(\Op(\Pseu^m)\) in the case of the double dilation or representation groupoid, which coincide for \(G=\RR^n\) with the dilations \(\alpha_\lambda(x)=\beta_\lambda(x)=\lambda x\) for \(\lambda>0\). In particular, one can use the Euclidean norm as a homogeneous quasi-norm.
\begin{lemma}\label{res:lemma_rn}
	For \(\alpha=\beta\) being the usual scalings on \(G=\RR^n\), every \(P\in \Op(\Pseu^m)\) belongs to \(\Psi^m(\RR^n)\) and its symbol is given by \(\widehat \P_1\).
\end{lemma}
\begin{proof}
	This is clear from the representation \eqref{eq:rep_as_osc_int} where in this case \((\OrbitMap{1}{x})^{-1}(y)=y-x\)   and the symbol estimates in \cref{res:symbol_estimates}.
\end{proof}
\begin{theorem}\label{res:comparison}
	For \(\alpha=\beta\) being the usual scalings on \(G=\RR^n\), the classes \(\Op(\Pseu^m)\) (for the double dilation groupoid and the representation groupoid) and \(\Psi^m_{cl}(\RR^n)\) coincide.
\end{theorem}
\begin{proof}
	Suppose first that the symbol \(p\) of \(P\in\Psi^m_{cl}(\RR^n)\) has the homogeneous expansion \(p\sim \sum_{i=0}^\infty p_i\) with \(p_i\) being homogeneous of degree \(m-i\) for \(\norm{(x,\xi)}\geq 1\). For every \(i\in\NN\) the function \(p_i\in \Smooth(\RR^{2n})\) is essentially homogeneous of degree \(m-i\). In particular,  \(\widecheck p_i=\mathcal F^{-1}_{\xi\to v}{p_i}\in \Ess ^{m-i}\). Using the linear split \(r_i\) in \cref{res:short_exact_sequence} there is a \(\Q^i_1\in\Pseu^{m-i}\) with \(\sigma_i(\Q^i_1)=[\widecheck p_i]\) and \(\widehat\Q^i_1= p_i\). By the asymptotic completeness shown in \cref{res:asymptotic_completeness}, there is \(\Q_1\in\Pseu^m\) with \(\Q_1\sim \sum_{i=0}^\infty\Q^i_1\). Let \(Q=\Op(\Q_1)\) and \(Q_i=\Op(\Q^i_1)\). We show that the Schwartz kernel of \(P-Q\) is in \(\Schwartz(\RR^n\times\RR^n)\), then \(P\in\Op(\Pseu^m)\) follows. It suffices to verify \(P-Q\in\Psi^k(\RR^n)\) for every integer \(k\leq m\) by \cite{Shu87}*{p.~179}. One writes
	\begin{equation*}
		P-Q= \left(P-\sum_{i=0}^{m-k-1}Q_i\right)+\left(\sum_{i=0}^{m-k-1}Q_i-Q\right).
	\end{equation*}
	The right term is by definition of the asymptotic sum in \(\Op(\Pseu^{k})\) and, therefore, in \(\Psi^k(\RR^n)\) by \cref{res:lemma_rn}. The left term is also contained in \(\Psi^k(\RR^n)\) as its symbol is given by \(p-\sum_{i=0}^{m-k-1}p_i\in\Gamma^{k}\).
	
	Conversely, let \(P=\Op(\P_1)\in\Op(\Pseu^m)\) and let \(\P\) be an essentially homogeneous extension. By \cref{res:lemma_rn} it suffices to show that \(\widehat \P_1\) admits a homogeneous expansion. 
	
	We proceed by iteratively constructing the terms in the homogeneous expansion. As \(\widehat{\P}_0\) is essentially homogeneous of degree \(m\), there is by \cref{res:equivalence-pseudo-symbol} a function \(p_0\) which is \(m\)-homogeneous for \(\norm{(x,\xi)}\geq 1\) such that \(\widehat{\P}_0-p_0\in\Schwartz(\RR^{2n})\). Choose, using the split in \cref{res:short_exact_sequence}, \(\Q^0\in\PPseu^m\) with principal cosymbol \([\P_0]\) and  \(\widehat\Q^0_1=p_0\). We set \(Q_0=\Op(\Q^0_1)\). Note that \(\sigma_m(\P_1-\Q^0_1)=0\) and therefore \(P-Q_0\in\Op(\Pseu^{m-1})\). In particular, by \cref{res:lemma_rn} one has \(\widehat{\P}_1-p_0\in \Gamma^{m-1}\). 
	
	Suppose now, we found for \(i=0,\ldots,k\) smooth functions \(p_i\) homogeneous of degree \(m-i\) outside \(\norm{(x,\xi)}\geq 1\)  and \(\Q^i\in\PPseu^{m-i}\) with \(\widehat\Q^i_1=p_i\) such that \(P-\sum_{i=0}^k \Op(\Q^i_1)\in\Op(\Pseu^{m-k-1})\), so that by \cref{res:lemma_rn} \(\widehat{\P}_1-\sum_{i=0}^k p_i\in\Gamma^{m-k-1}\). Extend \(\P_1-\sum_{i=0}^k \Q^i_1\in\Pseu^{m-k-1}\) to an element \(\mathbb A_{k+1}\in\PPseu^{m-k-1}\). Let \(p_{k+1}\) be the \((m-k-1)\)-homogeneous function outside \(\norm{(x,\xi)}\geq 1\) such that \( {(\widehat {\mathbb A}_{k+1})}_0-p_{k+1}\in\Schwartz(\RR^{2n})\). Use the split in \cref{res:short_exact_sequence} again to extend \(\widecheck p_{k+1}\) to \(\mathbb Q^{k+1}\) such that \(\widehat\Q^{k+1}_1=p_{k+1}\). Then \(P-\sum_{i=0}^{k+1} \Op(Q^i_1)\in\Op(\Pseu^{m-k-2})\) and \(\widehat{\P}_1-\sum_{i=0}^{k+1} p_i\in\Gamma^{m-k-2}\). This shows that \(\widehat \P_1\) has the homogeneous expansion \(\widehat \P_1\sim \sum_{i=0}^\infty p_i\) and is, hence, classical. 
\end{proof}

\begin{remark}Suppose \(P\in\Psi^\ell_{cl}(\RR^n)\) and \(Q\in\Psi^m_{cl}(\RR^n)\) are classical Shubin pseudodifferential operators on \(G=\RR^n\). Then their commutator satisfies \([P,Q]=PQ-QP\in\Psi^{\ell+m-2}_{cl}(\RR^n)\). In the groupoid calculus, this corresponds to the appearance of \(t^2\) in the definition of the convolution. Namely, in this case one has for \(f,g\in\Schwartz(\mathcal G)\)
	\begin{equation*}
		f*g(x,t,v)=\int f(x,t,w)g\left(x+t^2w,t,v-w\right)\D w .
	\end{equation*}
Hence, \([\P_1,\Q_1]\in\Pseu^{\ell+m-2}(\RR^n)\) can be observed from the fact that \(\P*\Q-\Q*\P\) vanishes to order \(2\) at \(t=0\). However, this argument only works on the Abelian group \(G=\RR^n\) as otherwise the convolution at \(t=0\) is noncommutative. 
\end{remark}

\subsection{Comparison to anisotropic calculi}
The anisotropic calculus described in \cite{BN03} or \cite{NR10}*{Sections~2.5, 2.7} is a variant of the calculus of Shubin and Helffer--Robert, where one equips the phase space \(\RR^{2n}\) with different weights. Let \(q_1,\ldots,q_n,r_1,\ldots,r_n,\) be positive integers and equip \(\RR^n\) with the dilations \begin{equation*}\alpha_\lambda(x_1,\ldots,x_n)=(\lambda^{q_1}x_1,\ldots,\lambda^{q_n}x_n)\quad\text{ and }\beta_\lambda(x_1,\ldots,x_n)=(\lambda^{r_1}x_1,\ldots,\lambda^{r_n}x_n)\end{equation*}
and fix corresponding homogeneous quasi-norms \(\norm{\,\cdot\,}_\alpha\) and \(\norm{\,\cdot\,}_\beta\).
\begin{definition}[\cite{BN03}*{Definition~2.1}]
	Let \(\Gamma_{\alpha,\beta}^m\) denote the space of all \(p\in C^\infty(\RR^{2n})\) such that for all \(a,b\in\NN^n_0\) there is \(C_{\alpha\beta}>0\) such that for all \((x,\xi)\in\RR^{2n}\)\begin{equation*}\abs{\partial^a_\xi\partial^b_x p(x,\xi)}\leq C_{ab} (1+\norm{\xi}_\alpha+\norm{x}_\beta)^{m-[a]_\alpha-[b]_\beta}.\end{equation*}
\end{definition}
Similar to before, we use the notation \(\Psi^m_{\alpha,\beta,cl}(\RR^n)\) for pseudodifferential operators with symbol in \(\Gamma^m_{\alpha,\beta}\) admitting a homogeneous expansion with respect to \(\lambda\cdot(x,\xi)=(\beta_\lambda(x),\alpha_\lambda(\xi))\).
\begin{remark}
	In \cite{BN03}*{Definition~2.1} actually positive rational weights \(q_1,\ldots,q_n,r_1,\ldots,r_n\) are permitted.  In this case, let \(M\in\NN\) be such that \(q_jM\in\NN\) and \(r_jM\in\NN\) for all \(j=1,\ldots,n\) and let 
	\begin{equation*}\widetilde\alpha_\lambda(x_1,\ldots,x_n)=(\lambda^{q_1M}x_1,\ldots,\lambda^{q_nM}x_n)\quad\text{ and }\widetilde{\beta}_\lambda(x_1,\ldots,x_n)=(\lambda^{r_1M}x_1,\ldots,\lambda^{r_nM}x_n).\end{equation*}
	Then \(p\in\Gamma^m_{\alpha,\beta}\) if and only if \(p\in\Gamma^{mM}_{\widetilde\alpha,\widetilde\beta}\). Hence, for the calculus, we may assume that the weights are integers.
\end{remark}
\begin{example} Let	\(q,r\in\NN\) and consider the anharmonic oscillators \(P_{q,r}=(-\Delta)^q+\norm{x}^{2r}\), where \(\norm{\,\cdot\,}\) denotes the Euclidean norm. Equip \(G=\RR^n\) with the dilations \(\alpha_\lambda(x)=\lambda^rx\) and \(\beta_\lambda(x)=\lambda^qx\). Then \(P_{q,r}\in\Psi^{2qr}_{\alpha,\beta,cl}\) with respect to these dilations and is elliptic. 
\end{example}
Using the double dilation groupoid for the dilations \(\alpha,\beta\) on the Abelian group \(G=\RR^n\) and adapting the proof of \cref{res:comparison} to these dilations we obtain the following result. 
\begin{theorem}\label{res:comparison:2}
	For integer dilations \(\alpha,\beta\) on \(\RR^n\), the classes \(\Op(\Pseu^m)\) for the double dilation groupoid and the anisotropic classes \(\Psi^m_{\alpha,\beta,cl}(\RR^n)\) coincide.
\end{theorem}

\begin{appendix}
	
\section{(Essentially) homogeneous distributions on graded groups}\label{sec:homogeneous_distributions}
Distributions on a graded group \(G\) which interact with the dilations nicely naturally play an important role in the calculi on graded groups. 
While \cites{Tay84, vEY19} use essentially homogeneous distributions, see \cref{def:ess_graded_group},
the calculus of \cite{CGGP92} is based on homogeneous distributions (and certain distributions with logarithmic singularities).
The purpose of this appendix is to summarize known results in a concise form.
Applying these results to our setting of essentially homogeneous distributions gives the following two statements used in the main text: 

\begin{theorem}\label{res:rockland:ess-hom}
	Let \(G\) be a graded Lie group. An element \(u \in \Sigma^m(G)\) for \(m\in\RR\) has a left inverse in \(\Sigma^{-m}(G)\) if and only if \(u\) satisfies the Rockland condition (\Cref{def:rockland:general}).
	
	 If \(u_P\) is the distribution associated to a left-invariant differential operator \(P\) as in \cref{ex:diffop-as-hom-distr}, the Rockland condition for \(u_P\) holds if and only if \(P\) satisfies the Rockland condition for differential operators (\cref{def:rockland-cond-diff}).
\end{theorem}

\begin{proposition}\label{res:c-star-and-psi-star-algebra}
	The algebra \(\Sigma^0(G)\) has a \(C^*\)-closure denoted by \(C^*(\Sigma^0 G)\) and \(\Sigma^0(G)\) is closed under holomorphic functional calculus.
\end{proposition}
\subsection{Essentially homogeneous vs. homogeneous distributions}
Recall from \cref{def:ess_graded_group} the space of essentially \(m\)-homogeneous distributions and its quotient space
	\begin{align*}
		\ess^m(G)&=\{u\in\Algebra E'(G)+\Schwartz(G)\colon \singsupp(u)\subset\{0\}\text{ and }{\alpha_\lambda}_* u-\lambda^m u\in\Schwartz(G) \text{ for all $\lambda \in \Rp$}\},\\
		\Sigma^m(G)&= \ess^m(G)/\Schwartz(G).
	\end{align*}
	used in \cites{Tay84, vEY19}. Here, we use again the embedding of \(\Schwartz(G)\) into distributions by multiplying them with the Lebesgue measure, see \cref{rem:schwartz-embedding}.
Moreover, we set \(\ess(G)=\bigoplus_{m\in\RR}\ess^m(G)\) and \(\Sigma(G)=\bigoplus_{m\in\RR}\Sigma^m(G)\).
Recall that with respect to the group convolution of distributions on \(G\), one has \(\ess^\ell(G)*\ess^m(G)\subseteq\ess^{\ell+m}(G)\), \(\ess^m(G)*\Schwartz(G)\subseteq \Schwartz(G)\) and  \(\Schwartz(G)*\ess^m(G)\subseteq\Schwartz(G)\). Furthermore, \(\ess^m(G)\) and \(\Schwartz(G)\) are closed under involution. Extending linearly, one sees that \(\ess(G)\) and \(\Sigma(G)\) are \(^*\)-algebras.

\begin{example} 
	For \(G=\RR^n\), one typically applies Fourier transform to work with functions instead of distributions. Suppose \(a_\infty(\xi)\) is a \(m\)-homogeneous function on \(\RR^n\setminus\{0\}\), that is \(a_\infty(\lambda\xi)=\lambda^m a_\infty(\xi)\) for all \(\lambda>0\). Let \(\chi\in \Smooth(\RR^n)\) vanish near \(\xi=0\) and be constant \(1\) outside a compact set. Then \(a(\xi)=\chi(\xi)a_\infty(\xi)\) is smooth function on all of \(\RR^n\), but only essentially homogeneous. For some purposes it might be useful to work with \(a\), corresponding to \(\mathcal F^{-1}(a)=\widecheck a\in \ess^m(G)\), whereas for others it is beneficial to consider the homogeneous function at infinity \(a_\infty\), corresponding to an element of \(\Sigma^m(G)\).
	
	With regard to the orders, note that for \(\alpha_\lambda(v)=\lambda v\) and viewing \(a\) as a tempered distribution, one has \({\alpha_\lambda}_* a-\lambda^{-m-Q(\alpha)}a\in\Schwartz(\RR^n)\) for all \(\lambda>0\). Hence, under Fourier transform \({\alpha_\lambda}_*\widecheck a -\lambda^m \widecheck a\in\Schwartz(\RR^n)\) for all \(\lambda>0\).
\end{example}

Recall that a \(m\)-homogeneous function \(a_\infty\) on \(\RR^n\setminus\{0\}\) with respect to dilations \(\alpha\), that is \(\alpha^*_\lambda a_\infty =\lambda^m a_\infty\), extends to a unique homogeneous distribution on \(\RR^n\) when \(-m-Q(\alpha)\notin \NN_0\), see for example \cite{Hoe90a}*{Theorem~3.2.3}. For \(-m-Q(\alpha)\in\NN_0\) the extension is not unique and also not necessarily homogeneous, see \cite{Hoe90a}*{Theorem~3.2.4}. Due to this fact, under inverse Fourier transform \cite{CGGP92} have to work with the following spaces of distributions, see also \cite{Gel90}*{Proposition~1.2}:
\begin{definition}[\cite{CGGP92}*{p.~33}]
	For \(-m-Q(\alpha)\notin\NN_0\) let
	\begin{equation*}
		\mathbf{K}^m=\{u_\infty\in\Schwartz'(G)\colon \singsupp (u_\infty)\subseteq \{0\} \text{ and }{\alpha_\lambda}_*u_\infty=\lambda ^m u_\infty \text{ for all $\lambda \in \Rp$}\}.
	\end{equation*} 
	and for \(-m-Q(\alpha)\in\NN_0\)
	\begin{equation*}
		\mathbf{K}^m=\left\{u_\infty+p\log(\norm{\,\cdot\,})\in\Schwartz'(G)\colon \begin{array}{l}
			p\in\mathcal P^{-m-Q(\alpha)}\text{, }\singsupp (u_\infty)\subseteq \{0\}
			\\
			\text{ and }{\alpha_\lambda}_*u_\infty=\lambda ^m u_\infty\text{ for all }\lambda>0
		\end{array} \right\}  .
	\end{equation*} 
	Here, \(\mathcal P^k\) denotes for \(k\in\NN_0\) the space of all \(k\)-homogeneous polynomials with respect to \(\alpha\) on \(G\) and \(\norm{\,\cdot\,}\) denotes a homogeneous quasi-norm on \(G\).
\end{definition}
Note that the shift in the order as compared to the definition of \(\mathbf K^m\) in \cites{Gel90, CGGP92} stems from our different convention of homogeneity for distributions, which coincides with the one used in \cite{vEY19}.

Let \([u]\in \Sigma^m(G)\) be represented by an essentially homogeneous \(u\). Then by \cite{DH22}*{Lemma~3.6}  applied to the essentially homogeneous \(u\), one can write
\begin{equation}\label{eq:decomposition_ess_hom}
	u =u_\infty+f,
\end{equation}
where \(u_\infty\in \mathbf K^m\) and \(f\in\Smooth(G)\).
For \(-m-Q(\alpha)\notin\NN_0\), this decomposition is unique, whereas for  \(-m-Q(\alpha)\in\NN_0\) it is unique up to adding an element of \(\mathcal P^{-m-Q(\alpha)}\) to \(u_\infty\) and subtracting it from \(f\).

Let \(\Schwartz_0(G)\) denote the space of all \(f\in\Schwartz(G)\) such that \(\int_G p(x)f(x)=0\) for all polynomials \(p\). For \(G=\RR^n\), this space corresponds under Fourier transform to the space of Schwartz functions that vanish with all derivatives at \(0\). This is the space on which homogeneous functions naturally act as multiplier. 

By \cite{CGGP92}*{Proposition~2.2} convolution with \(v\in\mathbf K^m\) defines an operator \(\Op(v)\colon \Schwartz_0(G)\to\Schwartz_0 (G)\) with \(f\mapsto v*f\). Note that this operator uniquely determines \(v\) if \(-m-Q(\alpha)\notin\NN\), otherwise upto to a polynomial in \(\mathcal P^{-m-Q(\alpha)}\). Denote \(\Op(\mathbf K)=\bigoplus_{m\in\RR}\Op(\mathbf K^m)\) the space of all such operators, which is an algebra with respect to composition by \cite{CGGP92}*{Proposition~2.3}. It is also closed under formal adjoints and \(\Op(v)^*=\Op(v^*)\), see \cite{CGGP92}*{p.~41}. These are for general graded groups the equivalent of homogeneous functions on \(\RR^n\!\setminus\!\{0\}\) acting as multiplication operators on \(\mathcal F(\Schwartz_0(\RR^n))\).
\begin{lemma}\label{res:equivalence_elliptic_cggp}
	For a graded Lie group \(G\) the map \(\Phi_\infty\colon \Sigma(G)\to \Op(\mathbf K) \), defined by extending linearly \([u]\mapsto \Op(u_\infty)\) for \(u\in\ess^m(G)\), is well-defined and a \(^*\)-isomorphism. 
\end{lemma}
\begin{proof}
	The possible non-uniqueness of the decomposition in \eqref{eq:decomposition_ess_hom} causes no problems as \(\Op(p)=0\) for every polynomial~\(p\), hence the map is well-defined. Let \([ u]\in \Sigma^\ell(G)\) and \([ v]\in \Sigma^m(G)\) be represented by \(u =u_\infty+f\) and \(v=v_\infty +g\). We need to check \(\Op((u*v)_\infty)=\Op(u_\infty)\Op(v_\infty)\). By \cite{CGGP92}*{Proposition~2.3} one knows \(\Op(u_\infty)\Op(v_\infty)=\Op(w)\) for some \(w\in \mathbf K^{\ell+m}\). Let \(\chi\in\SmoothCompactSupp(G)\) be constant \(1\) near \(0\). Then by \cite{CGGP92}*{Proposition~3.3} one has \(w-\chi\cdot u_\infty*\chi\cdot v_\infty\in\Smooth(G)\). Note that \(u=\chi\cdot u+(1-\chi)u = \chi\cdot u_\infty+\tilde f\) with \(\tilde f=\chi\cdot f+(1-\chi)u\in\Schwartz(G)\) and similarly for \(v\). Consequently
	\begin{equation*}
		u*v\in \chi\cdot u_\infty*\chi\cdot v_\infty +\Schwartz(G).
	\end{equation*}
	Hence, we obtain that \((u*v)_\infty = w\). For the adjoints note that \((u^*)_\infty=(u_\infty)^*\). To see that the map is bijective, let \(\Psi_\infty(\Op(u_\infty))=[\chi\cdot u_\infty]\). This is well-defined as \(u_\infty\) is unique up to a polynomial and one checks that \(\Psi_\infty=\Phi^{-1}_\infty\).
\end{proof}

\begin{example}\label{ex:diffop-and_phi_infty}
	Suppose \(P\) is a \(m\)-homogeneous left-invariant differential operator on \(G\) and \(u_P\in\ess^m(G)\) the corresponding distribution defined by \(\langle u_P, f\rangle =Pf(0)\). Then \(\Phi_\infty([u_P])=\Op(u_P)\), where \(\Op(u_P)f=Pf\) for \(f\in\Schwartz_0(G)\). 
\end{example}
\subsection{The Rockland condition}\label{sec:rockland}
By \cref{res:equivalence_elliptic_cggp}, we see that \(u\in\Sigma^m(G)\) is invertible in \(\Sigma(G)\)
if and only if the operator \(\Op(u_\infty)\) is invertible in \(\Op(\mathbf K)\).
The latter can be characterized by a Rockland condition in terms of the representations of \(G\).

Let \(\pi\) be a non-trivial irreducible unitary representation of \(G\) on a Hilbert space \(\mathcal H_\pi\). Then one has the integrated representation \(\pi\colon\Schwartz(G)\to\Bound(\mathcal H_\pi)\) with \(f\mapsto\pi(f)\) given by
\[\pi(f)=\int_G f(x)\pi(x)\D x\in \mathbb B(\mathcal H_\pi).\]
The integrated representation \(\pi\) is a \(^*\)-homomorphism. Note that \(\pi(f)\) can be also viewed as the operator-valued Fourier transform of \(f\) at \(\pi\) often denoted by \(\widehat{f}(\pi)\).

As described in \cite{CGGP92}*{p.~36}, see also \cite{Pon08}*{Section~3.3.2}, one defines for \(\Op(u)\in\Op(\mathbf K^m)\) a possibly unbounded operator \(\pi(\Op(u))\) on \(\mathcal H_\pi\) with domain being the span of \(\pi(f)v\) for \(f\in\Schwartz_0(G)\) and \(v\in\mathcal H_\pi\) by
\[\pi(\Op(u))\pi(f)=\pi(u*f).\]
Then \(\pi(\Op(u))\) is densely defined and the adjoint of \(\pi(\Op(u))\) is \(\pi(\Op(u^*))\).
Therefore, \(\pi(\Op(u))\) is closable. The map \(\Op(u)\mapsto \overline{\pi(\Op(u))}\) yields a \(^*\)-representation of \(\Op(\mathbf K)\). The smooth vectors \(\mathcal H^\infty_\pi\) are contained in the domain of \(\overline{\pi(\Op(u))}\) by \cite{CGGP92}*{Lemma~6.3}. 

For \(P\in\lie{U}^m(\lie g)\), denote by \(u_P\in \mathbf K^m\) the corresponding distribution as in \cref{ex:diffop-as-hom-distr}. Then \(\D\pi(P)\) as defined in \cref{sec:diff-op-rockland} and \(\overline{\pi(\Op(u_P))}\) coincide on \(\mathcal H^\infty_\pi\), see \cite{Pon08}*{Remark~3.3.7}.
In particular, the following Rockland condition generalizes the one introduced by Rockland \cite{Roc78} for differential operators which was recalled in \cref{def:rockland-cond-diff}.
\begin{definition}[\cites{Roc78, CGGP92}] \label{def:rockland:general}
	The operator \(\Op(u)\in \Op(\mathbf K^m)\) satisfies the \emph{Rockland condition} on \(G\) if \(\overline{\pi(\Op(u))}\) is injective on the space of smooth vectors \(\mathcal H^\infty_\pi\) for all \(\pi\in\widehat{G}\setminus\{\pi_{\mathrm{triv}}\}\).
\end{definition}
\begin{example}
	For \(G=\RR^n\), every irreducible unitary representation is equivalent to a character \(\pi_\xi(x)=e^{i\langle \xi,x\rangle}\) for \(\xi\in\widehat\RR^n\), where \(\xi=0\) defines the trivial representation. For \(u\in\mathbf K^m\) and \(\xi\neq 0\), the operator \(\overline{\pi_\xi(\Op(u))}\) is multiplication by \(a_\infty(\xi)\), where \(a_\infty\) is the homogeneous function coinciding with \(\widehat u\) on \(\RR^n\!\setminus\!\{0\}\). Hence, the Rockland condition is the usual ellipticity condition in this case.
\end{example}
\begin{theorem}[\cite{CGGP92}*{Theorem~6.2}]\label{res:Rockland-theorem}
	An operator \(\Op(u)\in\Op(\mathbf K^m)\) has a left inverse if and only if \(\Op(u)\) satisfies the Rockland condition. Moreover, an operator \(u\in\Op(\mathbf K^m)\) is invertible if and only if \(\Op(u)\) and \(\Op(u^*)\) satisfy the Rockland condition.
\end{theorem}
 Hence, \cref{res:equivalence_elliptic_cggp} and the statements above give \cref{res:rockland:ess-hom}.
\begin{remark}\label{rem:rockland-matrix} \cref{res:Rockland-theorem} also holds for matrices of operators, see \cite{DH22}*{Lemma~3.9}. Namely \(\Op(u)=(\Op(u_{ij}))_{i,j=1}^k\in\mathbb M_k(\Op(\mathbf K^m))\) has a left inverse if and only if for all \(\pi\in\widehat{G}\setminus\{\pi_{\mathrm{triv}}\}\) the operator \(\left(\overline{\pi(\Op(u_{ij}))}\right)_{i,j=1}^k\) is injective on \((\mathcal H^\infty_\pi)^k\).
\end{remark}
\subsection[\texorpdfstring{The \(\Cst\)-algebra of order zero homogeneous distributions}{The C*-algebra of order zero homogeneous distributions}]{\texorpdfstring{The \boldmath\(\Cst\)-algebra of order zero homogeneous distributions}{The C*-algebra of order zero homogeneous distributions}}
It is well-known that zero homogeneous distributions extend to bounded operators.
\begin{lemma}
	Let \(v\in\mathbf K^0\). Then \(\Op(v)\colon \Schwartz_0(G)\to\Schwartz_0(G)\) extends to a bounded operator on \(L^2(G)\).
\end{lemma}
\begin{proof}
	By \cite{Chr88}*{Lemma~2.4} \(v\) can be written as \(v=c\delta_0+\mathrm{PV}(K)\) with \(c\in\CC\) and \(\mathrm{PV}(K)\) denoting a principal-value distribution. These extend to bounded operators on \(L^2(G)\), see \cite{KS71}*{Theorem~1} or \cite{Goo80}*{Theorem~2.1}.
\end{proof}
In particular, one can consider the \(C^*\)-completion of the algebra \(\Op(\mathbf K^0)\) in \(\Bound(L^2G)\).
	We remark that \(\Cst(\Op(\mathbf K^0))\) is the generalized fixed point algebra \(\mathrm{Fix}^{\Rp}(J_G)\) considered in \cite{Ewe23a}, where its spectrum and \(\mathrm K\)-theory are computed.

Recall the following notion of a \(\Psi^*\)-algebra by Gramsch \cite{Gra84}*{Definition~5.1}. Compiling known facts from the literature, we shall see in the following that \(\Op(\mathbf K^0)\) is a \(\Psi^*\)\nb-algebra.
\begin{definition}
	A subalgebra \(\Algebra A\) of a \(\Cst\)-algebra \(B\) with unit \(1\in B\) is called a \emph{\(\Psi^*\)\nb-algebra of \(B\)} if
	\begin{enumerate}
		\item \(1\in \Algebra A\) and \(\mathcal A^*=\mathcal A\),		
		\item \(\mathcal A\) is spectrally invariant, that is, \(\mathcal A^{-1}=B^{-1}\cap \mathcal A\),
		\item\label{item:Frechet} \(\mathcal A\) can be equipped with a Fr\'echet algebra structure such that \(\mathcal A\hookrightarrow B\) is continuous.
	\end{enumerate}
\end{definition}
One can show, see \cite{Gra84}*{Remark~5.2},  that in this case \(a\mapsto a^{-1}\) is a continuous map \(\mathcal A\to\mathcal A\) and that therefore \(\Psi^*\)-algebras are closed under holomorphic functional calculus. 
\begin{proposition}\label{res:psi-star-algebra}
	The algebra \(\Op(\mathbf K^0)\) is a \(\Psi^*\)-algebra of \(\Cst(\Op(\mathbf K^0))\). In particular, it is closed under holomorphic functional calculus.
\end{proposition}
\begin{proof}
	It is clear that \(\Op(\mathbf K^0)\) contains the identity operator and that is closed under the involution. Spectral invariance is also known, see for example \cite{CGGP92}*{Theorem~5.1}. For the Fr\'echet algebra structure, we use that
	Fermanian-Kammerer and Fischer define in \cite{FKF20}*{Definition~5.1} an algebra of \(0\)-homogeneous symbols \(\widetilde{S}^0\subset L^\infty(\widehat G, \Bound(\mathcal H_\rho))\). Here, \(\widehat G\) is equipped with the Plancherel measure. Under group Fourier transform this algebra is isomorphic to \(\Op(\mathbf K^0)\). 
	They equip \(\widetilde{S}^0\) with Fr\'echet seminorms.
	Arguing similar to the proof of \cite{FR16}*{Theorem~5.2.22} one can show that \(\widetilde{S}^0\) is a Fr\'echet algebra with respect to these seminorms. One seminorm is given by \(\sup_\pi\norm{\widehat v(\pi)}\) for \(v\in\mathbf K^0\). By the Plancherel theorem this is precisely the norm of \(v\) as a convolution operator on \(L^2(G)\). Hence, \refitem{item:Frechet} is satisfied. 
\end{proof}
\end{appendix}

\bibliographystyle{alpha} 
\bibliography{references.bib}

\begin{bibdiv}
\begin{biblist}

\bib{AMY21}{article}{
      author={Androulidakis, Iakovos},
      author={Mohsen, Omar},
      author={Yuncken, Robert},
       title={The convolution algebra of {S}chwartz kernels along a singular
  foliation},
        date={2021},
        ISSN={0379-4024,1841-7744},
     journal={J. Operator Theory},
      volume={85},
      number={2},
       pages={475\ndash 503},
         url={https://doi.org/10.7900/jot},
      review={\MR{4249560}},
}

\bib{AMY22}{article}{
      author={Androulidakis, Iakovos},
      author={Mohsen, Omar},
      author={Yuncken, Robert},
       title={A pseudodifferential calculus for maximally hypoelliptic
  operators and the {H}elffer-{N}ourrigat conjecture},
        date={2022},
     journal={arXiv preprint arXiv:2201.12060},
}

\bib{AS68}{article}{
      author={Atiyah, M.~F.},
      author={Singer, I.~M.},
       title={The index of elliptic operators. {I}},
        date={1968},
        ISSN={0003-486X},
     journal={Ann. of Math.},
      volume={87},
      number={3},
       pages={484\ndash 530},
         url={https://doi.org/10.2307/1970715},
      review={\MR{236950}},
}

\bib{BC22}{article}{
      author={Bruno, Tommaso},
      author={Calzi, Mattia},
       title={Schr\"{o}dinger operators on {L}ie groups with purely discrete
  spectrum},
        date={2022},
        ISSN={0001-8708,1090-2082},
     journal={Adv. Math.},
      volume={404},
       pages={Paper No. 108444, 45},
         url={https://doi.org/10.1016/j.aim.2022.108444},
      review={\MR{4418885}},
}

\bib{Bea77}{incollection}{
      author={Beals, R.},
       title={Op\'{e}rateurs invariants hypoelliptiques sur un groupe de {L}ie
  nilpotent},
        date={1977},
   booktitle={S\'{e}minaire {G}oulaouic-{S}chwartz 1976--1977: \'{E}quations
  aux d\'{e}riv\'{e}es partielles et analyse fonctionnelle},
   publisher={\'{E}cole Polytech., Palaiseau},
       pages={Exp. No. 19, 8},
      review={\MR{494313}},
}

\bib{BG88}{book}{
      author={Beals, Richard},
      author={Greiner, Peter},
       title={Calculus on {H}eisenberg manifolds},
      series={Annals of Mathematics Studies},
   publisher={Princeton University Press, Princeton, NJ},
        date={1988},
      volume={119},
        ISBN={0-691-08500-5; 0-691-08501-3},
         url={https://doi.org/10.1515/9781400882397},
      review={\MR{953082}},
}

\bib{BN03}{article}{
      author={Boggiatto, Paolo},
      author={Nicola, Fabio},
       title={Non-commutative residues for anisotropic pseudo-differential
  operators in {$\mathbb R^n$}},
        date={2003},
        ISSN={0022-1236,1096-0783},
     journal={J. Funct. Anal.},
      volume={203},
      number={2},
       pages={305\ndash 320},
         url={https://doi.org/10.1016/S0022-1236(03)00194-0},
      review={\MR{2003350}},
}

\bib{CDR21}{article}{
      author={Chatzakou, Marianna},
      author={Delgado, Julio},
      author={Ruzhansky, Michael},
       title={On a class of anharmonic oscillators},
        date={2021},
        ISSN={0021-7824,1776-3371},
     journal={J. Math. Pures Appl.},
      volume={153},
       pages={1\ndash 29},
         url={https://doi.org/10.1016/j.matpur.2021.07.006},
      review={\MR{4299820}},
}

\bib{CG90}{book}{
      author={Corwin, Lawrence~J.},
      author={Greenleaf, Frederick~P.},
       title={Representations of nilpotent {L}ie groups and their applications.
  {P}art {I}},
      series={Cambridge Studies in Advanced Mathematics},
   publisher={Cambridge University Press, Cambridge},
        date={1990},
      volume={18},
        ISBN={0-521-36034-X},
        note={Basic theory and examples},
      review={\MR{1070979}},
}

\bib{CGGP92}{article}{
      author={Christ, Michael},
      author={Geller, Daryl},
      author={{G\l{}owacki}, Pawe\l},
      author={Polin, Larry},
       title={Pseudodifferential operators on groups with dilations},
        date={1992},
        ISSN={0012-7094,1547-7398},
     journal={Duke Math. J.},
      volume={68},
      number={1},
       pages={31\ndash 65},
         url={https://doi.org/10.1215/S0012-7094-92-06802-5},
      review={\MR{1185817}},
}

\bib{Chr88}{article}{
      author={Christ, Michael},
       title={On the regularity of inverses of singular integral operators},
        date={1988},
        ISSN={0012-7094,1547-7398},
     journal={Duke Math. J.},
      volume={57},
      number={2},
       pages={459\ndash 484},
         url={https://doi.org/10.1215/S0012-7094-88-05721-3},
      review={\MR{962516}},
}

\bib{Con94}{book}{
      author={Connes, Alain},
       title={Noncommutative geometry},
   publisher={Academic Press, Inc., San Diego, CA},
        date={1994},
        ISBN={0-12-185860-X},
      review={\MR{1303779}},
}

\bib{Cor95}{book}{
      author={Cordes, H.~O.},
       title={The technique of pseudodifferential operators},
      series={London Mathematical Society Lecture Note Series},
   publisher={Cambridge University Press},
        date={1995},
}

\bib{CP19c}{article}{
      author={Choi, Woocheol},
      author={Ponge, Rapha\"{e}l},
       title={Tangent maps and tangent groupoid for {C}arnot manifolds},
        date={2019},
        ISSN={0926-2245,1872-6984},
     journal={Differential Geom. Appl.},
      volume={62},
       pages={136\ndash 183},
         url={https://doi.org/10.1016/j.difgeo.2018.11.002},
      review={\MR{3880751}},
}

\bib{DH20}{article}{
      author={Dave, Shantanu},
      author={Haller, Stefan},
       title={The heat asymptotics on filtered manifolds},
        date={2020},
        ISSN={1050-6926,1559-002X},
     journal={J. Geom. Anal.},
      volume={30},
      number={1},
       pages={337\ndash 389},
         url={https://doi.org/10.1007/s12220-018-00137-4},
      review={\MR{4058517}},
}

\bib{DH22}{article}{
      author={Dave, Shantanu},
      author={Haller, Stefan},
       title={Graded hypoellipticity of {BGG} sequences},
        date={2022},
        ISSN={0232-704X,1572-9060},
     journal={Ann. Global Anal. Geom.},
      volume={62},
      number={4},
       pages={721\ndash 789},
         url={https://doi.org/10.1007/s10455-022-09870-0},
      review={\MR{4493251}},
}

\bib{DS14}{article}{
      author={Debord, Claire},
      author={Skandalis, Georges},
       title={Adiabatic groupoid, crossed product by $\mathbb{R}_+^*$ and
  pseudodifferential calculus},
        date={2014},
        ISSN={0001-8708,1090-2082},
     journal={Adv. Math.},
      volume={257},
       pages={66\ndash 91},
         url={https://doi.org/10.1016/j.aim.2014.02.012},
      review={\MR{3187645}},
}

\bib{Dyn75}{article}{
      author={Dynin, A.~S.},
       title={Pseudodifferential operators on the {H}eisenberg group},
        date={1975},
        ISSN={0002-3264},
     journal={Dokl. Akad. Nauk SSSR},
      volume={225},
      number={6},
       pages={1245\ndash 1248},
      review={\MR{423431}},
}

\bib{ENN88}{article}{
      author={Elliott, G.~A.},
      author={Natsume, T.},
      author={Nest, R.},
       title={Cyclic cohomology for one-parameter smooth crossed products},
        date={1988},
        ISSN={0001-5962,1871-2509},
     journal={Acta Math.},
      volume={160},
      number={3-4},
       pages={285\ndash 305},
         url={https://doi.org/10.1007/BF02392278},
      review={\MR{945014}},
}

\bib{ENN96}{article}{
      author={Elliott, George~A.},
      author={Natsume, Toshikazu},
      author={Nest, Ryszard},
       title={The {A}tiyah-{S}inger index theorem as passage to the classical
  limit in quantum mechanics},
        date={1996},
        ISSN={0010-3616,1432-0916},
     journal={Comm. Math. Phys.},
      volume={182},
      number={3},
       pages={505\ndash 533},
         url={http://projecteuclid.org/euclid.cmp/1104288300},
      review={\MR{1461941}},
}

\bib{Ewe23a}{article}{
      author={Ewert, Eske},
       title={Pseudo-differential extension for graded nilpotent {L}ie groups},
        date={2023},
     journal={Doc. Math.},
      volume={28},
      number={6},
       pages={1323\ndash 1379},
}

\bib{Fed70}{article}{
      author={Fedosov, Boris~Vasil'evich},
       title={Direct proof of the formula for the index of an elliptic system
  in {E}uclidean space},
        date={1970},
     journal={Functional Analysis and Its Applications},
      volume={4},
      number={4},
       pages={339\ndash 341},
}

\bib{FKF20}{article}{
      author={Fermanian-Kammerer, Clotilde},
      author={Fischer, V\'{e}ronique},
       title={Defect measures on graded {L}ie groups},
        date={2020},
        ISSN={0391-173X,2036-2145},
     journal={Ann. Sc. Norm. Super. Pisa Cl. Sci.},
      volume={21},
       pages={207\ndash 291},
      review={\MR{4288603}},
}

\bib{Fol94}{incollection}{
      author={Folland, Gerald~B.},
       title={Meta-{H}eisenberg groups},
        date={1994},
   booktitle={Fourier analysis ({O}rono, {ME}, 1992)},
      series={Lecture Notes in Pure and Appl. Math.},
      volume={157},
   publisher={Dekker, New York},
       pages={121\ndash 147},
      review={\MR{1277821}},
}

\bib{FR16}{book}{
      author={Fischer, V\'{e}ronique},
      author={Ruzhansky, Michael},
       title={Quantization on nilpotent {L}ie groups},
      series={Progress in Mathematics},
   publisher={Birkh\"{a}user Cham},
        date={2016},
      volume={314},
        ISBN={978-3-319-29557-2; 978-3-319-29558-9},
         url={https://doi.org/10.1007/978-3-319-29558-9},
      review={\MR{3469687}},
}

\bib{FS74}{article}{
      author={Folland, G.~B.},
      author={Stein, E.~M.},
       title={Estimates for the {$\bar \partial \sb{b}$} complex and analysis
  on the {H}eisenberg group},
        date={1974},
        ISSN={0010-3640,1097-0312},
     journal={Comm. Pure Appl. Math.},
      volume={27},
       pages={429\ndash 522},
         url={https://doi.org/10.1002/cpa.3160270403},
      review={\MR{367477}},
}

\bib{FS82}{book}{
      author={Folland, G.~B.},
      author={Stein, Elias~M.},
       title={Hardy spaces on homogeneous groups},
      series={Mathematical Notes},
   publisher={Princeton University Press, Princeton, NJ; University of Tokyo
  Press, Tokyo},
        date={1982},
      volume={28},
        ISBN={0-691-08310-X},
      review={\MR{657581}},
}

\bib{Gar04}{misc}{
      author={Garrett, Paul},
       title={Weil-{S}chwartz envelopes for rapidly decreasing functions},
        date={2004},
  url={https://www-users.cse.umn.edu/~garrett/m/fun/weil_schwartz_envelope.pdf},
}

\bib{Gel90}{book}{
      author={Geller, Daryl},
       title={Analytic pseudodifferential operators for the {H}eisenberg group
  and local solvability},
      series={Mathematical Notes},
   publisher={Princeton University Press, Princeton, NJ},
        date={1990},
      volume={37},
        ISBN={0-691-08564-1},
         url={https://doi.org/10.1515/9781400860739},
      review={\MR{1030277}},
}

\bib{GK22}{unpublished}{
      author={Goffeng, Magnus},
      author={Kuzmin, Alexey},
       title={Index theory of hypoelliptic operators on {C}arnot manifolds},
        date={2022},
}

\bib{Glo91}{article}{
      author={G{\l}owacki, Pawe\l},
       title={The {R}ockland condition for nondifferential convolution
  operators. {II}},
        date={1991},
        ISSN={0039-3223,1730-6337},
     journal={Studia Math.},
      volume={98},
      number={2},
       pages={99\ndash 114},
         url={https://doi.org/10.4064/sm-98-2-99-114},
      review={\MR{1100916}},
}

\bib{Gof24}{unpublished}{
      author={Goffeng, Magnus},
       title={Solving the index problem for (curved)
  {B}ernstein-{G}elfand-{G}elfand sequences},
        date={2024},
}

\bib{Goo80}{article}{
      author={Goodman, Roe},
       title={Singular integral operators on nilpotent {L}ie groups},
        date={1980},
        ISSN={0004-2080,1871-2487},
     journal={Ark. Mat.},
      volume={18},
      number={1},
       pages={1\ndash 11},
         url={https://doi.org/10.1007/BF02384677},
      review={\MR{608323}},
}

\bib{Gra84}{article}{
      author={Gramsch, Bernhard},
       title={{R}elative {I}nversion in der {S}törungstheorie von {O}peratoren
  und {$\Psi$}-{A}lgebren.},
        date={1984},
     journal={Mathematische Annalen},
      volume={269},
       pages={27\ndash 72},
         url={http://eudml.org/doc/163920},
}

\bib{Gro55}{article}{
      author={Grothendieck, Alexandre},
       title={Produits tensoriels topologiques et espaces nucl\'{e}aires},
        date={1955},
        ISSN={0065-9266,1947-6221},
     journal={Mem. Amer. Math. Soc.},
      volume={16},
       pages={Chapter 1: 196 pp.; Chapter 2: 140},
      review={\MR{75539}},
}

\bib{Hoe85c}{book}{
      author={H\"{o}rmander, Lars},
       title={The analysis of linear partial differential operators. {III}},
      series={Grundlehren der mathematischen Wissenschaften [Fundamental
  Principles of Mathematical Sciences]},
   publisher={Springer-Verlag, Berlin},
        date={1985},
      volume={274},
        ISBN={3-540-13828-5},
      review={\MR{781536}},
}

\bib{Hoe90a}{book}{
      author={H\"{o}rmander, Lars},
       title={The analysis of linear partial differential operators. {I}},
     edition={Second ed.},
      series={Grundlehren der mathematischen Wissenschaften [Fundamental
  Principles of Mathematical Sciences]},
   publisher={Springer-Verlag, Berlin},
        date={1990},
      volume={256},
        ISBN={3-540-52345-6},
         url={https://doi.org/10.1007/978-3-642-61497-2},
      review={\MR{1065993}},
}

\bib{Hel79}{book}{
      author={Helgason, S.},
       title={Differential geometry, {L}ie groups, and symmetric spaces},
      series={Pure and Applied Mathematics},
   publisher={Academic Press},
        date={1979},
        ISBN={9780080873961},
         url={https://books.google.de/books?id=DWGvsa6bcuMC},
}

\bib{Hel84}{book}{
      author={Helffer, Bernard},
       title={Th\'{e}orie spectrale pour des op\'{e}rateurs globalement
  elliptiques},
      series={Ast\'{e}risque},
   publisher={Soci\'{e}t\'{e} Math\'{e}matique de France, Paris},
        date={1984},
      volume={112},
        note={With an English summary},
      review={\MR{743094}},
}

\bib{HN78}{article}{
      author={Helffer, B.},
      author={Nourrigat, F.},
       title={Hypoellipticit\'{e} pour des groupes nilpotents de rang de
  nilpotence {$3$}},
        date={1978},
        ISSN={0360-5302,1532-4133},
     journal={Comm. Partial Differential Equations},
      volume={3},
      number={8},
       pages={643\ndash 743},
         url={https://doi.org/10.1080/03605307808820075},
      review={\MR{499352}},
}

\bib{HN79}{article}{
      author={Helffer, B.},
      author={Nourrigat, J.},
       title={Caracterisation des op\'{e}rateurs hypoelliptiques homog\`enes
  invariants \`a gauche sur un groupe de {L}ie nilpotent gradu\'{e}},
        date={1979},
        ISSN={0360-5302,1532-4133},
     journal={Comm. Partial Differential Equations},
      volume={4},
      number={8},
       pages={899\ndash 958},
         url={https://doi.org/10.1080/03605307908820115},
      review={\MR{537467}},
}

\bib{Hor66}{book}{
      author={Horv\'{a}th, John},
       title={Topological vector spaces and distributions. {V}ol. {I}},
   publisher={Addison-Wesley Publishing Co., Reading, Mass.-London-Don Mills,
  Ont.},
        date={1966},
      review={\MR{205028}},
}

\bib{HR82}{article}{
      author={Helffer, Bernard},
      author={Robert, Didier},
       title={Propri\'{e}t\'{e}s asymptotiques du spectre d'op\'{e}rateurs
  pseudodiff\'{e}rentiels sur {${\bf R}\sp{n}$}},
        date={1982},
        ISSN={0360-5302,1532-4133},
     journal={Comm. Partial Differential Equations},
      volume={7},
      number={7},
       pages={795\ndash 882},
         url={https://doi.org/10.1080/03605308208820239},
      review={\MR{662451}},
}

\bib{SH18}{article}{
      author={Haj Saeedi~Sadegh, Ahmad~Reza},
      author={Higson, Nigel},
       title={Euler-like vector fields, deformation spaces and manifolds with
  filtered structure},
        date={2018},
        ISSN={1431-0635,1431-0643},
     journal={Doc. Math.},
      volume={23},
       pages={293\ndash 325},
      review={\MR{3846057}},
}

\bib{Kum81}{book}{
      author={Kumano-Go, Hitoshi},
       title={Pseudodifferential operators},
   publisher={MIT Press, Cambridge, Mass.-London},
        date={1981},
        ISBN={0-262-11080-6},
        note={Translated from the Japanese by the author, R\'{e}mi Vaillancourt
  and Michihiro Nagase},
      review={\MR{666870}},
}

\bib{KN65}{article}{
      author={Kohn, J.~J.},
      author={Nirenberg, L.},
       title={An algebra of pseudo-differential operators},
        date={1965},
        ISSN={0010-3640,1097-0312},
     journal={Comm. Pure Appl. Math.},
      volume={18},
       pages={269\ndash 305},
         url={https://doi.org/10.1002/cpa.3160180121},
      review={\MR{176362}},
}

\bib{KS71}{article}{
      author={Knapp, A.~W.},
      author={Stein, E.~M.},
       title={Intertwining operators for semisimple groups},
        date={1971},
        ISSN={0003-486X},
     journal={Ann. of Math. (2)},
      volume={93},
       pages={489\ndash 578},
         url={https://doi.org/10.2307/1970887},
      review={\MR{460543}},
}

\bib{Lar13}{article}{
      author={Larcher, Julian},
       title={Multiplications and convolutions in {L}. {S}chwartz' spaces of
  test functions and distributions and their continuity},
        date={2013},
        ISSN={0174-4747,2196-6753},
     journal={Analysis},
      volume={33},
      number={4},
       pages={319\ndash 332},
         url={https://doi.org/10.1524/anly.2013.1200},
      review={\MR{3143770}},
}

\bib{LMV17}{article}{
      author={Lescure, Jean-Marie},
      author={Manchon, Dominique},
      author={Vassout, St\'{e}phane},
       title={About the convolution of distributions on groupoids},
        date={2017},
        ISSN={1661-6952,1661-6960},
     journal={J. Noncommut. Geom.},
      volume={11},
      number={2},
       pages={757\ndash 789},
         url={https://doi.org/10.4171/JNCG/11-2-10},
      review={\MR{3669118}},
}

\bib{LR01}{incollection}{
      author={Landsman, N.~P.},
      author={Ramazan, B.},
       title={Quantization of {P}oisson algebras associated to {L}ie
  algebroids},
        date={2001},
   booktitle={Groupoids in analysis, geometry, and physics ({B}oulder, {CO},
  1999)},
      series={Contemp. Math.},
      volume={282},
   publisher={Amer. Math. Soc., Providence, RI},
       pages={159\ndash 192},
         url={https://doi.org/10.1090/conm/282/04685},
      review={\MR{1855249}},
}

\bib{Mel82}{unpublished}{
      author={Melin, Anders},
       title={Lie filtrations and pseudo-differential operators},
        date={1982},
        note={unpublished},
}

\bib{Mel94}{incollection}{
      author={Melrose, R.~B.},
       title={Spectral and scattering theory for the {L}aplacian on
  asymptotically {E}uclidean spaces},
        date={1994},
   booktitle={Spectral and scattering theory},
      editor={Ikawa, M.},
       pages={85\ndash 130},
}

\bib{Moh21}{article}{
      author={Mohsen, Omar},
       title={On the deformation groupoid of the inhomogeneous
  pseudo-differential calculus},
        date={2021},
        ISSN={0024-6093,1469-2120},
     journal={Bull. Lond. Math. Soc.},
      volume={53},
      number={2},
       pages={575\ndash 592},
         url={https://doi.org/10.1112/blms.12443},
      review={\MR{4239197}},
}

\bib{Moh22}{article}{
      author={Mohsen, Omar},
       title={Index theorem for inhomogeneous hypoelliptic differential
  operators},
        date={2022},
        ISSN={1867-5778,1867-5786},
     journal={M\"{u}nster J. Math.},
      volume={15},
      number={2},
       pages={305\ndash 331},
      review={\MR{4541286}},
}

\bib{Moh22u}{unpublished}{
      author={Mohsen, Omar},
       title={On the index of maximally hypoelliptic differential operators},
   publisher={arXiv preprint arXiv:2201.13049},
        date={2022},
}

\bib{Nie83}{book}{
      author={Nielsen, Ole~A.},
       title={Unitary representations and coadjoint orbits of low-dimensional
  nilpotent {L}ie groups},
      series={Queen's Papers in Pure and Applied Mathematics},
   publisher={Queen's University, Kingston, ON},
        date={1983},
      volume={63},
      review={\MR{773296}},
}

\bib{NR10}{book}{
      author={Nicola, Fabio},
      author={Rodino, Luigi},
       title={Global pseudo-differential calculus on {E}uclidean spaces},
      series={Pseudo-Differential Operators. Theory and Applications},
   publisher={Birkh\"{a}user Verlag, Basel},
        date={2010},
      volume={4},
        ISBN={978-3-7643-8511-8},
         url={https://doi.org/10.1007/978-3-7643-8512-5},
      review={\MR{2668420}},
}

\bib{Par72}{article}{
      author={Parenti, Cesare},
       title={Operatori pseudo-differenziali in {$R^n$} e applicazioni},
        date={1972},
     journal={Annali di Matematica Pura ed Applicata},
      volume={93},
       pages={359\ndash 389},
         url={https://doi.org/10.1007/BF02412028},
}

\bib{Pon08}{article}{
      author={Ponge, Rapha\"{e}l~S.},
       title={Heisenberg calculus and spectral theory of hypoelliptic operators
  on {H}eisenberg manifolds},
        date={2008},
        ISSN={0065-9266,1947-6221},
     journal={Mem. Amer. Math. Soc.},
      volume={194},
      number={906},
       pages={viii+ 134},
         url={https://doi.org/10.1090/memo/0906},
      review={\MR{2417549}},
}

\bib{Ren80}{book}{
      author={Renault, Jean},
       title={A groupoid approach to {$C\sp{\ast} $}-algebras},
      series={Lecture Notes in Mathematics},
   publisher={Springer, Berlin},
        date={1980},
      volume={793},
        ISBN={3-540-09977-8},
      review={\MR{584266}},
}

\bib{Roc78}{article}{
      author={Rockland, Charles},
       title={Hypoellipticity on the {H}eisenberg
  group-representation-theoretic criteria},
        date={1978},
        ISSN={0002-9947,1088-6850},
     journal={Trans. Amer. Math. Soc.},
      volume={240},
       pages={1\ndash 52},
         url={https://doi.org/10.2307/1998805},
      review={\MR{486314}},
}

\bib{RR20}{article}{
      author={Rottensteiner, David},
      author={Ruzhansky, Michael},
       title={The harmonic oscillator on the {H}eisenberg group},
        date={2020},
        ISSN={1631-073X,1778-3569},
     journal={C. R. Math. Acad. Sci. Paris},
      volume={358},
      number={5},
       pages={609\ndash 614},
         url={https://doi.org/10.5802/crmath.78},
      review={\MR{4149860}},
}

\bib{RR22}{article}{
      author={Rottensteiner, David},
      author={Ruzhansky, Michael},
       title={Harmonic and anharmonic oscillators on the {H}eisenberg group},
        date={2022},
        ISSN={0022-2488,1089-7658},
     journal={J. Math. Phys.},
      volume={63},
      number={11},
       pages={Paper No. 111509, 23},
         url={https://doi.org/10.1063/5.0106068},
      review={\MR{4503714}},
}

\bib{RS76}{article}{
      author={Rothschild, Linda~Preiss},
      author={Stein, E.~M.},
       title={Hypoelliptic differential operators and nilpotent groups},
        date={1976},
        ISSN={0001-5962,1871-2509},
     journal={Acta Math.},
      volume={137},
      number={3-4},
       pages={247\ndash 320},
         url={https://doi.org/10.1007/BF02392419},
      review={\MR{436223}},
}

\bib{Schw54}{article}{
      author={Schwartz, Laurent},
       title={Espaces de fonctions diff\'{e}rentiables \`a valeurs
  vectorielles},
        date={1954/55},
        ISSN={0021-7670,1565-8538},
     journal={J. Analyse Math.},
      volume={4},
       pages={88\ndash 148},
         url={https://doi.org/10.1007/BF02787718},
      review={\MR{80268}},
}

\bib{Schw66}{book}{
      author={Schwartz, Laurent},
       title={Th\'{e}orie des distributions},
      series={Publications de l'Institut de Math\'{e}matique de
  l'Universit\'{e} de Strasbourg},
   publisher={Hermann, Paris},
        date={1966},
      volume={IX-X},
        note={Nouvelle \'{e}dition, enti\'{e}rement corrig\'{e}e, refondue et
  augment\'{e}e},
      review={\MR{209834}},
}

\bib{Sch87}{incollection}{
      author={Schrohe, Elmar},
       title={Spaces of weighted symbols and weighted {S}obolev spaces on
  manifolds},
        date={1987},
   booktitle={{P}seudo-{D}ifferential {O}perators: {P}roceedings of a
  {C}onference held in {O}berwolfach, {F}ebruary 2--8, 1986},
      editor={Cordes, Heinz~O.},
      editor={Gramsch, Bernhard},
      editor={Widom, Harold},
   publisher={Springer Berlin Heidelberg},
     address={Berlin, Heidelberg},
       pages={360\ndash 377},
         url={https://doi.org/10.1007/BFb0077751},
}

\bib{Sch93}{article}{
      author={Schweitzer, Larry~B.},
       title={Dense {$m$}-convex {F}r\'{e}chet subalgebras of operator algebra
  crossed products by {L}ie groups},
        date={1993},
        ISSN={0129-167X,1793-6519},
     journal={Internat. J. Math.},
      volume={4},
      number={4},
       pages={601\ndash 673},
         url={https://doi.org/10.1142/S0129167X93000315},
      review={\MR{1232986}},
}

\bib{Shu87}{book}{
      author={Shubin, M.~A.},
       title={Pseudodifferential operators and spectral theory},
      series={Springer Series in Soviet Mathematics},
   publisher={Springer-Verlag, Berlin},
        date={1987},
        ISBN={3-540-13621-5},
         url={https://doi.org/10.1007/978-3-642-96854-9},
        note={Translated from the Russian by Stig I. Andersson},
      review={\MR{883081}},
}

\bib{Tay84}{article}{
      author={Taylor, Michael~E.},
       title={Noncommutative microlocal analysis. {I}.},
        date={1984},
        ISSN={0065-9266,1947-6221},
     journal={Mem. Amer. Math. Soc.},
      volume={52},
      number={313},
       pages={iv+182},
         url={https://doi.org/10.1090/memo/0313},
      review={\MR{764508}},
}

\bib{Tre67}{book}{
      author={Tr\`eves, Fran\c{c}ois},
       title={Topological vector spaces, distributions and kernels},
   publisher={Academic Press, New York-London},
        date={1967},
      review={\MR{225131}},
}

\bib{vEY17}{article}{
      author={van Erp, Erik},
      author={Yuncken, Robert},
       title={On the tangent groupoid of a filtered manifold},
        date={2017},
        ISSN={0024-6093,1469-2120},
     journal={Bull. Lond. Math. Soc.},
      volume={49},
      number={6},
       pages={1000\ndash 1012},
         url={https://doi.org/10.1112/blms.12096},
      review={\MR{3743484}},
}

\bib{vEY19}{article}{
      author={van Erp, Erik},
      author={Yuncken, Robert},
       title={A groupoid approach to pseudodifferential calculi},
        date={2019},
        ISSN={0075-4102,1435-5345},
     journal={J. Reine Angew. Math.},
      volume={756},
       pages={151\ndash 182},
         url={https://doi.org/10.1515/crelle-2017-0035},
      review={\MR{4026451}},
}

\end{biblist}
\end{bibdiv}

\end{document}